\documentclass[11pt,reqno,oneside]{amsart}

\usepackage[margin=1in]{geometry}
\usepackage[reqno]{amsmath}
\usepackage{mathbbol}
\usepackage{amssymb}
\usepackage[thmtools-compat]{keytheorems}
\usepackage{mathrsfs}
\usepackage{tikz,tikz-cd,graphicx,caption,subcaption,bm,cancel}
\usetikzlibrary{hobby}
\usetikzlibrary{positioning}

\usepackage[symbol*]{footmisc}
\usepackage{import}
\usepackage{xcolor}

\usepackage[T1]{fontenc}
\usepackage{lmodern}
\usepackage{overpic}

\usepackage{enumitem}
\setlist[enumerate,1]{label=\normalfont(\roman*)}

\declaretheorem[numberwithin=section]{theorem}
\declaretheorem[sibling=theorem]{proposition}
\declaretheorem[sibling=theorem]{corollary}
\declaretheorem[sibling=theorem]{lemma}
\declaretheorem[sibling=theorem]{conjecture}
\declaretheorem[sibling=theorem,style=definition]{definition}

\declaretheorem[numbered=no,style=remark]{remark}
\declaretheorem[numbered=no,style=remark]{remarks}

\declaretheorem{claim}
\declaretheorem[numbered=no, name=Claim]{claim*}

\declaretheorem[name=Theorem]{thmintro}

\numberwithin{equation}{section}

\usepackage{hyperref}

\usepackage{cleveref}

\crefname{thmintro}{theorem}{theorems}
\Crefname{thmintro}{Theorem}{Theorems}

\crefname{conjecture}{conjecture}{conjectures}
\Crefname{conjecture}{Conjecture}{Conjectures}

\DeclareMathOperator{\id}{id}

\DeclareMathOperator{\Rat}{Rat}
\DeclareMathOperator{\ARat}{\overline{Rat}}

\newcommand{\cA}{\mathcal{A}}

\newcommand{\cF}{\mathcal{F}}
\newcommand{\cG}{\mathcal{G}}
\newcommand{\cN}{\mathcal{N}}

\newcommand{\cS}{\mathcal{S}}
\newcommand{\sT}{\mathscr{T}}
\newcommand{\sH}{\mathscr{H}}

\newcommand{\C}{\mathbb{C}}
\newcommand{\D}{\mathbb{D}}
\newcommand{\N}{\mathbb{N}}
\newcommand{\R}{\mathbb{R}}
\newcommand{\T}{\mathbb{T}}
\newcommand{\GGamma}{\mathbb{\Gamma}}
\newcommand{\Z}{\mathbb{Z}}
\newcommand{\CC}{\mathbb{\hat{C}}}
\newcommand{\cC}{\mathcal{C}}
\newcommand{\cO}{\mathcal{O}}
\newcommand{\cT}{\mathcal{T}}
\newcommand{\TT}{\mathsf{T}}
\newcommand{\QC}{QC_0}
\newcommand{\Homeo}{\operatorname{Homeo}}

\newcommand{\Mod}{\operatorname{Mod}}
\newcommand{\Mob}{\textnormal{Möb}}
\newcommand{\Ico}{\mathcal{I}}

\newif\ifcomments

\commentstrue

\newcommand{\antiThurston}{(anti-)\allowbreak{}Thurston\ }
\newcommand{\antiRational}{(anti-)\allowbreak{}rational\ }


\begin{document}
\title{Classification of critically fixed anti-Thurston maps}
\author{Lukas Geyer}
\address{Montana State University \\ Department of Mathematical
  Sciences \\ Bozeman, MT 59717--2400 \\ USA}
\email{geyer@montana.edu}
\author{Mikhail Hlushchanka}
\address{Korteweg-de Vries Instituut voor Wiskunde, Universiteit van Amsterdam,  1090 GE \newline Amsterdam, The Netherlands}
\address{Mathematisch Instituut,  Universiteit Utrecht,
3508 TA Utrecht,  The Netherlands}
\email{mikhail.hlushchanka@gmail.com}

\thanks{This material is based partially upon work supported by the
  National Science Foundation under Grant No.\ DMS-1928930 for the
  first author and Grant No.\ 1440140 for the second author, while the
  authors participated in a program hosted by the Mathematical
  Sciences Research Institute in Berkeley, California, during the
  Spring semester of 2022. The second author was also partially
  supported by the Marie Sk{\l}odowska-Curie Postdoctoral Fellowship under the EU's Horizon Europe Programme (Grant No.\ 101068362).}

\begin{abstract}
  We provide a complete combinatorial classification of critically
  fixed anti-Thurston maps, i.e., orientation-reversing branched
  covers of the $2$-sphere that fix every critical point. The first
  step in the proof, and an interesting result in its own right, is a
  combinatorial classification of critically fixed anti-rational maps
  as ``Schottky maps'' associated with certain plane graphs. Both of
  these classification results rely heavily on an
  orientation-reversing version of Thurston's theory, including the
  canonical decomposition of anti-Thurston maps, which we develop in
  this paper. We use our combinatorial framework to answer special
  instances of the global curve attractor and twisting problems, and
  to resolve the Hurwitz existence problem in the setting of
  critically fixed anti-rational maps. Finally, we give some
  applications to the study of anti-rational maps with symmetries and
  to critically fixed anti-polynomials.
\end{abstract}

\maketitle

\setcounter{tocdepth}{1}
\tableofcontents

\section{Introduction}
\label{sec:introduction}
\subsection{Overview}
The main subject of this paper is anti-holomorphic dynamics, i.e., the orientation-reversing analog of holomorphic dynamics. The prefix ``anti-'' (as in
``anti-holomorphic'', ``anti-conformal'', or ``anti-Thurston map'')
will always refer to orientation-reversing versions of familiar concepts
from complex analysis and complex dynamics. For example, an
\emph{anti-rational} map is the complex conjugate of a rational
map. Many of the results in holomorphic dynamics also hold in
the anti-holomorphic setting, but there are several subtle and
not-so-subtle differences; see
Section~\ref{sec:revi-anti-holom}. In addition to exploring anti-holomorphic dynamics as an 
extension of complex dynamics, this subject also 
has intriguing connections to Kleinian groups
\cite{lazebnikUnivalentPolynomialsHubbard2021,
  lodgeDynamicalGasketsGenerated2023, lodgeCirclePackingsKissing2022}
and to the theory of gravitational lensing
\cite{khavinsonFundamentalTheoremAlgebra2008}.

The two main results in this paper are the complete combinatorial
classifications of \emph{critically fixed anti-rational maps}
(\Cref{thm:a}) and \emph{critically fixed anti-Thurston maps}
(\Cref{thm:b}). Here, ``critically fixed'' means that every critical
point is fixed. The analogous classification problems in the
orientation-preserving setting were studied and resolved by the second
author in
\cite{hlushchanka2017Thesis,hlushchankaTischlerGraphsCritically2019,hlushchankaCriticallyFixedThurston2022}.
Further partial results in this direction were obtained in
\cite{geyerSharpBoundsValence2008} and
\cite{lazebnikUnivalentPolynomialsHubbard2021} in the anti-polynomial
case. A closely related, and ultimately equivalent, classification of
critically fixed anti-rational maps was independently obtained in
\cite{lodgeCirclePackingsKissing2022}.

Our classification results provide a natural combinatorial framework
for further study of the family of critically fixed anti-rational
maps. To illustrate this, we resolve special instances of the
\emph{global curve attractor} and \emph{twisting problems}, as well as
settle the \emph{Hurwitz existence problem} in our setting; see
\cite{hlushchanka2017Thesis, hlushchankaTischlerGraphsCritically2019,
  hlushchankaCriticallyFixedThurston2022,
  cordwellClassificationCriticallyFixed2015} for related results for
critically fixed rational maps. A natural direction for further work is to analyze combinatorially the structure of the Hurwitz
spaces associated with critically fixed anti-rational maps. We
emphasize that the orientation-reversing setting appears to be more
subtle than the orientation-preserving one. In particular, the global
curve attractor problem for the whole family of critically fixed
anti-rational maps remains open, even though it has been resolved for
critically fixed rational maps by the second author in
\cite{hlushchanka2017Thesis, hlushchankaTischlerGraphsCritically2019}.

As a major technical tool, we develop an orientation-reversing version
of the theory of Thurston maps, including analogs of Thurston's
characterization of rational maps
\cite{douadyProofThurstonsTopological1993,
  buffTeichmullerSpacesHolomorphic2014} and the canonical
decomposition of obstructed maps
\cite{pilgrimCanonicalThurstonObstructions2001,
  pilgrimCombinationsComplexDynamical2003}; see
\Cref{sec:thurst-theor-orient}. The specialization of this theory to
the setting of critically fixed maps is a crucial ingredient in our
classification results. In upcoming joint work with Lodge, Luo, and
Mukherjee \cite{GeyerAntiThurstonAndReflectionGroups}, we apply the
decomposition results from this paper to provide an explicit relation
between canonical decompositions of (Schottky) anti-Thurston maps and
geometric decompositions of (Schottky) pared 3-manifolds. In this way,
we formalize a ``decomposition theory'' link between the theory of
Kleinian groups and rational dynamics, thereby adding a novel entry to
the \emph{Sullivan dictionary}
\cite{sullivanQuasiconformalHomeomorphismsDynamics1985a}.

We now introduce and discuss our main results in more detail. 

\subsection{Classification of critically fixed anti-rational maps}

Given a critically fixed anti-rational map $f:\CC\to\CC$, its
(\emph{alternating}) \emph{Tischler graph} $T_f$ is the plane graph in
$\CC$ whose edges are the fixed internal rays in the immediate
superattracting basins of $f$ and whose vertices are the endpoints of
these rays; see \Cref{fig:Julia_sets}. We named this graph by analogy
with the \emph{Tischler graphs} of critically fixed rational maps
studied in \cite{cordwellClassificationCriticallyFixed2015,
  hlushchanka2017Thesis, hlushchankaTischlerGraphsCritically2019}; see
also \cite{tischlerCriticalPointsValues1989}.

In \Cref{thm:characterization-1}, we establish the following
properties of the Tischler graph $T_f$:
\begin{enumerate}
\item $T_f$ is a connected bipartite plane graph.
\item The vertices of $T_f$ are exactly the fixed points of $f$, and
  each edge of $T_f$ connects a critical fixed point to a repelling
  fixed point.
\item Every repelling vertex of $T_f$ has degree 2, and every
  critical vertex has degree $m+2\ge 3$, where $m$ is the multiplicity
  of the corresponding critical point.
\item $T_f$ has exactly $d+1$ faces, where $d$ is the degree of $f$.
\item Every face $A$ of $T_f$ is a Jordan domain, and $f$ maps $A$
  anti-conformally onto $\CC \setminus \overline{A}$. 
\end{enumerate}

\begin{figure}[t]
\newcommand{\juliawidth}{7cm}    
\newcommand{\juliagap}{1cm}  

\begin{tikzpicture}[
    x=\juliawidth/1668, y=\juliawidth/1668,
    edge/.style   = {black, line width=0.7pt, line cap=round, line join=round},
    vertex/.style = {circle, fill=black, inner sep=0pt, minimum size=3.5pt}, repelling/.style = {circle, fill=white, line width=0.7pt, draw=black,
                        inner sep=0pt, minimum size=3.5pt},
  ]

 \begin{scope}
 \node[anchor=south west, inner sep=0pt] at (0,0)
      {\includegraphics[width=\juliawidth]{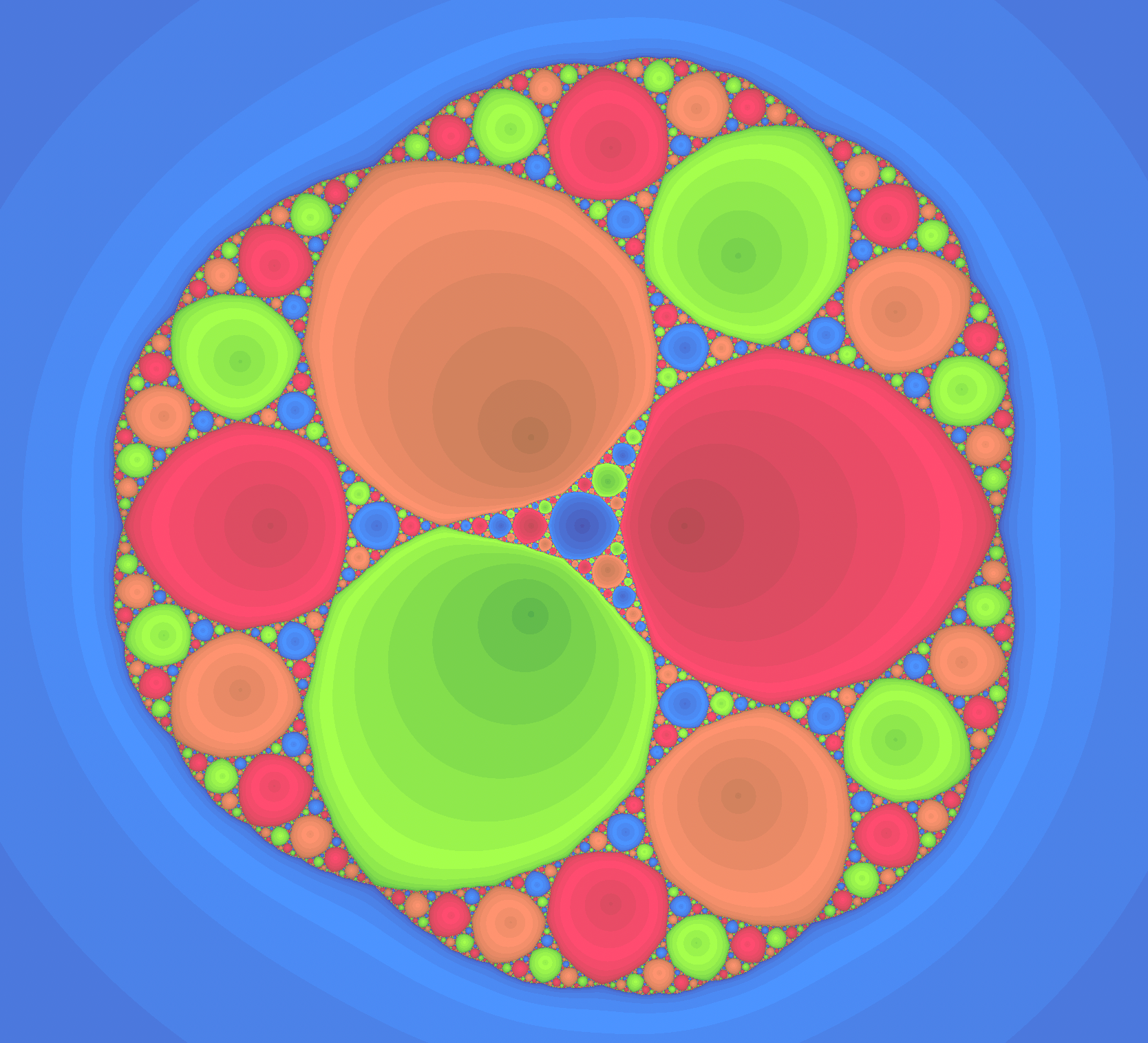}};
    \begin{scope}[shift={(846,751.6)}]
 
  \coordinate (c) at (0,0);          %
  \coordinate (r) at (0:148.5);      %
  \coordinate (o) at (120:148.5);    %
  \coordinate (g) at (240:148.5);    %
 
  \coordinate (ro) at (60:203.5);    %
  \coordinate (og) at (180:203.5);   %
  \coordinate (gr) at (300:203.5);   %

   \coordinate (bo) at (-27,50);    %
  \coordinate (bg) at (-27,-50);   %
  \coordinate (br) at (60,0);   %

  \draw[edge] (c) -- (r)  (c) -- (o)  (c) -- (g);
 
  \draw[edge] (r) to[bend right=45] (ro) to[bend right=45] (o);
  \draw[edge] (o) to[bend right=45] (og) to[bend right=45] (g);
  \draw[edge] (g) to[bend right=45] (gr) to[bend right=45] (r);
 
  \foreach \v in {c,r,o,g} \node[vertex] at (\v) {};

    \node[repelling] at (ro) {};
    \node[repelling] at (og) {};
    \node[repelling] at (gr) {};
    \node[repelling] at (bo) {};
    \node[repelling] at (bg) {};
    \node[repelling] at (br) {};
\end{scope}
\end{scope}

\begin{scope}[xshift=\juliawidth+\juliagap,
                x=\juliawidth/6.39633, y=\juliawidth/6.39633]
    \node[anchor=south west, inner sep=0pt] at (0,0)
      {\includegraphics[width=\juliawidth]{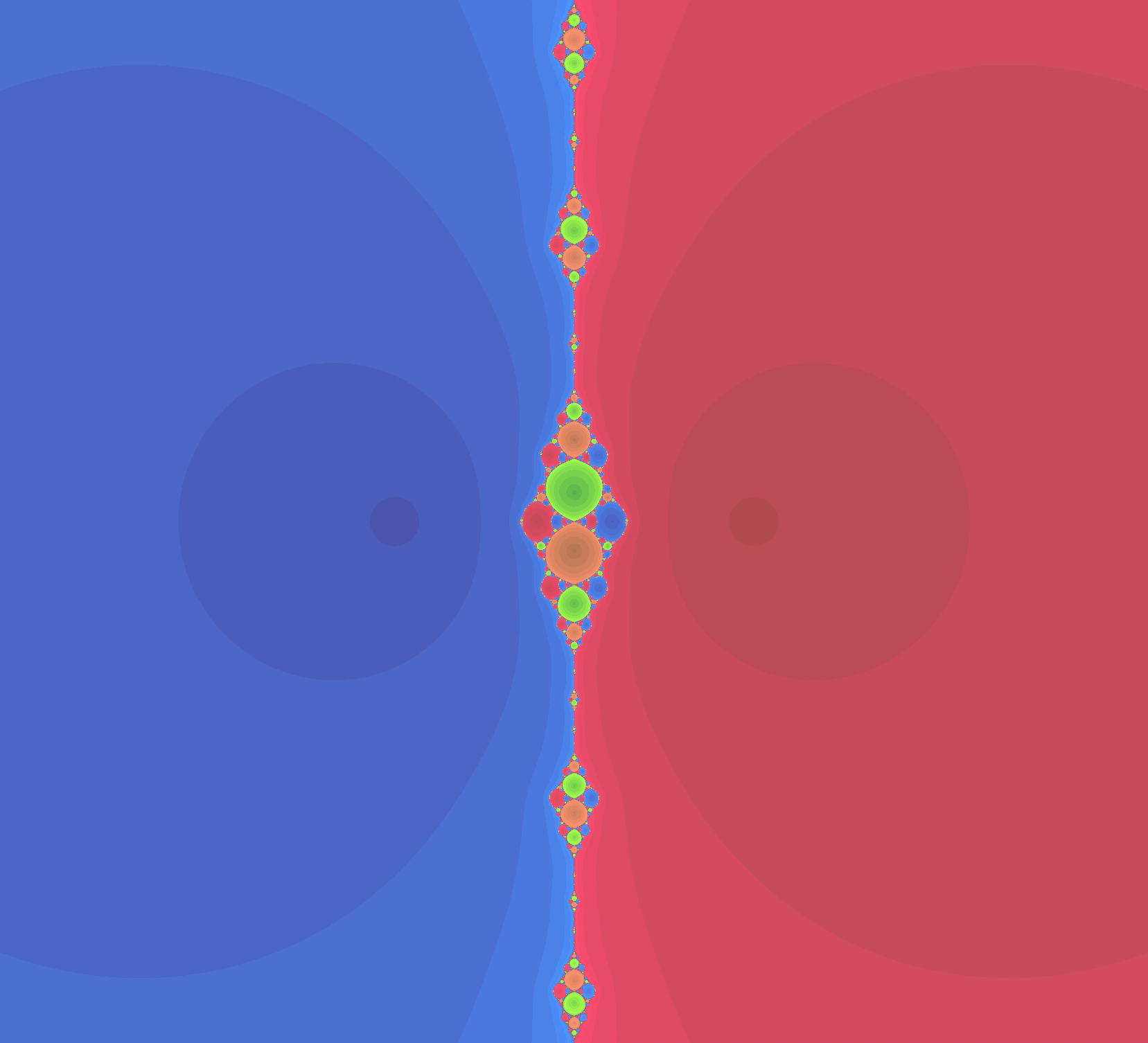}};
    \begin{scope}[shift={(3.19965,2.90573)}]
      \coordinate (L) at (-1,0);           %
      \coordinate (R) at ( 1,0);           %
      \coordinate (T) at (0, 0.165556);    %
      \coordinate (B) at (0,-0.165556);    %
      \coordinate (q0)   at (0,0);
      \coordinate (qtop) at (0, 2.261134);
      \coordinate (qbot) at (0,-2.261134);
      \coordinate (qLT)  at (-0.155579, 0.235049);
      \coordinate (qRT)  at ( 0.155579, 0.235049);
      \coordinate (qLB)  at (-0.155579,-0.235049);
      \coordinate (qRB)  at ( 0.155579,-0.235049);
 
      \draw[edge] (L) to[out=135,  in=180, looseness=1.2] (qtop)
                      to[out=0,    in=45,  looseness=1.2] (R);
      \draw[edge] (L) to[out=-135, in=180, looseness=1.2] (qbot)
                      to[out=0,    in=-45, looseness=1.2] (R);
      \draw[edge] (L) to[out=45,   in=162]  (qLT) to[out=-18,  in=150]  (T);
      \draw[edge] (L) to[out=-45,  in=-162] (qLB) to[out=18,   in=-150] (B);
      \draw[edge] (R) to[out=135,  in=18]   (qRT) to[out=-162, in=30]   (T);
      \draw[edge] (R) to[out=-135, in=-18]  (qRB) to[out=162,  in=-30]  (B);
      \draw[edge] (T) -- (B);              %
 
      \foreach \v in {L,R,T,B} \node[vertex] at (\v) {};
      \foreach \q in {q0,qtop,qbot,qLT,qRT,qLB,qRB} \node[repelling] at (\q) {};
   \end{scope}
   \end{scope}
\end{tikzpicture}

\mbox{}

\begin{tikzpicture}[
    edge/.style      = {black, line width=0.7pt,
                        line cap=round, line join=round},
    vertex/.style    = {circle, fill=black, inner sep=0pt,
                        minimum size=3.5pt},
    repelling/.style = {circle, draw=black, fill=white, line width=0.7pt,
                        inner sep=0pt, minimum size=3.5pt},
  ]
 
  \begin{scope}[x=\juliawidth/2456, y=\juliawidth/2456]
    \clip (0,0) rectangle (2456,1586);
    \node[anchor=south west, inner sep=0pt] at (0,0)
      {\includegraphics[width=\juliawidth]{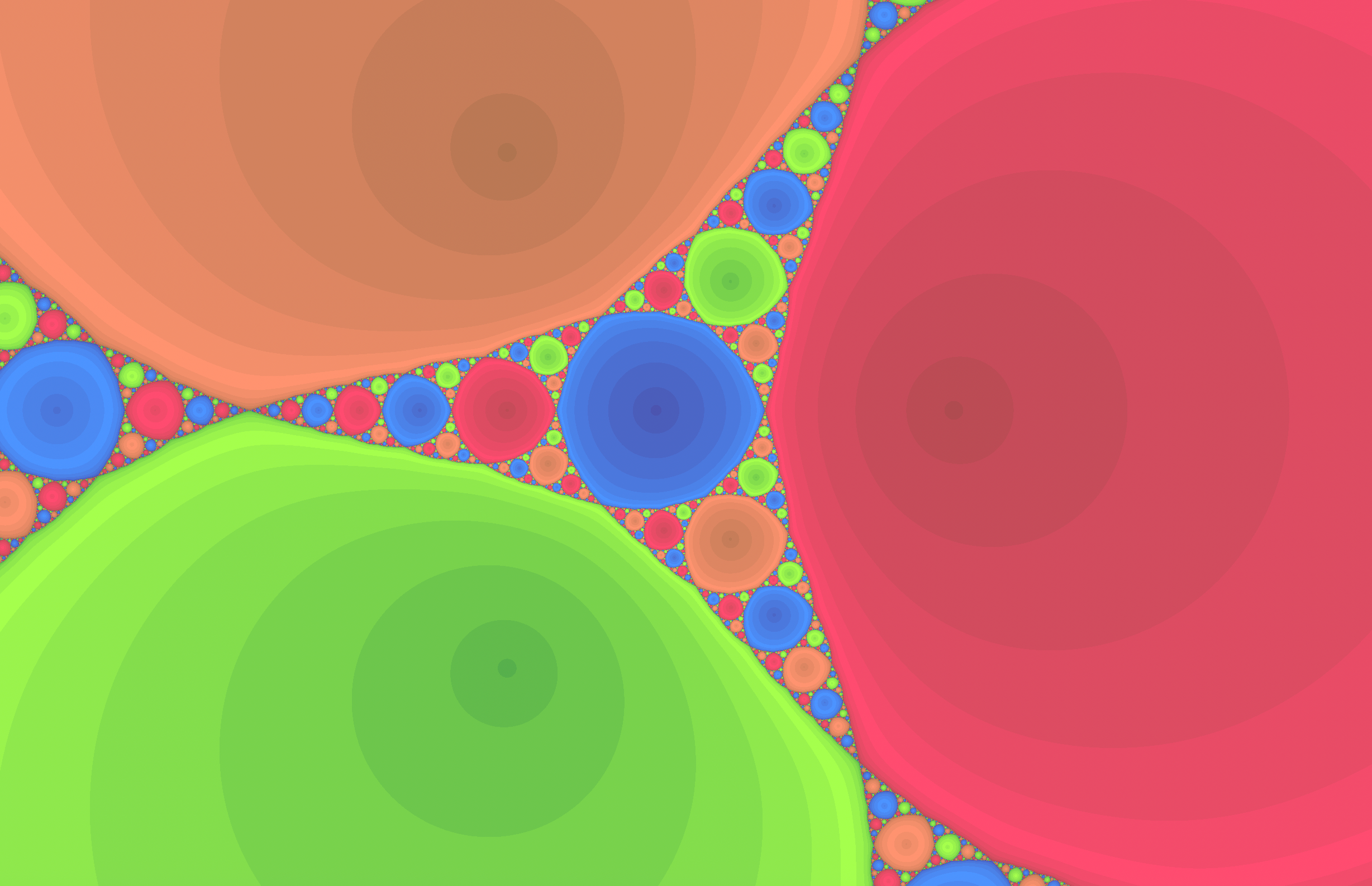}};
    \begin{scope}[shift={(1174.47,851.49)}]
      \coordinate (c) at (0,0);            %
      \coordinate (r) at (0:533.1);         %
      \coordinate (o) at (120:533.1);       %
      \coordinate (g) at (240:533.1);       %
      \coordinate (ro) at (60:728);       %
      \coordinate (og) at (180:728);      %
      \coordinate (gr) at (300:728);      %
      \coordinate (cr) at (0:194.9);        %
      \coordinate (co) at (120:194.9);      %
      \coordinate (cg) at (240:194.9);      %
 
      \draw[edge] (c) -- (r)  (c) -- (o)  (c) -- (g);
      \draw[edge] (r) to[bend right=45] (ro) to[bend right=45] (o);
      \draw[edge] (o) to[bend right=45] (og) to[bend right=45] (g);
      \draw[edge] (g) to[bend right=45] (gr) to[bend right=45] (r);
 
      \foreach \v in {c,r,o,g} \node[vertex] at (\v) {};
      \foreach \q in {ro,og,gr,cr,co,cg} \node[repelling] at (\q) {};
    \end{scope}
  \end{scope}
 
  \begin{scope}[xshift=\juliawidth+\juliagap,
                x=\juliawidth/2.57361, y=\juliawidth/2.57361]
    \clip (0,0) rectangle (2.57361,1.66195);
    \node[anchor=south west, inner sep=0pt] at (0,0)
      {\includegraphics[width=\juliawidth]{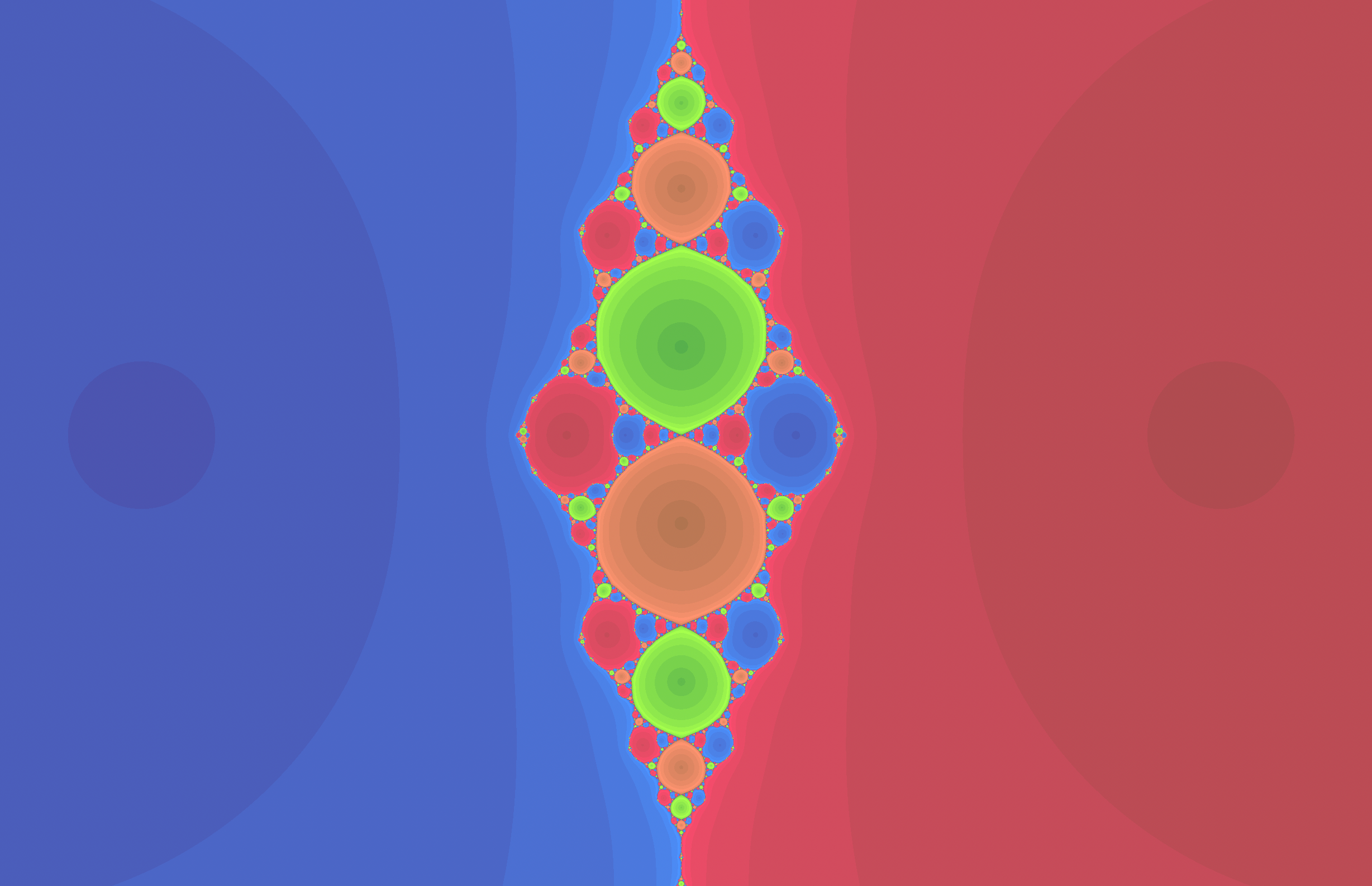}};
    \begin{scope}[shift={(1.27802,0.84567)}]
      \coordinate (L) at (-1,0);           %
      \coordinate (R) at ( 1,0);           %
      \coordinate (T) at (0, 0.165556);    %
      \coordinate (B) at (0,-0.165556);    %
      \coordinate (q0)   at (0,0);
      \coordinate (qtop) at (0, 2.261134);
      \coordinate (qbot) at (0,-2.261134);
      \coordinate (qLT)  at (-0.155579, 0.235049);
      \coordinate (qRT)  at ( 0.155579, 0.235049);
      \coordinate (qLB)  at (-0.155579,-0.235049);
      \coordinate (qRB)  at ( 0.155579,-0.235049);
 
      \draw[edge] (L) to[out=135,    in=180,    out looseness=1.11, in looseness=1.28] (qtop)
                      to[out=0,      in=45,     out looseness=1.28, in looseness=1.11] (R);
      \draw[edge] (L) to[out=-135,   in=180,    out looseness=1.11, in looseness=1.28] (qbot)
                      to[out=0,      in=-45,    out looseness=1.28, in looseness=1.11] (R);
      \draw[edge] (L) to[out=45,     in=162.2,  out looseness=1.18, in looseness=0.73] (qLT)
                      to[out=-17.8,  in=150,    out looseness=0.66, in looseness=1.07] (T);
      \draw[edge] (L) to[out=-45,    in=-162.2, out looseness=1.18, in looseness=0.73] (qLB)
                      to[out=17.8,   in=-150,   out looseness=0.66, in looseness=1.07] (B);
      \draw[edge] (R) to[out=135,    in=17.8,   out looseness=1.18, in looseness=0.73] (qRT)
                      to[out=-162.2, in=30,     out looseness=0.66, in looseness=1.07] (T);
      \draw[edge] (R) to[out=-135,   in=-17.8,  out looseness=1.18, in looseness=0.73] (qRB)
                      to[out=162.2,  in=-30,    out looseness=0.66, in looseness=1.07] (B);
      \draw[edge] (T) -- (B);              %
 
      \foreach \v in {L,R,T,B} \node[vertex] at (\v) {};
      \foreach \q in {q0,qtop,qbot,qLT,qRT,qLB,qRB} \node[repelling] at (\q) {};
    \end{scope}
  \end{scope}
\end{tikzpicture}
\caption{Dynamical planes of critically fixed anti-rational maps
  \eqref{eq:tetr_map} (on the left) and \eqref{eq:diam_map} (on the
  right), showing the Julia sets and distinct basins of attraction
  shaded in different colors. The pictures also illustrate the
  corresponding alternating Tischler graphs, with black and white
  vertices representing the critical and repelling fixed points,
  respectively.}
  \label{fig:Julia_sets}
\end{figure}

The \emph{reduced Tischler graph} of $f$ is the plane (multi-)graph
obtained from the Tischler graph $T_f$ by removing all its vertices of
degree $2$ and joining the incident edges. These reduced graphs may
have multiple edges, but no loops. As topological analogs of these
graphs, we define \emph{reduced topological Tischler graphs} to be
connected plane (multi-)graphs $T \subset S^2$ such that all of their
vertices have degree $\ge 3$ and such that all of their faces are
Jordan domains. To any reduced topological Tischler graph $T$, we
associate a \emph{Schottky map} $f_T: S^2\to S^2$, which is a
continuous map whose restriction to the closure $\overline{A}$ of
every face $A$ of $T$ is a \emph{topological reflection} in
$\partial A$ (see
\Cref{def:top-reflection,def:schottky-map}). Equivalently, $f_T$ is
the identity on the graph $T$, and maps each face $A$ homeomorphically
onto $S^2\setminus \overline{A}$.  In particular, $f_T$ is an
orientation-reversing branched covering map on the 2-sphere whose
critical points are exactly the vertices of $T$, and thus they are
fixed by the map. In other words, $f_T$ is a critically fixed
anti-Thurston map.

The name ``Schottky map'' is chosen by analogy with Schottky
groups. There is indeed a rigorous connection between Schottky maps
and associated Kleinian reflection groups; see
\cite{lodgeDynamicalGasketsGenerated2023} and
\cite{lodgeCirclePackingsKissing2022}.  The constructions in these
papers are closely related, but not quite identical, to our
constructions. Instead of Tischler graphs, the authors of
\cite{lodgeDynamicalGasketsGenerated2023,
  lodgeCirclePackingsKissing2022} focus on their dual graphs and
associated circle packings.

In \Cref{sec:plane-graphs-schottk}, we prove the first main result of this paper.

\begin{thmintro}
\label{thm:a}
There is an explicit and canonical one-to-one correspondence between
the following sets:
\begin{itemize}
    \item the set of Möbius conjugacy classes of critically
    fixed anti-rational maps;
    \item the set of isomorphism classes of unobstructed reduced 
    topological Tischler graphs;
        \item the set of equivalence classes of realizable
    Schottky maps.
\end{itemize}
\end{thmintro}

Here we say that a reduced topological Tischler graph $T$ is \emph{unobstructed} if any two distinct faces of $T$ share at most one boundary edge. (This property is equivalent to the requirement that the dual graph of $T$ has no multiple edges.) Two reduced topological Tischler graphs $T$ and $T'$
are called \emph{isomorphic} if there is an orientation-preserving homeomorphism of
the sphere mapping $T$ bijectively onto $T'$. Equivalence of Schottky maps is
defined as combinatorial equivalence for anti-Thurston maps; see \Cref{def:Thurston-equiv}. Roughly speaking, two Schottky maps are \emph{equivalent} if they are conjugate up to isotopy relative to their critical points. Finally, a Schottky map is called \emph{realizable} if it is combinatorially equivalent to a (critically fixed) anti-rational map. 

We use \Cref{thm:a} to provide a combinatorial classification of
critically fixed anti-polynomial maps in
\Cref{sec:crit-fixed-anti-poly}.  Another useful consequence is that
symmetries of unobstructed reduced topological Tischler graphs are
inherited by the associated anti-rational maps; see
\Cref{thm:symmetries}.

\subsection{Classification of critically fixed anti-Thurston maps} Not all critically fixed anti-Thurston maps are combinatorially equivalent to a Schottky map; see \Cref{subsec:non-schottky-ex} for a non-Schottky example. In light of this, we introduce more general combinatorial models of critically fixed anti-Thurston maps, which we call \emph{multi-Schottky maps}. Here, we only provide an informal description of these models, and refer the reader to \Cref{subsec:multi-schottky} for the precise definition. 

Let $\sT=\{S^2_t\}_t$ and $\sH=\{S^2_h\}_h$ be two finite collections
of (disjoint) 2-spheres with $\#\sT \geq 1$ and $\#\sH \geq 0$. We
make the following assumptions:
\begin{itemize}
\item every sphere $S^2_t$ contains an unobstructed reduced
  topological Tischler graph $T_t$;
\item every sphere $S^2_h$ contains a plane graph $H_h$ with at least
  2 connected components and such that each connected component of $H_h$ consists of a single edge
  together with its two (distinct) endpoints;
\item each sphere $S^2_h$ comes with an orientation-reversing
  self-homeomorphism $\xi_h$ that fixes the graph $H_h$ pointwise.
\end{itemize}

We glue the spheres in $\sT$ and $\sH$ together to form a topological space $\cS$
while respecting the following conditions:
\begin{enumerate}
\item\label{item:S-i} Whenever two spheres are glued together, one of
  them must be from $\sT$ and the other one must be from
  $\sH$. Furthermore, the gluing happens along (the closure of)
  exactly one edge in each of the respective graphs. The common image of
  these closed arcs in $\cS$ is called a \emph{seal}.
\item\label{item:S-ii} For all $S^2_h\in \sH$, every edge of $H_h$
  must be used in the gluing process (to form a seal).
\item\label{item:S-iii} The resulting space $\cS$ must be connected,
  and the removal of any seal from $\cS$ must split it into exactly
  two connected components.
\end{enumerate}
Topologically, the space $\cS$ is a tree-like configuration of finitely many $2$-spheres that are glued to each other along
closed arcs (the seals in $\cS$), where each seal glues together exactly two spheres.

\begin{figure}[t]
  \centering
  \includegraphics[width=.9\textwidth]{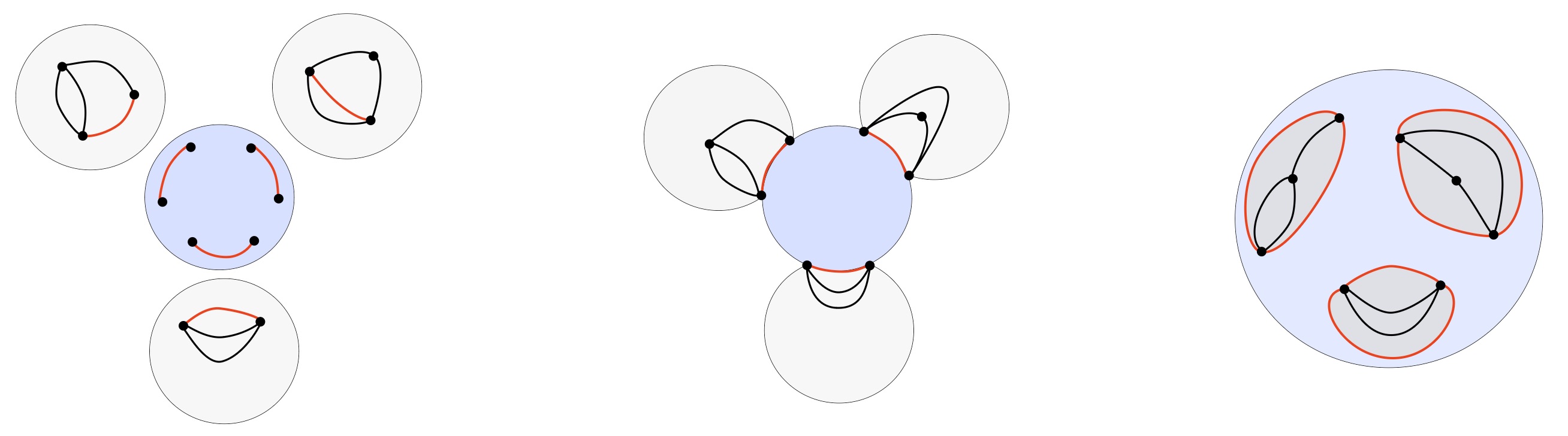}
  \caption{Constructing a multi-Tischler graph. In the left picture,
    each of the three gray spheres contains an unobstructed reduced
    topological Tischler graph, and the blue sphere contains a graph
    formed by three disjoint edges. The picture in the middle shows
    the topological space $\cS$ obtained by gluing these four spheres
    along the red edges. Finally, the picture on the right illustrates
    the $2$-sphere obtained by opening up each seal in $\cS$ to a
    Jordan curve (in red) and the resulting multi-Tischler graph.}
  \label{fig:multi-tischler}
\end{figure}

Let us now ``open up'' each seal in $\cS$ to a Jordan curve; we refer
the reader to \Cref{fig:multi-tischler} for an illustration. One can
check that the resulting space is actually a $2$-sphere. We denote by
$M$ the plane graph in this sphere induced by the graphs $T_t\subset
S^2_t$; we call this graph a \emph{reduced multi-Tischler graph}. Note
that each edge of $T_t$ that forms a seal in $\cS$ is doubled in the
graph $M$. Moreover, the spheres $S^2_h\in \sH$ are in (canonical)
bijective correspondence with the multiply connected faces $A_h$ of
$M$. Clearly, each homeomorphism $\xi_h: S^2_h\to S^2_h$ induces an
orientation-reversing self-homeomorphism $\mu_h$ of the respective
closure $\overline{A_h}$, such that $\mu_h$ acts as an involution on
each boundary component of $A_h$. We denote by $M^\mu$ the union of the
closures of all multiply connected faces of $M$, and by $\mu: M^\mu\to
M^\mu$ the orientation-reversing homeomorphism that agrees with
$\mu_h$ on each multiply connected face $A_h$.

We call the pair $(M,\mu)$ constructed above a \emph{reduced
  multi-Tischler pair}. We now define a self-map $f=f_{(M,\mu)}$ of
the underlying $2$-sphere in the following way. First, we require that
$f$ fixes the vertex set $V(M)$ of $M$ pointwise.  Next, given an edge
$e$ of $M$, we set $f|_e=\mu|_e$ if $e\subset \partial M^\mu$, and
$f|_e = \id_e$ otherwise. Finally, for every multiply connected face $A$
of $M$, we define $f|_A=\mu|_A$, and for each simply connected face
$A$ we define $f|_A$ to be an orientation-reversing homeomorphism
extending the boundary map $f:\partial A \to f(\partial A)$. Note
that, in the latter case, since $\partial A$ and $f(\partial A)$ are
Jordan curves, $f|_{\overline{A}}$ is uniquely defined up to isotopy rel.\
$\partial A$ by the Alexander trick (see \Cref{thm:alexander}). It is
straightforward to check that $f$ is an orientation-reversing branched
covering map with the critical set $C_f=V(M)$, and thus a critically fixed
anti-Thurston map. We call such a map $f=f_{(M,\mu)}$ a \emph{multi-Schottky map}
associated with the reduced multi-Tischler pair $(M,\mu)$.

With this setup, we can now state our second main result.

\begin{thmintro}
  \label{thm:b}
There is a canonical one-to-one correspondence between the following sets:
\begin{itemize}
    \item the set of combinatorial (resp.\ isotopy) equivalence classes of critically
    fixed anti-Thurston maps;
    \item the set of equivalence (resp.\ isotopy) classes of reduced multi-Tischler pairs;
    \item the set of equivalence (resp.\ isotopy) classes of multi-Schottky maps.
\end{itemize}
\end{thmintro}

Here, equivalence (resp.\ isotopy) of multi-Schottky maps is defined as combinatorial (resp.\ isotopy) equivalence for anti-Thurston maps; see \Cref{def:Thurston-equiv,def:Thurston-isot}. Furthermore, two multi-Tischler pairs $(M,\mu)$ and $(M',\mu')$ are called \emph{equivalent} if there is an orientation-preserving homeomorphism $\phi$ of the underlying sphere that maps $M$ onto $M'$ and such that $\phi^{-1}\circ \mu' \circ \phi$ is isotopic to $\mu$ on $M^\mu$ rel.\ $M^\mu\cap V(M)$. If, in addition, $\phi$ is isotopic to the identity rel.\ $ V(M)$, we say that the pairs $(M,\mu)$ and $(M',\mu')$ are \emph{isotopic}.

A crucial ingredient in the proof of \Cref{thm:b} is the description
of canonical decompositions of critically fixed anti-Thurston maps provided by \Cref{cor:can_crit_fix_decomposition}. A similar statement for critically fixed
Thurston maps was shown by the second author and Prochorov in
\cite{hlushchankaCriticallyFixedThurston2022}. In this paper, we work
out the relevant details in the orientation-reversing and marked setting.

\subsection{Applications} Our classification results provide a useful combinatorial framework for studying critically fixed anti-rational maps. We illustrate this with applications to three well-studied problems; see \Cref{sec:exampl-appl} for more details and further applications. 

\subsubsection{Hurwitz existence problem.} This classical question asks whether a given combinatorial branch data set can be realized by a holomorphic branched covering map between compact Riemann surfaces; see, for example, \cite{PetronioHurwitzProblem2020} for more context. In our setting of critically fixed anti-rational maps it specializes as follows. 

The \emph{branch data} of a critically fixed anti-rational map is simply the list of the multiplicities of its critical points. We say that a tuple $(m_1,\dots,m_r)$ of positive integers is \emph{admissible for degree $d\geq 2$} if $m_k \le d-1$ for all $k$ and $\sum_{k=1}^r m_k = 2d-2$. These are immediate topological constraints that the multiplicities of the critical points of a critically fixed anti-rational (or anti-Thurston) map of degree $d$ must satisfy, so admissibility is necessary for such a tuple to be realized as branch data. We show that it is also sufficient: 

\begin{thmintro}
  \label{thm:c}
  Every admissible tuple for degree $d \ge 2$ is the branch data of some critically fixed anti-rational map of degree $d$.
\end{thmintro}

\Cref{thm:c} parallels the analogous result \cite[Thm.~1]{cordwellClassificationCriticallyFixed2015} for critically fixed rational maps. The proof is constructive: we build an unobstructed reduced topological Tischler graph matching the prescribed data. We expect that our combinatorial framework will also be useful for the study of Hurwitz spaces and correspondences on moduli spaces associated with critically fixed anti-rational maps, which play an important role in understanding properties of the associated Thurston pullback maps; see, for example, \cite{kochTeich2013}.

\subsubsection{Global curve attractor problem} Given an (anti-)rational map $f: \CC\to \CC$ with a finite postcritical set $P_f$, a \emph{finite global curve attractor} for $f$ is the smallest finite set $\mathcal{N}$ of homotopy classes of simple closed curves in $\CC\setminus P_f$ that absorbs every pullback orbit with respect to $f$: for every simple closed curve $\gamma\subset \CC\setminus P_f$, there exists $N=N(\gamma)\geq 1$ such that the homotopy class $[\delta]$ of every connected component $\delta$ of $f^{-n}(\gamma)$ satisfies $[\delta]\in\mathcal{N}$ for all $n\geq N$. 

It is conjectured by Pilgrim that such a finite attractor exists for
all non-Latt\`{e}s postcritically finite rational maps (see
\Cref{conj:fgca}); we also refer the reader to
\cite{pilgrimPullbackRelationCurves2022} for the history of the
problem. Although this conjecture was recently established for all
postcritically finite polynomial maps
\cite{belkRecognizingTopologicalPolynomials2022} and all rational maps
with four postcritical points \cite{bartholdiCorrespondences2024}, the
general case remains open. The second author used intersection number
techniques to settle the conjecture for critically fixed rational maps
in \cite{hlushchanka2017Thesis,
  hlushchankaTischlerGraphsCritically2019}. In our setting of
critically fixed anti-rational maps, the problem turns out to be
genuinely more delicate: double intersections with the reduced
Tischler graph can persist under pullback (see \Cref{fig:pullback-2}
and the remark after \Cref{conj:decreasing-complexity}); this
phenomenon has no orientation-preserving analog. Consequently, the
conjecture remains open for the full family of critically fixed
anti-rational maps. Nevertheless, we establish it for the following
natural subclass:

\begin{thmintro}
  \label{thm:d}
  Let $f$ be a critically fixed anti-rational map. If every edge of
  the reduced Tischler graph of $f$ lies on the boundary of a digonal
  or triangular face, then $f$ has a finite global curve attractor.
\end{thmintro}

\subsubsection{Twisting problem.} Given a postcritically finite
(anti-)rational or, more generally, \antiThurston map $f$ and an
orientation-preserving homeomorphism $\phi$ of the underlying sphere
that fixes the postcritical set $P_f$ pointwise, this problem asks for
the combinatorial equivalence class of the \emph{twisted map}
$\phi \circ f$; see \Cref{subsec:twisting}. This question was
originally posed by Hubbard in the 1980s for the so-called ``rabbit
polynomial'' $f(z)\approx z^2-0.1226+0.7449i$ and (iterates of) a
certain Dehn twist $\phi$, and subsequently resolved by Bartholdi and
Nekrashevych in \cite{bartholdiTwistedRabbit2006} using algebraic
methods. Since then, the twisting problem has been studied in various
settings using both algebraic methods and methods from combinatorial topology; see, for example, \cite{bartholdiTwistedRabbit2006,
  lodgeBoundaryValues2013, kelseyQuadraticThurston2019,
  belkRecognizingTopologicalPolynomials2022, mukundanTwisting2025,
  lanierTwisting2025, hlushchankaCriticallyFixedThurston2022}.

In this paper, we solve the twisting problem for Schottky maps $f_T$
and (iterates of) Dehn twists $\phi_\gamma$ about simple closed curves
$\gamma$ intersecting each face of the corresponding reduced
topological Tischler graph $T$ in at most one arc. Namely, for each
$n\in \Z$, we prove that the twisted map $\phi_\gamma^n\circ f_T$ is
again a Schottky map, associated with the so-called
\emph{$n$-rotation} of $T$ about $\gamma$; see \Cref{subsec:twisting}
and \Cref{thm:twisting} for details.

\subsection{Structure of the paper}
Sections~\ref{sec:preliminaries}--\ref{sec:plane-graphs} are
devoted to the discussion of necessary background. We introduce the
basic notation used throughout this paper in \Cref{sec:preliminaries}
and provide a review of anti-holomorphic dynamics in
\Cref{sec:revi-anti-holom}. In \Cref{sec:thurst-theor-orient}, we
briefly review Thurston's theory of branched covers of the $2$-sphere
and extend it to the orientation-reversing setting. In particular, we
prove an orientation-reversing version of Thurston's characterization
of rational maps (see \Cref{thm:anti-thurston}) and
discuss the orientation-reversing versions of canonical obstructions and decompositions in
\Cref{subsec:canon-obstr,subsec:canon-decomp}. Finally, we introduce
the necessary terminology and results from the theory of plane graphs
in \Cref{sec:plane-graphs}.

In \Cref{sec:crit-fixed-anti-rational}, we define Tischler
graphs of critically fixed anti-rational maps and describe their
properties in \Cref{thm:characterization-1}. Following this, in
Section~\ref{sec:plane-graphs-schottk}, we introduce topological
Tischler graphs and prove our combinatorial classification of
critically fixed anti-rational maps (see
\Cref{thm:tischler-fixed-correspondence,thm:crit-fix-rational-classification}). In
Section~\ref{sec:crit-fixed-anti-poly}, we use this to classify
critically fixed anti-polynomials via \emph{topological Tischler
  trees}, which can be thought of either as topological versions of
the corresponding Tischler graphs with the vertex at $\infty$ removed, or
as extensions of Hubbard trees for polynomials.

In \Cref{sec:decomp-crit-fix}, we discuss special properties of
completely invariant multicurves and corresponding decompositions for
(marked) critically fixed (anti-)Thurston maps and, in particular, of
canonical obstructions and corresponding canonical decompositions (see
\Cref{cor:can_crit_fix_decomposition}). We also construct a
non-Schottky critically fixed anti-Thurston map in
\Cref{subsec:non-schottky-ex}. In \Cref{sec:multi-schottky}, we first
provide formal definitions of multi-Tischler pairs and associated
multi-Schottky maps, and then prove our combinatorial
classification of critically fixed anti-Thurston maps (see
\Cref{thm:class-crit-fix-anti-thurst}).

Finally, we provide various applications of our results in
Section~\ref{sec:exampl-appl}. We begin with applications to the
global curve attractor problem for critically fixed anti-rational maps
(\Cref{subsec:glob-curve-attr}) and to the twisting problem for
Schottky maps (\Cref{subsec:twisting}). Next, we give a complete
solution to the Hurwitz existence problem in our setting in
\Cref{subsec:hurwitz} and provide a census of critically fixed
anti-rational maps of small degrees in \Cref{subsec:census}. Finally,
we discuss an application to anti-rational maps with symmetries in
\Cref{subsec:maps-with-symmetries}.

\subsection{Acknowledgments}
Discussions, feedback, and encouragement from many people considerably
improved this paper. Special thanks are due to Russell Lodge, Yusheng
Luo, Chris McKay, Sabyasachi Mukherjee, and Palina Salanevich. We are
also indebted to the organizers and other participants of the MSRI
special program on Complex Dynamics in Spring~2022. In addition, the
first author extends his gratitude to the Department of Mathematical
Sciences at Montana State University (in particular to department head
Beth Burroughs) and the second author to the Department of Mathematics
at Utrecht University (in particular to department head Jason Frank)
for supporting their visits to MSRI.

\subsection{AI Declaration}
Some of the figures were produced with assistance of Google Gemini (free version)
and Anthropic Claude (Opus 5.5). 
Furthermore, 
Claude (Opus 5.5) and ChatGPT (GPT-6 Astra) were used to check the final version of the paper for grammar, style, readability, as well as soundness of mathematical arguments. The authors have independently checked all calculations, proofs, and figures, and take sole responsibility for the correctness and completeness of all mathematical content in this paper.

\section{Notation and preliminaries}
\label{sec:preliminaries}

We use the notation $S^2$ for the topological $2$-sphere and $\CC$ for
the Riemann sphere, equipped with the standard complex structure. We write $\Mob(\CC)$ for the group of Möbius transformations on $\CC$. We
denote the disjoint union of sets $A$ and $B$ by $A \sqcup B$, the
identity map on a set $X$ by $\id_X$, the restriction of a map $f$ to
a set $X$ by $f|_X$, and the $n$-th iterate of a map $f$ by $f^n = f
\circ f \circ \ldots \circ f$. Given an equivalence relation on a set
$X$ and an element $x\in X$, we write $[x]$ for the equivalence class
of $x$.

In this paper, by ``analytic'' we will always mean ``complex
analytic'', and an \emph{anti-analytic} function of the complex
variable $z$ is an analytic function of $\bar{z}$, defined in some
open subset of $\C$. If either the domain or codomain contains
$\infty$, analyticity and anti-analyticity are defined in the usual
way, using the coordinate $w=1/z$ in a neighborhood of $\infty$. We
will use the notation $\partial f$ and $\bar{\partial} f$ for the
Wirtinger derivatives
$\frac{\partial f}{\partial z} = \frac12 \left(\frac{\partial
    f}{\partial x} - i \frac{\partial f}{\partial y} \right)$ and
$\frac{\partial f}{\partial \bar{z}} = \frac12 \left(\frac{\partial
    f}{\partial x} + i \frac{\partial f}{\partial y} \right)$ of a
quasiconformal, analytic, or anti-analytic map $f$ of the complex
variable $z=x+iy$.

By a \emph{plane curve} we mean the image of a continuous
parameterization map $\eta: I \to S^2$, where $I$ is an interval in
$\R$. We say that two plane curves $\gamma_1$ and $\gamma_2$
\emph{intersect transversely at a point $p\in \gamma_1\cap \gamma_2$}
if $p$ is an isolated point in $\gamma_1\cap \gamma_2$ and if
$\gamma_1$ locally separates $\gamma_2$ at $p$, and vice versa. (Less
formally, this means that $\gamma_1$ and $\gamma_2$ cross each other
at $p$.) We say that the plane curves $\gamma_1$ and $\gamma_2$
\emph{intersect transversely} (or that they have \emph{transverse
  intersections}) if $\gamma_1\cap\gamma_2$ is finite and if
$\gamma_1$ and $\gamma_2$ intersect transversely at each point
$p\in \gamma_1\cap \gamma_2$.

For a subset $Q\subset S^2$, we denote by $\Homeo^+(S^2,Q)$ the
set of all orientation-preserving homeomorphisms of $S^2$ that fix $Q$
pointwise, and we write $\Homeo^+_0(S^2,Q)$ for those elements of
$\Homeo^+(S^2,Q)$ that are isotopic to the identity on $S^2$ relative
to $Q$. We will refer to the pair $(S^2,Q)$ as a \emph{marked sphere}.

We will need the following elementary but very useful result.
\begin{theorem}[Alexander trick]
  \label{thm:alexander}
  Let $B^n$ be the closed unit ball in $\R^n$, and let $S^{n-1} =
  \partial B^n$ be its boundary sphere. Then the following statements are
  true for all $n\geq 1$:
  \begin{enumerate}
  \item\label{item:alexander_1} Every homeomorphism
    $f: S^{n-1} \to S^{n-1}$ can be extended to a homeomorphism
    $F:B^n \to B^n$.
  \item\label{item:alexander_2} If $F_0, F_1: B^n \to B^n$ are
    homeomorphisms with $F_0 = F_1$ on $S^{n-1}$, then $F_0$ and $F_1$
    are isotopic relative to $S^{n-1}$.
  \item\label{item:alexander_3}
    If $F_0,F_1: B^n \to B^n$ are homeomorphisms extending two
    isotopic homeomorphisms $f_0,f_1:S^{n-1} \to S^{n-1}$, then $F_0$
    and $F_1$ are isotopic.
  \end{enumerate}
\end{theorem}
\begin{proof}
  A proof of at least \ref{item:alexander_1} and
  \ref{item:alexander_2} can be found in many places, e.g.,
  \cite[Appendix~C2]{hubbardTeichmullerTheoryApplications2016}.  While
  \ref{item:alexander_3} is not explicitly proved in this reference,
  the methods used to prove \ref{item:alexander_1} and
  \ref{item:alexander_2} easily yield it as well. Namely, it follows
  from \ref{item:alexander_2} that any two homeomorphic extensions of
  each $f_k:S^{n-1}\to S^{n-1}$, $k\in\{0,1\}$, are isotopic relative
  to $S^{n-1}$. Thus, without loss of generality, we may assume that
  each homeomorphism $F_k: B^n\to B^n$ is the \emph{radial extension}
  of $f_k$ (which is used to prove \ref{item:alexander_1}), that is,
  $F_k(rx) = r f_k(x)$, where $x \in S^{n-1}$ and $r \in [0,1]$. For
  these, given an isotopy $(f_t)_{t \in [0,1]}$ from $f_0$ to $f_1$,
  the radial extensions $F_t(rx) = r f_t(x)$, $t\in[0,1]$, provide an
  isotopy from $F_0$ to $F_1$.
\end{proof}
\begin{remark}
  Obviously, this theorem remains true if we replace $B^n$ by a
  homeomorphic set $D$, with $S^{n-1}$ replaced by $\partial D$. We
  will need it in the case where $D$ is the closure of a Jordan domain
  in $S^2$.
\end{remark}

\section{Review of anti-holomorphic dynamics}
\label{sec:revi-anti-holom}
In this section, we summarize a few basic definitions and results from anti-holomorphic dynamics.
Most of the previous work in this area has focused on the iteration of anti-polynomials, 
particularly on the study of analogs of the Mandelbrot set, now called ``tricorns'' and
``multicorns'' for the families $f_{c,d}(z) = \bar{z}^d+c$; see
\cite{croweStructureMandelbarSet1989,
  mukherjeeMulticornsUnicornsII2017,
  nakaneMulticornsUnicornsAntiholomorphic2003,
  nakaneConnectednessTricorn1993, hubbardMulticornsAreNot2014}.

\begin{definition}
  Let $f$ be an anti-analytic map, and let $z_0 \in \C$ be a fixed point
  of $f$ with \emph{multiplier} $\lambda = \bar{\partial} f(z_0)$ and
  \emph{real multiplier} $L = |\lambda| = |\bar{\partial}
  f(z_0)|$. Then $z_0$ is
  \begin{itemize}
  \item \emph{superattracting} (or \emph{critical}) if $L = 0$,
  \item \emph{attracting} if $0<L<1$,
  \item \emph{repelling} if $L>1$, and
  \item \emph{indifferent} if $L=1$.
  \end{itemize}
  If $f$ has a fixed point at $\infty$, we
  define its multiplier as the multiplier of the conjugated map $g(w)=1/f(1/w)$ at $w=0$. The (real) multiplier of a periodic point $z_0$ of period $n\geq 2$ is defined to be the (real) multiplier of the $n$-th iterate $f^n$ at $z_0$; in particular, its real multiplier is $|{\partial}
  f^n(z_0)|$ when $n$ is even, and $|\bar{\partial}
  f^n(z_0)|$ when $n$ is odd. Periodic points are classified analogously to fixed points in terms of their multipliers.
\end{definition}
\begin{remark} 
  Note that the multiplier of $z_0$ with respect to the second iterate $f^2=f\circ f$ is not
  $\lambda^2$, but $\lambda \bar{\lambda} = |\lambda|^2$. In
  particular, $f^2$ has a superattracting, attracting, or repelling
  fixed point at $z_0$ if $f$ does. If $z_0$ is an indifferent fixed point of $f$, then $f^2$
  has a fixed point with multiplier $1$ at $z_0$, i.e., $z_0$ is a fixed point of multiplicity $\ge 2$ for $f^2$.  Note also that the multipliers for
  anti-analytic maps are not invariant under conjugacy by analytic
  maps. For instance, if $f(z) = \lambda \bar{z} + O(z^2)$ and
  $\phi(z) = az+O(z^2)$ with $a \ne 0$, then
  $\phi \circ f \circ \phi^{-1}(z) = \lambda' \bar{z} + O(z^2)$ with
  $\lambda' = (a/\bar{a}) \lambda = \frac{a^2}{|a|^2} \lambda$.  However, it
  is easy to see that the real multiplier $L = |\lambda|$ is invariant
  under analytic conjugacy.
\end{remark}

\begin{definition}
  An \emph{anti-rational map} $f(z)$ on $\CC$ is a rational map of $\bar{z}$,
  or alternatively, it is a map $f:\CC \to \CC$ such that its complex conjugate $\bar{f}:\CC \to \CC$, 
  $z\mapsto \bar{f}(z) = \overline{f(z)}$, is a rational map on $\CC$. The \emph{degree} of an
  anti-rational map $f$ is the degree of the corresponding rational map
  $\bar{f}$. Similarly, the \emph{local degree} of $f$ at a
  point $z_0\in \CC$, denoted $\deg(f, z_0)$, is defined as the local degree of the rational map
  $\bar f$ at $z_0$.
\end{definition}
\begin{remark}
  If we were to keep track of the orientation of preimages, the degree of
  an orientation-reversing map would be negative. However, in our
  context not much is gained by this convention, so we define (local) degrees to be non-negative, both for anti-rational maps and later for
  anti-Thurston maps.
\end{remark}

 We will always assume that our functions are
  non-constant, and typically also that the degree is at least 2.
  Note also that the second iterate $f^2$ of an anti-rational
  map $f$ of degree $d$ is a rational map of degree $d^2$.

Let $f: \CC\to \CC$ be an anti-rational or rational map. Just as in the theory of complex dynamics of rational maps
\cite{beardonIterationRationalFunctions1991,
  carlesonComplexDynamics1993,
  steinmetzRationalIteration1993,milnorDynamicsOneComplex2006}, we
define the \emph{Fatou set} $F_f$ of $f$ as the set of points in $\CC$ where the
family of iterates $(f^n)$ is normal, the \emph{Julia set} $J_f$ as
its complement, the set of \emph{critical points} $C_f$ as the set of
points where the local degree of $f$ is larger than one, the set of
\emph{critical values} as $f(C_f)$, and the \emph{postcritical set}
$P_f := \bigcup\limits_{n=1}^\infty f^n(C_f)$ as the forward orbit of
the critical set. We say that $f$ is \emph{postcritically finite} if
$P_f$ is a finite set. 

It is easy to see that $J_{f^2} = J_f$,
$F_{f^2}= F_f$, $C_{f^2} = C_f \cup f^{-1}(C_f)$, and $P_{f^2} = P_f$,
so iteration of an anti-rational map $f$ is in some sense equivalent to iteration of the
rational map $f^2$. However, second iterates of anti-rational maps form a rather special subclass of rational maps, so anti-rational maps have some special properties of their own.

Whereas a rational map of degree $d\geq 2$ always has $d+1$ fixed points,
when counted with multiplicity, the situation is more complicated for
anti-rational maps. The following is well known and has been proved
using the harmonic argument principle in
\cite{khavinsonNumberZerosCertain2006}
and using the Lefschetz fixed point theorem in
\cite{geyerSharpBoundsValence2008}.

\begin{theorem}
  \label{thm:counting-fixed-points}
  Let $f$ be an anti-rational map of degree $d \ge 2$ with a total of $N_{attr}$
  attracting and superattracting fixed points, with $N_{rep}$ repelling
  fixed points, and with no indifferent fixed points. Then $N_{rep}-N_{attr}
  = d-1$. In particular, the total number of fixed points of $f$ is $N =
  N_{rep}+N_{attr} = 2N_{attr}+d-1$.
\end{theorem}

It should be noted that we will restrict ourselves to maps with only
superattracting and repelling periodic points, so that we can ultimately ignore
attracting periodic points. However, in the interest of giving a slightly more comprehensive background of results in
anti-holomorphic dynamics, we also include results on attracting fixed
points.

The following orientation-reversing versions of the local normal
forms as provided by Koenigs' and Böttcher's theorems in complex
dynamics are straightforward consequences of the corresponding
theorems for analytic maps.
\begin{theorem}\label{thm:koenigs}\cite{nakaneMulticornsUnicornsAntiholomorphic2003}
  Let $f$ be an anti-analytic map defined in a neighborhood of a fixed
  point $z_0$ with real multiplier $L \notin \{ 0,1\}$. Then there
  exists an analytic map $\phi$ defined in a neighborhood
  of $z_0$, with $\phi(z_0)=0$ and $|\phi'(z_0)|=1$, such that
  \[
    \phi(f(z)) = L \overline{\phi(z)}
  \]
  in a neighborhood of $z_0$.
\end{theorem}

\begin{theorem}\label{thm:boettcher}\cite{nakaneConnectednessTricorn1993}
  Let $f$ be an anti-analytic map defined in a neighborhood of a
  superattracting fixed point $z_0$ with local expansion $f(z) = z_0 +
  a_k \overline{(z-z_0)}^k + O((z-z_0)^{k+1})$, where $k \ge 2$ and
  $a_k \ne 0$. Then there exists an analytic map $\phi$ defined in a
  neighborhood of $z_0$, with $\phi(z_0)=0$ and $\phi'(z_0) \ne 0$,
  such that
  \[
    \phi(f(z)) = \overline{\phi(z)}^k
  \]
  in a neighborhood of $z_0$.
\end{theorem}
We will call the conjugating maps $\phi$ as above \emph{Koenigs maps} (in the attracting
and repelling cases) and \emph{Böttcher maps} (in the superattracting case). 

Suppose now that $f$ is an anti-rational map. The \emph{basin of attraction} of an attracting or superattracting fixed point $z_0$ is the set of all points in $\CC$ that converge to $z_0$ under iteration of $f$. The \emph{immediate basin} $A_f(z_0)$ is defined to be the connected component of the basin of attraction that contains $z_0$.

One can use the functional equations from Theorems~\ref{thm:koenigs}
and \ref{thm:boettcher} to try to extend the respective conjugating
maps $\phi$ to the whole basin of attraction of the fixed point. As in
the rational case, one may then deduce the following two results.

\begin{theorem}
  \label{thm:attracting-basin}
  Let $f$ be an anti-rational map of degree $d \ge 2$ with an
  attracting fixed point $z_0$. Then the immediate basin $A_f(z_0)$
  must contain a critical point of $f$. If it contains only one
  critical point of $f$, then it is simply connected.
\end{theorem}
\begin{theorem}
  \label{thm:superattracting-basin}
  Let $f$ be an anti-rational map of degree $d \ge 2$ with a
  superattracting fixed point~$z_0$. If the immediate basin $A_f(z_0)$
  contains no other critical point besides $z_0$, then it is simply
  connected and the Böttcher map $\phi$ extends to a conformal map
  from $A_f(z_0)$ onto the unit disk $\D$.
\end{theorem}
The following weak version of a theorem by Fatou is an immediate
corollary.
\begin{corollary}
  An anti-rational map $f$ of degree $d \ge 2$ can have at most $2d-2$
  attracting or superattracting fixed points.
\end{corollary}

In this paper, we are mainly interested in \emph{critically fixed}
anti-rational maps, i.e., those anti-rational maps $f$ for which all
points in $C_f$ are fixed. It is easy to show that the second
iterate $f^2$ of such a map $f$ is never critically fixed,
unless $f$ is conjugate to a map of the form
$z \mapsto \bar{z}^{d}$. The next result summarizes some basic facts about the dynamics of critically fixed
anti-rational maps.

\begin{theorem}
  \label{thm:crit-fixed-basic}
  Let $f$ be a critically fixed anti-rational map of degree $d \ge
  2$. Then all periodic points of $f$ are either superattracting or
  repelling, every orbit in the Fatou set
 converges to a superattracting fixed point, all Fatou components of $f$ are simply connected, and the
  Julia set $J_f$ is connected.
\end{theorem}
\begin{proof}
  Since $f$ is critically fixed, the rational map $g=f^2$ has the same
  critical values as $f$, and so every critical value of $g$ is a
  superattracting fixed point of $g$. By standard results of complex
  dynamics \cite{milnorDynamicsOneComplex2006}, this shows that $g$
  cannot have any attracting or indifferent periodic points (because
  such periodic points require an infinite critical orbit), which
  implies that $f$ cannot have any attracting or indifferent periodic
  points either. Similarly, since Sullivan's classification of Fatou
  components applied to $g$ implies that every orbit in the Fatou set
  of $g$ must converge to a superattracting fixed point of $g$, the
  same is true for $f$, because the Fatou sets (and superattracting
  fixed points) of $f$ and $g=f^2$ coincide.

  By Theorem~\ref{thm:superattracting-basin}, the immediate basin of
  every superattracting fixed point of $f$ is simply connected. For
  any other Fatou component $U$, there exists a minimal $n \ge 1$ such
  that $V = f^n(U)$ is the immediate basin of a superattracting fixed
  point of $f$. Under our assumptions, $f^n: U\to V$ is a branched
  covering with no critical points in $U$, so $f^n$ is a homeomorphism
  from $U$ onto the simply connected domain $V$, and thus $U$ is
  simply connected as well.  Finally, since all Fatou components of
  $f$ are simply connected, the Julia set $J_f$ must be connected.
\end{proof}

\section{Thurston theory for orientation-reversing maps}
\label{sec:thurst-theor-orient}

\subsection{Background and setup}
\label{sec:background-setup}

\emph{Thurston maps} are topological versions of postcritically
finite rational maps, and since their introduction by William Thurston, they
have been studied extensively in complex dynamics; see, e.g.,
\cite{douadyProofThurstonsTopological1993,
  pilgrimCombiningRationalMaps1998,
  pilgrimCanonicalThurstonObstructions2001,
  haissinskyCoarseExpandingConformal2009,
  bonnotThurstonEquivalenceRational2012,
  selingerThurstonsPullbackMap2012,
  selingerTopologicalCharacterizationCanonical2013,
  buffTeichmullerSpacesHolomorphic2014,
  hubbardTeichmullerTheoryApplications2016,
  thurstonRubberBandsRational2016, bonkExpandingThurstonMaps2017,
  thurstonElasticGraphs2019}. 
  In this section, we define
orientation-reversing versions of Thurston maps and of their combinatorial
equivalence, and show that the analog of Thurston's characterization
of rational maps still holds in the orientation-reversing setting. We also introduce the notions of canonical obstruction and decomposition and summarize their properties. We remark that \cite{lodgeDynamicalGasketsGenerated2023}
implicitly contains the same analog of Thurston's characterization result, even though the explicitly stated
theorem in that paper appears somewhat weaker. An orientation-reversing version of Thurston's characterization has also appeared in \cite{belkThurston2023}. In \cite{belkThurston2023}, the authors provide a unified proof of the Nielsen--Thurston classification theorem from the theory of mapping class groups and of Thurston's characterization of (anti-)rational maps, using the ideas of Bers' proof of the former result \cite{bersExtremalProblem1978,farbPrimerMappingClass2012}. In contrast, our arguments here are based on the ideas and results from \cite{douadyProofThurstonsTopological1993,pilgrimCanonicalThurstonObstructions2001,selingerThurstonsPullbackMap2012,selingerTopologicalCharacterizationCanonical2013,buffTeichmullerSpacesHolomorphic2014} on Thurston's characterization in the orientation-preserving case. Although many of the proofs there can be adapted directly to the orientation-reversing setting, verifying all the details of such an adaptation would be rather tedious. Therefore, whenever possible, we instead deduce our results from the classical orientation-preserving ones by passing to the second iterate.  

Let $f:S^2 \to S^2$ be a branched covering.  The \emph{critical set}
$C_f$ of $f$ is the set of its branch points, i.e., the set of points
where $f$ is not a local homeomorphism.  It is well known that there
is an integer $d \ge 1$ such that every $w \notin f(C_f)$ has $d$
preimages under $f$, and we say that $d$ is the \emph{degree} of $f$ and denote it by $\deg(f)$. We also denote by $\deg(f,q)$ the \emph{local degree} of $f$ at a point $q\in S^2$. Then $\deg(f)=\sum_{q\in f^{-1}(\{p\})} \deg(f,q)$ for all $p\in S^2$. Just as in the case of rational maps, we define the
\emph{postcritical set} $P_f := \bigcup\limits_{n=1}^\infty f^n(C_f)$
as the forward orbit of the critical set, and we say that $f$ is
\emph{postcritically finite} if $P_f$ is a finite set. 

\begin{definition}
  A \emph{marked Thurston map} $(f,Q)$ is an orientation-preserving
  branched covering $f:S^2 \to S^2$ of degree $d \ge 2$ together with
  a finite forward-invariant set $Q \subset S^2$ containing the
  postcritical set $P_f$. A \emph{marked anti-Thurston map} $(f,Q)$
  is an orientation-reversing branched covering $f:S^2 \to S^2$ with
  the same properties.  An (\emph{anti-})\emph{Thurston map} is a marked
  \antiThurston map with $Q = P_f$; we will also call such \antiThurston maps \emph{unmarked}. 
\end{definition}

Marked Thurston maps were introduced in
\cite{buffTeichmullerSpacesHolomorphic2014}, where the
generalization of Thurston's characterization to these maps was
proved. Marked anti-Thurston maps are simply the orientation-reversing
variants of these maps. It follows directly from the definition that
such maps are postcritically finite.

\begin{definition}\label{def:Thurston-equiv}
  Two marked \antiThurston maps $(f,Q_f)$ and $(g,Q_g)$ are
  \emph{(combinatorially) equivalent} if there exist
  orientation-preserving homeomorphisms $\phi,\psi:S^2 \to S^2$
  such that $\phi(Q_f) = \psi(Q_f) = Q_g$ and
  $\phi \circ f = g \circ \psi$, and such that $\phi$ and $\psi$ are
  isotopic relative to $Q_f$. We say that a marked \antiThurston map
  is \emph{realizable} if it is equivalent to a marked (anti-)rational
  map.
\end{definition}

\begin{remark}
  It is an exercise in homotopy lifting techniques to show that
  equivalence between $f$ and $g$ implies equivalence between the
  iterates $f^n$ and $g^n$ for all $n\geq 1$; see
  \cite[Cor.~11.6]{bonkExpandingThurstonMaps2017}. It is also straightforward to prove that  every marked \antiThurston map $(f,Q_f)$ with $\#Q_f\leq 3$ is equivalent to a marked \antiRational map.
\end{remark}

\begin{definition}\label{def:Thurston-isot}
  Two marked \antiThurston maps $(f,Q_f)$ and $(g,Q_g)$ are called
  \emph{isotopic} if $Q_f = Q_g$ and there exist
  $\phi, \psi\in \Homeo_0^+(S^2,Q_f)$ such that
  $\phi \circ f = g \circ \psi$.
\end{definition}

\begin{remark}
  By homotopy lifting, two marked \antiThurston maps $(f,Q_f)$ and
  $(g,Q_g)$ are isotopic if and only if $Q_f = Q_g$ and
  $f = g \circ \psi$ for some $\psi\in \Homeo_0^+(S^2,Q_f)$.
\end{remark}

Let $(f,Q)$ be a marked Thurston or anti-Thurston map with
postcritical set $P_f \subseteq Q$. The \emph{orbifold $\cO_f$
  associated with $f$} is the topological orbifold with underlying
space $S^2$ and cone points at every $z \in P_f$ of order $\nu(z)$,
where $\nu(z)$ is the least common multiple of the local degrees of
$f^n$ at all points $w \in f^{-n}(\{z\})$, over all $n \ge 1$.  Note
that $2\leq \nu(z)\leq \infty$ for all $z\in P_f$. Furthermore,
$ \nu(z)= \infty$ if and only if $z$ lies in a periodic cycle containing
a critical point, and we use the common convention that a cone point
of order $\infty$ is viewed as a puncture in $\cO_f$. The
\emph{signature} of the orbifold $\cO_f$ is the tuple
$(\nu(z_1), \nu(z_2), \ldots, \nu(z_m))$, $m=\# P_f$, of the orders of
all its cone points, where we adopt the convention to list them in
increasing order, i.e.,
$\nu(z_1) \le \nu(z_2) \le \ldots \le \nu(z_m)$.

The \emph{Euler characteristic} of $\cO_f$ is 
\[
  \chi(\cO_f) := 2-\sum\limits_{z \in P_f}
  \left(1-\frac{1}{\nu(z)}\right).
\] 
In our context we always have $\chi(\cO_f)\le 0$, and in the case
$\chi(\cO_f) = 0$ the map $f$ is an orbifold covering map
\cite[Prop.~9.1]{douadyProofThurstonsTopological1993}. (The proof
given in the reference does not use the fact that $f$ is
orientation-preserving and applies both to Thurston and anti-Thurston
maps.)  We say that $\cO_f$ is \emph{hyperbolic} if $\chi(\cO_f)<0$,
and \emph{parabolic} otherwise.  There are exactly 6 possible
parabolic signatures: $(\infty, \infty)$, $(2,2,\infty)$, $(2,4,4)$,
$(2,3,6)$, $(3,3,3)$, and $(2,2,2,2)$. All of them except the last one
induce a unique complex structure on $S^2$, and thus the corresponding
unmarked \antiThurston map is always equivalent to an (anti-)rational
map. In fact, one can list all possible equivalence classes of such
anti-Thurston maps similarly to the orientation-preserving case, as in
\cite[Prop.~9.2]{douadyProofThurstonsTopological1993}.  In the
following, a marked \antiThurston map whose associated orbifold has
signature $(2,2,2,2)$ will be called a \emph{$(2,2,2,2)$-map}.  Note
that these are the only maps with parabolic orbifold and more than 3
postcritical points.

In the particular case where all critical points of $f$ are periodic,
the orbifold $\cO_f$ is just the topological surface
$S^2 \setminus P_f$, and it is hyperbolic if and only if $P_f$
contains at least 3 points.  Most \antiThurston maps have hyperbolic
orbifolds, and in our applications we will not consider 
anti-rational maps with
non-hyperbolic orbifolds, except for the simple examples
$f(z)=\bar{z}^d$ 
($d
\geq 2$), with the parabolic orbifold
$\C^* := \C \setminus \{0\}$ of signature $(\infty,\infty)$.

(Marked) Thurston maps are a very flexible tool for constructing
postcritically finite branched coverings of $S^2$ with prescribed dynamical behavior, and it
is of great interest to know whether a given Thurston map is
equivalent to a rational map.  
In order to state Thurston's characterization from
\cite{douadyProofThurstonsTopological1993} and
\cite{buffTeichmullerSpacesHolomorphic2014}, as well as its orientation-reversing version (see \Cref{thm:anti-thurston}), we first have to introduce
the notion of \emph{multicurve obstructions}. 
We should note that our
definition differs slightly from the original one in that we do not
require our obstructions to be \emph{stable}, i.e., backward
invariant. However, we will show below in
Lemma~\ref{lem:obstruction-types} that a marked \antiThurston map has a
stable obstruction if and only if it has an obstruction, so this
distinction is immaterial.

In the following, let $(f,Q)$ be a marked Thurston or anti-Thurston
map.  In this context, ``curves'' will always mean simple closed
(unoriented) curves in $S^2 \setminus Q$, and we will adopt the
convention that \emph{homotopy} will always be understood as free
homotopy in $S^2 \setminus Q$.

A simple closed curve $\gamma \subset S^2 \setminus Q$ is \emph{essential}
if each component of $S^2 \setminus \gamma$ contains at least two
points from $Q$, is \emph{peripheral} if one component of $S^2
\setminus \gamma$ contains exactly one point from $Q$, and is
\emph{null-homotopic} otherwise. A \emph{multicurve} in $S^2 \setminus
Q$ is a (possibly empty) tuple $\Gamma=(\gamma_1, \ldots, \gamma_m)$ of pairwise
disjoint and pairwise non-homotopic essential simple closed curves in
$S^2 \setminus Q$.  

We say that a multicurve $\Gamma$ is \emph{stable} if for all
$\gamma \in \Gamma$, every essential component of $f^{-1}(\gamma)$ is
homotopic to an element of $\Gamma$. The multicurve $\Gamma$ is called
\emph{completely invariant} if it is stable and if, in addition, for
every $\gamma' \in \Gamma$ there exists $\gamma \in \Gamma$ such that
$\gamma'$ is homotopic to a component of $f^{-1}(\gamma)$.

\begin{remark}
  Our notion of ``completely invariant'' multicurves follows
  \cite{selingerTopologicalCharacterizationCanonical2013}. Originally,
  this property was introduced in
  \cite{pilgrimCombinationsComplexDynamical2003}, where it was simply
  called ``invariant''. Some authors use ``invariant'' for the
  property that we call ``stable'', so our terminology is intended to reduce possible confusion between these two notions.
\end{remark}

Associated with any (non-empty) multicurve $\Gamma=(\gamma_1, \ldots, \gamma_m)$ is the
\emph{Thurston linear map} $L = L_{f,Q,\Gamma}: \R^\Gamma \to
\R^\Gamma$ with the \emph{Thurston matrix} $A =A_{f,Q,\Gamma}=
(a_{jk})_{j,k=1}^m$ defined as follows. Let $(\gamma_{j,k,\alpha})_\alpha$ be
the components of $f^{-1}(\gamma_k)$ homotopic to $\gamma_j$, and let
$d_{j,k,\alpha}=\deg(f: \gamma_{j,k,\alpha}\to \gamma_k)$ be the
degree of $f$ on each $\gamma_{j,k,\alpha}$. Then
\begin{equation*}
  L(\gamma_k) = \sum_{j=1}^m a_{jk} \gamma_j  \qquad \text{with}
  \qquad a_{jk} =  \sum_{\alpha} \frac{1}{d_{j,k,\alpha}} \ge 0.
\end{equation*}
In the definition of $L$ and $A$, we ignore any components of
$f^{-1}(\gamma_k)$ that are not essential or not homotopic to a
curve in $\Gamma$. Permuting the elements in $\Gamma$ clearly produces 
matrices that are conjugate by permutation matrices. Since we will only be interested in the equivalence class of $A$ under this conjugacy relation, with a slight abuse of notation, we may sometimes treat $\Gamma$ as a set of curves rather than an ordered tuple.
We also note that the linear map $L$ and the matrix $A$ depend only on
the homotopy classes of the curves in $\Gamma$.

By the Perron--Frobenius theorem, $A$ has a real non-negative leading 
eigenvalue $\lambda(f, Q, \Gamma)$, so that all other
eigenvalues $\lambda'$ satisfy $|\lambda'| \le \lambda(f, Q, \Gamma)$.

\begin{definition}
  Let $(f,Q)$ be a marked \antiThurston map.  A \emph{(multicurve)
    obstruction} for $(f,Q)$ is a multicurve $\Gamma$ in $S^2
  \setminus Q$ with $\lambda(f,Q,\Gamma) \ge 1$. It is
  \emph{irreducible} if the associated matrix $A = A_{f,Q,\Gamma}$ is
  irreducible, i.e., if no permutation of the curves in
  $\Gamma$ puts the matrix $A$ in the block lower-triangular form $
  \begin{bmatrix}
    A_{11} & 0 \\
    A_{21} & A_{22} \\
  \end{bmatrix}$. It is \emph{simple} if no permutation of the curves in
  $\Gamma$ puts the matrix $A$ in the block form $
  \begin{bmatrix}
    A_{11} & 0 \\
    A_{21} & A_{22} \\
  \end{bmatrix}
  $ with $A_{11}$ having leading eigenvalue strictly less than 1. 
  We say that $(f,Q)$ is \emph{obstructed} if there exists a multicurve obstruction for $(f,Q)$; otherwise, we say that $(f,Q)$ is \emph{unobstructed}.
\end{definition}

Simple obstructions were introduced in
\cite{pilgrimCanonicalThurstonObstructions2001}; see also
\cite{selingerTopologicalCharacterizationCanonical2013}.  
In \cite{douadyProofThurstonsTopological1993}, only stable
obstructions are considered. However, the following shows that this
distinction is immaterial.

\begin{lemma}
  \label{lem:obstruction-types}
  Let $(f,Q)$ be a marked \antiThurston map. Then every obstruction
  for $(f,Q)$ contains an irreducible obstruction, and every
  irreducible obstruction for $(f,Q)$ is contained in a completely
  invariant simple obstruction. In particular, if $(f,Q)$ has an
  obstruction, then it has both a completely invariant (and thus also
  stable) obstruction and an irreducible (and thus also simple)
  obstruction.
\end{lemma}
\begin{proof}
  Perron--Frobenius theory (see, e.g., \cite[Vol.~2,
  Chap.~XIII]{gantmacherTheoryMatricesVols1959}) implies that any
  obstruction $\Gamma$ for $(f,Q)$ contains an irreducible
  submulticurve $\Gamma'$ with $\lambda(f,Q,
  \Gamma')=\lambda(f,Q,\Gamma)\geq 1$. To prove the second part,
  suppose that $\Gamma_0$ is an irreducible obstruction for $(f,Q)$.
  Let $\Gamma_n$, $n\geq 0$, be a multicurve in $S^2\setminus Q$
  obtained from $f^{-n}(\Gamma_0)$ by removing the non-essential
  curves and identifying the homotopic ones. Note that for every
  $\gamma \in \Gamma_0$ there is a component of $f^{-1}(\Gamma_0)$
  that is homotopic to $\gamma$; for otherwise, the $\gamma$-row of
  the matrix $A_{f,Q,\Gamma_0}$ consists of zeros, and $A_{f,Q,\Gamma_0}$
  is not irreducible. It follows that we may assume that
  $\Gamma_0\subseteq \Gamma_1$, and, by using induction, that
  $\Gamma_n\subseteq \Gamma_{n+1}$ for all $n\geq 0$. Since the curves
  in every $\Gamma_n$ are pairwise non-homotopic and $Q$ is finite,
  the multicurves $\Gamma_n$ must stabilize, that is, we must have
  $\Gamma_N=\Gamma_{N+1}=\Gamma_{N+2}=\dots$ for some sufficiently
  large $N$. It is immediate that the multicurve $\Gamma_N$ is
  completely invariant. To show that $\Gamma_N$ is also simple, assume that
  some permutation of the curves in $\Gamma_N$ puts the associated
  matrix $A_{f,Q,\Gamma_N}$ in the block lower-triangular form $
  \begin{bmatrix}
    A_{11} & 0 \\
    A_{21} & A_{22} \\
  \end{bmatrix}$.
  Let $\Gamma' \subset \Gamma_N$ be the non-empty submulticurve that
  corresponds to the block $A_{11}$. Suppose that $\Gamma'$ contains a
  curve $\gamma_k\in \Gamma_k$ with $k\geq 1$. By construction, there
  is a curve $\gamma_{k-1} \in \Gamma_{k-1}$ such that one of the
  components of $f^{-1}(\gamma_{k-1})$ is homotopic to
  $\gamma_k$. Thus $\gamma_{k-1}$ must be in $\Gamma'$, and by reverse
  induction it follows that $\Gamma'$ must contain a curve from
  $\Gamma_0$. Since $\Gamma_0$ is irreducible, we have
  $\Gamma_0\subseteq \Gamma'$, and therefore the leading eigenvalue of
  $A_{11}$ is at least $\lambda(f,Q,\Gamma_0)\geq 1$. This finishes
  the proof of the lemma.
\end{proof}

One particularly important type of obstruction is a \emph{Levy cycle},
defined as follows.

\begin{definition}
  Let $(f,Q)$ be a marked \antiThurston map.  A multicurve $\Gamma =
  (\gamma_1, \ldots, \gamma_m )$ in $S^2 \setminus Q$ is called a
  \emph{Levy cycle} for $f$ if, for all $k=1,\ldots,m$, there exists a
  component $\gamma_k'$ of $f^{-1}(\gamma_k)$ that is homotopic to
  $\gamma_{k-1}$ (with the convention $\gamma_0 = \gamma_m$) and such
  that $\deg(f\colon \gamma'_k \to \gamma_k)=1$. If a Levy cycle
  $\Gamma$ consists of a single curve $\gamma$, we call $\gamma$ a
  \emph{fixed Levy curve}.
\end{definition}
It is easy to see that Levy cycles are irreducible obstructions. They
were first introduced and studied in
\cite{levyCriticallyFiniteRational1985}, where it was shown that every obstructed topological polynomial (and every obstructed Thurston map of degree $2$) has a Levy cycle. However, this is no longer true for
general Thurston maps of degree $\ge 3$.

We close this subsection by finally stating Thurston's topological
characterization of marked rational maps
\cite{buffTeichmullerSpacesHolomorphic2014}; see
\cite{douadyProofThurstonsTopological1993} for the unmarked case.
\begin{theorem}\label{thm:Thurston_theorem}
  Let $(f,Q_f)$ be a marked Thurston map that is not a
  $(2,2,2,2)$-map. Then $(f,Q_f)$ is equivalent to a marked rational
  map $(g,Q_g)$ if and only if $(f,Q_f)$ does not have a stable
  obstruction. Moreover, if $(g',Q_{g'})$ is another marked rational
  map equivalent to $(f,Q_f)$, then there exists a Möbius
  transformation $\theta $ with $\theta(Q_g)=Q_{g'}$ that conjugates 
  $g$ and $g'$, i.e., such that $\theta\circ g \circ \theta^{-1} =
  g'$.
\end{theorem}

\subsection{Induced map on Teichmüller space}
\label{sec:thurstons-theorem}

The proof of Thurston's theorem (\Cref{thm:Thurston_theorem}) employs
the iteration of the induced \emph{Thurston pullback map}
$\sigma_f:\cT_f \to \cT_f$ on a Teichmüller space $\cT_f$ associated
with a Thurston map $f$.  In order to prove its orientation-reversing
version (see \Cref{thm:anti-thurston}), we first have to define this
induced map $\sigma_f$ on the corresponding Teichmüller space in the
orientation-reversing setting. We will assume that the reader is
familiar with quasiconformal mappings and the Teichmüller theory of
Riemann surfaces, as well as their applications to Thurston maps as
used in \cite{douadyProofThurstonsTopological1993}. Good expositions
of this material can be found in
\cite{buffTeichmullerSpacesHolomorphic2014} and
\cite{hubbardTeichmullerTheoryApplications2016}.

We call a map $f:\CC \to \CC$ \emph{quasiregular} if it is a
non-constant map of the form $f=g \circ \phi$ with $g: \CC\to \CC$
rational and $\phi: \CC \to \CC$ quasiconformal.  We say that $f:\CC
\to \CC$ is \emph{anti-quasiregular} if its complex conjugate
$\bar{f}$ is quasiregular, or equivalently if the map
  $z\mapsto f(\bar z)$ is quasiregular.  
  It is well known and not too
hard to show that every marked \antiThurston map is isotopic to a marked (anti-)quasiregular \antiThurston map. For an idea of the proof, see
\cite[Exerc.~10.6.7]{hubbardTeichmullerTheoryApplications2016}.

We say that a quasiconformal map $\phi: \CC \to \CC$ is
\emph{normalized} if it fixes $0$, $1$, and $\infty$, and we denote
the space of all such normalized quasiconformal maps by $\QC$.  For a
quasiconformal or quasiregular map $f$, we denote its \emph{complex
  dilatation} by $\mu_f = \bar{\partial} f / \partial f$, and
equalities involving complex dilatations are always understood to hold
almost everywhere. By the measurable Riemann mapping theorem, for
every Lebesgue measurable function $\mu: \CC\to \C$ with $\| \mu
\|_\infty < 1$ there exists a unique $\phi \in \QC$ with $\mu_\phi =
\mu$.

In the following, the \antiThurston maps we consider are always
defined on the standard Riemann sphere $\CC$, equipped with the
standard complex structure.  Suppose $(f,Q_f)$ is a marked
(anti-)quasiregular \antiThurston map with
$Q_f\supseteq \{0,1,\infty\}$. We call a combinatorial equivalence
$\phi \circ f = g \circ \psi$ (see \Cref{def:Thurston-equiv}) between
$(f,Q_f)$ and a marked (anti-)rational map $(g,Q_g)$ \emph{normalized}
if the homeomorphisms $\phi$ and $\psi$ are normalized quasiconformal
maps, that is, if $\phi, \psi \in \QC$.  Note that the existence of
such a normalized equivalence implies
$\{0, 1, \infty\} \subseteq \phi(Q_f)= Q_g$. We also remark that
considering only normalized equivalences is not too
restrictive. Indeed, suppose $(f,Q_f)$ is equivalent to a marked
\antiRational map $(g,Q_g)$ via a not necessarily normalized
equivalence $\phi \circ f = g \circ \psi$.  First, up to modifying
$\phi$ by an isotopy rel.\ $Q_f$ so that it is quasiconformal and
then replacing $\psi$ by the corresponding isotopic lift, we may
assume that both $\phi$ and $\psi$ are quasiconformal. Let $\theta$ be
the unique Möbius transformation such that $\theta \circ \phi$ fixes
0, 1, and $\infty$. If we now define
$g_0 := \theta \circ g \circ \theta^{-1}$,
$\phi_0 := \theta \circ \phi$, and $\psi_0 := \theta \circ \psi$, then
$g_0$ is Möbius conjugate to $g$, and $(f,Q_f)$ is equivalent to
$(g_0,\theta(Q_g))$ via the normalized equivalence
$\phi_0 \circ f = g_0 \circ \psi_0$.

A \emph{marked Riemann sphere} is a pair $(\CC, Q)$, where $Q$ is a finite subset of $\CC$. We will always assume
that $Q$ contains at least three points, and for convenience we will also 
assume that $\{0,1,\infty\} \subseteq Q$. We denote the complex
conjugation map by $\iota(z) = \bar{z}$, and write $Q^* := \iota(Q)$
for the complex conjugates of the points in $Q$. The marked Riemann
sphere $(\CC, Q^*)$ is the \emph{mirror image} of $(\CC, Q)$, and
$\iota$ is a canonical (anti-conformal) isomorphism between
$(\CC, Q)$ and $(\CC, Q^*)$.

For a marked Riemann sphere $(\CC, Q)$, the \emph{Teichmüller
  equivalence relation} $\overset{Q}{\sim}$ on the space of
quasiconformal maps on $\CC$ is defined by declaring $\phi
\overset{Q}{\sim} \psi$ if there is a Möbius transformation $\theta$
such that $\theta \circ \phi$ and $\psi$ are isotopic relative to $Q$;
it is easy to see that the isotopy can be assumed to be through
quasiconformal maps. The associated \emph{Teichmüller space} $\cT_Q$
is the space of Teichmüller equivalence classes $[\phi]$ of
quasiconformal maps $\phi: \CC\to\CC$, or equivalently the space of
all normalized quasiconformal maps $\phi\in \QC$ modulo isotopy
relative to $Q$. This space carries a natural complex analytic structure
\cite[Thm.~6.5.1]{hubbardTeichmullerTheoryApplications2006} and a
natural metric, the \emph{Teichmüller metric} $d$, which makes
$\cT_Q$ a complete metric space \cite[Prop.~and
Def.~6.4.4]{hubbardTeichmullerTheoryApplications2006}.

The standard (orientation-preserving) setup for proving Thurston's
theorem is as follows. Given a marked quasiregular Thurston map
$(f,Q)$ with $Q\supseteq \{0,1,\infty\}$ and $\phi \in \QC$, there is
a unique $\psi \in \QC$ that makes the map $g := \phi \circ f \circ
\psi^{-1}$ analytic (this follows from the measurable Riemann mapping theorem applied to $\mu_{\phi\circ f}$). One useful way to think about the induced map $\phi \mapsto
\psi$ on $\QC$ is as a pullback of conformal structures on $\CC$,
where $\phi$ and $\psi$ provide the global charts. This pullback
respects the Teichmüller equivalence relation $\overset{Q}{\sim}$ on $\QC$, and hence 
it induces a map $\sigma = \sigma_{f,Q}: \cT_Q \to \cT_Q$ on the
corresponding Teichmüller space, called \emph{Thurston's pullback
  map}. The map $\sigma$ is analytic (with respect to the natural
complex analytic structure on $\cT_Q$; see \cite[Sec.~1.3]{buffTeichmullerSpacesHolomorphic2014}) and thus does not increase the
Teichmüller distance~$d$, i.e., $d(\sigma(\tau), \sigma(\tau')) \le
d(\tau,\tau')$ for all $\tau,\tau' \in \cT_Q$ (see \cite[Cor.~6.10.7]{hubbardTeichmullerTheoryApplications2006}). Furthermore, if $(f,Q)$
is not a $(2, 2, 2, 2)$-map, then some iterate $\sigma^k$ is a strict
(but not necessarily uniform) contraction, i.e., $d(\sigma^k(\tau),
\sigma^k(\tau')) < d(\tau,\tau')$ for all distinct $\tau,\tau' \in
\cT_Q$; see \cite[Lem.~2.9]{buffTeichmullerSpacesHolomorphic2014}.

\begin{figure}[t]
  \centering
  \begin{tikzcd}
    (\CC, Q^*) \arrow[r, "\iota"] \arrow[d, "\widetilde{\psi}"] \arrow[rr,
    bend left, "\widetilde{f}"] & (\CC, Q)
    \arrow[r, "f"] \arrow[d, "\psi"] & (\CC, Q) \arrow[d, "\phi"] \\ 
    (\CC, \widetilde{\psi}(Q^*)) \arrow[r, "\iota"] \arrow[rr, bend right,
    "\widetilde{g}"] & (\CC, \psi(Q)) \arrow[r, "g"] & (\CC, \phi(Q))
  \end{tikzcd}
  \caption{Orientation-reversing setup. The map $f$ is
    anti-quasiregular, $\widetilde{f}$ is quasiregular, $g$ is
    anti-rational, and $\widetilde{g}$ is rational. The maps $\phi$,
    $\psi$, and $\widetilde{\psi}$ are normalized quasiconformal
    mappings.}
  \label{fig:anti-teich}
\end{figure}

The following lemma will enable us to give a unified definition of
Thurston's pullback map in the orientation-preserving and orientation-reversing
cases.

\begin{lemma}
  \label{lem:anti-sigma}
  Let $(f,Q)$ be a marked (anti-)quasiregular \antiThurston map of
  degree $d$ with $\{ 0,1,\infty \} \subseteq Q$.  Then for every
  $\phi\in \QC$ there exists a unique $\psi \in \QC$ such that $g
  :=\phi \circ f \circ \psi^{-1}$ is an (anti-)rational map of degree
  $d$.  Furthermore, if we have another map $\phi' \in \QC$ with
  $\phi' \overset{Q}\sim \phi$ and a map $\psi' \in \QC$ such that
  $g':= \phi' \circ f \circ (\psi')^{-1}$ is (anti-)rational, then
  $\psi'\overset{Q}\sim \psi$ and $g' = g$.
\end{lemma}
\begin{proof}
  The orientation-preserving case of these (or equivalent) statements
  is covered in \cite{douadyProofThurstonsTopological1993} and
  \cite{buffTeichmullerSpacesHolomorphic2014}, so we will only provide
  the argument in the orientation-reversing case. For the following
  discussion, the commutative diagram in Figure~\ref{fig:anti-teich}
  should be helpful.

  Let $(f,Q)$ be a marked anti-quasiregular anti-Thurston map with
  $\{ 0,1,\infty \} \subseteq Q$. Then $\widetilde{f}:=f
  \circ \iota$ is quasiregular, so there exists a unique
  $\widetilde{\psi} \in \QC$ with complex dilatation $\mu_{\widetilde{\psi}} =
  \mu_{\phi\circ \widetilde{f}}$, or equivalently, such that $\widetilde{g} := \phi \circ
  \widetilde{f} \circ \widetilde{\psi}^{-1}$ is rational. Then $\psi := \iota
  \circ \widetilde{\psi} \circ \iota \in \QC$ is again a normalized
  quasiconformal map, and $g = \phi \circ f \circ \psi^{-1} =
  \widetilde{g} \circ \iota$ is an anti-rational map (of degree~$d$). From the uniqueness of
  $\widetilde{\psi}$, it immediately follows that $\psi = \iota
  \circ \widetilde{\psi} \circ \iota$ is the
  unique map in $\QC$ with this property.

  Now assume we have $\phi' \in \QC$ with $\phi' \overset{Q}\sim
  \phi$. Then an isotopy $\phi_t$, $t\in[0,1]$, from $\phi=\phi_0$ to $\phi' =
  \phi_1$ relative to $Q$ lifts to an isotopy $\psi_t$ relative to $Q$ with $\psi_0
  = \psi$ and $\phi_t \circ f = g \circ \psi_t$ for all $t\in[0,1]$. This shows in
  particular that $\psi_1 \in \QC$ with $\psi_1 \overset{Q}\sim
  \psi$, and that $g = \phi_1 \circ f \circ \psi_1^{-1} = \phi' \circ
  f \circ \psi_1^{-1}$ is anti-rational. By the uniqueness argument
  above, it follows that $\psi'=\psi_1\overset{Q}\sim
  \psi$ and $g' = g$, which completes the proof.
\end{proof}

For the following corollary, let $\Rat_d$ and $\ARat_d$ denote the
spaces of rational and anti-rational maps of degree $d$, respectively.

\begin{corollary}
  \label{cor:sigma-def}
  With the same assumptions and notation as in
  Lemma~\ref{lem:anti-sigma}, the maps 
  \[S_f :\QC
  \to \QC \text{ and } R_f:\QC\to \Rat_d \cup \ARat_d\]
  given by $S_f(\phi) = \psi$ and $R_f(\phi) = g$ induce maps \[\sigma = \sigma_{f,Q} :\cT_Q
  \to \cT_Q \text{ and } \rho = \rho_{f,Q}:\cT_Q \to \Rat_d \cup \ARat_d\]
  defined by $\sigma([\phi]) = [\psi]$ and $\rho([\phi]) = g$.
\end{corollary}

\begin{proof}
  This follows immediately from Lemma~\ref{lem:anti-sigma}. The map
  $S_f$ respects the Teichmüller equivalence relation, while the map
  $R_f$ is constant on Teichmüller equivalence classes.
\end{proof}

In analogy with the classical setting, we will refer to the map $\sigma=\sigma_{f,Q}$ from \Cref{cor:sigma-def} as the \emph{Thurston pullback map} induced by $(f,Q)$. We now record its basic properties.

\begin{proposition}
  \label{prop:anti-sigma}
 Let $(f,Q)$ be a marked (anti-)quasiregular \antiThurston map with $\{ 0,1,\infty \} \subseteq Q$. Then the
  following statements hold true for the induced maps $\sigma =
  \sigma_{f,Q}$ and $\rho = \rho_{f,Q}$:
  
  \begin{enumerate}
  \item \label{it:1} $\sigma$ does not increase Teichmüller distances, i.e.,
    $d(\sigma(\tau), \sigma(\tau')) \le d(\tau, \tau')$ for all
    $\tau, \tau' \in \cT_Q$.
  \item \label{it:2} The map $f \mapsto \sigma_{f,Q}$ commutes with iteration,
    i.e., $\sigma_{f^n,Q} = (\sigma_{f,Q})^n$.
   \item \label{it:3} The maps $\sigma_{f,Q}$ and $\rho_{f,Q}$ depend only on the isotopy class of $(f,Q)$.   
  \item \label{it:4} 
     If $\tau = [\phi]$ is a fixed point of
    $\sigma$, then $\phi \circ f = g\circ \psi$ with $\psi = S_f(\phi)$ and $g=\rho(\tau)$ is a
    normalized combinatorial equivalence between $(f,Q)$ and the marked (anti-)rational map $(g,\psi(Q))$.
  \item \label{it:5} If $\phi \circ f = g \circ \psi$ is a normalized
    combinatorial equivalence between $(f,Q)$ and a marked (anti-)rational map $(g, Q_g)$, then $g = \rho(\tau)$ and 
    $\tau = [\phi]=[\psi]$ is a fixed point of $\sigma$.
  \end{enumerate}
\end{proposition}
\begin{remark}     
 Together with the discussion above, \ref{it:3} means that we may define the Thurston pullback map $\sigma_{f,Q}$ for any marked (anti-)Thurston map $(f,Q)$ by passing to some isotopic (anti-)quasiregular representative. Moreover, \ref{it:4} and \ref{it:5} show that the fixed points of $\sigma_{f,Q}$ are in one-to-one correspondence with the pairs $(g,[\phi])$ defining a normalized combinatorial equivalence between $(f,Q)$ and a marked \antiRational map $(g,Q_g)$. It follows that every Möbius conjugacy class of marked \antiRational maps equivalent to $(f,Q)$ arises from a fixed point of $\sigma_{f,Q}$ in this way. If $(f,Q)$ is not a $(2,2,2,2)$-map, then $\sigma_{f,Q}$ has at most one fixed point (since some iterate of $\sigma_{f,Q}$ is a strict contraction). Therefore, in this case there is at most one marked \antiRational map $(g,Q_g)$ admitting a normalized equivalence with $(f,Q)$, and thus at most one Möbius conjugacy class of marked \antiRational maps equivalent to $(f,Q)$.
\end{remark}

\begin{proof}
  The reason why we include the orientation-preserving case in the statement is that
  even iterates of orientation-reversing maps are
  orientation-preserving. However, the orientation-preserving part of
  this proposition is already proved in
  \cite{douadyProofThurstonsTopological1993} and
  \cite{buffTeichmullerSpacesHolomorphic2014}, so we will assume that
  $f$ is anti-quasiregular in the following.

  \ref{it:1} With the same notation as in the proof of Lemma~\ref{lem:anti-sigma}
  and Figure~\ref{fig:anti-teich}, the map $S_{\widetilde{f}}:\QC\to \QC$, given by $\phi\mapsto \widetilde\psi=\iota\circ \psi\circ \iota$ with $\psi=S_f(\phi)$, is the ordinary pullback of a complex structure by the
  quasiregular map $\widetilde{f}=f\circ \iota$, and it therefore induces an analytic map
  $\sigma_{\widetilde{f}}: \cT_Q \to \cT_{Q^*}$ between the Teichmüller
  spaces of the marked Riemann sphere $(\CC, Q)$ and its mirror image $(\CC,
  Q^*)$. This implies that $\sigma_{\widetilde{f}}$ does not increase the 
  Teichmüller distance (see, for example, \cite[Sec.~6.10]{hubbardTeichmullerTheoryApplications2006}). 
  Consider now the (anti-analytic) isometry $I$ between $\cT_Q$
  and $\cT_{Q^*}$ given by conjugating by the complex conjugation $\iota$, that is,
  $I([\widetilde{\psi}]) = [\iota \circ \widetilde{\psi} \circ \iota] =
  [\psi]$.
  Since
  $\sigma_{f,Q} = I \circ \sigma_{\widetilde{f}}$ by construction, this implies
  that $d(\sigma_{f,Q}(\tau), \sigma_{f,Q}(\tau')) \le d(\tau, \tau')$
  for all $\tau, \tau' \in \cT_Q$.

  \ref{it:2} The argument that the map $f \mapsto \sigma_{f,Q}$ commutes with iteration
  is nearly the same as in the orientation-preserving case. 
  The basic
  idea is to use the characterization of $\sigma_{f,Q}$ given by
  Lemma~\ref{lem:anti-sigma}, combined with the fact that the
  composition of $n$ anti-rational maps is rational or anti-rational,
  depending on whether $n$ is even or odd. For an illustration of the
  argument, see \Cref{fig:sigma-iterate}.

  \ref{it:3}--\ref{it:5} These 
  claims are direct consequences of our definitions
  and results above.
\end{proof}

\begin{figure}[t]
  \centering
  \begin{tikzcd}
    (\CC, Q) \arrow[r, "f"] \arrow[d, "\phi_n"]
    \arrow[rrrr, bend left, "f^n"]
    & \cdots \arrow
    [r,"f"] & (\CC, Q) \arrow[r, "f"] \arrow[d, "\phi_2"]
    & (\CC, Q) \arrow[r, "f"]
    \arrow[d, "\phi_1"] & (\CC, Q) \arrow[d, "\phi_0"] \\
    (\CC, \phi_n(Q)) \arrow[r, "g_n"]
    \arrow[rrrr, bend right, "G_n"]
    &
    \cdots \arrow [r,"g_3"] &
    (\CC, \phi_2 (Q)) \arrow[r, "g_2"]
    & (\CC, \phi_1(Q)) \arrow[r, "g_1"] & (\CC, \phi_0(Q))
  \end{tikzcd}
  \caption{Illustration of the fact that $\sigma_{f^n,Q} =
    \sigma_{f,Q}^n$. Given a point $\tau_0 \in \cT_Q$ represented by a
    normalized quasiconformal map
    $\phi_0$, its iterates $\tau_n = \sigma_{f,Q}^n(\tau_0)$ are
    represented by the unique normalized quasiconformal maps
    $\phi_n$ for which the maps $g_n = \phi_{n-1} \circ f \circ \phi_n^{-1}$
    are anti-rational. This implies that $\phi_n$ is the unique
    normalized quasiconformal map that makes $G_n := g_1 \circ g_2
    \circ \ldots \circ g_n = \phi_0 \circ f^n \circ \phi_n^{-1}$
    rational or anti-rational, depending on whether $n$ is even or
    odd, and thus $\tau_n = \sigma_{f^n,Q}(\tau_0)$.}
  \label{fig:sigma-iterate}
\end{figure}

\subsection{Canonical obstructions}
\label{subsec:canon-obstr}

The notion of a \emph{canonical obstruction} for a Thurston map $f$ was introduced by Pilgrim
in \cite{pilgrimCanonicalThurstonObstructions2001}, and was later
refined by Selinger in \cite{selingerThurstonsPullbackMap2012,
  selingerTopologicalCharacterizationCanonical2013}. It consists of
all homotopy classes of essential curves whose hyperbolic length tends to zero under
iteration of Thurston's pullback map. If the canonical obstruction is
non-empty, it indeed induces a (simple and completely invariant) multicurve
obstruction; see \cite[Thm.~1.1(2)]{pilgrimCanonicalThurstonObstructions2001} and \cite[Thm.~3.3(2)]{selingerThurstonsPullbackMap2012}. If it is empty and $f$ has a hyperbolic orbifold, then
$f$ is always combinatorially equivalent to a rational map; see \cite[Thm.~1.1(1)]{pilgrimCanonicalThurstonObstructions2001} and \cite[Thm.~3.3(1)]{selingerThurstonsPullbackMap2012}.  In
\cite[Sec.~10]{pilgrimCombinationsComplexDynamical2003}, canonical
obstructions were used to define a \emph{canonical decomposition} of Thurston maps. In the orientation-reversing case, all of this still works, and we
summarize the relevant definitions and theorems in the following.

As in \Cref{sec:thurstons-theorem}, we fix an identification of $S^2$ with the Riemann sphere $\CC$ equipped with its standard complex structure, and assume that the set $Q$ of marked points is
finite and contains at least three points. Let us consider a point $\tau=[\phi]\in \cT_Q$ in the corresponding Teichmüller space, and an essential simple closed curve $\gamma$ in $S^2 \setminus Q$. Since $\#Q\ge 3$, the Riemann surface $\CC\setminus\phi(Q)$ carries a unique complete hyperbolic metric, and we denote by $l_\tau(\gamma)$ the length of the unique closed geodesic homotopic to $\phi(\gamma)$ in this metric. Clearly, $l_\tau(\gamma)$ depends only on the homotopy class of $\gamma$ and is independent of the representative $\phi$. We will also need the following property; see \cite[Prop.~7.2]{douadyProofThurstonsTopological1993}. 

\begin{lemma}\label{lem:w_Lipschitz}
    Let $\gamma$ be an essential curve in $S^2 \setminus Q$. Then the function $\tau\in \cT_Q\mapsto w(\gamma,\tau):=-\log l_\tau(\gamma)$ is Lipschitz with respect to the Teichmüller metric $d$ on $\cT_Q$: for all
    $\tau, \tau' \in \cT_Q$, we have
    \[|w(\gamma, \tau) - w(\gamma, \tau')| \le 2 d(\tau, \tau').\] 
\end{lemma}

\begin{definition}
  \label{def:canonical-obstruction}
  Let $(f,Q)$ be a marked \antiThurston map, $\tau_0\in \cT_Q$
  be a basepoint in the associated Teichmüller space, and $\tau_n = \sigma_{f,Q}^n(\tau_0)$ be its iterates under the induced
  Thurston pullback map. The \emph{canonical obstruction} $\GGamma_{f,Q}$ for $(f,Q)$ is the set of all homotopy classes of
  essential curves $\gamma$ in $S^2\setminus Q$ for which
  $l_{\tau_n}(\gamma) \to 0$ as $n \to \infty$. 
\end{definition}

  It easily follows from \Cref{lem:w_Lipschitz} and   \Cref{prop:anti-sigma}\ref{it:1} that the canonical obstruction for $(f,Q)$ is independent of the choice of the basepoint $\tau_0\in\cT_Q$. Moreover, the set $\GGamma_{f,Q}$ is finite, and any two distinct homotopy classes in  $\GGamma_{f,Q}$ admit disjoint representatives. We denote by  $\Gamma_{f,Q}$ any multicurve in $S^2\setminus Q$ consisting of one representative from each homotopy class of the canonical obstruction $\GGamma_{f,Q}$. With a slight abuse of terminology, we refer to this multicurve as a \emph{canonical (multicurve) obstruction} for $(f,Q)$.

The following are orientation-reversing versions of results from
\cite{pilgrimCanonicalThurstonObstructions2001,selingerThurstonsPullbackMap2012,
  selingerTopologicalCharacterizationCanonical2013}.
\begin{lemma}
  \label{lem:canonical-obstruction-iterates}
  Let $(f,Q)$ be a marked \antiThurston map and $k \ge 1$ be an integer. Then the
  canonical obstructions for $(f,Q)$ and its iterate $(f^k,Q)$ coincide.
\end{lemma}
\begin{proof}
  The argument is the same as in the orientation-preserving case, see \cite[Prop.~3.3]{selingerTopologicalCharacterizationCanonical2013}. Let us fix $\tau_0\in \cT_Q$ and consider an essential curve $\gamma\subset S^2\setminus Q$. Writing $\tau_n:=\sigma_{f,Q}^n(\tau_0)$ and $w_n(\gamma):=-\log l_{\tau_n}(\gamma)$, we have by \Cref{lem:w_Lipschitz} and \Cref{prop:anti-sigma} that 
  $|w_{n+1}(\gamma)-w_n(\gamma)| \le 2 d(\tau_{n+1}, \tau_n) \le 2
  d(\tau_1, \tau_0)$.  In particular, this implies that
  $\lim_{n\to\infty} w_{kn}(\gamma) = \infty$ if and only if
  $\lim_{n\to\infty} w_n(\gamma)=\infty$. Since the canonical
  obstruction for $(f,Q)$ consists of the homotopy classes of curves $\gamma$ for which
  $\lim_{n\to\infty} w_n(\gamma) = \infty$, and the one for $(f^k,Q)$
  consists of those for which
  $\lim_{n\to\infty} w_{kn}(\gamma) = \infty$, this shows that the
  canonical obstructions for $(f,Q)$ and $(f^k,Q)$ coincide.
\end{proof}

\begin{theorem}
  \label{thm:canonical-obstruction-is-empty}
  Let $(f,Q)$ be a marked \antiThurston map and $\Gamma_{f,Q}$ be its canonical multicurve obstruction. Then the following statements are true: 
  \begin{enumerate}
  \item \label{item:can_1}
  If $(f,Q)$ is not a $(2,2,2,2)$-map and $\Gamma_{f,Q} = \emptyset$, then $(f,Q)$ is unobstructed.
  
  \item \label{item:can_2} If $\Gamma_{f,Q} \neq \emptyset$, then $\Gamma_{f,Q}$ is a simple and completely invariant obstruction for $(f,Q)$.
  \end{enumerate}
\end{theorem}
\begin{proof}
  The orientation-preserving case is established in \cite{pilgrimCanonicalThurstonObstructions2001,selingerThurstonsPullbackMap2012}; see also \cite[Thm.~3.2.3 and Cor.~4.3.3]{selingerThurstonsCharacterizationTheorem2011}. The orientation-reversing case is an easy consequence. 

  Indeed, let $(f,Q)$ be a marked anti-Thurston map of degree $d\geq 2$, so that $(f^2,Q)$ is a marked Thurston map. By \Cref{lem:canonical-obstruction-iterates} the canonical obstructions $\GGamma_{f,Q}$ and $\GGamma_{f^2,Q}$ coincide, and thus we may also suppose that $\Gamma_{f,Q}=\Gamma_{f^2,Q}$ for the corresponding canonical multicurve obstructions. We now prove each statement separately.

  \ref{item:can_1} Assume that $(f,Q)$ is not a $(2,2,2,2)$-map and has an obstruction $\Gamma$. Then $(f^2,Q)$ is not a $(2,2,2,2)$-map either, and $\Gamma$ is also an obstruction for $(f^2,Q)$. (For the latter, the entry-wise inequality $A_{f^2,Q,\Gamma}\geq (A_{f,Q,\Gamma})^2$ implies that  $\lambda(f^2,Q,\Gamma) \geq (\lambda(f,Q, \Gamma))^2\geq 1$.) By the orientation-preserving case of the statement, it follows that $\Gamma_{f,Q} = \Gamma_{f^2,Q} \ne \emptyset$, which implies~\ref{item:can_1}.
  
  \ref{item:can_2} Let us suppose that
  $\Gamma:=\Gamma_{f,Q}=\Gamma_{f^2,Q} \ne \emptyset$. Then $\Gamma$
  is a simple and completely invariant obstruction for $(f^2,Q)$ by
  the orientation-preserving case. We first show that $\Gamma$ is
  stable for $(f,Q)$. Indeed, let $\tau=[\phi]\in \cT_Q$ and
  $\sigma_{f,Q}(\tau)=[\psi]$ with $g:=\phi\circ f\circ \psi^{-1}$
  anti-rational. Then $g$ is a covering map of degree $d$ between the
  Riemann surfaces
  $\CC\setminus g^{-1}(\phi(Q))\subset \CC\setminus \psi(Q)$ and
  $\CC\setminus \phi(Q)$. It follows that for every essential curve
  $\gamma\subset S^2\setminus Q$ and every essential component
  $\gamma'$ of $f^{-1}(\gamma)$ we have
  $l_{\sigma_{f,Q}(\tau)}(\gamma')\leq d\cdot l_\tau(\gamma)$. This
  implies that for any such curve $\gamma$ with
  $[\gamma] \in \GGamma_{f,Q}$ we also have
  $[\gamma'] \in \GGamma_{f,Q}$, which shows that the multicurve
  $\Gamma$ is stable. Stability of $\Gamma$ now implies that the
  associated Thurston matrix $A=A_{f,Q,\Gamma}$ satisfies
  $A^2=A_{f^2,Q,\Gamma}$, and thus $\Gamma$ is a simple and completely
  invariant obstruction for $(f,Q)$, because $\Gamma$ is a simple and
  completely invariant obstruction for $(f^2,Q)$. Indeed, if for some
  $\gamma' \in \Gamma_{f,Q}$ there is no $\gamma \in \Gamma_{f,Q}$
  such that $\gamma'$ is homotopic to a component of $f^{-1}(\gamma)$,
  then the $\gamma'$-row of $A$, and thus also of $A^2$, consists of
  zeros. This contradicts complete invariance of $\Gamma$ for
  $(f^2,Q)$. The remaining properties of $\Gamma$ also follow
  immediately from the corresponding definitions. This finishes the
  proof.
\end{proof}

We now discuss properties of the canonical obstruction for marked \antiThurston maps $(f,Q)$ whose associated orbifold $\cO_f$ has parabolic signature $(2,2,2,2)$; our exposition essentially follows \cite[Secs.~4 and 5]{selingerTopologicalCharacterizationCanonical2013}, with necessary adjustments for the orientation-reversing case. The orbifold $\cO_f$ is the quotient $\T/\langle \upsilon\rangle$ of a torus $\T$ by a hyperelliptic involution $\upsilon\colon \T \to \T$. If we denote the canonical projection
from $\T$ to $\cO_f$ by $p$,  then $p$ is a 2-to-1 branched covering,
ramified over the four points of the postcritical set $P_f$, which are also the $p$-images of the fixed
points of the involution $\upsilon$. By
\cite[Prop.~4.2]{selingerTopologicalCharacterizationCanonical2013} and
the discussion following it (see also
\cite[Prop.~2.6]{farbPrimerMappingClass2012}), every homotopy class $[\gamma]$ of
simple closed curves in $S^2 \setminus P_f$ defines (via lifting along~$p$) an element $\langle \gamma \rangle$ of the first homology group $H_1(\T, \Z) \cong
\Z^2$, unique up to sign. Conversely, for every $h \in H_1(\T, \Z)$ there exists a simple closed
curve $\gamma$ in $S^2 \setminus P_f$ and an integer $n \in \Z$ such
that $h = n \langle \gamma \rangle$. By
\cite[Lem.~9.4]{douadyProofThurstonsTopological1993} (which also holds
in the orientation-reversing case, with the same proof), the map $f:S^2\to S^2$
lifts to a covering map $\widetilde{f}\colon \T \to \T$ of the same degree $d\geq 2$, which is uniquely defined up to post-composition with the involution~$\upsilon$. The induced push-forward map 
$\widetilde{f}_*\colon H_1(\T, \Z) \to H_1(\T, \Z)$ on homology is then linear with
determinant $d$ in the orientation-preserving case, and $-d$ in the
orientation-reversing case.

The following statement relates arithmetic properties of the eigenvalues of the push-forward map $\widetilde{f}_*$ and topological properties of the canonical multicurve obstruction of $(f,Q)$. 

\begin{proposition}\label{prop:can-obstruction-and-arithmetics}
A canonical multicurve obstruction $\Gamma_{f,Q}$ of a marked $(2,2,2,2)$-map $(f,Q)$ contains a curve that has two postcritical points of $f$ in each complementary
component if and only if $\widetilde{f}_*$ has two integer eigenvalues with distinct absolute values.
\end{proposition}

Note that the two conditions in \Cref{prop:can-obstruction-and-arithmetics} are independent of the corresponding choices (of the multicurve $\Gamma_{f,Q}$ and the lift $\widetilde{f}$).

\begin{proof}
   The orientation-preserving case is proved in \cite[Thm.~5.3]{selingerTopologicalCharacterizationCanonical2013}. Let us suppose that $(f,Q)$ is a marked anti-Thurston $(2,2,2,2)$-map, and  fix lifts $\widetilde{f}$ and $\widetilde{f^2}$ of $f$ and $f^2$ to $\T$, respectively. Clearly, $(\widetilde{f})^2$ is also a lift of $f^2$, and thus $\widetilde{f^2}=(\widetilde{f})^2$ or $\widetilde{f^2}=\upsilon\circ (\widetilde{f})^2$, which means $\big(\widetilde{f^2}\big)_*=\pm\big(\widetilde{f}_*\big)^2$. 

    \begin{claim*}
     $\widetilde{f}_*$ has two integer eigenvalues with distinct absolute values if and only if $\big(\widetilde{f^2}\big)_*$ does.
  \end{claim*}
   
   The proof of the claim  is a routine check using the fact that $\widetilde{f}_*$ has integer coefficients and determinant $-d$, where $d\geq 2$ is the degree of $f$; we leave this verification to the reader. The desired statement now follows from the claim above combined with the orientation-preserving case and \Cref{lem:canonical-obstruction-iterates}.
\end{proof}

The next statement provides a topological characterization of  $(2,2,2,2)$-maps whose canonical obstruction is empty. 

\begin{theorem}\label{thm:can-obstruction-is-empty-parabolic}
  The canonical obstruction of a marked $(2,2,2,2)$-map $(f,Q)$ is empty if
  and only if the following two conditions hold:
  \begin{enumerate}
      \item \label{item:can_par_1} every curve of every simple obstruction for
  $(f,Q)$ has two postcritical points of $f$ in each complementary
  component; and 
    \item \label{item:can_par_2} the two eigenvalues of $\widetilde{f}_*$ are
  non-integer or have the same absolute value.
  \end{enumerate}
\end{theorem}
\begin{proof}
 For Thurston maps this is
  \cite[Cor.~5.4]{selingerTopologicalCharacterizationCanonical2013}. Now let $(f,Q)$ be a marked anti-Thurston $(2,2,2,2)$-map, so that $(f^2,Q)$ is a marked Thurston $(2,2,2,2)$-map. We will write $(\text{i}^\prime)$ and $(\text{ii}^\prime)$ for the corresponding conditions \ref{item:can_par_1} and \ref{item:can_par_2} for $(f^2,Q)$, respectively.

  Suppose first that the canonical multicurve obstruction
  $\Gamma_{f,Q}$ is empty. Then $\Gamma_{f^2,Q}=\emptyset$ by
  \Cref{lem:canonical-obstruction-iterates}, and thus
  $(\text{i}^\prime)$ and $(\text{ii}^\prime)$ hold by the
  orientation-preserving case of the statement. Since
  \ref{item:can_par_2} $\Longleftrightarrow$ $(\text{ii}^\prime)$ by
  the claim in the proof of
  \Cref{prop:can-obstruction-and-arithmetics}, it remains to show
  \ref{item:can_par_1}.
  This will immediately follow from the following claim for general marked \antiThurston maps.
  
  \begin{claim*}
  If a multicurve $\Gamma$ is a simple obstruction for a marked \antiThurston map $(f,Q)$, then it is a simple obstruction for $(f^2,Q)$. 
  \end{claim*}

  We use the following characterization of simple obstructions: a multicurve is a simple obstruction for a given marked \antiThurston map if and only if its associated Thurston matrix $A$ satisfies $Av\geq v$ for some vector $v$ with positive entries; see \cite[Prop.~2.2]{selingerTopologicalCharacterizationCanonical2013} (the proof there also works in the orientation-reversing case). Suppose $\Gamma$ is a simple obstruction for $(f,Q)$, and let $v$ be a vector with  positive entries such that $A_{f,Q,\Gamma}\, v\geq v$. Since  $A_{f^2,Q,\Gamma} \geq  (A_{f,Q,\Gamma})^2$ entry-wise, we have $A_{f^2,Q,\Gamma}\, v \geq (A_{f,Q,\Gamma})^2 \, v = A_{f,Q,\Gamma} \,(A_{f,Q,\Gamma} \, v ) \geq A_{f,Q,\Gamma} \, v \geq v$, and thus $\Gamma$ is also a simple obstruction for $(f^2,Q)$.

\smallskip
  
  To prove the converse, let us suppose that \ref{item:can_par_1} and \ref{item:can_par_2} are true, but $\Gamma_{f,Q}\neq \emptyset$. By \Cref{thm:canonical-obstruction-is-empty}\ref{item:can_2}, $\Gamma_{f,Q}$ is a simple obstruction for $(f,Q)$, and thus by \ref{item:can_par_1} every curve of $\Gamma_{f,Q}$ has two postcritical points of $f$ in each complementary component. 
  \Cref{prop:can-obstruction-and-arithmetics} then implies that $\widetilde{f}_*$ has two integer eigenvalues with distinct absolute values, which contradicts our assumption \ref{item:can_par_2}. This finishes the proof of the theorem.
\end{proof}

\subsection{Orientation-reversing Thurston theorem}
The following theorem is the orientation-reversing version of the
generalized Thurston theorem from
\cite{buffTeichmullerSpacesHolomorphic2014}, allowing for additional
marked points outside the postcritical set, and including some maps
with parabolic orbifolds.

\begin{theorem}
  \label{thm:anti-thurston}
  Let $(f,Q_f)$ be a marked anti-Thurston map with $\#Q_f\geq 3$ that is not a
  $(2,2,2,2)$-map. Then the following are equivalent:
  \begin{itemize}
  \item $(f,Q_f)$ is equivalent to a marked anti-rational map.
  \item $(f,Q_f)$ does not have an obstruction.
  \item The canonical obstruction $\Gamma_{f,Q_f}$ is empty.
  \end{itemize}
  Furthermore, if these equivalent conditions hold, then the equivalent marked anti-rational map
  is unique up to Möbius conjugacy. More precisely, if
  $\{0,1,\infty\} \subseteq Q_f$ and $f$ is anti-quasiregular, then there is exactly one marked
  anti-rational map $(g,Q_g)$ such that there exists a normalized
  combinatorial equivalence between $(f,Q_f)$ and $(g,Q_g)$.
\end{theorem}

Recall that, as in the orientation-preserving case, every marked anti-Thurston map $(f,Q_f)$ with $\#Q_f\leq 3$ is equivalent to a marked anti-rational map. 

\begin{proof}
  From Lemma~\ref{lem:canonical-obstruction-iterates} and Theorem~\ref{thm:canonical-obstruction-is-empty} we know that the canonical
  obstructions for $(f,Q_f)$ and $(f^2,Q_f)$ coincide, and that either map 
  is unobstructed if and only if its canonical obstruction is empty. In particular, if one of the two maps is obstructed, then so is the other. In that case $(f^2,Q_f)$, being an obstructed marked Thurston map, is not equivalent to a marked rational map by \Cref{thm:Thurston_theorem}, and consequently $(f,Q_f)$ cannot be equivalent to a marked anti-rational map.

  It remains to show that $(f,Q_f)$ is equivalent to a marked anti-rational
  map if it is unobstructed.  For this direction we use the properties
  of $\sigma = \sigma_{f,Q_f}$ given in \Cref{prop:anti-sigma}.
  By the above, $(f^2, Q_f)$ is an unobstructed orientation-preserving marked
  Thurston map, and thus the even
  iterates $\sigma^{2n}(\tau)$ of any $\tau \in \cT_{Q_f}$ converge to
  a unique fixed point $\tau^*$ of $\sigma^2=\sigma_{f^2,Q_f}$ 
  (see \cite[Thm.~2.2]{buffTeichmullerSpacesHolomorphic2014}). Since $\sigma^2(\sigma(\tau^*))=\sigma(\sigma^2(\tau^*))=\sigma(\tau^*)$, we conclude that 
  $\tau^* = \sigma(\tau^*)$ is the unique fixed point of $\sigma$, and
  thus $(f,Q_f)$ is equivalent to a marked anti-rational map $(g, Q_g)$, unique up to Möbius conjugacy. The
  precise normalized uniqueness statement follows directly from parts 
  \ref{it:4} and \ref{it:5} in \Cref{prop:anti-sigma}.
\end{proof}

The uniqueness part of \Cref{thm:anti-thurston} now implies that symmetries of
$(f,Q_f)$ are inherited by $(g,Q_g)$.

\begin{corollary}\label{cor:symmetric-thurston-maps}
  Let $(f,Q_f)$ be a marked \antiThurston map with $\#Q_f\ge 3$ that
  is not a $(2,2,2,2)$-map, and let $\phi\circ f=g\circ\psi$ be a
  combinatorial equivalence between $(f,Q_f)$ and a marked
  \antiRational map $(g,Q_g)$.  Suppose $\zeta: S^2\to S^2$ is a
  homeomorphism with $\zeta(Q_f)=Q_f$ that commutes with $f$ up to
  isotopy, that is, the \antiThurston maps $(\zeta\circ f, Q_f)$ and
  $(f\circ \zeta, Q_f)$ are isotopic. Then there exists a unique
  Möbius or anti-Möbius transformation $\theta$ that is isotopic to
  $\phi\circ \zeta \circ\phi^{-1}$ rel.\ $Q_g$. Moreover, $g$ is
  $\theta$-symmetric, that is, $\theta\circ g=g\circ\theta$ and
  $\theta(Q_g)=Q_g$.
\end{corollary}

\begin{proof}
We first note the following consequence of \Cref{prop:anti-sigma} and the subsequent remark, combined with the normalization discussion in \Cref{sec:thurstons-theorem}; it refines the uniqueness parts of \Cref{thm:Thurston_theorem,thm:anti-thurston}. If $\phi_j\circ f=g_j\circ\psi_j$, $j=1,2$, are combinatorial equivalences
between $(f,Q_f)$ and marked \antiRational maps $(g_1, Q_{g_1})$ and $(g_2, Q_{g_2})$, then there exists a unique Möbius transformation $\theta$ such that $g_2=\theta\circ g_1\circ\theta^{-1}$ and $\theta\circ\phi_1$ is isotopic to $\phi_2$
rel.\ $Q_f$. 

Now let $\varepsilon$ be the identity map on $\CC$ when $\zeta$ is orientation-preserving, or the complex conjugation map $\iota$ otherwise. Moreover, let $\omega\in\Homeo_0^+(S^2,Q_f)$ be such that $f\circ \zeta \circ \omega=\zeta \circ f$. Set $\phi':=\varepsilon\circ \phi\circ \zeta^{-1}$, $\psi':=\varepsilon \circ \psi \circ (\zeta \circ \omega)^{-1}$, and $g'=\varepsilon \circ g \circ \varepsilon^{-1}$; see \Cref{fig:thurston-symmetries}. Note that the maps $\phi',\psi': S^2 \to \CC$ are orientation-preserving, isotopic rel.\ $Q_f$, and satisfy $\phi'\circ f = g'\circ \psi'$ and $\phi'(Q_f)=\psi'(Q_f)=\varepsilon(Q_g)$. Hence the pair $\phi',\psi'$ defines a combinatorial equivalence between $(f,Q_f)$ and the marked \antiRational map $(g',\varepsilon(Q_g))$. By the claim above, there is a unique Möbius transformation $\theta'$ such that $g=\theta'\circ g'\circ(\theta')^{-1}$ and $\theta'\circ\phi'$ is isotopic to $\phi$
rel.\ $Q_f$. It follows that $\theta:=\theta'\circ \varepsilon$ is a Möbius or anti-Möbius transformation that is isotopic to $\phi\circ \zeta \circ\phi^{-1}$ rel.\ $Q_g$ and satisfies $\theta\circ g=g\circ\theta$ and 
$\theta(Q_g)=Q_g$, as desired. Since $\#Q_f\geq 3$ and a Möbius or anti-Möbius transformation is
determined by its orientation
and its values at three points, we also have that such $\theta$ is unique. This finishes the proof. 
\begin{figure}[t]
  \centering
  \begin{tikzcd}
    (S^2, Q_f) \arrow[rrr, bend left, "\psi'"] \arrow[d,
    "f"] 
   & (S^2,Q_f)   \arrow[l, "\zeta\circ \omega"] \arrow[d,
    "f"] 
    \arrow[r, "\psi"] &(\CC, Q_g) \arrow[r, "\varepsilon"] \arrow[d, "g"] &
    (\CC, \varepsilon(Q_g)) \arrow[d, "g'"]\\
    (S^2, Q_f) \arrow[rrr, bend right, "\phi'"]   &(S^2,Q_f)\arrow[l, "\zeta"] 
    \arrow[r, "\phi"]  & (\CC, Q_g)
    \arrow[r, "\varepsilon"] 
     & (\CC, \varepsilon(Q_g)) 
  \end{tikzcd}
  \caption{Diagram illustrating the argument in Corollary~\ref{cor:symmetric-thurston-maps}.}
  \label{fig:thurston-symmetries}
\end{figure}
\end{proof}

\subsection{Canonical decomposition}
\label{subsec:canon-decomp}

Let $(f,Q)$ be a marked Thurston or anti-Thurston map with a
completely invariant multicurve $\Gamma$. The \emph{decomposition} of
$(f,Q)$ with respect to $\Gamma$ is obtained along the lines of
\cite[Thm.~5.1]{pilgrimCombinationsComplexDynamical2003}, resulting
in a postcritically finite marked branched cover of a \emph{tree of
  spheres} with finitely many components. The \emph{canonical
  decomposition} is the decomposition of $(f,Q)$ with respect to a canonical
multicurve obstruction $\Gamma_{f,Q}$.

Here is a brief non-technical description of this process. We consider
two copies of the marked sphere $(S^2, Q)$, one for the domain and the
other one for the range of $f$. First, we pinch every curve $\gamma
\in \Gamma$ to a point in the range sphere, and we pinch every
component $\gamma'$ of $f^{-1}(\Gamma)$ to a point in the domain
sphere. The resulting topological spaces are \emph{nodal spheres},
i.e., finite collections of $2$-spheres (which we will refer to as
\emph{small spheres}) glued together at finitely many special points
(called \emph{nodes}), which correspond to the pinched curves from
$\Gamma$ and $f^{-1}(\Gamma)$. (Since the removal of any node
disconnects these nodal spheres, they are also called \emph{trees of
  small spheres}.) We keep the points from $Q$ marked in both nodal
spheres; in addition, we also mark all of their nodes. We denote by
$(S^2, Q)/\Gamma$ and $(S^2, Q)/f^{-1}(\Gamma)$ the resulting marked nodal
spheres, obtained from the range and domain of $f$, respectively. The original map $f$
induces a branched cover between the two constructed marked nodal
spheres, which respects their markings.

Next, we identify the nodal spheres in the domain and range to get a
dynamical system. For this, we forget the marked points in
$(S^2, Q)/f^{-1}(\Gamma)$ that correspond to the null-homotopic curves
in $f^{-1}(\Gamma)$, and then collapse every small sphere with at most
2 of the remaining marked points to a point. The nodal sphere obtained
after this stabilization procedure is canonically homeomorphic (up to
isotopy relative to the marked points) to the nodal sphere in the
range. The original map $f$ together with this identification thus
induces a self-map $\widehat{f}:
\widehat{\mathscr{S}}_\Gamma \to \widehat{\mathscr{S}}_\Gamma$ on the set $\widehat{\mathscr{S}}_\Gamma$ of all small spheres of $(S^2,Q)/\Gamma$, which maps each small sphere onto
a small sphere.

The marking on the nodal sphere $(S^2,
Q)/\Gamma$ induces a marking on every small sphere $\widehat{S}\in
\widehat{\mathscr{S}}_\Gamma$. Namely, besides the points from $Q$,
$\widehat{S}$ is also marked by the points corresponding to the
pinched curves from $\Gamma$. If we denote by $Q(\widehat S)$ the set
of all marked points in the small sphere $\widehat{S}$, then $\widehat
f(Q(\widehat S)) \subset Q(\widehat{f(S)})$ and
$\bigcup_{\widehat{S}\in \widehat{\mathscr{S}}_\Gamma} Q(\widehat S)$
contains all the critical values of $\widehat f$. (Note that the
critical points of $\widehat f$ may arise only from the critical
points in $f^{-1}(Q)$ or from the components $\gamma'$ in
$f^{-1}(\Gamma)$ such that $\deg(f: \gamma'\to \gamma)>1$.)

Since there are only finitely many small spheres, every small sphere
in $(S^2, Q)/\Gamma$ is eventually periodic under $\widehat f$. For
each such periodic small sphere $\widehat{S}\in
\widehat{\mathscr{S}}_\Gamma$, we may consider the first return map
$F:=\widehat{f}^k: \widehat{S}\to\widehat{S}$, which we will also call a \emph{small map} (associated with the periodic sphere $\widehat{S}$). By construction,
$Q(\widehat{S})$ is forward-invariant under the small map $F$ and
contains all of its critical values. It follows that either $\deg(F)=1$
(i.e., $F$ is a homeomorphism) or $(F,Q(\widehat S))$ is a marked
\antiThurston map, uniquely defined up to isotopy equivalence. For details of this construction in the
orientation-preserving case  
see
\cite{pilgrimCombinationsComplexDynamical2003,
  bartholdiAlgorithmicAspectsBranched2021}.

We now state a characterization of canonical multicurve obstructions through the corresponding decompositions, conjectured by Pilgrim in \cite{pilgrimCombinationsComplexDynamical2003} and proved by Selinger in \cite{selingerTopologicalCharacterizationCanonical2013} in the orientation-preserving case.

\begin{theorem}\label{thm:dec-char-can-obstruction}

Let $(f,Q)$ be a marked \antiThurston map. Then a canonical multicurve obstruction of $(f,Q)$ is the unique minimal (with respect to inclusion, up to
homotopy) multicurve $\Gamma$ in $S^2\setminus Q$ with the following two
properties:
\begin{enumerate}
  \item\label{item:char_can_1} either $\Gamma=\emptyset$, or $\Gamma$ is a completely invariant
        obstruction for $(f,Q)$;
  \item\label{item:char_can_2} for each periodic small sphere $\widehat S\in\widehat{\mathscr S}_\Gamma$,
        the first return map $F:=\widehat f^{\,k}\colon \widehat S\to\widehat S$
        satisfies the following condition: if $\deg(F)>1$, then the
        canonical obstruction of $(F,Q(\widehat S))$ is empty.
\end{enumerate}
\end{theorem}

Here, for $\Gamma=\emptyset$, we use the convention that
$\widehat{\mathscr S}_\emptyset$ consists of the single small sphere
$(S^2,Q)$, with $\widehat f=f$.

\begin{proof}
  In the orientation-preserving case, this follows from \cite[Cor.~5.4 and Thm.~5.6]{selingerTopologicalCharacterizationCanonical2013}. Now 
  let $(f,Q)$ be an anti-Thurston map and $\Gamma_{f,Q}$ be its
  canonical multicurve obstruction.

  Suppose that a multicurve $\Gamma$ in $S^2\setminus Q$ possesses the two properties from the statement. If $\Gamma=\emptyset$, then $\Gamma_{f,Q}=\emptyset$ as well by \ref{item:char_can_2}. Assume now that $\Gamma\neq \emptyset$, and thus it is a completely invariant obstruction for
  $(f,Q)$. Note that then
  $\Gamma$ is also a completely invariant obstruction for
  $(f^2,Q)$. Clearly, $\widehat{S}\in \widehat{\mathscr{S}}_\Gamma$ is
  a periodic small sphere for $\widehat f$ if and only if it is a
  periodic small sphere for $\widehat{f^2}$. Furthermore, if $F$ is
  the first return map to $\widehat S$ for $\widehat f$, then the
  first return map to $\widehat S$ for $\widehat{f^2}$ equals $F$ or
  $F^2$. By \Cref{lem:canonical-obstruction-iterates}, it follows that
  $\Gamma$ satisfies the corresponding property \ref{item:char_can_2} for $(f^2,Q)$. Using the orientation-preserving version of the statement, we conclude that $\Gamma\supseteq \Gamma_{f^2,Q}=\Gamma_{f,Q}$ (up to
  homotopy), as desired. 

  Conversely, the orientation-preserving version of the statement and \Cref{lem:canonical-obstruction-iterates} immediately imply that $\Gamma_{f,Q}$ satisfies \ref{item:char_can_2}. At the same time, \Cref{thm:canonical-obstruction-is-empty}\ref{item:can_2} implies that $\Gamma_{f,Q}$ also satisfies \ref{item:char_can_1}. This finishes the proof of the theorem.
\end{proof}

In \cite[Thm.~5.6]{selingerTopologicalCharacterizationCanonical2013},
Selinger provides a purely topological characterization of canonical obstructions in the orientation-preserving case. An analogous
characterization also works in the orientation-reversing case, and we
present a unified version below.

\begin{theorem}\label{thm:top-char-can-obstruction}
  Let $(f,Q)$ be a marked \antiThurston map. Then a canonical multicurve obstruction of $(f,Q)$ is the unique minimal (with respect to inclusion, up to
homotopy) multicurve $\Gamma$ in $S^2\setminus Q$ with the following
properties:
\begin{enumerate}
  \item either $\Gamma=\emptyset$, or $\Gamma$ is a completely invariant
        obstruction for $(f,Q)$;        
   \item if $\widehat S \in \widehat{\mathscr{S}}_\Gamma$ is a periodic small sphere and the associated small map  $F$ is a $(2,2,2,2)$-map, then every curve of every simple
  obstruction for $(F,Q(\widehat S))$ has two postcritical points of
  $F$ in each complementary component and the two eigenvalues of 
  $\widetilde{F}_*$ are
  non-integer or have the same absolute value;
  \item if $\widehat S \in \widehat{\mathscr{S}}_\Gamma$ is a periodic small sphere and the associated small map $F$ is neither a $(2,2,2,2)$-map nor a homeomorphism, then $\big(F,
  Q(\widehat S)\big)$ is combinatorially equivalent to a marked (anti-)rational map or, equivalently, $\big(F, Q(\widehat S)\big)$
  does not have a multicurve obstruction.      
\end{enumerate}

\end{theorem}

\Cref{thm:top-char-can-obstruction} is an immediate consequence of
\Cref{thm:dec-char-can-obstruction} combined with
\Cref{thm:canonical-obstruction-is-empty,thm:can-obstruction-is-empty-parabolic}
(which characterize maps with empty canonical obstructions) and
Thurston's characterization theorems
(\Cref{thm:Thurston_theorem,thm:anti-thurston}).

\section{Plane graphs}
\label{sec:plane-graphs}

We will need some terminology and results from the theory of plane
graphs. We refer the reader to \cite{diestelGraphTheory2017} for
general background in graph theory.

By an \emph{open arc} in $S^2$ we mean the image $e:=\eta\big((0,1)\big)$ under a continuous injective map $\eta : (0,1)\to S^2$ such that the limits $\eta(0):=\lim_{t\to 0^+} \eta(t)$ and $\eta(1):=\lim_{t\to 1^-} \eta(t)$ exist and are disjoint from $e$. The points $v_0=\eta(0)$ and $v_1=\eta(1)$ are called
the \emph{endpoints} of $e$, and we say that $e$ \emph{connects} $v_0$
and $v_1$. If $e$ is an open arc and $V$ is a finite set in $S^2$, we say that $e$ is an \emph{open arc in $(S^2,V)$} whenever
$e\subset S^2\setminus V$ and both endpoints of $e$ are in $V$.  In
this paper, homotopy/isotopy of open arcs in $(S^2,V)$ is always
considered within the family of open arcs in $(S^2,V)$; in
particular, the endpoints must be fixed under such homotopy/isotopy.

We record the following useful fact, which we frequently invoke implicitly throughout this paper (the proof is immediate from \cite[Thm.~A.6]{buserGeometrySpectraCompact2010}).

\begin{lemma}\label{lem:arc_isotopy}
  Let $U\subset S^2$ be a Jordan domain, and $e, e'$ be two open arcs in $U$ with the 
  same (possibly equal) endpoints $v_0,v_1$. Then $e$ and $e'$
 are isotopic rel.\ $\{v_0,v_1\}\cup (S^2\setminus U)$.
\end{lemma}

For our purposes, a \emph{plane graph} in $S^2$ is a pair $G=(V,E)$
consisting of a finite set $V \subset S^2$ of \emph{vertices} and a
finite set $E$ of pairwise disjoint open arcs in $(S^2,V)$,
called \emph{edges}.  Note that our definition of a plane graph allows
\emph{multiple edges} (that is, distinct edges connecting the same
pair of vertices), as well as \emph{loops} (that is, edges
connecting a vertex to itself). If a plane graph $G$ has neither loops nor multiple edges, we call $G$ \emph{simple}.

A \emph{face} of a plane graph $G=(V,E)$ is a connected component of
$S^2 \setminus G$. Here and in the following, if there is no risk of
confusion, we use the same letter $G$ to denote the subset of the
sphere given by $V \cup \left(\bigcup_{e\in E} e\right)$, called
the \emph{realization} of $G$. Note that we do not include endpoints
in the edges, and also do not include boundary edges and vertices in
the faces.  Given a plane graph $G$, we denote by $V(G)$, $E(G)$, and
$F(G)$ its sets of vertices, edges, and faces, respectively.

In the following, let $G=(V,E)$ be a plane graph.  We say that an edge
$e\in E$ is \emph{incident} with its endpoints, two vertices $v_1,v_2\in V$ are called \emph{adjacent} if they are connected by an edge, and two edges
$e_1,e_2\in E$ are called \emph{adjacent} if they have a common endpoint.
The \emph{degree} of a vertex $v\in V$,
denoted $\deg_G(v)$, is the number of edges incident with $v$ (where an edge is counted twice if it is a loop). A
\emph{leaf} of $G$ is a vertex of degree~$1$; an \emph{isolated vertex} of $G$ is a vertex of degree $0$. 

A (finite) \emph{walk} in the graph $G$ is a finite alternating
sequence $W = (v_0,e_1,\allowbreak v_1,\allowbreak e_2, \allowbreak
\ldots,e_n,v_n)$ of vertices and edges such that $v_{j-1}$ and $v_j$
are the endpoints of the edge $e_j$ for each $j=1,\ldots,n$. The number $n$ of (not necessarily distinct) edges in $W$ is the
\emph{length} of the walk; the vertices $v_1, \dots, v_{n-1}$ are the 
\emph{interior} vertices of $W$, and $v_0$ and $v_n$ are its \emph{initial}
and \emph{terminal} vertices, respectively. The walk $W$ is called a
\emph{simple path} if $v_j \ne v_k$ for $j \ne k$, i.e., if no vertex
appears more than once in $W$ (which also implies that no edge appears
more than once). If a simple path visits each vertex of $G$ exactly once, it is called \emph{Hamiltonian}.
 Walks naturally lead to the notions of
\emph{connectedness} and \emph{connected components} of the graph $G$. Euler's formula states that the number of connected components of $G$ equals $\#V(G)-\#E(G)+\#F(G)-1$.

If the initial and terminal vertices of a walk are the same, we say
that it is a \emph{closed walk} or a \emph{circuit}. A \emph{cycle} is
a circuit with distinct edges in which only the initial and terminal vertices are the same.

Let $A$ be a face of $G$. A \emph{boundary circuit} of $A$ is a
circuit in $G$ obtained by traversing (once) a connected component of
$\partial A$.
Note that an edge of $G$ may appear zero times, once, or twice in a boundary
circuit. We say that an edge $e \subset \partial A$ is a \emph{simple
  boundary edge} of $A$ if $e\subset \partial B$ for a face $B \ne
A$. Otherwise (if $A$ is on both sides of $e$), we say that $e$ is a
\emph{double boundary edge} of $A$. It is easy to see that a face is
simply connected if and only if it has exactly one boundary component,
and a simply connected face is a Jordan domain if and only if its
boundary circuit is a cycle. 

A face of $G$ is called a \emph{digon} (or \emph{digonal}) if it is
simply connected and its boundary circuit has length~$2$. Similarly, a
\emph{triangular} face (or simply a \emph{triangle}) is a simply
connected face whose boundary circuit has length $3$.

The graph $G=(V,E)$ is called \emph{bipartite} if there is a
partition $V = V_1 \sqcup V_2$ of its vertices into two disjoint sets
$V_1$ and $V_2$ such that every edge connects some $v_1 \in V_1$ to
some $v_2 \in V_2$. The graph $G$ is a \emph{tree} if it is connected
and has no cycles. We say that $G$ is \emph{$2$-connected} if $\#V> 2$ and $G$ remains connected after the removal of any vertex of $G$ (together with the incident edges). Note that every $2$-connected graph is connected.

In analogy with \antiThurston maps, we now introduce two equivalence
relations on plane graphs.

\begin{definition}
  Two plane graphs $G=(V,E)$ and $G'=(V',E')$ are called
  \emph{isomorphic} if there is an orientation-preserving
  homeomorphism $\phi: S^2 \to S^2$ such that $\phi(G) = G'$ and
  $\phi(V)=V'$.
\end{definition}

\begin{remark}
  The requirement that $\phi$ maps vertices to vertices is
  automatically satisfied for all vertices of degree not equal to two.
\end{remark}

\begin{definition}
  Two plane graphs $G=(V,E)$ and $G'=(V',E')$ are called
  \emph{isotopic} if $V=V'$ and $\phi(G) = G'$ for some $\phi\in
  \Homeo_0^+(S^2,V)$.
\end{definition}

The following result provides a criterion for two plane graphs to be
isotopic; the proof is straightforward from \cite[Thms.~A.5 and
A.6]{buserGeometrySpectraCompact2010}.

\begin{proposition}\label{prop:graph-isotopy-criterion}
  Let $G$ and $G'$ be two plane graphs without loops and with a common vertex
  set~$V$. Then $G$ and $G'$ are isotopic if and only if the following
  two conditions are satisfied:
    \begin{enumerate}
    \item $\#E(G)=\#E(G')$; 
    \item for each edge $e\in E(G)$ there is an edge $e'\in E(G')$
      such that $e$ and $e'$ are homotopic and $m_G(e)=m_{G'}(e')$.
    \end{enumerate}
\end{proposition}

Here, $m_G(e)$ denotes the \emph{multiplicity} of an edge $e$ in a
plane graph $G$, that is, the total number of edges of $G$ that are
homotopic to $e$.

Finally, we introduce the notion of a dual graph. Let $G=(V,E)$ and $G^*=(V^*,E^*)$ be two plane graphs in $S^2$. We say that $G^*$ is a \emph{dual graph of $G$} if it satisfies the following conditions:
\begin{enumerate}
\item $\#V^* =\#F(G)$ and $\#E^*=\#E$; 
\item every face $A$ of $G$ contains exactly one vertex of $G^*$, which we denote $v^*(A)$;
\item for every edge $e$ of $G$, there is an edge $e^*$ of $G^*$ such that
\begin{itemize}
    \item $e^*$ connects the vertices $v^*(A)$ and $v^*(B)$ in $G^*$, where $A, B$ are the two faces of $G$ with $e$ on the boundary (here, $A=B$ if $e$ is a double boundary edge); 
    \item $\#(e^*\cap G) = 1$;
    \item $e^*$ intersects $e$ transversely (in a unique point). 
\end{itemize}
\end{enumerate}

The next result summarizes some basic properties of dual graphs. 

\begin{lemma}\label{lem:dual-graphs} The following statements are true:
\begin{enumerate}
    \item Every plane graph $G$ has a dual graph. Moreover, if $G$ is connected, then any two dual graphs of $G$ are isomorphic.
    \item If two connected plane graphs are isomorphic, then their dual graphs are isomorphic as well.  
    \item If $G$ is a connected plane graph with a dual graph $G^*$, then every dual graph of $G^*$ is isomorphic to $G$.
\end{enumerate}    
\end{lemma}

We also record the following easy lemma. 

\begin{lemma}\label{lem:dual-to-Tischler}
The following statements are true:
\begin{enumerate}
\item Let $G$ be a plane graph with $\#V(G)\geq 3$. Then every face of $G$ is a Jordan domain if and only if $G$ is $2$-connected and has no loops.
\item Let $G$ be a plane graph with $\#F(G)\geq 3$ and no isolated vertices. Then every face of $G$ is a Jordan domain if and only if the dual graph $G^*$ of $G$ is $2$-connected and has no loops. 
\end{enumerate}
\end{lemma}

\section{Critically fixed anti-rational maps and Tischler graphs}
\label{sec:crit-fixed-anti-rational}

We say that an \antiThurston map is \emph{critically fixed} if it
fixes each of its critical points. Our aim in this and the following
section is to combinatorially classify all critically fixed
anti-rational maps. As in the orientation-preserving setting
\cite{tischlerCriticalPointsValues1989,
  cordwellClassificationCriticallyFixed2015,
  hlushchanka2017Thesis, hlushchankaTischlerGraphsCritically2019}, we start by introducing
canonical combinatorial models for such maps, which, in analogy with
the rational case, we call \emph{Tischler graphs}.

Let $f$ be a critically fixed anti-rational map of degree $d \ge
2$. By \Cref{thm:superattracting-basin}, the map $f$ on the immediate
basin $A_f(c)$ of a critical fixed point $c$ is conformally conjugate
to the map $g(w) = \bar{w}^{m+1}$ on the unit disk $\D$, where
$m=\deg(f,c)-1 \ge 1$ is the multiplicity of the critical point $c$.
The map $g$ fixes (setwise) the $m+2$ open radii $r_{\theta_k} = \{ t
e^{i \theta_k} : t \in (0,1) \}$ in $\D$ with arguments
$\theta_k:=\frac{2\pi k}{m+2}$ for $k=0,\ldots,m+1$. We call the
preimages of these radii under the corresponding Böttcher map
$\phi_c: A_f(c)\to \D$ the \emph{fixed internal rays} in the immediate
basin $A_f(c)$. Clearly, every such ray $R_{\theta_k}:=
\phi_c^{-1}(r_{\theta_k})$ is fixed (setwise) by the map $f$, i.e.,
$f(R_{\theta_k})=R_{\theta_k}$.  As in the case of rational maps, it
is straightforward to show that every accumulation point of the path
$t\mapsto\phi^{-1}_{c}(te^{i\theta_k})$ as $t\to 1^-$ must be a fixed
point of $f$ in $\partial A_f(c)\subset J_f$.  At the same time, it is
well known that the accumulation set of such a path is non-empty and
connected (see, e.g., \cite[Problem
5-b]{milnorDynamicsOneComplex2006}).  Since $f$ has only finitely many
fixed points, it follows that the limit $q_k:=\lim_{t\to
  1^-}{\phi^{-1}_{c}(te^{i\theta_k})}$ exists for every
$k=0,\dots,m+1$.  This shows that $R_{\theta_k}$ is an open arc
connecting $c$ and $q_k$.  We call $c$ the \emph{basepoint} and $q_k$
the \emph{landing point} of the fixed internal ray $R_{\theta_k}$, and
say that $R_{\theta_k}$ \emph{lands} at $q_k$. For critically fixed
anti-rational maps, all fixed points in the Julia set are repelling
(see Theorem \ref{thm:crit-fixed-basic}), so all the landing points
$q_0, \ldots, q_{m+1}$ are repelling fixed points.

\begin{remarks}\mbox{}
  \begin{enumerate}
  \item Note that we consider fixed internal rays as open arcs, not
    including either the basepoint or the landing point.
  \item The argument for the existence of a landing point of a fixed
    internal ray given above generally applies to periodic internal
    rays in the immediate basin $A_f(c)$ of a superattracting fixed
    point $c$ with no other critical point inside. Since critically
    fixed maps are hyperbolic, something even stronger is true, though
    with a less elementary proof: the boundary $\partial A_f(c)$ is
    locally connected \cite[Lem.~19.3]{milnorDynamicsOneComplex2006},
    so that every internal ray in $A_f(c)$ lands, not just
    the periodic ones.
  \end{enumerate}
\end{remarks}

With these observations, we define our central object of
study in this section; see \Cref{fig:Julia_sets} for an illustration.

\begin{definition}
  \label{def:tischler-graph} 
  Let $f$ be a critically fixed anti-rational map of degree $d\geq
  2$. Then the \emph{alternating Tischler graph} of $f$ is the plane
  graph $T_f$ in $\CC$ whose edges are the fixed internal rays of $f$
and whose vertices are the endpoints
  of these rays (which are the basepoints of the rays together with
  their landing points).
\end{definition}

Note that a priori there might be a repelling fixed point that is not
the landing point of any fixed internal ray, and thus it would not be
included in the alternating Tischler graph.  However, the following
theorem, which is the central result in this section, shows that this
cannot happen.

\begin{theorem}
  \label{thm:characterization-1}
  Let $f$ be a critically fixed anti-rational map of degree $d \ge
  2$. Then its alternating Tischler graph $T = T_f$ is a connected
  bipartite plane graph with $d+1$ faces, whose vertices are exactly
  the fixed points of $f$.  Each vertex corresponding to a critical
  point of multiplicity $m$ has degree $m+2 \ge 3$ in $T$, and each
  vertex corresponding to a repelling fixed point has degree
  2. Furthermore, every face $A$ of $T$ is a Jordan domain, and $f$
  maps $A$ anti-conformally onto $\CC \setminus \overline{A}$.
\end{theorem}
\begin{proof}
  By construction, different fixed internal rays do not intersect, and
  each edge of $T$ connects a critical point of $f$ to a repelling
  fixed point, so $T$ is a bipartite plane graph. Furthermore, as
  every critical point $c$ of multiplicity $m$ is the basepoint of
  exactly $m+2$ fixed internal rays, we have $\deg_T(c)=m+2\geq
  3$. Since $f$ is orientation-reversing, it reverses the cyclic order
  of the fixed internal rays landing at a repelling fixed point. This
  is only possible if there are at most two fixed internal rays
  landing at each repelling fixed point, and thus $\deg_T(q)\leq 2$
  for every $q\in V(T)\setminus C_f$. We will show that we always have
  equality, i.e., that every repelling fixed point is the landing
  point of two fixed internal rays.

  Assume that $f$ has $r$ critical points $c_1, \ldots, c_r$ with
  corresponding multiplicities $m_1, \ldots, m_r$. Then
  $\sum_{k=1}^r m_k = 2d-2$ by the Riemann--Hurwitz formula. Note also
  that the number $N_{attr}$ of attracting and superattracting fixed
  points equals $r$, because $f$ has exactly $r$ superattracting fixed
  points and no attracting fixed points (see Theorem
  \ref{thm:crit-fixed-basic}). Moreover,
  Theorem~\ref{thm:counting-fixed-points} implies that $f$ has exactly
  $N_{rep} = r+d-1$ repelling fixed points. Since critical points form
  one of the parts of the bipartite graph $T$, we have
  \begin{equation}\label{eq:tischler-edges-count-1}
    \#E(T) = \sum_{k=1}^r{\deg_T(c_k)}=\sum_{k=1}^r (m_k+2) = 2r +
    2d-2 = 2 N_{rep}.
  \end{equation}
  At the same time, since the landing points of fixed internal rays
  (which are repelling fixed points of $f$) form the other part of the
  bipartite graph $T$, we also have
  \begin{equation}\label{eq:tischler-edges-count-2}
    \#E(T) = \sum_{q\in V(T)\setminus C_f} \deg_T(q)  \leq 2 N_{rep}.
  \end{equation}
  Combining \eqref{eq:tischler-edges-count-1} and
  \eqref{eq:tischler-edges-count-2}, we conclude that every repelling
  fixed point $q$ must be a vertex of $T$ and that it must be the
  landing point of exactly two fixed internal rays, i.e.,
  $\deg_T(q)=2$. In particular, this implies that the vertices of $T$
  are exactly the fixed points of $f$, so
  \begin{equation}\label{eq:tischler-vertex-count}
      \#V(T) = N_{rep}+N_{attr}=2r+d-1.
  \end{equation} 

  Let $A$ be a face of the alternating Tischler graph $T$ and
  $W:=(v_0,e_1,v_1,\dots, e_n, v_n=v_0)$ be a boundary circuit of $A$.
  (Note that a priori $\partial A$ may have several connected
  components.)  Let
  $W':=(v_k,e_{k+1},v_{k+1},\dots, e_{\ell}, v_{\ell})$,
  $0\leq k < \ell \leq n$, be a subwalk of $W$ such that $v_k=v_\ell$
  and all the vertices in $\{v_k,v_{k+1},\dots, v_{\ell-1}\}$ are
  different, that is, $W'$ is a cycle. We will show that $k= 0$ and
  $\ell=n$, that is, $W=W'$ is a cycle and $A$ is a Jordan domain.
  
  Let $U$ be the connected component of $\CC\setminus W'$ disjoint
  from the face $A$. (Here, with a slight abuse of notation, we denote
  by $W'$ the Jordan curve in $\CC$ traced by the walk $W'$.)  Note
  that $U$ is a Jordan domain with $U \cap \partial A = \emptyset$.
  Let $U'$ be the connected component of $f^{-1}(U)$ with
  $e_{k+1} \subset \partial U'$.  (Such a component exists because $f$
  fixes the edge $e_{k+1}\subset \partial U$ and it is a local
  orientation-reversing homeomorphism near every $z\in e_{k+1}$.)
 
  We claim that $e_j \subset \partial U'$ for each $j$ with
  $k+1 \le j \le \ell$. We argue by induction on $j$. The base case
  $j=k+1$ is satisfied by the choice of $U'$. Let $k+1< j \le \ell$
  and assume that we have $e_{j-1} \subset \partial U'$. If $v_{j-1}$
  is a repelling fixed point, then $e_{j-1}$ and $e_j$ are the two
  fixed internal rays landing at $v_{j-1}$, so
  $e_j \subset \partial U'$ (since $f$ is a local
  orientation-reversing homeomorphism at $v_{j-1}$). If $v_{j-1}$ is a
  fixed critical point, then $e_{j-1}$ and $e_j$ are two consecutive
  fixed internal rays based at $v_{j-1}$, and all of the other fixed
  internal rays based at $v_{j-1}$ are contained in $U$, so they are
  not in $\overline{A}$. Using the local dynamics of $f$ at $v_{j-1}$,
  we conclude that $f$ maps a small neighborhood of $v_{j-1}$ in
  $\overline{A}$ homeomorphically onto a small neighborhood of
  $v_{j-1}$ in $\overline{U}$, which directly implies that
  $e_j \subset \partial U'$. This establishes the claim and shows that
  $\partial U'$ contains $\partial U$.

  Now note that $U'\subseteq \CC\setminus \overline{U}$ and thus it
  cannot contain any critical points of $f$ (as those are fixed by the
  map). Since $U$ is a Jordan domain, it follows that $f: U'\to U$ is
  a homeomorphism. Moreover, we claim that $\partial U'=\partial U$, 
  and thus $U'=\CC\setminus\overline{U}$. Indeed,
  a point
  $x\in\partial U'\setminus\partial U$ would be a non-critical point with
$f(x)\in f(\partial U') \subseteq \partial U$, so by the local analysis above, points of $U'$ near
$x$ would have a second preimage in $U'$ near the point $x_0\in\partial U$
with $f(x_0)=f(x)$, contradicting the injectivity of $f$ on $U'$. Since $f$
maps each vertex and each edge of $T$ onto itself, while $f(U')=U$ is
disjoint from $U'=\CC\setminus \overline{U}$, we conclude that $U'\cap T=\emptyset$. It follows that $A=U'$, so that $W=W'$ is a cycle and
$A$ is a Jordan domain. Furthermore, since $f: U'\to U$ is a
homeomorphism, $f$ maps the face $A=U'$ anti-conformally onto
$\CC\setminus\overline{A}=U$.
 
  As every face of $T$ is a Jordan domain, the graph $T$ is
  connected. Finally, Euler's formula together with
  \eqref{eq:tischler-edges-count-1} and
  \eqref{eq:tischler-vertex-count} implies that
  \[
    \#F(T) = \#E(T) - \#V(T) + 2 =  (2r + 2d-2) - (2r+d-1)+ 2= d+1.
  \]
  This finishes the proof of the theorem.
\end{proof}

\begin{remark}
  Theorem \ref{thm:characterization-1} suggests the following very
  simple combinatorial model for a critically fixed anti-rational map
  $f: \CC\to \CC$, which we justify in the next section. First, draw
  the alternating Tischler graph $T=T_f$. Next, for each face $A$ of
  $T$ pick a topological reflection in $\partial A$, i.e., an
  orientation-reversing homeomorphism $F_A$ from the Riemann sphere
  $\CC$ to itself with $F_A(A) = \CC \setminus \overline{A}$,
  $F_A|_{\partial A} = \id_{\partial A}$, and $F_A\circ
  F_A=\id_\CC$ (see Definition
  \ref{def:top-reflection}). Now define an orientation-reversing branched covering
  $F_T: \CC\to \CC$ by setting $F_T$ to be $F_A$ on the closure
  $\overline A$ of every face $A$ of $T$ (see Definition
  \ref{def:schottky-map}). Then $F_T$ is a critically fixed
  anti-Thurston map that is isotopic to the original map $f$ (see
  \Cref{prop:equivalent-tischler-graphs} and
  Corollary~\ref{cor:crit-fixed-schottky}).
\end{remark}

We close this section with the following natural definition.

\begin{definition}
  \label{def:reduced-tischler}
  Let $f$ be a critically fixed anti-rational map of degree $d\geq 2$
  and $T_f$ be its alternating Tischler graph. The plane graph
  $\mathring T_f$ with the same realization as $T_f$ but with the
  vertex set $V(\mathring T_f)=C_f$ is called the \emph{reduced Tischler
    graph} of $f$.
\end{definition}

In other words, the reduced Tischler graph $\mathring T_f$ is obtained
from the alternating Tischler graph $T_f$ by ``forgetting'' all
vertices of degree $2$, which are exactly the repelling fixed points
of $f$ (see Theorem \ref{thm:characterization-1}).

\section{Topological Tischler graphs and Schottky  maps}
\label{sec:plane-graphs-schottk}

From the previous section we know what a topological model of a
critically fixed anti-rational map looks like. Our goal in this
section is to formally introduce these models and to understand
precisely when they correspond to an anti-rational map. In particular,
we will provide a combinatorial classification of all critically fixed
anti-rational maps (Theorem \ref{thm:tischler-fixed-correspondence}).

\begin{definition}
  A \emph{marked topological Tischler graph} is a connected plane
  graph $T$ in $S^2$ without loops, but possibly with multiple edges,
  such that $T$ has at least three faces and each face is a Jordan
  domain.
\end{definition}

We note that since all faces of a marked topological Tischler graph
$T$ are Jordan domains, every vertex of $T$ has degree $\geq 2$.
Moreover, since $T$ has at least three faces, it must have at least
two vertices of degree $\geq 3$.

\begin{definition}  
  Suppose $T$ is a marked topological Tischler graph.
  \begin{itemize}
  \item We say that $T$ is an \emph{alternating topological Tischler
      graph} if every edge of $T$ connects a vertex of degree $2$ with
    a vertex of degree $\geq 3$; in other words, $T$ is alternating if
    it is bipartite with one of the parts formed by all the vertices
    of degree $2$.
  \item We say that $T$ is a \emph{reduced topological Tischler graph}
    if it has no vertices of degree $2$.
  \item The plane graph $\mathring T$ with the same realization as the
    plane graph $T$ but with the vertex set $V(\mathring T) = \{v\in
    V(T): \deg_T(v)\geq 3\}$ is called the reduced topological
    Tischler graph \emph{corresponding} to $T$.
  \end{itemize}
\end{definition}

\begin{remarks}\mbox{}
  \begin{enumerate}
  \item By Theorem~\ref{thm:characterization-1} and
    Definition~\ref{def:reduced-tischler}, the alternating (resp.\
    reduced) Tischler graph of every critically fixed anti-rational
    map of degree $\geq 2$ is also an alternating (resp.\ reduced)
    topological Tischler graph.
  \item Reduced topological Tischler graphs can be identified with
    their realizations in $S^2$, whereas for marked topological
    Tischler graphs, one also needs to ``remember'' the location of
    the vertices of degree 2.  In order to classify all critically
    fixed anti-rational maps, reduced topological Tischler graphs are
    sufficient. However, for the classification of general (possibly
    obstructed) critically fixed anti-Thurston maps, we will need the
    marked version.
  \end{enumerate}
\end{remarks}

We now formally define topological models for critically fixed
anti-rational maps, which we call \emph{Schottky maps}.
  
\begin{definition}
  \label{def:top-reflection}
  Given a Jordan domain $U \subset S^2$, an associated
  \emph{topological reflection} (\emph{in $\partial U$)} is an
  orientation-reversing homeomorphism $f_U: S^2 \to S^2$ such that
  $f_{U}=\id$ on $\partial U$ and $f_U\circ f_U = \id$ on $S^2$.
\end{definition}

\begin{definition}
  \label{def:schottky-map}
  Let $T$ be a marked topological Tischler graph in $S^2$. A
  \emph{Schottky map} associated with $T$ is a self-map
  $f_T: S^2\to S^2$ of the sphere such that the restriction of $f_T$
  to the closure $\overline{A}$ of every face $A$ of $T$ is the
  restriction to $\overline{A}$ of a topological reflection in
  $\partial A$. An associated \emph{marked Schottky map} $(f_{T}, Q)$
  is a Schottky map $f_T$ associated with $T$ together with the set
  $Q$ consisting of all vertices of $T$.
\end{definition}

Equivalently, a continuous self-map $f: S^2 \to S^2$ is a Schottky map
associated with a marked topological Tischler graph $T$ in $S^2$ if
$f|_T = \id_T$ and $f$ sends every face $A$ of $T$ homeomorphically
onto $S^2\setminus \overline{A}$. Indeed, the map $f_A: S^2\to S^2$
given by
\[
  f_A(p)=
  \begin{cases}
    f(p) & \text{if $p\in \overline A$}\\
    (f|_A)^{-1}(p) & \text{if $p\in S^2\setminus \overline A$}
  \end{cases}
\]
is a topological reflection in $\partial A$ that coincides with $f$ on
$\overline A$.

It is straightforward from the definitions above and the Alexander
trick (Theorem~\ref{thm:alexander}) that a Schottky map $f_T$ is
uniquely defined up to isotopy relative to $T$. Moreover, the Schottky
map $f_T$ coincides (up to isotopy rel.\ $T$) with a Schottky map
$f_{\mathring T}$ associated with the corresponding reduced Tischler
graph $\mathring T$. We summarize these and further basic properties
of Schottky maps in the proposition below.

\begin{proposition}
  \label{prop:equivalent-tischler-graphs}
  Let $T=(Q,E)$ be a marked topological Tischler graph in $S^2$ and
  $(f,Q)$ be an associated marked Schottky map. Suppose $P \subseteq
  Q$ is the set of vertices of $T$ of degree $\ge 3$.
  Then the following statements are true: 
  \begin{enumerate}
  \item\label{item:schottky-1} $f$ is an anti-Thurston map fixing all
    points in $Q$, with $P$ as both its critical and postcritical
    set. In particular, $(f,Q)$ is a marked critically fixed
    anti-Thurston map.
  \item\label{item:schottky-2} The degree of $f$ equals $\#F(T)-1$,
    and $\deg(f,v)=\deg_T(v)-1$ for all $v\in Q$.
  \item\label{item:schottky-3} $f$ is not a $(2,2,2,2)$-map. Moreover,
    $f$ has hyperbolic orbifold if and only if $\# P\geq 3$.
  \item\label{item:schottky-4} If $T' = (Q',E')$ is an isomorphic
    (resp.\ isotopic) topological Tischler graph and $(f',Q')$ is an
    associated marked Schottky map, then $(f,Q)$ and $(f',Q')$ are
    combinatorially equivalent (resp.\ isotopic).
  \end{enumerate} 
\end{proposition}
\begin{proof}
  \ref{item:schottky-1}--\ref{item:schottky-2} It is immediate from the
  definition that $f$ is an orientation-reversing branched covering of
  $S^2$ fixing all points in $Q$. Moreover, the branch points of $f$
  are exactly the vertices in $P$ and $\deg(f,v)=\deg_T(v)-1$ for all
  $v\in P$. By construction, any point inside a face $A$ of $T$ has
  exactly one (simple) preimage under $f$ in each face except $A$, so
  it has exactly $\#F(T)-1$ preimages under $f$ (counting
  multiplicities). This implies that the degree of $f$ is
  $\#F(T)-1\geq 2$. Combining all of these observations, we get that
  $(f,Q)$ is a marked critically fixed anti-Thurston map with $C_f =
  P_f = P$.

  \ref{item:schottky-3} By the above, the orbifold $\cO_f$ associated
  with $f$ is $S^2\setminus P$ with the Euler characteristic
  $\chi(\cO_f) = 2-\# P$, which implies the statement.

  \ref{item:schottky-4} First, let us show that if $(g,Q)$ is another
  marked Schottky map associated with $T$, then $(f,Q)$ and $(g,Q)$
  are isotopic. Given an arbitrary face $A$ of $T$, suppose $f_A$ and
  $g_A$ denote the topological reflections in $\partial A$ defining
  the restrictions $f|_{\overline A}$ and $g|_{\overline A}$. Then the
  map $\psi: S^2\to S^2$ given by $\psi|_T=\id_T$ and
  $\psi|_A=f^{-1}_A\circ g_A$ for $A\in F(T)$ is an
  orientation-preserving homeomorphism. Moreover, since every face $A$
  is a Jordan domain and $\psi|_{\overline A}: \overline A\to
  \psi(\overline A)={\overline A}$ is a homeomorphism with
  $\psi|_{\partial A} = \id_{\partial A}$, the Alexander trick
  (Theorem~\ref{thm:alexander}) implies that $\psi\in
  \Homeo^+_0(S^2,T) \subset \Homeo^+_0(S^2,Q)$. Since $g=f\circ \psi$
  by construction, the claim follows.

  Suppose now that $T' = (Q', E')$ is a marked topological Tischler
  graph isomorphic (resp.\ isotopic) to $T=(Q,E)$, and let $\phi:S^2
  \to S^2$ be an orientation-preserving homeomorphism providing the
  isomorphism, that is, $\phi(T)=T'$ and $\phi(Q)=Q'$. (If $T'$ is
  isotopic to $T$, we assume further that $\phi\in
  \Homeo^+_0(S^2,Q)$.) Consider now the map $g:=\phi^{-1} \circ f'
  \circ \phi$. We claim that $(g,Q)$ is a marked Schottky map
  associated with $T$, and thus $(f',Q')$ is equivalent
  (resp. isotopic) to $(f,Q)$ by the discussion above. Indeed, $g|_T=
  (\phi^{-1} \circ f' \circ \phi)|_T = \id_T$, because
  $f'|_{T'}=\id_{T'}$ and $\phi$ sends $T$ homeomorphically onto
  $T'$. Now if $A$ is an arbitrary face of $T$, then $A':=\phi(A)$ is
  a face of $T'$.  Since $f'$ maps $A'$ homeomorphically onto
  $S^2\setminus \overline {A'}$, the map $g=\phi^{-1} \circ f' \circ
  \phi$ maps $A$ homeomorphically onto $S^2\setminus \overline A$.  It
  follows that $(g,Q)$ is a marked Schottky map associated with $T$,
  which finishes the proof.
\end{proof}

The following proposition provides a sufficient condition for a given
anti-Thurston map to be isotopic to a Schottky map.

\begin{proposition}\label{prop:fixed-graph-implies-schottky}
  Let $T=(Q,E)$ be a marked topological Tischler graph, $f\colon
  S^2\to S^2$ be an orientation-reversing branched covering map, and
  suppose the following conditions are satisfied:
  \begin{enumerate}
  \item\label{sch-crit-1} the degree $d$ of $f$ equals $\#F(T)-1$;
  \item\label{sch-crit-2} $f(v)=v$ for all $v\in Q$;
  \item\label{sch-crit-3} there is a plane graph $T'=(Q,E')$ such that
    $f|_{T'}: T' \to T$ is a homeomorphism;
  \item\label{sch-crit-4} for every $e'\in E'$, $f(e')$ is homotopic
    to $e'$;
  \item\label{sch-crit-5} $f|_{T'}$ preserves the cyclic order of
    edges at every vertex of $T'$: if $e'_1, \dots, e'_k$ are the
    edges of $T'$ incident with a vertex $v\in Q$ and listed in
    counterclockwise order, then $f(e'_1), \dots, f(e'_k)$ are the
    edges of $T$ incident with $v$ and listed in counterclockwise order.
  \end{enumerate}
  Then $(f,Q)$ is isotopic to a marked Schottky map $(f_T, Q)$.
\end{proposition}

According to our convention in Section~\ref{sec:plane-graphs}, the
homotopy between $f(e')$ and $e'$ in condition \ref{sch-crit-4} is assumed
to be within the family of open arcs in $(S^2,Q)$.

\begin{proof} Conditions~\ref{sch-crit-3} and \ref{sch-crit-4} imply
  that the plane graphs $T'$ and $T$ are isotopic (see \Cref{prop:graph-isotopy-criterion}), and condition~\ref{sch-crit-5} ensures the existence of $\psi\in \Homeo^+_0(S^2,
  Q)$ such that $\psi(e')=f(e')$ for all $e'\in E'$. Furthermore, we
  may assume that $\psi|_{e'}= f|_{e'}$ for each $e'\in E'$; this
  follows, for instance, from
  \cite[Thm.~A.6(ii)]{buserGeometrySpectraCompact2010}. Set
  $g:=f\circ \psi^{-1}$. Then $g$ is an orientation-reversing branched
  self-covering of $S^2$ with $g|_T = \id_T$. 

  We claim that $B:=S^2\setminus \overline{A}\subseteq g(A)$ for every
  face $A$ of $T$. In order to prove this, fix a point
  $z_0\in \partial A \setminus C_g$. For an arbitrary point
  $z_1\in B$, choose a path $\alpha:[0,1]\to \overline B$ with
  $\alpha(0)=z_0$, $\alpha(1)=z_1$, and
  $\alpha\left((0,1]\right) \subset B$. Let $\alpha': [0,1]\to S^2$ be
  a lift of $\alpha$ by $g$ starting at $z_0$, that is, $\alpha'$
  satisfies $\alpha'(0)=z_0$ and $g\circ \alpha' = \alpha$. (Such a
  lift always exists since $g$ is a branched covering.) Since $g$ is
  an orientation-reversing local homeomorphism near $z_0$ and
  $g|_{\partial A} = \id_{\partial A}$, we have that
  $\alpha'\left((0,\epsilon)\right) \subset A$ for some
  $\epsilon>0$. Furthermore, since $\partial A = \partial B \subset T$
  is forward-invariant under $g$ and
  $\alpha((0,1]) \subset S^2 \setminus \partial B$, we have that
  $\alpha'((0,1]) \subset S^2 \setminus \partial A$ and thus
  $\alpha'\left((0,1]\right)\subset A$ by connectedness. It follows
  that $z_1=\alpha(1) =g\left(\alpha'(1)\right)\in g(A)$, which
  establishes the claim.

  Since the degree of $g$ equals the degree of $f$, which is
  $\#F(T) -1$, we can conclude that $g$ maps every $A\in F(T)$
  homeomorphically onto $S^2\setminus \overline{A}$. It follows that
  $(g,Q)$ is a marked Schottky map associated with $T$. Since
  $P_g=C_g\subseteq Q$ and $f=g \circ \psi$ with
  $\psi\in \Homeo^+_0(S^2, Q)$, we get that $(f, Q)$ and $(g,Q)$ are
  isotopic and the statement follows.
\end{proof}

\begin{corollary}
  \label{cor:crit-fixed-schottky}
  Let $f$ be a critically fixed anti-rational map with the alternating
  (or reduced) Tischler graph $T = (Q,E)$, and $(f_T,Q)$ be an
  associated marked Schottky map. Then $(f,Q)$ is isotopic to
  $(f_T,Q)$.
\end{corollary}
\begin{proof}
  Since $f$ fixes all vertices of $T$ and maps every edge of $T$
  homeomorphically onto itself, the statement immediately follows from
  \Cref{prop:fixed-graph-implies-schottky} applied to the map $f$
  and the graph $T$ with $T'=T$.
\end{proof}

\subsection{Realizability of Schottky maps}
In the following, let $T = (Q,E)$ be a marked topological Tischler
graph with an associated marked Schottky map $(f,Q)$. Our task now is
to understand for which graphs $T$ the marked anti-Thurston map $(f,Q)$ is
realizable, that is, when $(f,Q)$ is combinatorially equivalent to a
marked anti-rational map.

Let us first analyze the situation where $\#C_f= 2$. (Recall from
\Cref{prop:equivalent-tischler-graphs}\ref{item:schottky-1} that
the critical set $C_f$ consists of all vertices of $T$ of degree $\ge
3$.) In this case, every edge of the corresponding reduced topological
Tischler graph $\mathring T$ has to connect the two vertices in
$C_f$. Then $\mathring T$ is isomorphic to the reduced Tischler graph
of $g(z) = \bar{z}^d$ with $d=\#F(\mathring T)-1 \ge 2$, and 
\Cref{prop:equivalent-tischler-graphs}\ref{item:schottky-4} and
\Cref{cor:crit-fixed-schottky} imply that the Schottky map
$f$ is combinatorially equivalent to $g$. Since every edge of the
reduced Tischler graph of $g$ contains exactly one fixed point of $g$,
it easily follows that the marked Schottky map $(f,Q)$ is obstructed
if and only if an edge of the reduced graph $\mathring T$ contains
more than one vertex of $T$ (or equivalently, if $T$ has two adjacent
vertices of degree 2). Otherwise, in the case where every edge of
$\mathring T$ contains at most one vertex of $T$, the marked Schottky
map $(f,Q)$ is combinatorially equivalent to a marked anti-rational map
$(g,Q')$, where $Q'$ consists of $0$, $\infty$, and some subset of the
$(d+1)$-st roots of unity (which are the repelling fixed points of
$g$).

In the general case, when $\#C_f\geq 2$, it is also easy to see that
the marked Schottky map $(f,Q)$ is obstructed whenever an edge of the
corresponding reduced topological Tischler graph $\mathring T$
contains at least two vertices of $T$. However, this is not the only
source of ``non-realizability''. For instance, a Schottky map
associated with the reduced topological Tischler graph from
\Cref{fig:obstruction} is obstructed.  The following definition
combines the structure of these two examples and provides a necessary
and sufficient condition for (non-)realizability of the marked
Schottky map $(f,Q)$; see \Cref{thm:schottky-obstructions} below.

\begin{definition}\label{def:Tischler-graph-obstructed}
  A marked topological Tischler graph $T$ is called \emph{obstructed}
  if there is a pair $A,B\in F(T)$ of distinct faces that share two
  distinct non-adjacent boundary edges (that is, $\partial A \cap
  \partial B$ contains at least two non-adjacent edges of $T$);
  otherwise, we say that $T$ is \emph{unobstructed}.
\end{definition}

\begin{remark}
  The above definition is easily seen to be equivalent to the
  following one: a marked topological Tischler graph $T$ is obstructed
  if and only if its dual graph $T^*$ has a cycle of length $2$ that
  does not bound a face of $T^*$. Moreover, we may simplify Definition
  \ref{def:Tischler-graph-obstructed} when $T$ is reduced. In this
  case, for all distinct faces $A,B\in F(T)$, any two edges of $T$ in
  $\partial A \cap \partial B$ are automatically non-adjacent. So the
  condition simply becomes that $T$ is obstructed if there are two
  distinct faces of $T$ sharing more than one boundary edge, or
  equivalently, if the dual graph $T^*$ has multiple edges.
\end{remark}

\begin{theorem}
  \label{thm:schottky-obstructions}
  Let $T=(Q,E)$ be a marked topological Tischler graph and $(f,Q)$ be
  an associated marked Schottky map. Then one of the following two mutually
  exclusive cases occurs.
  \begin{description}
  \item[$\bm{T}$ is obstructed] In this case, the marked anti-Thurston
    map $(f,Q)$ is not equivalent to a marked anti-rational
    map. Moreover, if $A,B\in F(T)$ are two distinct faces that share
    two distinct non-adjacent boundary edges $a,b\in E(T)$, then any
    simple closed curve $\gamma\subset S^2\setminus Q$ that intersects
    (transversely) the graph $T$ exactly twice, once in $a$ and once
    in $b$, is a fixed Levy curve for $(f,Q)$.
  \item[$\bm{T}$ is unobstructed] In this case, the marked
    anti-Thurston map $(f,Q)$ is equivalent to a marked critically
    fixed anti-rational map, unique up to Möbius conjugacy.
  \end{description}
\end{theorem}

\begin{remarks}\mbox{}
\begin{enumerate}
\item A direct consequence of this theorem is that the alternating and
  reduced Tischler graphs of every critically fixed anti-rational map
  are unobstructed, when viewed as marked topological Tischler graphs.
\item In the case where $T$ is reduced, a similar statement was shown
  in \cite{parkLevyThurstonObstructions2024} as an application of a
  more general framework.
\end{enumerate}
\end{remarks}

To prepare for the proof, we first have a closer look at the pullback
action of Schottky maps on simple closed curves in $S^2 \setminus
Q$. Some of the following parallels the concepts and arguments in the
orientation-preserving case from
\cite{hlushchankaTischlerGraphsCritically2019}. As in
Section~\ref{sec:thurst-theor-orient}, we adopt the convention that
homotopy of simple closed curves means free homotopy in
$S^2 \setminus Q$.

\begin{figure}[t]
  \begin{tikzpicture}
    \begin{scope}
      \coordinate (A) at (-1,-1);
      \coordinate (B) at ( 1,-1);
      \coordinate (C) at ( 1, 1);
      \coordinate (D) at (-1, 1);
      
      \coordinate (b) at (-1.5,0);
      
      \draw [fill] (A) circle (1mm);
      \draw [fill] (B) circle (1mm);
      \draw [fill] (C) circle (1mm);
      \draw [fill] (D) circle (1mm);
      \draw (A) -- (B) -- (C) -- (D) -- (A);
      \draw (D) arc (90:270:1); 
      \draw (B) arc (-90:90:1);
      \draw [dashed] (b) circle (1.5);
      \draw (0.25,0.4) node {$\gamma$};
    \end{scope}

    \begin{scope}[xshift=7cm]
      \coordinate (A) at (-1,-1);
      \coordinate (B) at ( 1,-1);
      \coordinate (C) at ( 1, 1);
      \coordinate (D) at (-1, 1);
      
      \draw [fill] (A) circle (1mm);
      \draw [fill] (B) circle (1mm);
      \draw [fill] (C) circle (1mm);
      \draw [fill] (D) circle (1mm);
      \draw (A) -- (B) -- (C) -- (D) -- (A);
      \draw (D) arc (90:270:1); 
      \draw (B) arc (-90:90:1);
      \draw (0.25,0.4) node {$\gamma$};

      \coordinate (a) at (-3,0);
      \coordinate (b) at (-1.5,0);
      \coordinate (c) at (0,0);
      \coordinate (d) at (1.5,0);
      \coordinate (e) at (1.5,2); %
      \coordinate (f) at (-3,2);  %
      \draw [dashed] (a) -- (b) -- (c) -- (d);
      \draw [dashed] (b) circle (1.5);
      \draw [dashed] (d) arc [
      start angle = -90,
      end angle = 90,
      x radius = 1,
      y radius = 1,
      ];
      \draw [dashed] (e) -- (f);
      \draw [dashed] (f) arc[
      start angle = 90,
      end angle = 270,
      x radius = 1,
      y radius = 1,
      ];
      \draw [fill=white] (a) circle (1mm);
      \draw [fill=white] (b) circle (1mm);
      \draw [fill=white] (c) circle (1mm);
      \draw [fill=white] (d) circle (1mm);
    \end{scope}
  \end{tikzpicture}
  \caption{On the left, a reduced topological Tischler graph $T$ whose
    associated Schottky map $f_T$ is obstructed. The dashed circle
    $\gamma$ is a fixed Levy curve, since $f_T^{-1}(\gamma)$ has one
    connected component $\gamma'$ that is homotopic to $\gamma$ and
    that is mapped homeomorphically onto $\gamma$. On the right, a
    dual graph of $T$ is indicated as well, with dashed edges
    and white vertices. The upper and lower halves of the fixed Levy
    curve are edges in the dual graph.}
  \label{fig:obstruction}
\end{figure}

\begin{definition}
  \label{def:complexity}
  Let $G = (Q,E)$ be a plane graph, and let $\gamma$ be a simple
  closed curve in $S^2 \setminus Q$. The \emph{complexity} of $\gamma$
  (with respect to $G$) is the intersection number $\gamma \cdot G$,
  defined as the minimal number of intersections of $\gamma'$ with $G$
  over all simple closed curves $\gamma'\subset S^2 \setminus Q$ that
  are homotopic to $\gamma$. We say that $\gamma$ is \emph{minimal}
  (with respect to $G$) if it realizes the minimal number of
  intersections with $G$, that is, if $\gamma\cdot G = \#(\gamma \cap
  G)$.
\end{definition}
\begin{lemma}
  \label{lem:complexity}
  Let $T = (Q,E)$ be a marked topological Tischler graph, $f=f_T$ be an
  associated Schottky map, and let $\gamma$ be a minimal simple closed
  curve in $S^2 \setminus Q$. Suppose $\gamma^1, \ldots, \gamma^n$ are the
  connected components of $f^{-1}(\gamma)$. Then
  \begin{equation*}
    \sum_{j=1}^n \gamma^j \cdot T \le \gamma \cdot T.
  \end{equation*}
  The inequality above is strict if and only if one of the curves $\gamma^j$ is not
  minimal. In particular, $\gamma^1 \cdot T \le \gamma \cdot T$, with
  equality if and only if $\gamma^1$ is minimal and
  $\gamma^j \cap T = \emptyset$ for $j=2, \ldots, n$.
\end{lemma}

For an illustration of the decrease of complexity under repeated
pullbacks in an example, see Figure~\ref{fig:pullback-1}.

\begin{proof}
  Because $\gamma$ is minimal, $\# (\gamma \cap T) = \gamma \cdot T$,
  and since $f|_T$ is the identity, we know that
  $\gamma \cap T = f^{-1}(\gamma) \cap T$. It follows that
  \[
    \sum_{j=1}^n \gamma^j \cdot T \le \sum_{j=1}^n \# (\gamma^j \cap T)
    = \# (\gamma \cap T) = \gamma \cdot T.
  \]
  The first inequality is strict if and only if one of the curves $\gamma^j$ is not
  minimal. The last claim of the lemma is an immediate consequence.
\end{proof}

\begin{proof}[Proof of Theorem \ref{thm:schottky-obstructions}]

  Let $T=(Q,E)$ be a marked topological Tischler graph and $(f,Q)$ be
  an associated marked Schottky map. Since $f$ is not a
  $(2,2,2,2)$-map (see
  \Cref{prop:equivalent-tischler-graphs}\ref{item:schottky-3}), the
  orientation-reversing Thurston theorem (see
  Theorem~\ref{thm:anti-thurston}) implies that $(f,Q)$ is equivalent
  to a marked anti-rational map if and only if $(f,Q)$ does not have a
  multicurve obstruction, and that the equivalent marked anti-rational
  map is unique up to Möbius conjugacy.
  
  We start with the easy part. Assume that $T$ is obstructed, that is,
  it has a pair of distinct faces $A$ and $B$ sharing two distinct
  non-adjacent boundary edges $a$ and $b$.  Let
  $\gamma\subset S^2\setminus Q$ be a simple closed curve that
  intersects (transversely) the graph $T$ exactly twice, once in $a$
  and once in $b$. First of all, we claim that $\gamma$ is
  essential. Indeed, let $U$ be a component of $S^2 \setminus
  \gamma$. Since $a$ and $b$ are non-adjacent edges, $U$ must contain
  one vertex incident with $a$, and another vertex incident with $b$. So
  $U$ contains at least $2$ marked points, and thus $\gamma$ is
  essential. If we denote the intersection points of $\gamma$ with $a$
  and $b$ by $z$ and $w$, then $f^{-1}(\gamma) \cap A$ is an open arc
  $\gamma'_A$ in $A$ connecting $z$ and $w$, and
  $f^{-1}(\gamma) \cap B$ is an open arc $\gamma'_B$ in $B$ connecting
  $z$ and $w$. It follows that the union
  $\gamma'_A\cup \gamma'_B \cup \{z,w\}$ is a simple closed curve
  $\gamma'$ in $S^2\setminus Q$, which again intersects the graph $T$
  exactly at $z$ and $w$, and hence $\gamma'$ is homotopic to $\gamma$
  in $S^2 \setminus Q$. Since, by definition, $f$ maps the arcs
  $\gamma'_A$ and $\gamma'_B$ homeomorphically onto $\gamma\cap B$ and
  $\gamma\cap A$, respectively, we have that
  $\deg(f: \gamma'\to \gamma) = 1$, and thus $\gamma$ is a fixed Levy
  curve for $(f,Q)$. Hence, when the graph $T$ is obstructed, the marked anti-Thurston map $(f, Q)$ is not equivalent to a marked anti-rational map.

  To prove the converse implication, suppose that $(f,Q)$ has an
  obstruction $\Gamma=(\gamma_1,\dots,\gamma_m)$. Since the Thurston
  matrix $A_{f,Q,\Gamma}=(a_{jk})_{j,k=1}^m$ depends only on the
  homotopy classes of curves $\gamma_k\in\Gamma$ in $S^2 \setminus Q$,
  we may assume that the curves in $\Gamma$ are minimal with respect
  to the graph $T$. By Lemma~\ref{lem:obstruction-types}, we may
  further assume that $\Gamma$ is an irreducible obstruction, so that
  the associated matrix $A_{f,Q,\Gamma}$ is irreducible. Let $D_A$ be
  the \emph{digraph associated with $A_{f,Q,\Gamma}$}, that is, $D_A$
  is a directed graph with the vertex set $\Gamma$ and a directed edge
  from vertex $\gamma_j$ to vertex $\gamma_k$ precisely when
  $a_{jk}\neq 0$ (i.e., when $f^{-1}(\gamma_k)$ has a component
  homotopic to $\gamma_j$). Since $A_{f,Q,\Gamma}$ is irreducible,
  this digraph is \emph{strongly connected}, meaning that there is a
  path in each direction between every pair of vertices of $D_A$ (see,
  for instance, \cite{brualdiCombinatorialMatrixTheory1991}).

  From Lemma~\ref{lem:complexity}, we get that
  $\gamma_j \cdot T \le \gamma_k \cdot T$ whenever $f^{-1}(\gamma_k)$
  has a component homotopic to $\gamma_j$, so that the complexity of
  curves in $\Gamma$ (with respect to $T$) does not decrease along the
  edges in the digraph $D_A$. Since $D_A$ is strongly connected, it
  follows that $\gamma_k\cdot T$ is constant, independent of
  $k$. Again from Lemma~\ref{lem:complexity}, we conclude that there
  is only one essential connected component $\gamma'_k$ of
  $f^{-1}(\gamma_k)$, which must be minimal with respect to $T$ and
  homotopic to some $\gamma_j\in \Gamma$, while all other components
  of $f^{-1}(\gamma_k)$ are disjoint from $T$. This means that the
  digraph $D_A$ has only one incoming edge to every
  $\gamma_k \in \Gamma$. A strongly connected digraph with this
  property has to be a cycle. Up to permutation of indices, we may
  therefore assume that $\gamma'_k$ is homotopic to $\gamma_{k-1}$ for
  each $k=1,\ldots, m$ (with the convention that indices are
  understood modulo $m$). The associated Thurston linear map
  $L=L_{f,Q,\Gamma}$ is then given by
  $L(\gamma_k) = \frac{1}{d_{k-1,k}} \gamma_{k-1}$, where
  $d_{k-1,k}=\deg(f: \gamma'_k\to \gamma_k)$. Since the leading
  eigenvalue of $L$ is
  $\lambda = \left(\prod_{k=1}^m \frac{1}{d_{k-1,k}}\right)^{1/m}$ and
  $\Gamma$ is an obstruction, we conclude that $\lambda = 1$ and
  $d_{k-1,k} = 1$ for all $k=1,\ldots, m$. In other words, $\Gamma$ is
  a Levy cycle.

  By the discussion above, $f$ maps $\gamma_1'$ one-to-one onto $\gamma_1$, and
  $\#(\gamma_1 \cap T) = \gamma_1 \cdot T = \gamma_0\cdot T = \gamma_1' \cdot T =
  \#(\gamma_1' \cap T)$. Since $f$ is the identity on $T$,
  this implies that $\gamma_1 \cap T = \gamma_1' \cap T$. Let
  $A_1, \ldots, A_r$ be the distinct faces of $T$ that intersect
  $\gamma_1$. Then for each $k=1,\dots,r$ the set $f^{-1}(\gamma_1) \cap A_k$ consists of (disjoint) open 
  arcs connecting the points
  in $\gamma_1 \cap \partial A_k$ within $A_k$, so that
  $f^{-1}(\gamma_1) \cap A_k = \gamma_1' \cap A_k$. (All the other
  connected components of $f^{-1}(\gamma_1)$ are contained in faces
  disjoint from $\gamma_1$.) By the definition of Schottky maps, every point in $\gamma_1 \cap A_1$ has
  exactly one preimage in $\gamma_1' \cap A_k$ for each 
  $k=2,\ldots,r$. Since $\deg(f: \gamma_1'\to \gamma_1) = 1$, we must
  have $r=2$, so any edge crossed by $\gamma_1$ is a common boundary edge
  between $A_1$ and $A_2$. Finally, because $\gamma_1$
  is simple and essential, $\partial A_1\cap \partial A_2$ must contain at least two 
  non-adjacent edges, which means that $T$ is obstructed. This completes the proof of the theorem.
\end{proof}

\subsection{Uniqueness of Schottky models}
\label{sec:uniq-schottky}

To establish a one-to-one correspondence between the isomorphism
classes of unobstructed reduced topological Tischler graphs and the
Möbius conjugacy classes of critically fixed anti-rational maps (see
Theorem \ref{thm:tischler-fixed-correspondence}), we have to check
that equivalence of Schottky maps implies isomorphism of the
corresponding graphs (see \Cref{prop:schottky-uniqueness}).  More
precisely, we will show the following: If $(f,Q)$ is a marked Schottky
map associated with two marked topological Tischler graphs $T$ and
$T'$, then these graphs are isotopic.

First, we introduce some auxiliary terminology.  Recall that an open
arc $\gamma$ in $(S^2,Q)$ is an open arc
$\gamma \subset S^2 \setminus Q$ with endpoints $p,q \in Q$.  We say
that such an open arc is \emph{essential} (in $(S^2,Q)$) if $p \ne q$,
i.e., if the endpoints are distinct.  Just as for open arcs, homotopy
and isotopy of essential open arcs in $(S^2,Q)$ are always considered
within the family of essential open arcs; in particular, the endpoints
must be fixed.

Let $(f, Q)$ be a marked \antiThurston map of degree $d$, and let
$\gamma$ be an open arc in $(S^2,Q)$. Then the preimage
$f^{-1}(\gamma)$ consists of exactly $d$ disjoint open arcs in
$S^2 \setminus Q$, called the \emph{preimage arcs} of $\gamma$ (under
$f$), each of which is mapped homeomorphically onto $\gamma$. Note
that some of these preimage arcs may have common endpoints. Moreover,
even if $\gamma$ is essential in $(S^2,Q)$, its preimage arcs are not
necessarily essential in $(S^2,Q)$. By homotopy lifting, the marked
\antiThurston map $(f,Q)$ induces a \emph{pullback relation
  $\leftarrow$} on the set of homotopy classes of essential open arcs
in $(S^2,Q)$. Namely, we define $[\gamma]\leftarrow[\gamma']$
whenever $\gamma'$ is an essential preimage arc of an essential open
arc $\gamma$.

We say that an essential open arc $\gamma$ in $(S^2,Q)$ is
\emph{invariant} (with respect to $(f,Q)$) if one of its preimage arcs
is homotopic to $\gamma$, that is,
$[\gamma]\leftarrow[\gamma]$. Furthermore, $\gamma$ is
\emph{$n$-periodic} with \emph{period} $n \ge 1$ if it is invariant
with respect to
$(f^n,Q)$.

The following definition mimics the analogous one for simple closed
curves (cf.\ Definition~\ref{def:complexity}).

\begin{definition}
  \label{def:complexity-arcs}
  Let $G = (Q, E)$ be a plane graph, and let $\gamma$ be an essential
  open arc in $(S^2,Q)$. The \emph{complexity} of $\gamma$ (with
  respect to $G$) is the intersection number $\gamma \cdot G$, defined
  as the minimal number of intersections of $\gamma'$ with $G$ over
  all essential open arcs $\gamma'$ in $(S^2,Q)$ that are homotopic to
  $\gamma$.  We say that $\gamma$ is \emph{minimal} (with respect to
  $G$) if $\gamma\cdot G = \#(\gamma\cap G)$.
\end{definition}
\begin{remark}
  Since we consider arcs to be open, we do not count the endpoints of
  $\gamma$ in $Q$ as intersection points with $G$. In particular, if
  $\gamma$ is an edge of $G$, then $\gamma\cdot G = 0$, since we can
  homotope $\gamma$ to a ``parallel'', slightly displaced disjoint
  arc.
\end{remark}

\begin{lemma}
  \label{lem:arc-pullbacks}
  Let $T=(Q,E)$ be a marked topological Tischler graph, let $f=f_T$ be
  an associated Schottky map, and let $\gamma$ be an essential open
  arc in $(S^2,Q)$ that is minimal with respect to $T$.  Suppose
  $\gamma'$ is a preimage arc of $\gamma$ that is essential in
  $(S^2,Q)$. Then $\gamma'\cdot T \leq \gamma \cdot T$.  Moreover, if
  equality occurs, then $\gamma'$ is also minimal with respect to $T$,
  and $\gamma \cup \gamma'$ is contained in the closure of two faces
  of the graph $T$.
\end{lemma}
\begin{proof}
  Let $T$, $f$, and $\gamma$ be as in the statement, and suppose that
  $\gamma'$ is an essential preimage arc of $\gamma$. Since $\gamma$
  is minimal and $f$ is the identity on $T$, we have that
  \begin{equation}\label{eq:arc-complexity}
    \gamma'\cdot T \leq \#\left(\gamma' \cap T\right) \leq
    \#\left(f^{-1}(\gamma) \cap T \right)=  \#\left(\gamma \cap
      T\right)  = \gamma \cdot T,
  \end{equation}
  which establishes the first part of the lemma. 

  Suppose now that $\gamma'\cdot T = \gamma \cdot T$. Then it follows
  from \eqref{eq:arc-complexity} that $\gamma'$ is minimal with
  respect to $T$, and that no other component of $f^{-1}(\gamma)$
  intersects $T$. Moreover, we have $\gamma' \cap T = \gamma \cap T$,
  because $f|_T = \id_T$. We also note that, since $\gamma'$ is
  essential in $(S^2,Q)$ and $f$ fixes all points in $Q$, the
  endpoints of $\gamma'$ and $\gamma$ must coincide.
  
  Assume that $\gamma$ meets more than two faces of $T$.  Then there
  is a sub-arc of $\gamma$ that meets three different faces
  $A_0, A_1, A_2\in F(T)$ consecutively and that intersects $T$ at
  $z_1 \in \partial A_0 \cap \partial A_1$ and
  $z_2 \in \partial A_1 \cap \partial A_2 $. Let $\gamma_1$ be the
  sub-arc of $\gamma$ in $A_1$ with endpoints $z_1$ and $z_2$. Then
  $\gamma_1$ has one preimage arc in $A_0$ with endpoint $z_1$ and one
  preimage arc in $A_2$ with endpoint $z_2$. Since $\gamma'$ contains
  both $z_1$ and $z_2$ (because $\gamma'\cap T = \gamma\cap T$), it
  also contains both of these preimage arcs of $\gamma_1$ as
  sub-arcs. This contradicts the fact that $f$ maps $\gamma'$
  homeomorphically onto $\gamma$.

  It follows that $\gamma$ meets at most two faces of $T$. If $\gamma$
  meets only one face $A_0$, then
  $\gamma' \cdot T = \gamma \cdot T = 0$, so $\gamma'$ also meets only
  one face of $T$, and since $f(A_0) \cap A_0 = \emptyset$, it must be
  a different face $A_1$. Hence
  $\gamma\cup\gamma' \subset \overline{A_0\cup A_1}$.

  If $\gamma$ meets two faces $A_0$ and $A_1$, then all points in
  $\gamma' \cap T = \gamma \cap T$ are in
  $\partial A_0 \cap \partial A_1$, and thus $\gamma'$ also meets
  exactly $A_0$ and $A_1$. (See the left-hand side of \Cref{fig:figure8} for an
  illustration of this case.) This finishes the proof of the lemma.
\end{proof}

The following corollary describes the structure of periodic essential
open arcs with respect to a marked Schottky map.

\begin{corollary}
  \label{cor:periodic-arcs}
  Let $T=(Q,E)$ be a marked topological Tischler graph, let $f=f_T$ be
  an associated Schottky map, and let $\gamma$ be an essential open
  arc in $(S^2,Q)$ that is minimal with respect to $T$.  Suppose
  $ \gamma_0=\gamma, \gamma_1,\dots, \gamma_{n}$ is a sequence of open
  arcs in $(S^2,Q)$ such that
  \begin{itemize}
      \item $\gamma_{j}$ is a preimage arc of $\gamma_{j-1}$ for each $j=1,\dots, n$;
      \item $\gamma_n$ is homotopic to $\gamma$;
      \item $\gamma_0, \dots, \gamma_{n-1}$ are pairwise non-homotopic. 
  \end{itemize}
  Then one of the following three mutually exclusive cases occurs: 
  \begin{enumerate}
  \item\label{item:periodic-i} $n=1$ and $\gamma$ is homotopic to an edge of $T$;
  \item\label{item:periodic-ii} $n=2$ and $\gamma$ meets exactly two faces of $T$;
  \item\label{item:periodic-iii} $n\geq 2$ and every $\gamma_j$, $j=0,\dots, {n}$, meets exactly one face $A_j$ of $T$. Moreover, the faces $A_0,\dots, A_{n-1}$ are all distinct and they all share the same two boundary vertices, namely, the endpoints of $\gamma$. 
  \end{enumerate}
\end{corollary}

See the left-hand side of \Cref{fig:figure8} for an illustration
of case \ref{item:periodic-ii}, and its right-hand side for an
illustration of case \ref{item:periodic-iii}.

\begin{proof}
  The proof is straightforward and we leave some of the details to the
  reader. Let $\gamma_0=\gamma, \gamma_1,\dots, \gamma_{n}$ be a
  sequence of open arcs in $(S^2,Q)$ as in the statement. Since
  $\gamma_n$ is homotopic to $\gamma$, all these arcs must be
  essential in $(S^2,Q)$. By Lemma~\ref{lem:arc-pullbacks}, we have
  \begin{equation}\label{eq:periodic-arc}
    \gamma \cdot T= \gamma_0\cdot T \geq \gamma_1 \cdot T \geq \dots \geq
    \gamma_n\cdot T = \gamma \cdot T.
  \end{equation}
  It follows that all inequalities in \eqref{eq:periodic-arc} are in
  fact equalities, and $\gamma_j\cup \gamma_{j+1}$ is contained in the
  closure of two faces of $T$ for all $j=0,\dots, n-1$. Now, if
  $\gamma$ meets exactly two faces of $T$, then $n=2$ and $\gamma$ is
  $2$-periodic (this easily follows from \Cref{lem:arc_isotopy}). Otherwise, every open arc $\gamma_j$, $j=0,\dots, n$,
  meets exactly one face $A_j$ of $T$. Furthermore, the endpoints of
  $\gamma_0, \gamma_1,\dots, \gamma_{n}$ must coincide, and thus the
  faces $A_0, A_1,\dots, A_{n}$ share the same two boundary vertices,
  given by the endpoints of $\gamma$. Since
  $\gamma_0, \dots, \gamma_{n-1}$ are pairwise non-homotopic and the
  faces of $T$ are Jordan domains, we have that the faces
  $A_0, A_1,\dots, A_{n-1}$ are all distinct. Finally, when $n=1$, one
  of the two complementary components of the Jordan curve
  $\overline{\gamma_0}\cup \overline{\gamma_1}$ cannot contain the vertices of $T$, and
  therefore $\gamma$ is homotopic to an edge of $T$.
\end{proof}

\begin{figure}[t]
  \centering
  \begin{tikzpicture}
  \begin{scope}[xscale=.9]
    \coordinate (P) at (0,0);
    \coordinate (P1) at (1,0);
    \coordinate (P2) at (2,0);
    \coordinate (Z1) at (3,0);
    \coordinate (Z1+) at (3,2);
    \coordinate (Z1-) at (3,-2);
    \coordinate (A0) at (3,1.2);
    \coordinate (A1) at (3,-1.2);
    \coordinate (Q2) at (4,0);
    \coordinate (Q1) at (5,0);
    \coordinate (Z2) at (6,0);
    \coordinate (Q) at (7,0);
    \draw (Z1) node [label=90:$z_1$] {};
    \draw (Z2) node [label={[label distance=-2mm]below right:$z_2$}] {};
    \draw (A0) node {$A_0$};
    \draw (A1) node {$A_1$};
    \draw [fill] (P) circle (1mm);
    \draw [fill] (P1) circle (1mm);
    \draw [fill] (P2) circle (1mm);
    \draw [fill] (Q2) circle (1mm) node[label={[label
      distance=-2mm]135:$q$}] {};
    \draw [fill] (Q1) circle (1mm);
    \draw [fill] (Q) circle (1mm);
    \draw (P) node[label=180:$p$] {}  -- (P1);
    \draw (Q) -- (Q1);
    \draw (P2) -- (Q2);
    \draw (P1) to [out=60,in=120] (P2);
    \draw (P1) to [out=-60,in=-120] (P2);
    \draw (Q2) to [out=60,in=120] (Q1);
    \draw (Q2) to [out=-60,in=-120] (Q1);
    \draw (P) to [out=120,in=180] (Z1+) to [out=0,in=60] (Q);
    \draw (P) to [out=-120,in=180] (Z1-) to [out=0,in=-60] (Q);
    \draw [red] (P) to [out=60,in=120] node[midway,above] {$\gamma$}
    (Z1);
    \draw [blue] (P) to [out=-60,in=-120] (Z1);
    \draw [red] (Z1) to [out=-60,in=-90] (Z2);
    \draw [red] (Z2) to [out=90,in=90] (Q2);
    \draw [blue] (Z1) to [out=60,in=90] node[midway,above]
    {$\gamma'$} (Z2);
    \draw [blue] (Z2) to [out=-90,in=-90] (Q2); 
    \draw [fill] (P) circle (1mm);
    \draw [fill] (Q1) circle (1mm);
    \draw [fill] (Q2) circle (1mm);
  \end{scope}
  \begin{scope}[xshift=11cm]
     \coordinate (L) at (-2,0);
     \coordinate (R) at (2,0);
     \coordinate (T3) at (0,2.8);
     \coordinate (T2) at (0,2);
     \coordinate (T1) at (0,1.2);
     \coordinate (B1) at (0,-1.2);
     \coordinate (B2) at (0,-2);
     \coordinate (B3) at (0,-2.8);
     \coordinate (B4) at (0,-3.6);
     \coordinate (A0) at (2.5,2.2);
     \coordinate (A1) at (0.7,0.7);
     \coordinate (A2) at (0.7,-0.7);
     \coordinate (C1) at (.6,0);
     \coordinate (C2) at (1.2,0);
     \coordinate (D1) at (-.5,-3.6);
     \coordinate (D2) at (.5,-3.6);
     \draw[fill] (L) circle (1mm);
     \draw[fill] (R) circle (1mm);
     \draw[fill] (T2) circle (1mm);
     \draw[fill] (T1) circle (1mm);
     \draw[fill] (B1) circle (1mm);
     \draw[fill] (B2) circle (1mm);
     \draw[fill] (C1) circle (1mm);
     \draw[fill] (C2) circle (1mm);
     \draw (L) node[label=180:$p$] {} -- (C1);
     \draw (C2) --(R) node[label=0:$q$] {};
     \draw (C1) to [out=60,in=120] (C2);
     \draw (C1) to [out=-60,in=-120] (C2);
     \draw (L) to [out=45, in=180] (T1);
     \draw (T1) to [out=0, in=135] (R);
     \draw (L) to [out=90, in=180] (T2);
     \draw (T2) to [out=0, in=90] (R);
     \draw (L) to [out=-45, in=180] (B1);
     \draw (B1) to [out=0, in=-135] (R);
     \draw (L) to [out=-90, in=180] (B2);
     \draw (B2) to [out=0, in=-90] (R);
     \draw (T2) -- (T1);
     \draw (B2) -- (B1);
     \draw (A0) node {$A_0=A_3$};
     \draw (A1) node {$A_1$};
     \draw (A2) node {$A_2$};
     \draw[red] (L) to [out=25, in=155] node[midway, below] {$\gamma_1$} (R);
     \draw[red] (L) to [out=-25, in=-155] node[midway, above] {$\gamma_2$} (R);
     \draw[red] (L) to [out=120, in=180] (T3) node[below] {$\gamma = \gamma_0$} to [out=0, in=60] (R);
     \draw[red] (L) to [out=-120, in=-180] (B3) node[below] {$\gamma_3$} to [out=0, in=-60] (R);
     \draw [fill] (L) circle (1mm);
     \draw [fill] (R) circle (1mm);
  \end{scope}
  \end{tikzpicture}
  \caption{On the left, the open arc $\gamma'$ (in blue) is an
      essential preimage arc of the open arc $\gamma$ (in
      red) under the marked Schottky map associated with the marked
      topological Tischler graph shown in black. Note that $\gamma$ is
      $2$-periodic, namely
      $[\gamma]\leftarrow [\gamma']\leftarrow[\gamma]$. On the right,
      each open arc $\gamma_j$ (in red), $j=1,2,3$, is an
      essential preimage arc of the open arc $\gamma_{j-1}$
      (in red) under the marked Schottky map associated with the
      marked topological Tischler graph shown in black. Note that the
      arc $\gamma=\gamma_0$ is both $3$-periodic and $2$-periodic,
      because
      $[\gamma]=[\gamma_0]\leftarrow
      [\gamma_1]\leftarrow[\gamma_2]\leftarrow[\gamma_3]=[\gamma]$ and
      $[\gamma]\leftarrow [\gamma_1]\leftarrow [\gamma]$. Hence,
      $\gamma$ is $n$-periodic for each $n\geq 2$. However, $\gamma$
      is not $1$-periodic, since it is not homotopic to an edge of the
      corresponding marked topological Tischler graph.}
  \label{fig:figure8}
\end{figure}

The following proposition provides the converse to
\Cref{prop:equivalent-tischler-graphs}\ref{item:schottky-4}.
\begin{proposition}
  \label{prop:schottky-uniqueness}
  Let $T=(Q,E)$ and $T'=(Q',E')$ be two marked topological Tischler
  graphs, and let $f = f_T$ and $f'=f_{T'}$ be associated Schottky
  maps. If the marked Schottky maps $(f, Q)$ and $(f', Q')$ are
  equivalent (resp.\ isotopic), then the marked topological Tischler
  graphs $T$ and $T'$ are isomorphic (resp.\ isotopic).
\end{proposition}
\begin{proof}
  Let $T$, $T'$, $f$, and $f'$ be as in the statement. The case $\#Q=2$ is immediate, so we assume throughout the following that $\#Q\geq 3$.  First we prove
  that the plane graphs $T$ and $T'$ are isotopic if the maps $(f,Q)$
  and $(f',Q')$ are isotopic. For this we need the following claim
  whose proof is straightforward from the definition of Schottky maps
  and Lemma~\ref{lem:arc-pullbacks}; we thus leave it to the reader.
  Recall that the multiplicity $m_T(e)$ of an edge $e$ in a plane
  graph $T$ is defined as the total number of edges of $T$ that are
  homotopic to $e$.

  \begin{claim*}
    For every edge $e \in E$, the multiplicity $m_T(e)$ equals the
    number of preimage arcs of $e$ under $f$ that are homotopic to
    $e$.
  \end{claim*}

  The above claim, together with homotopy lifting and  \Cref{cor:periodic-arcs}\ref{item:periodic-i}, implies that for
  every edge $e\in E$ there is an edge $e'\in E'$ homotopic to $e$
  such that $m_T(e)=m_{T'}(e')$ (and vice versa). Hence $\# E = \#E'$,
  and the plane graphs $T$ and $T'$ are isotopic by
  \Cref{prop:graph-isotopy-criterion}.

  Suppose now that $\phi,\psi: S^2\to S^2$ provide a combinatorial
  equivalence between $f$ and $f'$, so that
  $\phi \circ f = f' \circ \psi$. Set
  $g:=\phi^{-1}\circ f' \circ \phi$. By the proof of
  \Cref{prop:equivalent-tischler-graphs}\ref{item:schottky-4}, the
  map $(g, Q)$ is a marked Schottky map associated with the graph
  $\phi^{-1}(T')$ with the vertex set $Q=\phi^{-1}(Q')$. At the same
  time, the maps $(g,Q)$ and $(f,Q)$ are isotopic, because
  $f=g\circ (\phi^{-1}\circ \psi)$ and
  $\phi^{-1}\circ \psi \in \Homeo^+_0(S^2,Q)$. By the above
  discussion, we have that the graphs $\phi^{-1}(T')$ and $T$ are
  isotopic, and thus $T'$ and $T$ are isomorphic. This finishes the
  proof of the proposition.
\end{proof}

We record the following immediate consequence.
\begin{corollary}
 \label{cor:schottky-classification}
  Two  marked topological Tischler
  graphs $T=(Q,E)$ and $T'=(Q',E')$ 
  are isomorphic (resp.\ isotopic) if and only if any associated marked Schottky
  maps $(f_T, Q)$ and $(f_{T'}, Q')$ are
  equivalent (resp.\ isotopic).
\end{corollary}  

\subsection{Classification of critically fixed
    anti-rational maps.}

  Combining our results so far, we can now precisely state the
  one-to-one correspondence between critically fixed anti-rational
  maps and unobstructed reduced topological Tischler graphs.

For $d \ge 2$, let $\cF_d$ be the set of Möbius conjugacy classes of
critically fixed anti-rational maps of degree $d$, and let $\cG_d$ be
the set of isomorphism classes of unobstructed reduced topological
Tischler graphs with $d+1$ faces. We define two maps $\Phi: \cF_d \to \cG_d$ and
  $\Psi: \cG_d \to \cF_d$ as follows.
  \begin{description}
  \item[Definition of $\Phi$] Given $[g] \in \cF_d$, we set
    $\Phi([g]) = [T_g]$, where $T_g$ is the reduced Tischler graph of the critically fixed anti-rational map 
    $g$. (By the remark after Theorem~\ref{thm:schottky-obstructions},
    $T_g$ is unobstructed.) In order to verify that this map is
    well-defined, suppose $g$ and $g'$ are conjugate by a Möbius
    transformation $\theta$, i.e.,
    $g' = \theta \circ g \circ \theta^{-1}$. Then the reduced Tischler graph $T_{g'}$ of $g'$ is given by $\theta(T_g)$, so $T_{g'}$ and $T_g$ are isomorphic.
  \item[Definition of $\Psi$] Given $[T] \in \cG_d$, let $f_T$ be an
    associated Schottky map. By
    Theorem~\ref{thm:schottky-obstructions}, $f_T$ is equivalent to a
    critically fixed anti-rational map $g_T$, 
    and we use this to
    define $\Psi([T]) = [g_T]$. (Note that  $\deg(g_T)=\deg(f_T)=\#F(T)-1=d$, as follows from \Cref{prop:equivalent-tischler-graphs}\ref{item:schottky-2}.) We claim that $\Psi$ is well-defined, that is, the Möbius conjugacy class
    $[g_T]$ is independent of the choice of both the representative
    $T$ and the anti-rational map $g_T$. Indeed, let $T'$ be a plane graph
    isomorphic to $T$, let $f_{T'}$ be an associated Schottky map,
    and let $g_{T'}$ be an anti-rational map equivalent to $f_{T'}$. By
    \Cref{prop:equivalent-tischler-graphs}\ref{item:schottky-4},
    the maps $f_T$ and $f_{T'}$ are combinatorially equivalent. Then
    by the uniqueness part of Theorem~\ref{thm:schottky-obstructions},
    the anti-rational maps $g_T$ and $g_{T'}$ are Möbius conjugate as desired.
  \end{description}

  As discussed, both maps $\Phi$ and $\Psi$ are canonically defined, i.e.,
  independent of the choices made in the constructions. Furthermore, \Cref{prop:schottky-uniqueness}
  implies that $\Psi$ is injective, and Corollary~\ref{cor:crit-fixed-schottky} shows that $\Psi \circ
  \Phi = \id_{\cF_d}$. It follows that $\Phi$ and $\Psi$ are bijections and are inverses of each other, so we have proved the following result.

\begin{theorem}
  \label{thm:tischler-fixed-correspondence}
  There is a canonical one-to-one correspondence between isomorphism
  classes of unobstructed reduced topological Tischler graphs and
  Möbius conjugacy classes of critically fixed anti-rational
  maps. More precisely, for every $d \ge 2$, the map
  $\Psi: \cG_d \to \cF_d$ is a bijection with
  inverse $\Psi^{-1}=\Phi$.
\end{theorem}

Combining the theorem above with
Theorem~\ref{thm:schottky-obstructions},
Lemma~\ref{lem:dual-to-Tischler}, and the remark after
Definition~\ref{def:Tischler-graph-obstructed}, we get the following
classification result, establishing Theorem~\ref{thm:a} from the
introduction. (The proof is straightforward and is left to the
reader.)

\begin{theorem}
\label{thm:crit-fix-rational-classification}
There is an explicit and canonical one-to-one correspondence between
the following sets:
\begin{itemize}
    \item the set of Möbius conjugacy classes of critically
    fixed anti-rational maps;
    \item the set of equivalence classes of realizable
    Schottky maps;
    \item the set of isomorphism classes of unobstructed reduced
    topological Tischler graphs;
    \item the set of isomorphism classes of unobstructed alternating
    topological Tischler graphs;
    \item the set of isomorphism classes of $2$-connected simple plane graphs.
\end{itemize}
\end{theorem}

\subsection{Symmetries of critically fixed  anti-rational maps}\label{subsec:symmetries}
A very useful consequence of the above combinatorial classification is
that critically fixed anti-rational maps inherit any symmetries of
their (topological) reduced Tischler graphs. Before providing the
precise statement, we need to introduce some notation.

Let $g$ be an (anti-)rational map. We denote by $\Mob^\pm(g)$ the
group of all Möbius and anti-Möbius transformations $\theta$ that
commute with $g$, i.e., for which $g\circ \theta = \theta \circ g$. We
call $\Mob^\pm(g)$ the \emph{symmetry group} of $g$. For a group
$\Theta$ of Möbius and anti-Möbius transformations, we say that
$g$ is \emph{$\Theta$-symmetric} if its symmetry group contains
$\Theta$.

Similarly, for a marked \antiThurston map $(f,Q)$, we denote by
$\Homeo^\pm(f,Q)$ the group of all homeomorphisms $\phi: S^2\to S^2$
with $\phi(Q)=Q$ that commute with $f$ up to isotopy, that is, for
which the \antiThurston maps $(f\circ \phi, Q)$ and
$(\phi\circ f, Q)$ are isotopic. We set
$\Mod^\pm(f,Q):=\Homeo^\pm(f,Q)/\sim$, where $\sim$ is the isotopy
equivalence relation rel.\ $Q$.

Furthermore, for a plane graph $T=(Q,E)$ in $S^2$, we denote by
$\Homeo^\pm(T)$ the group of all homeomorphisms $\phi: S^2\to S^2$
such that $\phi(T)=T$ and $\phi(Q)=Q$. We set
$\Mod^\pm(T):=\Homeo^\pm(T)/\sim$, where $\sim$ is the isotopy
equivalence relation rel.\ $Q$. Finally, when $T$ is a plane graph in
$\CC$, we denote by $\Mob^\pm(T)$ the group of all Möbius and
anti-Möbius transformations $\theta$ such that $\theta(T)=T$ and
$\theta(Q)=Q$. In the special case where $\Mod^\pm(T)\ni [\iota]$ for the
complex conjugation map $\iota(z) = \bar{z}$, we say that the graph
$T$ is \emph{$\R$-symmetric}.

\begin{theorem}\label{thm:symmetries}
The following statements are true:
  \begin{enumerate}
  \item\label{item:symmetries-i} Let $T=(Q,E)$ be a marked
    topological Tischler graph and $f_T$ be an associated Schottky
    map. Then $\Mod^\pm(f_T,Q)=\Mod^\pm(T)$.
  \item\label{item:symmetries-ii} Let $g$ be a critically fixed
    anti-rational map with $\#C_g\geq 3$, and $T_g$ and $\mathring T_g$ be its alternating and reduced
    Tischler graphs, respectively. Then $\Mob^\pm(g)=\Mob^\pm(T_g)=\Mob^\pm(\mathring T_g)$. Moreover, for
    every $Q$ with $C_g\subseteq Q\subseteq V(T_g)$ and the corresponding graph $T_g^Q$ with the same realization as $T_g$ but with the vertex set $Q$, we have
    \[
      \Mod^\pm(g,Q)=\Mod^\pm(T_g^Q)=\{[\theta]\colon \theta\in
      \Mob^\pm(g),\, \theta(Q)=Q\},
    \]
    where $[\theta]$ denotes the isotopy equivalence class of $\theta$
    rel.\ $Q$.
  \end{enumerate}
\end{theorem}

\begin{remark}
  When $g$ has only two critical points, that is, when
  $g(z)=\bar{z}^d$ up to conjugation, the equality
  $\Mob^\pm(g) = \Mob^\pm(\mathring T_g)$ fails, since
  $\Mob^\pm(\mathring T_g)$ also contains the scaling maps
  $z\mapsto \lambda z$ with $\lambda\in(0,\infty)$. However, the
  remaining statements in \ref{item:symmetries-ii} still hold,
  provided $\#Q \geq 3$.
\end{remark}

\begin{proof} \ref{item:symmetries-i} Given $[\phi]\in\Mod^\pm(T)$,
  the proof of
  \Cref{prop:equivalent-tischler-graphs}\ref{item:schottky-4}
  (which also applies when $\phi$ is orientation-reversing) implies
  that the marked anti-Thurston maps $(f_T,Q)$ and
  $(\phi^{-1}\circ f_T\circ \phi,Q)$ are isotopic, and thus
  $[\phi]\in \Mod^\pm(f_T, Q)$. Conversely, given
  $[\phi]\in\Mod^\pm(f_T,Q)$, the proof of
  \Cref{prop:schottky-uniqueness} (which also applies when $\phi$
  is orientation-reversing) implies that $\phi(T)$ is isotopic to $T$
  rel.\ $Q$, and thus $[\phi]\in \Mod^\pm(T)$.

  \ref{item:symmetries-ii} By definitions of the alternating and
  reduced Tischler graphs, we have
  $\Mob^\pm(g) \subset \Mob^\pm(T_g)\subset \Mob^\pm(\mathring
  T_g)$. Conversely, if $\theta\in \Mob^\pm(\mathring T_g)$, then
  $(\theta^{-1}\circ g \circ \theta, C_g)$ is isotopic to $(g,C_g)$ by
  \Cref{prop:equivalent-tischler-graphs}\ref{item:schottky-4} and
  \Cref{cor:crit-fixed-schottky}, and thus
  $\theta^{-1}\circ g \circ\theta= g$ by the uniqueness part of
  Theorem~\ref{thm:anti-thurston}. Hence,
  $\Mob^\pm(\mathring T_g) \subset \Mob^\pm(g)$. Similarly, by the
  uniqueness part of Theorem~\ref{thm:anti-thurston}, if
  $[\phi]\in \Mod^\pm(g,Q)$, then $\phi$ is isotopic rel.\ $Q$ to a
  Möbius or anti-Möbius transformation commuting with $g$ (see
  \Cref{cor:symmetric-thurston-maps}), that is,
  $\Mod^\pm(g,Q)\subset\{[\theta]\colon \theta\in \Mob^\pm(g),\,
  \theta(Q)=Q\}$. The remaining inclusion is immediate.
\end{proof}

\section{Critically fixed anti-polynomials}
\label{sec:crit-fixed-anti-poly}

In this section, we apply our classification machinery to the special
case of critically fixed anti-polynomials, providing an alternative
approach to the results in \cite{geyerSharpBoundsValence2008} and
\cite{lazebnikUnivalentPolynomialsHubbard2021}, as well as
generalizing them by allowing critical points of arbitrary
multiplicity. Our results also extend Tischler's classification of
critically fixed polynomials to the anti-holomorphic case
\cite{tischlerCriticalPointsValues1989}.

In the anti-polynomial setting, Tischler graphs are special, because
the point at $\infty$ is a critical point of multiplicity $d-1$, where
$d$ is the degree of a given critically fixed anti-polynomial
$f: \CC\to \CC$. It then easily follows from Theorem
\ref{thm:characterization-1} that $\infty$ is a vertex of degree $d+1$
in the corresponding alternating Tischler graph $T_f$ and that
$\infty$ lies on the boundary of each of the $d+1$ faces of $T_f$. We
define the \emph{alternating Tischler tree} of $f$ to be the plane
graph $\Delta_f$ in $\CC$ whose edges are the fixed internal rays
within the immediate (superattracting) basins of finite critical
points of $f$, and whose vertices are the endpoints of these rays. In
other words, the graph $\Delta_f$ is obtained from $T_f$ by removing
the vertex at $\infty$ and all the edges incident with it. It is easy to
see that $\Delta_f$ is a plane tree in $\C$ (e.g., this follows from
Lemma~\ref{lem:poly-graphs-and-trees} below). Moreover, the tree
$\Delta_f$ has the following properties (see
Theorem~\ref{thm:characterization-1}): 
\begin{itemize}
\item the vertices of $\Delta_f$ are exactly the finite fixed points
  of $f$; 
\item the tree $\Delta_f$ is ``alternating'': every edge connects a
  finite critical fixed point with a repelling fixed point; 
\item every finite critical fixed point of multiplicity $m$ has degree
  $m+2 \ge 3$ in $\Delta_f$; 
\item every repelling fixed point is either a leaf of $\Delta_f$ or a
  vertex of degree $2$; 
\item $\Delta_f$ has exactly $d+1$ leaves, where $d$ is the degree of
  $f$.
\end{itemize}
The \emph{reduced Tischler tree} of $f$ is the plane tree
$\mathring \Delta_f$ obtained from $\Delta_f$ by forgetting all the
vertices of degree $2$. If we also remove all the leaves from
$\mathring \Delta_f$ together with the incident edges, then we get the
\emph{Hubbard tree} of the anti-polynomial $f$ (because the resulting
tree is forward invariant, connects all the finite critical points of
$f$, and every leaf is a critical point).

For
convenience, in the following discussion we implicitly always work with a
marked sphere $(S^2,\{\infty\})$, with a designated point at infinity, and
we write $\R^2$ for $S^2 \setminus \{ \infty \}$. %

\begin{definition}\label{def:poly-Tischler-graph}
  A marked topological Tischler graph $T$ is called \emph{anti-polynomial} if $\infty$ is a vertex of $T$ and $\deg_T(\infty)$ equals the number of faces of $T$.
\end{definition}
\begin{remark}
  Note that the condition that the faces of a marked topological Tischler graph are Jordan domains implies
  that $\infty$ is on the boundary of every face of an anti-polynomial marked topological Tischler graph.
\end{remark}

Recall from \Cref{def:Tischler-graph-obstructed} that a marked topological Tischler graph $T$ is called unobstructed if for any two distinct faces $A$ and $B$, any two distinct edges of $T$ in $\partial A \cap \partial B$ are adjacent. The following corollary provides a simple characterization of (un)obstructed anti-polynomial marked topological Tischler graphs.

\begin{corollary}
  \label{cor:poly-tischler-graphs}
  Let $T=(Q,E)$ be an anti-polynomial marked topological Tischler graph. Then  $T$ is unobstructed if and only if $T$ has no two vertices of degree 2 joined by an edge. In particular, if $T$ is alternating or reduced, then $T$ is always unobstructed. 
  
  Furthermore, if $T$ is unobstructed and $(f_T,Q)$ is an associated marked Schottky map, then $(f_T,Q)$ is combinatorially equivalent to a marked critically
  fixed anti-polynomial map, unique up to affine conjugacy.
\end{corollary}

\begin{proof}
  First, suppose that $T$ is an anti-polynomial marked topological
  Tischler graph with an edge $e$ connecting two vertices $p, q$ of
  degree $2$. Suppose $a$ and $b$ are the edges of $T$ incident with $p$
  and $q$, respectively, and different from $e$. Since $T$ has at
  least three faces and all of them are Jordan domains, the edges
  $a, b$ are distinct and non-adjacent. Moreover, these edges
  (together with $e$) belong to the boundaries of two distinct faces
  of $T$, and thus $T$ is obstructed.

  Conversely, assume that $T$ is obstructed, and let $A \ne B$ be two
  faces of $T$ sharing two distinct non-adjacent edges
  $a, b \subset \partial A \cap \partial B$. Suppose $\gamma$ is a
  simple closed curve in $S^2 \setminus Q$ intersecting (transversely)
  $T$ exactly twice, once in $a$ and once in $b$. We denote by $U$ and
  $W$ the two components of $S^2 \setminus \gamma$ so that
  $\infty \in U$. Let $p$ and $q$ be the endpoints of $a$ and $b$ that
  lie in $W$, respectively. Note that $p \neq q$, because $a$ and $b$
  are non-adjacent edges. Since $\gamma$ does not intersect any face
  of $T$ other than $A$ and $B$, and since every face has $\infty$ on
  its boundary, we must have $C \subset U$ for all faces
  $C \notin \{A,B\}$. Now, as the faces $A$ and $B$ are Jordan
  domains, we have that all vertices of $T$ within $W$ have degree $2$
  and they all lie on the boundary of both faces $A$ and $B$. It
  follows that all vertices (and edges) of $T$ within $W$ lie on a
  simple path connecting the vertices $p$ and $q$. This implies that
  $W$ contains at least one edge connecting two vertices of degree
  $2$.
  
  Finally, if $(f_T,Q)$ is a marked Schottky map associated with an
  anti-polynomial marked topological Tischler graph $T=(Q,E)$, then
  the local degree $\deg(f_T,\infty)$ equals the degree of the map
  $f_T$. Assuming that $T$ is unobstructed, Theorem
  \ref{thm:schottky-obstructions} implies that $(f_T,Q)$ is equivalent
  to a marked critically fixed anti-polynomial map, unique up to
  Möbius (and thus also affine) conjugacy. This finishes the proof.
\end{proof}

We now introduce topological versions of the Tischler trees of
critically fixed anti-polynomial maps.

\begin{definition}
  An \emph{alternating topological Tischler tree} is a plane tree
  $\Delta \subset \R^2$ with at least three leaves and with every edge
  connecting a vertex of degree $\leq 2$ to a vertex of degree
  $\geq 3$. A \emph{reduced topological Tischler tree} is a plane tree
  $\Delta \subset \R^2$ with at least three leaves and without
  vertices of degree $2$.
\end{definition}

\begin{lemma}\label{lem:poly-graphs-and-trees}\mbox{}
\begin{enumerate}
\item\label{item:poly-trees-graphs-marked} Let $T$ be an
  anti-polynomial alternating topological Tischler graph, and
  $\Delta_T$ be the plane graph obtained from $T$ by removing the
  vertex at $\infty$ and all the incident edges. Then $\Delta_T$ is an
  alternating topological Tischler tree. Conversely, if $\Delta$ is an
  alternating topological Tischler tree and $T_\Delta$ is a plane
  graph obtained from $\Delta$ by connecting all leaves of $\Delta$ to
  $\infty$, then $T_\Delta$ is an anti-polynomial alternating
  topological Tischler graph.
\item\label{item:poly-trees-graphs-unmarked} Similarly, let $T$ be an
  anti-polynomial reduced topological Tischler graph. Suppose $T'$ is
  the plane graph obtained from $T$ by subdividing every edge incident
  with $\infty$ into two edges sharing a new common vertex. If
  $\Delta_T$ is the plane graph obtained from $T'$ by removing the
  vertex at $\infty$ and all the incident edges, then $\Delta_T$ is a
  reduced topological Tischler tree. Conversely, if $\Delta$ is a
  reduced topological Tischler tree, let $T_\Delta$ be a plane graph
  obtained from $\Delta$ by first connecting all leaves of $\Delta$ to
  $\infty$ and then forgetting all vertices of degree $2$ in the
  resulting graph. Then $T_\Delta$ is an anti-polynomial reduced
  topological Tischler graph.
\end{enumerate}
\end{lemma}
\begin{proof} We will only prove \ref{item:poly-trees-graphs-marked};
  the proof of \ref{item:poly-trees-graphs-unmarked} is analogous and
  left to the reader. Suppose $T$ is an anti-polynomial alternating
  topological Tischler graph, and $\Delta_T$ is the plane graph as in
  the statement. It is enough to check that $\Delta_T$ is a tree, as
  the other conditions are immediate from the definitions. By
  construction of $\Delta_T$ and the remark after Definition
  \ref{def:poly-Tischler-graph}, we have:
  \[
    \#V(\Delta_T)=\#V(T)-1, \quad \#E(\Delta_T)=\#E(T)-\#F(T), \quad
    \text{and}\quad  \#F(\Delta_T)=1.
  \]
  Since $T$ is connected, the Euler formula implies that
  \[
    \#V(\Delta_T)-\#E(\Delta_T)+\#F(\Delta_T)=
    \#V(T)-\#E(T)+\#F(T)=2,
  \]
  and thus $\Delta_T$ is connected as well. The claim that $\Delta_T$
  is a tree now follows from the fact that $\Delta_T$ has exactly one
  face.

  For the converse, let $ \Delta$ be an alternating topological
  Tischler tree and $T_\Delta$ be the plane graph as in the
  statement. Since $\Delta\subset \R^2$ is a plane tree, every face of
  $T_\Delta$ must have $\infty$ on the boundary. Suppose
  $a_1,\dots,a_n$ are all the edges of $T_\Delta$ incident with $\infty$
  and listed in cyclic order. Let $p_1,\dots,p_n$ be the endpoints of
  $a_1,\dots,a_n$ that are different from $\infty$, respectively. By
  construction, $p_1,\dots, p_n$ are exactly the leaves of $\Delta$,
  and thus $n\geq 3$. Now, let
  $P=(v_0=p_1,e_1,v_1,\dots, e_k,v_k=p_2)$ be the unique simple path
  in the tree $\Delta$ that connects the leaves $p_1$ and $p_2$. Then
  the concatenation of the path $P$ with the simple path
  $(p_1,a_1,\infty,a_2,p_2)$ forms a cycle in $T_\Delta$, and thus a
  Jordan curve $\gamma$ in $S^2$. Let $U$ be the complementary
  component of $\gamma$ that does not contain any of the edges
  $a_3,\dots, a_n$. Then $U$ must be a face of $T_\Delta$; for
  otherwise, $U$ contains at least one vertex of $\Delta$, and thus
  also a leaf of $\Delta$, which is a contradiction. The previous
  argument implies that $T_\Delta$ has exactly $n$ faces (namely, one
  face ``between'' any two consecutive edges at $\infty$), each of
  which is a Jordan domain. It follows that $T_\Delta$ is an
  anti-polynomial marked topological Tischler graph. Since the tree
  $\Delta$ is alternating, the graph $T_\Delta$ must be alternating as
  well. This completes the proof of
  \ref{item:poly-trees-graphs-marked}.
\end{proof}

This lemma implies that there is a canonical one-to-one correspondence
between the isomorphism classes of anti-polynomial alternating
topological Tischler graphs and the isomorphism classes of alternating
topological Tischler trees, and similarly between the reduced
versions.  (This easily follows from
\Cref{prop:graph-isotopy-criterion} and the fact that, up to homotopy,
there is a unique arc connecting each leaf of an alternating
topological Tischler tree with $\infty$.)  Furthermore, combining the
discussion in this section with
\Cref{thm:tischler-fixed-correspondence}, we obtain the following
combinatorial classification of critically fixed anti-polynomial
maps. (The proof is straightforward and is left to the reader.)

\begin{theorem}\label{thm:class-poly}
  There is an explicit and canonical one-to-one correspondence between
  the following sets:
  \begin{itemize}
  \item the set of affine conjugacy classes of critically fixed
    anti-polynomial maps;
  \item the set of isomorphism classes of anti-polynomial alternating
    topological Tischler graphs;
  \item the set of isomorphism classes of anti-polynomial reduced
    topological Tischler graphs;
  \item the set of isomorphism classes of alternating topological
    Tischler trees;
  \item the set of isomorphism classes of reduced topological Tischler
    trees.
\end{itemize}
\end{theorem}

\section{Decompositions of critically fixed (anti-)Thurston maps}
\label{sec:decomp-crit-fix}

In this section, we discuss decompositions of critically fixed (anti-)Thurston maps; see \Cref{subsec:canon-decomp} for general terminology and notation. Our main goal is to prove the following result, which generalizes \cite[Thm.~3.28]{hlushchankaCriticallyFixedThurston2022}, proved there in the orientation-preserving and unmarked setting.

\begin{theorem}\label{thm:crit-fix-decomposition} Let $(f,Q)$ be a marked critically fixed Thurston or anti-Thurston map that fixes $Q$ pointwise, and let $\Gamma$ be a non-empty completely invariant multicurve. Suppose $\widehat{f}:  \widehat{\mathscr{S}}_\Gamma\to\widehat{\mathscr{S}}_\Gamma$ is the corresponding map on the small spheres with respect to $\Gamma$. Then the following statements are true:
\begin{enumerate}
    \item\label{item:decomp_thm_1} For each curve $\gamma \in \Gamma$ there is exactly one component $\gamma'$ of $f^{-1}(\gamma)$ that is homotopic to $\gamma$ in $S^2\setminus Q$. All other components $\delta'$ of $f^{-1}(\gamma)$ are null-homotopic in $S^2\setminus Q$ with $\deg(f:\delta'\to\gamma)=1$.
    \item\label{item:decomp_thm_2} Every small sphere $\widehat S \in \widehat{\mathscr{S}}_\Gamma$ is fixed under $\widehat{f}$. Moreover, every point in $Q(\widehat S)$ is fixed under $\widehat f$.
\end{enumerate}
\end{theorem}

Although the proof of \Cref{thm:crit-fix-decomposition} is similar to the proof of \cite[Thm.~3.28]{hlushchankaCriticallyFixedThurston2022}, we include it here for the sake of completeness. First, we introduce some auxiliary notation and constructions for a general marked Thurston or anti-Thurston map $(f,Q)$ with a completely invariant multicurve $\Gamma$ in $S^2\setminus Q$. 

Let $f^{-1}(\Gamma)$ be the set of all components of $f^{-1}(\bigcup_{\gamma\in \Gamma} \gamma) \subset S^2\setminus f^{-1}(Q)\subset S^2\setminus Q$. We denote by $\mathscr{S}_\Gamma$ the set of all connected components of $S^2\setminus \bigcup_{\gamma\in \Gamma} \gamma$; the set $\mathscr{S}_{f^{-1}(\Gamma)}$ is defined in a similar way. If $S$ is a component in $\mathscr{S}_\Gamma$, we denote by $\widehat{S}$ the corresponding small sphere in $\widehat{\mathscr{S}}_\Gamma$ and by $Q(\widehat{S})$ the corresponding marked set of $\widehat S$, and similarly for the components in  $\mathscr{S}_{f^{-1}(\Gamma)}$. More specifically, the sphere $\widehat S$ is obtained from the closure $\overline S$ by collapsing each boundary curve of $\overline S$ to a point, and the marked set $Q(\widehat{S})$ corresponds to the points in $S\cap Q$ together with the (pinched) boundary curves of $\overline S$.

Since the multicurve $\Gamma$ is completely invariant, for every  $S \in \mathscr{S}_\Gamma$ there is a unique component $i(S)\in \mathscr{S}_{f^{-1}(\Gamma)}$ such that $i(S)\setminus Q$ is homotopic to $S\setminus Q$ in $S^2 \setminus Q$. We may therefore identify the corresponding small spheres $\widehat{S}$ in $\widehat{\mathscr{S}}_\Gamma$ and $\widehat{i(S)}$ in $\widehat{\mathscr{S}}_{f^{-1}(\Gamma)}$ via a homeomorphism $i^*:  \big(\widehat{S}, Q(\widehat{S})\big) \to \big(\widehat{i(S)}, Q(\widehat{i(S)})\big)$ sending marked points to marked points, but not
necessarily bijectively. (The homeomorphism $i^*$ is canonical up to isotopy rel. $Q(\widehat S)$.) Note also that by construction, $f$ maps each component $S'\in \mathscr{S}_{f^{-1}(\Gamma)}$ onto a component $f(S')\in \mathscr{S}_\Gamma$, which induces a branched covering map $f_*: \big(\widehat{S'},Q(\widehat{S'})\big)  \to \big(\widehat{f(S')},Q(\widehat{f(S')})\big)$ between the associated small spheres (respecting the marked points). The map $\widehat f: \widehat {\mathscr{S}}_\Gamma\to \widehat {\mathscr{S}}_\Gamma$ on the small spheres with respect to $\Gamma$ may then be defined as the composition $f_*\circ i^*$.

Let $T_\Gamma$ be an (abstract simple) graph with the vertex set $\mathscr{S}_\Gamma$ and the edge set $\Gamma$, where two distinct components $S_1, S_2 \in \mathscr{S}_\Gamma$ are connected by an edge $\gamma\in \Gamma$ if and only if $\gamma$ is a boundary curve in both $S_1$ and $S_2$. We will denote by $T_{f^{-1}(\Gamma)}$ the corresponding graph for $f^{-1}(\Gamma)$. It is easy to check that $T_\Gamma$ and $T_{f^{-1}(\Gamma)}$ are in fact trees. Moreover, $f$ sends adjacent vertices in $T_{f^{-1}(\Gamma)}$ to adjacent vertices in $T_\Gamma$.

Similarly to \cite[Sec.~3.4]{hlushchankaCriticallyFixedThurston2022}, we now introduce two special subtrees of $T_{f^{-1}(\Gamma)}$, denoted $T^{ess}_{f^{-1}(\Gamma)}$ and $T^{\bullet}_{f^{-1}(\Gamma)}$. Namely, the tree $T^{ess}_{f^{-1}(\Gamma)}$ is the (unique) minimal subtree of $T_{f^{-1}(\Gamma)}$ whose vertex set contains $\{i(S)\colon S\in \mathscr{S}_{\Gamma}\}$, and the tree $T^{\bullet}_{f^{-1}(\Gamma)}$ is the (unique) minimal subtree of $T_{f^{-1}(\Gamma)}$ whose vertex set contains $\mathscr{S}^{\bullet}_{f^{-1}(\Gamma)}:=\{S'\in \mathscr{S}_{f^{-1}(\Gamma)}\colon S'\cap Q \neq \emptyset\}$.  

\begin{remark} For a marked critically fixed Thurston or anti-Thurston map $(f,Q)$ fixing $Q$ pointwise, \Cref{thm:crit-fix-decomposition} will imply that the trees $T^{ess}_{f^{-1}(\Gamma)}$ and $T^{\bullet}_{f^{-1}(\Gamma)}$ coincide and that $f$ provides an isomorphism between $T^{\bullet}_{f^{-1}(\Gamma)}$ and $T_{\Gamma}$.
\end{remark}

The next lemma describes the structure of the two chosen subtrees of $T_{f^{-1}(\Gamma)}$ (the proof is straightforward and is left to the reader); compare \cite[Lems.~3.29 and 3.30]{hlushchankaCriticallyFixedThurston2022}.

\begin{lemma}\label{lem:trees_structure} 
The following statements are true:
\begin{enumerate}

\item\label{item:trees_edges} The edge sets of $T^{ess}_{f^{-1}(\Gamma)}$ and $T^{\bullet}_{f^{-1}(\Gamma)}$ consist of all essential and all non-null-homotopic curves in $f^{-1}(\Gamma)$, respectively.  In particular,  $T^{ess}_{f^{-1}(\Gamma)}$ is a subtree of $T^{\bullet}_{f^{-1}(\Gamma)}$.

\item\label{item:ess_tree} The tree $T^{ess}_{f^{-1}(\Gamma)}$ is obtained from $T_{\Gamma}$ by edge subdivision: if two components $S_1,S_2 \in \mathscr{S}_\Gamma$ are connected in $T_\Gamma$ by an edge $\gamma\in \Gamma$ then the components $i(S_1), i(S_2)$ are connected in $T^{ess}_{f^{-1}(\Gamma)}$ by a simple path consisting of all edges $\delta' \in f^{-1}(\Gamma)$ that are homotopic to $\gamma$ in $S^2\setminus Q$.

\item\label{item:marked_tree} Let $S'\in \mathscr{S}^{\bullet}_{f^{-1}(\Gamma)}$.  Suppose $q\in Q$ is a marked point in $S'$ and $S\in \mathscr{S}_{\Gamma}$ is the component containing $q$.  Then either $i(S)=S'$ and $S'$ is a vertex of $T^{ess}_{f^{-1}(\Gamma)}$,  or $i(S)\neq S'$ and $S'\in V(T^{\bullet}_{f^{-1}(\Gamma)}) \setminus V(T^{ess}_{f^{-1}(\Gamma)})$.  In the latter case,  $S'$ is a leaf of $T^{\bullet}_{f^{-1}(\Gamma)}$ with $S'\cap Q= \{q\}$.  Furthermore,  $S'$ and $i(S)$ are connected in  $T^{\bullet}_{f^{-1}(\Gamma)}$ by a simple path consisting of all $q$-peripheral curves $\delta' \in f^{-1}(\Gamma)$.
\end{enumerate}
\end{lemma}

Here a curve $\delta'\in f^{-1}(\Gamma)$ is called \emph{$q$-peripheral} for $q\in Q$ if for a component $U$ of $S^2\setminus \delta'$ we have $U\cap Q= \{q\}$.  %

We are finally ready to provide the proof of \Cref{thm:crit-fix-decomposition}. 

\begin{proof}[Proof of \Cref{thm:crit-fix-decomposition}]
    Suppose $(f,Q)$ and $\Gamma$ are as in the statement; in particular, $C_f\subset Q$ and $f(q) = q$ for all $q\in Q$. We will say that a vertex of $T_\Gamma$ or of $T^{\bullet}_{f^{-1}(\Gamma)}$ is \emph{marked} if it contains a point from $Q$. 
 
Let 
\begin{equation*}%
P':=\big(S'_0, \delta'_1, S'_1,\dots, \delta_n', S'_n \big),
\end{equation*}
where each $S'_j$ is a vertex of $T^\bullet_{f^{-1}(\Gamma)}$ and each $\delta'_j$ is an edge of $T^\bullet_{f^{-1}(\Gamma)}$,
be a maximal (i.e., non-extendable) simple path in $T^{\bullet}_{f^{-1}(\Gamma)}$ such that 
\begin{equation*}%
f(P'):=\big(f(S'_0), f(\delta'_1), f(S'_1),\dots, f(\delta'_n), f(S'_n)\big)
\end{equation*}
is a simple path in $T_\Gamma$.
We will call such a path $P'$ a \emph{maximal injective path} in $T^{\bullet}_{f^{-1}(\Gamma)}$. (Note that $P'$ must have positive length.) 

\begin{claim}
  \label{claim:1}
  The start and end vertices of $P'$ are marked.
\end{claim}

It is sufficient to show that $S_0'$ is a marked vertex of $T^{\bullet}_{f^{-1}(\Gamma)}$. We argue by contradiction and suppose that $S'_0 \cap Q = \emptyset$. 

Set $\gamma:=f(\delta'_1)$, and let $E(S'_0)$ be the set of all edges that are incident with $S'_0$ in $T^{\bullet}_{f^{-1}(\Gamma)}$.
Note that $\#E(S'_0) \geq 2$; for otherwise $S'_0$ is a leaf of $T^{\bullet}_{f^{-1}(\Gamma)}$ and thus must be marked by the definition of $T^{\bullet}_{f^{-1}(\Gamma)}$. Furthermore, we have $f(\delta')=\gamma=f(\delta'_1)$ for each edge $\delta'\in E(S'_0)$ due to maximality of $P'$.

Let $\widetilde S'$ be the unique component in $\mathscr{S}_{f^{-1}(\{\gamma\})}$ that contains $S'_0$. Then $\widetilde S' \cap C_f = \emptyset$, because $S'_0\cap Q= \emptyset$ (by assumption) and  every $\delta'\in E(S'_0)$ is a boundary curve of $\widetilde S'\supset S'_0$ (while any other boundary curve of $\widetilde S'$ must be null-homotopic). It follows that $f|_{\widetilde S'} : \widetilde S'\to  f(\widetilde S')$ is a covering map. Hence $\widetilde S'$ is simply connected, as $f(\widetilde S')\in \mathscr{S}_{\{\gamma\}}$ is a Jordan domain. This is a contradiction, because $\#E(S'_0) \geq 2$ and thus $\partial \widetilde S'$ has at least two components.

\medskip

By Claim 1, there are marked points $q_0,q_n \in Q$ such that $q_0\in S_0'$ and $q_n\in S_n'$. Let $S_0$ and $S_n$ be the vertices of $T_\Gamma$ that contain $q_0$ and $q_n$, respectively. \Cref{lem:trees_structure}\ref{item:marked_tree} implies that either $i(S_0)=S'_0$ or $S'_0$ is a leaf of $T^{\bullet}_{f^{-1}(\Gamma)}$ that is connected to $i(S_0)$ by a simple path consisting of all $q_0$-peripheral curves in $f^{-1}(\Gamma)$. In either case, the path $P'$ must pass through $i(S_0)$; moreover, $i(S_0)$ is the first vertex of $T^{ess}_{f^{-1}(\Gamma)}$ on the path $P'$. Similarly,  $i(S_n)$ is the last vertex of $T^{ess}_{f^{-1}(\Gamma)}$ on the path $P'$. 

Since $f$ fixes the marked set $Q$ pointwise, we get that $f(S'_0) = S_0$ and $f(S'_n)=S_n$. By the choice of $P'$, we conclude that $f(P')$ is a simple path connecting $S_0$ and $S_n$ in $T_\Gamma$. At the same time, \Cref{lem:trees_structure}\ref{item:ess_tree} implies that the subpath of $P'$ between $i(S_0)$ and $i(S_n)$ is not shorter than $f(P')$. Furthermore, this subpath has to pass through all the vertices
$i(f(S'_0))=i(S_0), i(f(S'_1)),\dots, i(f(S'_{n-1})),i(f(S'_n))=i(S_n)$
and in this particular order. Combining these facts, we obtain the following two claims.

\begin{claim}
  \label{claim:2}
For each $j=0,\dots, n$, we have $i(f(S'_j)) = S'_j$. In particular, all vertices of $P'$ are in $T^{ess}_{f^{-1}(\Gamma)}$.
\end{claim}

\begin{claim}
  \label{claim:3}
  For each $j=1,\dots, n$, the curve $\delta'_j$ is essential and homotopic to $f(\delta'_j)$ in $S^2\setminus Q$. Furthermore, $\delta'_j$ is the only curve in $f^{-1}(\Gamma)$ that is homotopic to $f(\delta'_j)$ in $S^2\setminus Q$.
\end{claim}

\medskip

To prove part \ref{item:decomp_thm_1} of
\Cref{thm:crit-fix-decomposition}, we fix a curve $\gamma\in
\Gamma$. Since the multicurve $\Gamma$ is completely invariant, there
is a component $\gamma'$ of $f^{-1}(\Gamma)$ that is homotopic to
$\gamma$ in $S^2\setminus Q$. By
\Cref{lem:trees_structure}\ref{item:trees_edges}, the curve $\gamma'$
corresponds to an edge of $T^{\bullet}_{f^{-1}(\Gamma)}$, and thus it
is contained in some maximal injective path in
$T^{\bullet}_{f^{-1}(\Gamma)}$. Claim 3 implies that $\gamma'$ is the
only component of $f^{-1}(\Gamma)$ that is homotopic to
$f(\gamma')\in \Gamma$ in $S^2\setminus Q$; in particular,
$f(\gamma')=\gamma$. Now if $\delta'$ is an arbitrary component of
$f^{-1}(\gamma)$, then either $\delta'=\gamma'$ or $\delta'$ is
null-homotopic in $S^2\setminus Q$. (Indeed, if $\delta'$ is not
null-homotopic, it is an edge in some maximal injective path in
$T^{\bullet}_{f^{-1}(\Gamma)}$, and thus it must be the curve
$\gamma'$ by Claim 3 and the discussion above.) In the latter case,
let $U$ be the component of $S^2\setminus \delta'$ with
$U\cap Q=\emptyset$ and $\widetilde S'$ be the component of
$\mathscr{S}_{f^{-1}(\{\gamma\})}$ with $\widetilde S' \subset U$ and
$\delta'\subset \partial \widetilde S'$.  Since
$\widetilde S' \cap C_f = \emptyset$, we have that
$f|_{ \widetilde S'} : \widetilde S' \to f( \widetilde S')$ is a
covering map. But $f( \widetilde S')\in \mathscr{S}_{\{\gamma\}}$ is a
Jordan domain, so $f|_{\widetilde S'}$ is a homeomorphism and
$\deg(f\colon \delta'\to \gamma)=1$. (We also obtain that
$U=\widetilde S'$.) This finishes the proof of part
\ref{item:decomp_thm_1} of the theorem.

To prove part \ref{item:decomp_thm_2}, let $\widehat{S}\in \widehat{\mathscr{S}}_\Gamma$ be an arbitrary small sphere and $S$ be the corresponding component of $\mathscr{S}_\Gamma$. Since $i(S)$ is a vertex of $T^{ess}_{f^{-1}(\Gamma)}$ (see \Cref{lem:trees_structure}), it is contained in some maximal injective path in $T^{\bullet}_{f^{-1}(\Gamma)}$. Claim 2 now implies that $i(f(i(S))) = i(S)$. Therefore, $f(i(S))=S$ and $\widehat f (\widehat S) = \widehat S$. Finally, suppose $\widehat{q}\in Q(\widehat S)$ is a marked point. If $\widehat{q}$ corresponds to a point in $S\cap Q$, it is fixed by $\widehat f$, because $f$ fixes $Q$ pointwise. Otherwise, $\widehat{q}$ corresponds to a boundary curve of $S$, and $\widehat f$ fixes $\widehat{q}$ by part \ref{item:decomp_thm_1}. This completes the proof of the theorem. 
\end{proof}

As in the orientation-preserving case \cite[Cor.~3.32]{hlushchankaCriticallyFixedThurston2022}, we may now deduce that fixed Levy curves are essentially the only relevant obstructions for critically fixed anti-Thurston maps.

\begin{corollary}\label{cor:crit_fix_obstructions}
Let $(f,Q)$ be a marked critically fixed Thurston or anti-Thurston map that fixes $Q$ pointwise. Then every obstruction of $(f,Q)$ contains a fixed Levy curve. In particular, $(f,Q)$ is realizable if and only if $(f,Q)$ does not have a fixed Levy curve.
\end{corollary}

\begin{proof}
    We first note that $(f,Q)$ is never a $(2,2,2,2)$-map, and thus we may use Thurston's characterization (\Cref{thm:Thurston_theorem} or \ref{thm:anti-thurston}). The statement now easily follows from \Cref{thm:crit-fix-decomposition}, since every obstruction contains an irreducible one, and every irreducible obstruction is contained in a completely invariant one (see  \Cref{lem:obstruction-types}). In fact, we get that every irreducible obstruction for $(f,Q)$ is a single fixed Levy curve.
\end{proof}

We also record another easy corollary to \Cref{thm:crit-fix-decomposition}, describing the properties of canonical obstructions and decompositions for critically fixed \antiThurston maps; compare \cite[Cor.~3.33]{hlushchankaCriticallyFixedThurston2022}.

\begin{corollary}\label{cor:can_crit_fix_decomposition}
    Let $(f,Q)$ be a marked critically fixed \antiThurston map that fixes $Q$ pointwise and $\Gamma:=\Gamma_{f,Q}$ be the canonical multicurve obstruction for $(f,Q)$. Then the following statements are true:
\begin{enumerate}
    \item\label{item:can_levy_1} For every curve $\gamma\in \Gamma$ there is exactly one component $\gamma'$ of $f^{-1}(\gamma)$ that is homotopic to $\gamma$ in $S^2\setminus Q$ and satisfies $\deg(f:\gamma'\to \gamma)=1$. All other components $\delta'$ of $f^{-1}(\gamma)$ are null-homotopic and also satisfy $\deg(f:\delta'\to \gamma)=1$. In particular, each curve $\gamma\in \Gamma$ is a fixed Levy curve.
    \item\label{item:can_levy_2} Every small sphere $\widehat{S}\in\widehat{\mathscr{S}}_\Gamma$ is fixed under $\widehat f$. Moreover, every point in $Q(\widehat S)$ is fixed under $\widehat f$.
    \item\label{item:can_levy_3} If $S\in \mathscr{S}_\Gamma$ satisfies $S\cap C_f=\emptyset$, then 
    $\widehat f|_{\widehat{S}}$ is a homeomorphism.
     \item\label{item:can_levy_4} If $S\in \mathscr{S}_\Gamma$ satisfies $S\cap C_f\neq\emptyset$, then 
     $\big(\widehat f|_{\widehat{S}}, Q(\widehat{S})\big)$ is equivalent to a marked critically fixed (anti-)rational map of degree $d(\widehat{S})=1+\frac{1}{2}\sum_{c\in S\cap C_f} (\deg(f,c)-1)$.
     \item\label{item:can_levy_5} If distinct components $S_1, S_2\in \mathscr{S}_\Gamma$ satisfy $\partial S_1\cap \partial S_2 \neq \emptyset$, then either $S_1\cap C_f \neq \emptyset$ or $S_2\cap C_f \neq \emptyset$.
\end{enumerate}    
\end{corollary}

\begin{proof}
  We suppose that $(f,Q)$ and $\Gamma$ are as in the statement, and assume that $\Gamma\neq \emptyset$; otherwise, the statement is trivial. Parts
  \ref{item:can_levy_1} and \ref{item:can_levy_2} follow immediately
  from \Cref{thm:crit-fix-decomposition}, because the canonical
  obstruction $\Gamma$ is simple and completely invariant (see
  \Cref{thm:canonical-obstruction-is-empty}\ref{item:can_2}).
  \Cref{thm:top-char-can-obstruction} and part \ref{item:can_levy_1} imply parts \ref{item:can_levy_3} and \ref{item:can_levy_4}, where the formula for $d(\widehat S)$ follows from
  the Riemann--Hurwitz formula.

  To show part \ref{item:can_levy_5}, suppose
  $S_1, S_2\in \mathscr{S}_\Gamma$ are distinct components with
  $\partial S_1\cap \partial S_2\neq \emptyset$, that is,
  $\partial S_1\cap \partial S_2 = \gamma$ for some curve
  $\gamma\in \Gamma$. Set $\Gamma':=\Gamma\setminus\{\gamma\}$ and
  $S'_\gamma:=S_1\cup S_2 \cup \gamma$. Then
  \[
    \mathscr{S}_{\Gamma'}= \big(\mathscr{S}_\Gamma\setminus \{S_1,
    S_2\}\big) \sqcup \{ S'_\gamma\}.
  \]
  By part \ref{item:can_levy_1}, the multicurve $\Gamma'$ is empty or is a
  completely invariant obstruction for $(f,Q)$. Suppose
  $\widehat f' : \widehat{\mathscr{S}}_{\Gamma'} \to
  \widehat{\mathscr{S}}_{\Gamma'}$ is the induced map on the small
  spheres with respect to $\Gamma'$. Then every small sphere
  $\widehat{S'}\in \widehat{\mathscr{S}}_{\Gamma'}$ is fixed under
  $\widehat f'$ by \Cref{thm:crit-fix-decomposition}\ref{item:decomp_thm_2}. Moreover, if
  $\widehat{S'}\neq \widehat {S'_\gamma}$, parts \ref{item:can_levy_1}--\ref{item:can_levy_4} imply that the respective small map
  $\widehat f' : \big(\widehat{S'}, Q(\widehat{S'})\big) \to
  \big(\widehat{S'}, Q(\widehat{S'})\big)$ is either a homeomorphism
  or a marked \antiThurston map that is equivalent to a marked critically
  fixed (anti-)rational map. Now if
  $S_1\cap C_f = S_2\cap C_f = \emptyset$, then the map
  $\widehat f'\colon \big(\widehat{S'_\gamma},
  Q(\widehat{S'_\gamma})\big) \to \big(\widehat{S'_\gamma},
  Q(\widehat{S'_\gamma})\big)$ is a homeomorphism.  It follows that
  the obstruction $\Gamma' \subsetneq \Gamma$ satisfies the conditions
  from \Cref{thm:top-char-can-obstruction}, which is a
  contradiction.
\end{proof}

\subsection{A critically fixed non-Schottky example}
\label{subsec:non-schottky-ex}

\begin{figure}[t]
  \centering
  \begin{tikzpicture} [line width=.8pt]
    \def\r{2cm}
    
    \def\s{.1*\r}
    
    \tikzstyle{vertex}=[circle,fill=black,minimum size=\s,inner sep=0pt]
    
    \foreach \name/\angle in {A/0, B/45, C/90, D/135, E/180,
                              F/225, G/270, H/315} {
      \node[vertex] (\name) at (\angle:\r) {};
    }
    \draw circle (\r);
    \foreach \i/\j in {A/B, C/D, E/F, G/H} {
      \draw (\i) to [bend left=20] (\j);
    }

    \foreach \angle in {22.5, 112.5, 202.5, 292.5} {
      \draw[red, rotate around={\angle:(\angle:.9*\r)}] (\angle:.9*\r) ellipse ({.3*\r} and {.5*\r});
    }
  \end{tikzpicture}
  \caption{A reduced topological Tischler graph $T=(Q,E)$ (in black)
    and the canonical multicurve obstruction $\Gamma$ for an associated marked
    Schottky map $(f_T,Q)$ (in red), consisting of four fixed Levy
    curves. Let $S\in \mathscr{S}_\Gamma$ be the unique component that
    does not contain any critical points of $f_T$. Replacing the
    original map $f_T$ on $S$ with a pseudo-Anosov map yields a
    critically fixed anti-Thurston map that is not equivalent to a
    Schottky map.}
  \label{fig:non-schottky-ex}
\end{figure}

In light of the discussion above, it is not hard to see how to
construct critically fixed anti-Thurston maps that are not
combinatorially equivalent to Schottky maps. Any anti-Thurston map $f$
equivalent to a Schottky map must admit a connected plane graph $T$
with $V(T)\supseteq C_f$, all of whose edges are invariant with respect
to $(f, C_f)$. In order to obtain a map without this property, we can
make sure that one of the degree-$1$ small maps in the canonical
decomposition is a pseudo-Anosov map. Since pseudo-Anosov maps do not fix any homotopy class of essential arcs (between marked points), the map $(f,C_f)$ does not admit such an invariant graph. The decomposition theory is not needed in this
subsection, but it provides a guideline and illustration of the
relevant properties of the example.

For an illustration of the construction, see \Cref{fig:non-schottky-ex}. We
start with a reduced topological Tischler graph $T=(Q,E)$ obtained
from an octagon, i.e., a cycle of length $8$ in $S^2$, by doubling
every other edge. Let $f=f_T$ be an associated Schottky map. Note that
$Q=C_{f}$, that is, $(f,Q)$ is an unmarked anti-Thurston map. Clearly,
$(f,Q)$ is obstructed, and it is easy to see that the canonical
obstruction $\Gamma=\Gamma_{f,Q}$ for $(f,Q)$ consists of four
disjoint fixed Levy curves, each surrounding one doubled pair of
edges. Let $S\in \mathscr{S}_\Gamma$ be the unique component that
contains no critical points of $f$. Up to modifying $f$ by isotopy
within the two non-digonal faces of $T$, we may assume that the
Schottky map $f$ fixes each curve in $\Gamma$ and $f|_{\overline S}$
is an involution. Replacing this involution with an
orientation-reversing pseudo-Anosov homeomorphism that agrees with $f$
on $\partial S$, we obtain the desired non-Schottky critically fixed
anti-Thurston map $g$.  Indeed, since $g|_{\overline{S}}$ is pseudo-Anosov, there can
be no invariant arc with respect to $(g,Q)$ that connects two critical
points lying in different components of $\mathscr{S}_\Gamma$. Hence,
there is no connected graph with invariant edges that joins all the
critical points of $g$.

\section{Multi-Schottky maps}
\label{sec:multi-schottky}

In \Cref{sec:plane-graphs-schottk}, we discussed how to associate a
Schottky map, i.e., a certain critically fixed anti-Thurston map, to
each marked topological Tischler graph. In
\Cref{subsec:non-schottky-ex}, we saw that not every critically fixed
anti-Thurston map admits a Schottky model. Thus, in order to classify
all critically fixed anti-Thurston maps (up to combinatorial
equivalence), we need to extend the notion of Schottky maps.

\subsection{Multi-Schottky maps} \label{subsec:multi-schottky}

Before introducing a generalization of Schottky maps, which we call \emph{multi-Schottky maps}, we must first provide some auxiliary terminology and constructions.

Let $T=(Q,E)$ be a marked topological Tischler graph. We say that a vertex $v$ of $T$ is \emph{critical} if $\deg_T(v)\geq 3$; in other words, $v$ is critical if it is a critical point of a marked Schottky map associated with $T$ (see \Cref{prop:equivalent-tischler-graphs}\ref{item:schottky-2}). An edge $e$ of $T$ is called \emph{critical} if it connects two critical vertices of $T$. For convenience, we will use the notation $\partial e$ to denote the set of endpoints of $e$. 

We now define a \emph{blow-up} of $T$ with respect to a (possibly empty) subset $\Lambda\subset E$ of critical edges of $T$. Roughly speaking, it is a plane graph obtained by doubling every edge in $\Lambda$. More formally, for every edge $e\in \Lambda$ we choose a Jordan domain $D_e$ so that the following conditions hold:
\begin{itemize}
    \item $e\subset D_e$ and $\partial e \subset \partial D_e$ for all $e\in \Lambda$. In other words, $e$ is a crosscut of $D_e$.
    \item $\overline{D_e}\cap T = \overline{e}$ for each $e\in \Lambda$. In particular, $D_e$ does not intersect any edge of $T$ other than $e$. 
    \item For all distinct edges $e_1,e_2\in \Lambda$, we have $\overline{D_{e_1}}\cap \overline{D_{e_2}} = \partial e_1\cap \partial e_2$. In particular, the Jordan domains $D_e$, $e\in \Lambda$, are pairwise disjoint.
\end{itemize}

By the above, the two endpoints of each edge $e\in \Lambda$ subdivide
$\partial D_e$ into two open arcs, which we call the \emph{blow-ups
of $e$} and denote by $e^+$ and $e^-$. Let $\TT_{\Lambda}$ be the plane
graph obtained from $T=(Q,E)$ by replacing every edge $e\in \Lambda$
with the two blow-ups of $e$, that is,
\[\text{$V(\TT_\Lambda)=Q$ \quad and \quad $E(\TT_\Lambda)= (E\setminus
    \Lambda) \cup \bigcup_{e\in \Lambda} \{e^+,e^-\}$.}\]
Note that every Jordan domain $D_e$, $e\in \Lambda$, is a digonal face of
$\TT_\Lambda$, and we will refer to such a face as a \emph{Levy face} of
$\TT_\Lambda$. The remaining faces of $\TT_\Lambda$ will be called \emph{regular}. 

The graph $\TT_\Lambda$ comes with a natural involution $\eta:\TT_\Lambda\to \TT_\Lambda$ that fixes every edge in $E\setminus \Lambda$ and swaps the two blow-ups of each $e\in \Lambda$. More concretely, $\eta$ satisfies the following properties:
\begin{itemize}
    \item $\eta|_Q= \id_Q$;
    \item $\eta|_e = \id_e$ for every $e\in E\setminus \Lambda$;
    \item $\eta(e^+)=e^-$, $\eta(e^-)=e^+$, and $\eta|_{e^+}=(\eta|_{e^-})^{-1}$ for each $e\in \Lambda$.
\end{itemize}

\begin{definition}\label{def:blow-up}
  The plane graph $\TT_\Lambda$ constructed as above is called a
  \emph{blow-up} of the marked topological Tischler graph $T=(Q,E)$
  with respect to the subset $\Lambda \subset E$ of critical edges of
  $T$. %
  The involution $\eta: \TT_{\Lambda}\to \TT_{\Lambda}$ defined above
  will be referred to as an \emph{edge involution} of $\TT_\Lambda$.
\end{definition}

\begin{remark}
    Clearly, for a given marked topological Tischler graph $T=(Q,E)$ and a subset $\Lambda \subset E$ of its critical edges, the respective blow-up $\TT_\Lambda$ and  edge
  involution $\eta: \TT_\Lambda\to \TT_\Lambda$ are not uniquely
  defined. However, they are uniquely specified up to isotopy rel.\
  $Q$.
\end{remark}

A plane graph $\TT$ may be viewed as a blow-up of some marked topological Tischler graph $T$ with the set of Levy faces $F_\Lambda\subset F(\TT)$ if and only if $\TT$ and $F_\Lambda$ satisfy the following conditions:
  \begin{itemize}
      \item every face of $\TT$ is a Jordan domain (in particular, $\TT$ is connected);
      \item $\# F(\TT) - \# F_\Lambda \geq 3$;
      \item every face in $F_\Lambda$ is digonal and no two distinct faces in $F_\Lambda$ share a common boundary edge (though they may have a common boundary vertex);
      \item if a vertex $v\in V(\TT)$ is on the boundary of $k\geq 1$ faces in $F_\Lambda$, then $\deg_\TT(v)-k\geq 3$. 
  \end{itemize}
  Namely, the marked topological Tischler graph $T$ is obtained from $\TT$ by replacing the two boundary edges of each digonal face $A\in F_\Lambda$ with a crosscut of $A$; these crosscuts then form the respective subset $\Lambda$ of critical edges of $T$. (Different choices of $F_\Lambda\subset F(\TT)$ lead to different, typically non-isotopic,  marked topological Tischler graphs $T$.) 

\begin{definition}\label{def:multi-tischler}
A \emph{multi-Tischler graph} is a plane graph $M$ with the following properties: 
\begin{enumerate}
\item every connected component of $M$ is either an isolated vertex, 
in which case we call it a \emph{vertex component}, or a blow-up of an unobstructed 
marked topological Tischler graph, in which case we call it a \emph{Tischler component};
\item $M$ has at least one Tischler component;
\item  for every Tischler component $\TT$, every regular face $A$ of $\TT$ is a face of $M$,
that is, $A$ does not contain any vertices of $M$;
\item for every Tischler component $\TT$, every Levy face $A$ of $\TT$ contains at least two vertices of $M$. (In particular, $A$ is not a face of $M$.)
\end{enumerate}
We say that a multi-Tischler graph $M$ is \emph{reduced} if it has no vertex components and each of its Tischler components is a blow-up of an (unobstructed) reduced topological Tischler graph.
\end{definition}

A priori, the description of a Tischler component $\TT$ of a multi-Tischler graph $M$ needs both the plane graph $\TT$ and the set $F_\Lambda \subset F(\TT)$ of its Levy faces. However, the last two conditions in the definition above enable us to uniquely identify $F_\Lambda$ from the plane graph $M$.
Furthermore, it is straightforward to check that every simply connected face of $M$ is a regular face of a Tischler component of $M$; in particular, it is a Jordan domain.

We will use the notation $M^\mu$ to denote the union of the closures of all multiply connected faces of the multi-Tischler graph $M$. By definition, for every such face, each of its boundary circuits is either a cycle of length $2$ (given by the boundary circuit of a Levy face of a Tischler component of $M$) or a circuit of length $0$ (given by a vertex component of $M$). In the former case, the pair of edges composing the cycle of length $2$ is called a \emph{Levy edge pair} in $M$. 

\begin{definition}\label{def:multi-Tischler-pair}
A \emph{multi-Tischler pair} is a pair $(M,\mu)$, where $M$ is a multi-Tischler graph and $\mu$ is an orientation-reversing self-homeomorphism of $M^\mu$ such that 
\begin{itemize}
    \item $\mu$ fixes every point in $V(M)\cap M^\mu$; and
    \item $\mu$ acts as an involution on every Levy edge pair $(e^+,e^-)$ in $M$, that is, $\mu|_{e^+}=(\mu|_{e^-})^{-1}$.
\end{itemize}
A multi-Tischler pair $(M,\mu)$ is called \emph{reduced} if $M$ is a reduced multi-Tischler graph.
\end{definition}

\begin{remark}
    It is immediate from this definition that, given a multi-Tischler pair $(M,\mu)$, the homeomorphism $\mu$ fixes (the closure of) every multiply connected face of $M$, as well as each of its boundary components. Furthermore, if we extend $\mu$ to $M\setminus M^\mu$ by the identity, then $\mu|_\TT$ is an edge involution for each Tischler component $\TT$ of $M$. 
\end{remark}

For our classification, we need the following two natural equivalence relations on multi-Tischler pairs. (Compare \cite[Def.~1.3]{hlushchankaCriticallyFixedThurston2022}.)

\begin{definition}\label{def:multi-schottky-equiv}
  Let $(M,\mu)$ and $(M',\mu')$ be two multi-Tischler pairs.  We say
  that $(M,\mu)$ and $(M',\mu')$ are \emph{equivalent} if there exists
  an orientation-preserving homeomorphism $\phi\colon S^2 \to S^2$
  such that $\phi(M)= M'$, $\phi(V(M)) = V(M')$, and such that
  $(\phi^{-1}\circ \mu' \circ \phi)|_{M^\mu}$ is isotopic to $\mu$
  rel.\ $M^\mu\cap V(M)$. Furthermore, if $\phi\in \Homeo^+_0(S^2,
  V(M))$, we say that $(M,\mu)$ and $(M',\mu')$ are \emph{isotopic}.
\end{definition}

With these preliminaries, we are ready to define multi-Schottky maps.

\begin{definition}\label{def:multi-schottky-map}
  Let $(M, \mu)$ be a multi-Tischler pair. A \emph{multi-Schottky map}
  associated with $(M,\mu)$ is a self-map $f:=f_{(M,\mu)}: S^2\to S^2$
  of the sphere that satisfies the following conditions:
\begin{itemize}
    \item For every multiply connected face $A$ of $M$, we have $f|_{\overline A} = \mu|_{\overline A}$.
    \item For every simply connected face $A$ of $M$, let $\TT$ be the unique Tischler component of $M$ with $\partial A \subset \TT$ and let $\eta$ be the corresponding edge involution of $\TT$ induced by $\mu$. 
    We require that $f|_{\overline A}$ is an orientation-reversing homeomorphism extending $\eta|_{\partial A}$. In particular, we have
    \[f(A)=S^2 \setminus \left(\overline{A} 
    \cup \bigcup_{D}
    \overline{D}\right),\]
    where the union is taken over all Levy faces $D$ of $\TT$ that share a boundary edge with~$A$.
\end{itemize}
A marked multi-Schottky map $(f,Q)$ associated with $(M,\mu)$ is a multi-Schottky map $f=f_{(M,\mu)}$ associated with $(M,\mu)$ together with the set $Q= V(M)$.
\end{definition}

It is straightforward from the definition and the Alexander trick that a multi-Schottky map $f_{(M,\mu)}$ is uniquely defined up to isotopy relative to $M$. We summarize this and other basic properties of multi-Schottky maps
in the proposition below. The proof is straightforward and analogous to the proof of \Cref{prop:equivalent-tischler-graphs}, so we leave it to the reader. (In particular, the expressions for the degree of $f_{(M,\mu)}$ follow easily from the Riemann--Hurwitz and Euler formulas; see the proof of \Cref{lem:mapping-prop-model} for a similar argument.)

\begin{proposition}
  \label{prop:equivalent-multi-schottky}
  Let $(M,\mu)$ be a multi-Tischler pair and $(f,Q)$ be an associated marked multi-Schottky map. Then the following are true:
  \begin{enumerate}
  \item\label{item:multi-i} $f$ is an anti-Thurston map fixing all points in $Q=V(M)$ with $C_f=P_f=\{v\in V(M)\colon \deg_M(v) \geq 3\}$. In particular, $(f,Q)$ is a marked critically fixed
    anti-Thurston map.
  \item\label{item:multi-ii} The degree of $f$ equals $\#F(M)-K_\TT-L=\#F_s(M)-2\cdot K_\TT+1$, where $K_\TT$ is the number of Tischler components of $M$, $L$ is the total number of Levy edge pairs in $M$, and $F_s(M)$ denotes the set of simply connected faces of $M$. Furthermore, for each $v\in C_f$ we have $\deg(f,v)=\deg_M(v)-L_v-1$, where
  $L_v$ is the number of Levy edge pairs incident with $v$.
  \item\label{item:multi-iii} $f$ is not a $(2,2,2,2)$-map. Moreover, $f$ has hyperbolic orbifold if and only if $\# C_f\geq 3$.
  \item\label{item:multi-iv} If $(f', Q')$ is another marked multi-Schottky map associated with a multi-Tischler pair $(M',\mu')$,
    where the pairs $(M,\mu)$ and $(M',\mu')$ are equivalent (resp.\ isotopic), then the marked anti-Thurston maps $(f,Q)$ and $(f',Q')$ are combinatorially equivalent (resp.\ isotopic).
  \end{enumerate} 
\end{proposition}

\subsection{Canonical obstructions for multi-Schottky
  maps}\label{subsec:can-obstr-multi-schottky}
In the following, let $(M,\mu)$ be a multi-Tischler pair and $(f,Q)$
be an associated marked multi-Schottky map. Our goal in this
subsection is to describe the canonical obstruction of $(f,Q)$.

Let $\TT$ be a Tischler component of $M$ and $A$ be a Levy face of
$\TT$. Pick a small closed regular neighborhood $R_{A}$ of the Jordan
curve $\partial A$, that is, choose a (small) closed annulus $R_{A}$
with core curve $\partial A$. We denote by $\gamma_{A}$ the component
of $\partial R_{A}$ that lies within the face $A$. Note that
$\gamma_{A}$ is a simple closed curve in $S^2\setminus Q$; moreover,
up to free homotopy in $S^2\setminus Q$, it is independent of the
choice of the annulus $R_{A}$.

Set
\[
  \Gamma'_M:=\{\gamma_{A}: \text{$A$ is a Levy face of a Tischler
    component $\TT$ of $M$}\}.
\]
Note that, by construction, the curves in $\Gamma'_M$ are pairwise
disjoint but not necessarily pairwise non-homotopic in $S^2\setminus
Q$. In fact, the curves in $\Gamma'_M$ are naturally in one-to-one
correspondence with the Jordan curves in the boundary of multiply
connected faces of $M$. So, two elements of $\Gamma'_M$ are homotopic
in $S^2\setminus Q$ if and only if they lie within the same annular
face of $M$. We denote by $\Gamma_M \subseteq \Gamma'_M$ a multicurve
in $S^2\setminus Q$ that contains exactly one curve from each
homotopy class represented in $\Gamma'_M$.

\begin{theorem}\label{thm:can-obstr-multi-schottky}
  Let $(M,\mu)$ be a multi-Tischler pair and $(f,Q)$ be an associated
  marked multi-Schottky map. Then the multicurve $\Gamma_M$
  constructed as above is a canonical multicurve obstruction for $(f,Q)$.
\end{theorem}

\begin{proof}
  Let $(M,\mu)$, $(f,Q)$, and $\Gamma:=\Gamma_M$ be as in the
  statement, and suppose that $\TT$ is a Tischler component of $M$.
  
  If $\TT$ has no Levy faces, then $M$ coincides with $\TT$ by
  \Cref{def:multi-tischler}. Hence $M=\TT$ is an unobstructed marked
  topological Tischler graph and $\Gamma=\emptyset$. It now follows
  from \Cref{def:multi-schottky-map} that $(f,Q)$ is a marked Schottky
  map associated with $\TT$. Thus $(f,Q)$ is realizable by
  \Cref{thm:schottky-obstructions} and the canonical obstruction for
  $(f,Q)$ is empty by \Cref{thm:anti-thurston}. The statement follows
  in this case.

  Suppose now that $\TT$ (and thus each Tischler component of $M$) has
  at least one Levy face. Let $\gamma=\gamma_{A}\subset S^2\setminus
  Q$ be a simple closed curve associated with a Levy face $A$ of $\TT$
  as above.

  \begin{claim*}
    There is exactly one component $\gamma'$ of $f^{-1}(\gamma)$ that
    is homotopic to $\gamma$ in $S^2\setminus Q$ and satisfies
    $\deg(f: \gamma'\to \gamma)=1$. All other components~$\delta'$ of
    $f^{-1}(\gamma)$ are null-homotopic and also satisfy
    $\deg(f:\delta'\to \gamma)=1$. In particular, $\gamma$ is a fixed
    Levy curve for $(f,Q)$.
  \end{claim*}

  By construction, $\gamma$ is contained in a multiply connected face
  $B$ of $M$ with $\partial A \subset \partial B$. By the definition
  of multi-Schottky maps, $f|_{\overline B}=\mu|_{\overline B}$ is an
  orientation-reversing self-homeomorphism of $\overline B$, and
  therefore there is a component $\gamma'$ of $f^{-1}(\gamma)$ 
  satisfying $\gamma'\subset B$ and $\deg(f: \gamma'\to
  \gamma)=1$. Since $f|_{\overline B}(\partial A)=\mu|_{\overline
    B}(\partial A)=\partial A$ and $\gamma$ is homotopic to $\partial
  A$ in $\overline B$, we get that $\gamma'$ must also be homotopic to
  $\partial A$ in $\overline B$, and thus $\gamma'$ is homotopic to
  $\gamma$ in $S^2\setminus Q$. By \Cref{def:multi-schottky-map}, any
  other component $\delta'$ of $f^{-1}(\gamma)$ must be within a
  simply connected face $A'$ of $M$. Hence $\delta'$ is null-homotopic
  in $S^2\setminus Q$, and since $f|_{\overline{A'}}$ is an
  orientation-reversing homeomorphism, we have $\deg(f:\delta'\to
  \gamma)=1$. The claim follows.

\smallskip
 
  It is immediate from the claim that $\Gamma=\Gamma_M$ is a completely invariant multicurve. We now look at the decomposition of $(f,Q)$ with respect to $\Gamma$. In the discussion below, we will follow the notation introduced in \Cref{sec:decomp-crit-fix}. 

  Suppose $\widehat S\in \widehat{\mathscr{S}}_{\Gamma}$ is a small sphere with respect to $\Gamma$, and $S$ is the respective component in $\mathscr{S}_\Gamma$. Recall that $\widehat S$ is marked by a finite set $Q(\widehat S)$ that corresponds to the points in $S\cap Q$ together with the pinched boundary components of $\overline{S}$. By  \Cref{thm:crit-fix-decomposition}\ref{item:decomp_thm_2}, $\widehat f: \widehat{\mathscr{S}}_{\Gamma}\to \widehat{\mathscr{S}}_{\Gamma}$ fixes $\widehat S$ and also every point $\widehat{q}\in Q(\widehat S)$. %

  By the construction of $\Gamma$, if $S\cap C_f = \emptyset$, then
  $S$ is contained in a multiply connected face $B$ of $M$ and thus
  $\widehat f : \big(\widehat{S}, Q(\widehat{S})\big)\to
  \big(\widehat{S}, Q(\widehat{S})\big)$ is an (orientation-reversing)
  homeomorphism induced by $\mu|_{\overline{B}}$. Otherwise,
  $M \cap S$ consists of a single Tischler component $\TT$ of $M$. Let
  $\widehat{\TT}$ be the plane graph in $\widehat S$ induced by
  $\TT\subset S$ when pinching every boundary component of $\overline
  S$ to a point. Note that $\widehat{\TT}$ is a blow-up of an
  (unobstructed) marked topological Tischler graph such that each Levy
  face $\widehat A$ of $\widehat{\TT}$ contains a unique point
  $\widehat{q}_{A}$ from $Q(\widehat S)$. Moreover, $\deg(\widehat
  f|_{\widehat S}, \widehat{q}_A) =1$ by the claim above. Since
  $f|_\TT$ is an edge involution of $\TT$, we may assume (up to
  modifying $\big(\widehat f|_{\widehat S}, Q(\widehat S)\big)$ by isotopy)
  that $\widehat f|_{\widehat{\TT}}$ is an edge involution of
  $\widehat{\TT}$. Then $\widehat f$ is an (orientation-reversing)
  self-homeomorphism of the closure of each Levy face of $\widehat
  \TT$. For each such face $\widehat A$, let us pick two disjoint open
  arcs $\widehat{e\,}'_{A,1}, \widehat{e\,}'_{A,2} \subset \widehat A$
  connecting the point $\widehat{q}_{A}\in Q(\widehat S) \cap \widehat
  A$ with the two vertices of $\widehat{\TT}$ in $\partial \widehat
  A$. Note that their images
  $\widehat{e}_{A,1}:=\widehat{f}(\widehat{e\,}'_{A,1})$ and
  $\widehat{e}_{A,2}:=\widehat{f}(\widehat{e\,}'_{A,2})$ are also
  disjoint open arcs in $\widehat A$ (connecting the same
  endpoints). Moreover, since $\widehat A$ is a Jordan domain, the
  arcs $\widehat{e}_{A,1}$ and $\widehat{e}_{A,2}$ are homotopic to
  $\widehat{e\,}'_{A,1}$ and $\widehat{e\,}'_{A,2}$, respectively,
  where homotopy is assumed to be within the family of open arcs in
  $\big(\widehat S,Q(\widehat S)\big)$; see \Cref{lem:arc_isotopy}. Now let $\widehat
  T$ (resp.\ $\widehat T'$) be a plane graph in $\widehat S$ obtained
  from $\widehat{\TT}$ by adding $\widehat q_A$ to the vertex set and
  replacing the two edges of $\widehat{\TT}$ in $\partial \widehat A$
  with $\widehat{e}_{A,1}$ and $\widehat{e}_{A,2}$ (resp.\ with
  $\widehat{e\,}'_{A,1}$ and $\widehat{e\,}'_{A,2}$) for all Levy
  faces $\widehat A$ of $\widehat{\TT}$. It is straightforward to
  check that $\widehat T$ is an unobstructed marked topological
  Tischler graph and, moreover, $\widehat f|_{\widehat S}$, $\widehat
  T$, and $\widehat T'$ satisfy the conditions in
  \Cref{prop:fixed-graph-implies-schottky}. We conclude that
  $\big(\widehat f|_{\widehat S}, Q(\widehat S)\big)$ is isotopic to a marked
  Schottky map associated with $\widehat T$, and thus $\big(\widehat f|_{\widehat
    S}, Q(\widehat S)\big)$ is equivalent to a marked critically fixed
  anti-rational map by \Cref{thm:schottky-obstructions}.

  It follows that the multicurve $\Gamma$ satisfies all the conditions in \Cref{thm:top-char-can-obstruction}, so, to complete the proof, it remains to check minimality. Suppose $\Gamma'\subsetneq \Gamma$ and $\gamma \in \Gamma\setminus \Gamma'$. Let $S'_\gamma\in \mathscr S_{\Gamma'}$ be the component containing $\gamma$, $\widehat{S}'_\gamma \in  \widehat{\mathscr S}_{\Gamma'}$ be the respective small sphere, and $\widehat \gamma \subset \widehat{S}'_\gamma \setminus Q(\widehat{S}'_\gamma)$ be a simple closed curve induced by $\gamma$. Note that $\widehat \gamma$ is essential in $\widehat{S}'_\gamma \setminus Q(\widehat{S}'_\gamma)$ and, moreover, $S'_\gamma \cap C_f \neq \emptyset$. Hence, the induced map on $\widehat{S}'_\gamma$ has degree $\geq 2$ and, using the claim, $\widehat \gamma$ is its fixed Levy curve. Consequently, the multicurve $\Gamma'$ does not meet the requirements from \Cref{thm:top-char-can-obstruction}, and thus $\Gamma$ must be a canonical multicurve obstruction for $(f,Q)$.
\end{proof}

\subsection{Existence of multi-Schottky models}\label{subsec:multi-shottky-model-exists} 
Let $(f,Q)$ be a marked critically fixed anti-Thurston map fixing
the marked set $Q$ pointwise. In this subsection, we show that $(f,Q)$
is isotopic to a marked multi-Schottky map.

Suppose $\Gamma:=\Gamma_{f,Q}$ is a canonical multicurve obstruction for
$(f,Q)$, and let us consider the decomposition of $(f,Q)$ with respect
to $\Gamma$. As in the previous subsection, we will follow the
notation from \Cref{sec:decomp-crit-fix}. We denote by
$\mathscr{S}_{\Gamma,\ARat}$ the set of all components $S \in
\mathscr{S}_\Gamma$ with $C_f\cap S \neq \emptyset $, by
$\widehat{\mathscr{S}}_{\Gamma,\ARat}$ the set of respective small
spheres $\widehat S\in \widehat{\mathscr{S}}_{\Gamma}$, and by
$Q(\widehat S, \Gamma)\subset Q(\widehat S)$ the set of marked points
in $\widehat S$ induced by the pinched boundary components of
$\overline S$.

By \Cref{cor:can_crit_fix_decomposition}, for every $\widehat S\in
\widehat{\mathscr{S}}_{\Gamma,\ARat}$ the respective small map
$\big(\widehat f|_{\widehat S}, Q(\widehat S)\big)$ is equivalent to a
marked critically fixed anti-rational map. This means that there is an
(unobstructed) marked topological Tischler graph
$\widehat{T}'_{S}\subset \widehat S$ with $V(\widehat{T}'_{S}) =
Q(\widehat S)$ such that an associated marked Schottky map is isotopic
to $\big(\widehat f|_{\widehat S}, Q(\widehat S)\big)$; see
\Cref{prop:fixed-graph-implies-schottky} and
\Cref{cor:crit-fixed-schottky}. Note also that every $\widehat q\in
Q(\widehat S,\Gamma)$ is incident with exactly two edges in
$\widehat{T}'_{S}$, because $\deg(\widehat f|_{\widehat S},\widehat q)
= 1$.

Let $\widehat{T}_{S}$ be the plane graph in $\widehat S$ with the same
realization as $\widehat{T}'_{S}$ but with the vertex set
$V(\widehat{T}_{S}) = V(\widehat{T}'_{S})\setminus Q(\widehat S,
\Gamma)$; in other words, $\widehat{T}_{S}$ is obtained from
$\widehat{T}'_{S}$ by removing all vertices in $Q(\widehat S, \Gamma)$
and joining the incident edges. Suppose $\widehat{\Lambda}_S\subset
E(\widehat T_S)$ is the set of all edges of $\widehat T_S$ 
containing points in $Q(\widehat S, \Gamma)$. It is immediate that
$\widehat{T}_{S}$ is an unobstructed marked topological Tischler graph
and that every edge in $\widehat\Lambda_S$ is critical. Let
$\widehat{\TT}_S$ be a blow-up of $\widehat{T}_S$ with respect to
$\widehat{\Lambda}_S$. By construction, every point in $Q(\widehat S,
\Gamma)$ is within a Levy face of $\widehat{\TT}_S$, and every Levy
face of $\widehat{\TT}_S$ contains exactly one point in $Q(\widehat S,
\Gamma)$. It follows that $\widehat{\TT}_S$ induces a plane graph
$\TT_S\subset S$ in $S^2$ with $V(\TT_S) = S\cap Q$. Note that $\TT_S$
is a blow-up of an unobstructed marked topological Tischler graph
whose Levy faces are characterized by containing a boundary component
of $\partial S$. In particular, every Levy face of $\TT_S$ contains at
least two points from $Q$.

We may now naturally view the union 
\begin{equation}\label{eq:multi-graph}
  M=M(f,Q):=Q\cup \bigsqcup_{S\in \mathscr{S}_{\Gamma,\ARat}} \TT_S
\end{equation}
as a plane graph in $S^2$ with the vertex set $V(M)=Q$. Note that, by
construction, $M$ is a multi-Tischler graph; moreover, it is uniquely defined up to isotopy rel.\ $Q$; see \Cref{prop:schottky-uniqueness}.

For each $\widehat S\in \widehat{\mathscr{S}}_{\Gamma,\ARat}$, up to
modifying the small map $\big(\widehat f|_{\widehat S}, Q(\widehat
S)\big)$ by isotopy, we may assume that $\widehat f|_{\widehat \TT_S}$
is an edge involution of $\widehat \TT_S$, because $\big(\widehat
f|_{\widehat S}, Q(\widehat S)\big)$ is isotopic to a marked Schottky
map associated with $\widehat{T}'_S$. It follows that $(f,Q)$ is
isotopic to a marked critically fixed anti-Thurston map $(g,Q)$ such
that $g|_{\TT_S}$ is an edge involution of $\TT_S$ for each $S\in
\mathscr{S}_{\Gamma,\ARat}$.

Now let $A$ be a simply connected face of $M$, and let $\TT$ be the unique Tischler component of $M$ with $\partial A \subset \TT$. Set
\begin{equation}\label{eq:a^g}
  A^g:=S^2 \setminus \left(\overline{A} \cup \bigcup_{D}
    \overline{D}\right),
\end{equation}
where the union is taken over all Levy faces $D$ of $\TT$ sharing a boundary edge with $A$. Since $A$ is a Jordan
domain, the set $A^g$ is a Jordan domain as well. Moreover, $g$ maps
$\partial A$ homeomorphically onto $\partial A^g$. It is then
straightforward to see that $g(A)\supset A^g$ by arguing as in the
proof of \Cref{prop:fixed-graph-implies-schottky}. Indeed, let us
connect an arbitrary point $z_1\in A^g$ to a point $z_0\in \partial
A^g\setminus C_g$ by a path $\alpha:[0,1]\to \overline{A^g}$ with
$\alpha(0)=z_0$, $\alpha(1)=z_1$, and $\alpha\left((0,1]\right)
\subset A^g$. Then a lift $\alpha'$ of $\alpha$ by $g$ starting at the
unique point $w_0\in \partial A$ with $g(w_0)=z_0$ must satisfy
$\alpha'\left((0,1]\right) \subset A$. Hence $g(A)\supset
A^g$. Similarly, for every multiply connected face $A$ of $M$, we can
deduce that $g(A) \supset A$. We claim that we actually have
equalities.
    
\begin{lemma} \label{lem:mapping-prop-model}
 With $M = M(f,Q)$ and $g$ defined as above, the following holds:
  \begin{itemize}
  \item $g$ maps every simply connected face $A$ of $M$
    homeomorphically onto $A^g$ as defined in \eqref{eq:a^g}, and
  \item $g$ maps every multiply connected face $A$ of $M$
    homeomorphically onto itself.
  \end{itemize}
\end{lemma}
    
\begin{proof}%
  We use a counting argument. (The same conclusion can also be derived
  from the properties of the canonical decomposition for $(f,Q)$.)

  We start by computing the degree $d$ of $g$. Fix any $S\in
  \mathscr{S}_{\Gamma,\ARat}$ and $q\in Q\cap S$, and let $\widehat q$
  be the corresponding marked point in $\widehat S$. Then we have
  \[
    \deg(g,q) = \deg(f,q)= \deg(\widehat f|_{\widehat S},\widehat {q})
    = \deg_{\widehat{T}'_S}(\widehat q)-1 =
    \deg_{\widehat{T}_S}(\widehat q)-1 =
    \deg_{\widehat{\TT}_S}(\widehat q)-L_{\widehat q} - 1,
  \]
  where $L_{\widehat q}$ is the number of Levy faces of
  $\widehat{\TT}_S$ with $\widehat q$ on the boundary (which is the
  same as the number $L_q$ of Levy edge pairs in $M$ incident with
  $q$). Then, by the Riemann--Hurwitz formula, we get that
  \begin{align*}
    2d-2 & = \sum_{q\in Q} \left(\deg(g,q)-1\right)
           = \sum_{S\in \mathscr{S}_{\Gamma,\ARat}} \,\,
           \sum_{q\in Q\cap S} \left(\deg(g,q)-1\right) \\
         & = \sum_{S\in \mathscr{S}_{\Gamma,\ARat}} \,\,
           \sum_{q\in Q\cap S} \left(\deg_{\widehat{\TT}_S}(\widehat
           q) - L_{\widehat q} - 2\right) \\
         & =\sum_{S\in \mathscr{S}_{\Gamma,\ARat}}
           \left( 2\cdot \# E(\widehat{\TT}_S) -
           2\cdot L_{\widehat{S}} - 2\cdot \#V(\widehat{\TT}_S)
           \right),
  \end{align*}  
  where $L_{\widehat{S}}$ is the number of Levy faces in
  $\widehat{\TT}_{S}$. Now, using the Euler formula for each
  $\widehat{\TT}_{S}$, we have that
  \begin{align*}
    d-1 & = \sum_{S\in \mathscr{S}_{\Gamma,\ARat}}
          \left(\# E(\widehat{\TT}_S) -
          L_{\widehat{S}}-\#V(\widehat{\TT}_S) \right) \\
        & = \sum_{S\in \mathscr{S}_{\Gamma,\ARat}}
          \left(\# F(\widehat{\TT}_S) - 2- L_{\widehat{S}} \right)
          = \sum_{S\in \mathscr{S}_{\Gamma,\ARat}}
          \left(\# F(\widehat{\TT}_S) - L_{\widehat{S}} \right)
          - 2\cdot K_{\TT},
  \end{align*}   
  where $K_\TT$ is the number of Tischler components of $M$. Since
  there is a one-to-one correspondence between the simply connected
  faces of $M$ and the non-Levy faces of its Tischler components, we
  finally get that
  \begin{equation}\label{eq:deg-multi-model}
    d=\#F_s(M)-2\cdot K_\TT+1,
  \end{equation}
  where $F_s(M)$ denotes the set of simply connected faces of $M$.

  Now let $A$ be a multiply connected face of $M$, and let $z\in A$ be
  arbitrary. By the discussion preceding
  \Cref{lem:mapping-prop-model}, there is at least one preimage of $z$
  under $g$ in $A$. Moreover, for every simply connected face $B$ of
  $M$, there is at least one preimage of $z$ in $B$ whenever
  $B^g\supset A$. The latter condition is always true unless $A
  \subset D$ for some Levy face $D$ of the unique Tischler component
  $\TT$ with $\partial B \subset \TT$. Since for every Tischler
  component $\TT$, there is exactly one Levy face $D$ with $D\supset
  A$, and every such face $D$ shares a boundary edge with exactly two
  non-Levy faces of $\TT$ (which are also simply connected faces of
  $M$), we deduce that
  \[\#\left(g^{-1}(z) \cap \bigcup_{B\in F_s(M)} B \right) \geq
    \#F_s(M) - 2\cdot K_\TT.\]
  From the formula \eqref{eq:deg-multi-model} for the degree of $g$,
  we conclude that we have accounted for all the preimages of $z$
  under $g$. In particular, the inequality above must in fact be
  equality, and also $A$ must contain a unique preimage of
  $z$. Analyzing the preimages of points in simply connected faces of
  $M$ in a similar way, we derive the desired mapping behavior of $g$.
\end{proof}

Set \[\mu=\mu(f,Q):=g|_{M^\mu},\]
where $M^\mu$ is the union of the closures of all multiply connected
faces of $M$.  \Cref{lem:mapping-prop-model} implies that $(M,\mu)$ is
a multi-Tischler pair and, moreover, $g$ is a multi-Schottky map
associated with $(M,\mu)$. Thus we have established the following
result.

\begin{theorem}\label{thm:multi-shottky-model}
Let $(f,Q)$ be a marked critically fixed anti-Thurston map that fixes the marked set $Q$ pointwise. Then $(f,Q)$ is isotopic to a marked multi-Schottky map associated with the multi-Tischler pair $(M,\mu)$ constructed as above.
\end{theorem}

\subsection{Classification of critically fixed anti-Thurston maps}
Our goal in this subsection is to prove the following classification
result, establishing \Cref{thm:b} from the introduction.

\begin{theorem}\label{thm:class-crit-fix-anti-thurst}
  There is a canonical one-to-one correspondence between the isotopy
  (resp.\ equivalence) classes of multi-Tischler pairs and the isotopy
  (resp.\ equivalence) classes of marked critically fixed
  anti-Thurston maps that fix their marked sets pointwise.
\end{theorem}

\begin{proof}
  Let $\mathscr{P}$ be the set of isotopy classes $[(M,\mu)]$ of
  multi-Tischler pairs $(M,\mu)$, and let $\mathscr{F}$ be the set of
  isotopy classes $[(f,Q)]$ of marked critically fixed anti-Thurston
  maps $(f,Q)$ that fix the marked set $Q$
  pointwise. \Cref{prop:equivalent-multi-schottky}\ref{item:multi-iv}
  implies that the map $\Psi: \mathscr{P} \to \mathscr{F}$ given
  by $$\big[(M,\mu)\big] \mapsto \big[(f_{(M,\mu)},V(M))\big],$$
  where $f_{(M,\mu)}$ denotes a multi-Schottky map associated with the
  multi-Tischler pair $(M,\mu)$, is well-defined. We also know that
  $\Psi$ is surjective by \Cref{thm:multi-shottky-model}.
    
  To prove the injectivity of $\Psi$, suppose $(M,\mu)$ and
  $(M',\mu')$ are two multi-Tischler pairs such that the respective
  marked multi-Schottky maps $(f,Q)$ and $(f',Q')$ are isotopic. Note
  that in this case $Q=Q'$. We will assume in the following that $\#Q=\#Q'\geq 3$; otherwise,  $(M,\mu)$ and
  $(M',\mu')$ are trivially isotopic. Let $\Gamma_M$ and $\Gamma_{M'}$ be the
  multicurves constructed in
  \Cref{subsec:can-obstr-multi-schottky}. By
  \Cref{thm:can-obstr-multi-schottky}, $\Gamma_M$ and $\Gamma_{M'}$
  are canonical multicurve obstructions for $(f,Q)$ and $(f',Q')$, and thus
  they must be isotopic rel.\ $Q$. Moreover, by the proof of this
  theorem, the ``anti-rational'' small maps of the respective
  canonical decompositions uniquely recover the Tischler components of
  $M$ and $M'$ up to isotopy rel.\ $Q$. Hence, there is a
  homeomorphism $\phi\in \Homeo^+_0(S^2,Q)$ such that
  $\phi(M)=M'$. Note that $g:=\phi^{-1} \circ f' \circ \phi$ is a
  multi-Schottky map associated with the multi-Tischler pair $(M,
  \phi^{-1} \circ \mu'\circ \phi)$. Moreover, $(g,Q)$ and $(f,Q)$ are
  isotopic, because $(f,Q)$ and $(f',Q')$ are isotopic by
  assumption, and thus there exists $\psi_0\in\Homeo^+_0(S^2,Q)$ with $f=g\circ\psi_0$. 

  Since $f$ and $g$ are multi-Schottky maps associated with multi-Tischler pairs
  with the same multi-Tischler graph $M$, they map the closure of each face and each edge of $M$
  homeomorphically onto the same set. Hence setting
  $\psi_1:=(g|_{\overline{A}})^{-1}\circ f|_{\overline{A}}$ on each closed face $\overline{A}$
  defines a homeomorphism $\psi_1\in \Homeo^+(S^2, Q)$ with $f=g\circ\psi_1$ that
  maps (the closure of) every face and every edge of $M$ onto itself. 
  It follows that  $\zeta:=\psi_0\circ\psi_1^{-1}\in\Homeo^+(S^2,Q)$ satisfies $g\circ \zeta=g$, that is, $\zeta$ is a deck 
  transformation of $g$. Fixing some complex structure on $(S^2,Q)$ and pulling it back by $g$, we may therefore 
  regard $\zeta$ as a conformal automorphism fixing $Q$. Since $\#Q\geq 3$, it follows
  that $\zeta$ is the identity, and thus $\psi_0=\psi_1$. In particular, we have $\psi_0(M^\mu)=M^\mu$ and
  $\mu=(\phi^{-1}\circ\mu'\circ\phi)|_{M^\mu}\circ\psi_0|_{M^\mu}$.
  Since $\psi_0\in\Homeo^+_0(S^2,Q)$ and it maps each component of
  $\partial M^\mu$ onto itself, we conclude that $\psi_0|_{M^\mu}$ is
  isotopic to the identity rel.\ $M^\mu\cap V(M)$.  (This easily
  follows from the Alexander trick and \cite[Thm.~A.5]{buserGeometrySpectraCompact2010} by partitioning $M^\mu$ into Jordan
  domains.) Consequently, $\mu$ is isotopic to
  $(\phi^{-1}\circ\mu'\circ\phi)|_{M^\mu}$ rel.\ $M^\mu\cap V(M)$, and
  thus the multi-Tischler pairs $(M,\mu)$ and $(M',\mu')$ are
  isotopic. This shows that the map $\Psi$ is injective.

  The corresponding statement for the equivalence classes is now a
  straightforward consequence using
  \Cref{def:Thurston-equiv,def:multi-schottky-equiv}.
\end{proof}

To derive \Cref{thm:b}, note that, given a multi-Tischler pair
$(M,\mu)$, the vertex set $V(M)$ coincides with the critical set $C_f$
of an associated multi-Schottky map $f=f_{(M,\mu)}$ precisely when the
multi-Tischler graph $M$ is reduced. The desired statement now follows
from the proof of \Cref{thm:class-crit-fix-anti-thurst}. In
particular, we have that the map $\Psi: \mathscr{P} \to \mathscr{F}$
provides a canonical one-to-one correspondence between isotopy classes
of reduced multi-Tischler pairs and isotopy equivalence classes of
critically fixed anti-Thurston maps.

\section{Applications and examples}
\label{sec:exampl-appl}

In this section, we present a few selected applications of our
combinatorial classification of critically fixed anti-rational maps.

\subsection{Pullbacks of curves and global curve attractors}
\label{subsec:glob-curve-attr}

Let $(f,Q)$ be a marked \antiThurston map. In the following
discussion, by a curve we will always mean a simple closed
(unoriented) curve in $S^2 \setminus Q$, and homotopy will as usual be
understood as free homotopy in $S^2 \setminus Q$.

A \emph{pullback} of a curve $\gamma$ under $f$ is a connected
component of $f^{-1}(\gamma)$. By homotopy lifting, homotopic curves
have homotopic pullbacks, which means that pullbacks are well-defined
for homotopy classes of curves.  With a slight abuse of notation, for
any set $\cA$ of homotopy classes of curves, we write $f^{-1}(\cA)$
for the set of all pullbacks of $[\gamma]\in \cA$ under $f$, and
similarly $f^{-n}(\cA)$ for the pullbacks under the $n$-th iterate
$f^n$.

A sequence $(\gamma_n)_{n\geq 0}$ of curves is a \emph{pullback
    orbit} of a curve $\gamma = \gamma_0$ if $\gamma_n$ is a pullback
  of $\gamma_{n-1}$ for each $n\geq 1$. A \emph{pullback orbit} of the
  homotopy class $[\gamma]$ is defined in a similar way. We say that
  the curve $\gamma$ (and its homotopy class $[\gamma]$) is
  \emph{periodic} if there is a periodic pullback orbit of
  $[\gamma]$.
  
\begin{figure}[t]
  \centering
  \begin{overpic}[width=.9\textwidth]{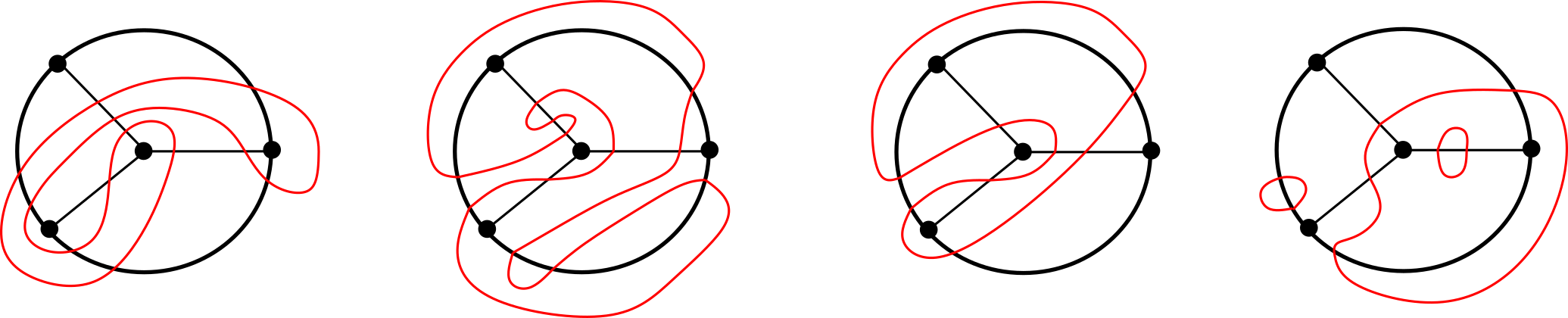}
    \put(22,10){$\xleftarrow{f}$} \put(50,10){$\simeq$}
    \put(77,10){$\xleftarrow{f}$} \put(0,0){\color{red} $\gamma_0$}
    \put(27,0){\color{red}{$\gamma_1^1$}} \put(58,0){\color{red}{$\gamma_1$}}
    \put(83,0){\color{red}{$\gamma_2^1$}}
  \end{overpic}
  \caption{Pulling back simple closed curves (in red) under a Schottky
    map $f$ associated with the tetrahedral Tischler graph $T$ (in
    black). Starting on the left with a minimal curve $\gamma_0$ of
    complexity $\gamma_0 \cdot T = 12$, the next picture shows its
    preimage $\gamma_1^1 = f^{-1}(\gamma_0)$, which is a simple closed
    curve with the same intersection points with $T$, but with
    $\gamma_1^1 \cdot T = 8$. The third picture shows a minimal curve
    $\gamma_1$ homotopic to $\gamma_1^1$, and the last picture shows
    the three components of its preimage
    $\gamma_2^1 \sqcup \gamma_2^2 \sqcup \gamma_2^3 =
    f^{-1}(\gamma_1)$, where $\gamma_2^1$ is minimal with
    $\gamma_2^1 \cdot T = 4$, and $\gamma_2^2$ and $\gamma_2^3$ are
    null-homotopic. The preimage $f^{-1}(\gamma_2^1)$ (not depicted)
    is a simple closed curve homotopic to $\gamma_2^1$.}
  \label{fig:pullback-1}
\end{figure}

\begin{definition}
  \label{def:fgca}
  Let $(f,Q)$ be a marked \antiThurston map. A \emph{finite global
    curve attractor} (or \emph{FGCA} for short) for $(f,Q)$ is a
  finite set $\cN$ of homotopy classes of curves with the following
  two properties:
  \begin{enumerate}
  \item \label{fgca:1} $f^{-1}(\cN) = \cN$; and
  \item \label{fgca:2} for every curve $\gamma$, there exists $n \geq
    1$ such that $f^{-n} \left(\{[\gamma]\}\right) \subseteq
    \cN$. (In particular, every pullback orbit of $[\gamma]$
    eventually lands in $\cN$.)
  \end{enumerate}
\end{definition}

\begin{remarks}\mbox{}
  \begin{enumerate}
  \item Since there are only finitely many homotopy classes of
    non-essential curves, and since pullbacks of non-essential curves
    are always non-essential, one can restrict to essential curves in
    the definitions of pullbacks and of finite global curve attractors.
  \item Uniqueness of FGCAs is straightforward from the
    definition. In fact, the FGCA (if it exists) consists of all
      periodic homotopy classes $[\gamma]$ of curves together with all their iterated pullbacks
$f^{-n}(\{[\gamma]\})$, $n\ge 1$.   
    \item It is easy to see that in order to establish existence of
      FGCAs, it is enough to find a finite set $\cN$ of homotopy
      classes of curves satisfying \ref{fgca:2}.
  \end{enumerate}
\end{remarks}
The central problem on FGCAs is the question of existence,
specifically the following conjecture.
\begin{conjecture}[FGCA conjecture]
  \label{conj:fgca}
  Let $(f,Q)$ be a marked postcritically finite \antiRational map
  that is not a $(2,2,2,2)$-map. Then $(f,Q)$ has a finite global
  curve attractor.
\end{conjecture}

The study of FGCAs for rational maps was initiated by Pilgrim in
\cite{pilgrimAlgebraicFormulationThurstons2012}, where he also
established the FGCA conjecture for certain quadratic
polynomials. Later, the conjecture was proved for all postcritically
finite polynomials in
\cite{belkRecognizingTopologicalPolynomials2022}, and for critically
fixed rational maps in
\cite{hlushchanka2017Thesis,hlushchankaTischlerGraphsCritically2019}. Furthermore, the FGCA conjecture has been recently established in the case when $\#Q=4$ by Bartholdi, Dudko, and Pilgrim \cite{bartholdiCorrespondences2024}. We refer the reader to the survey
\cite{pilgrimPullbackRelationCurves2022} for more context and
background. In this subsection, we will discuss the FGCA conjecture in
the special setting of (marked) critically fixed anti-rational maps.

In the following, we consider realizable Schottky maps instead of
critically fixed anti-rational maps, purely for convenience. Since
these are combinatorially equivalent notions
(\Cref{thm:schottky-obstructions,thm:tischler-fixed-correspondence}),
and all the questions in this subsection are invariant under
combinatorial equivalence, it does not substantially alter any of our
results. We also note that even though some of the subsequent
auxiliary results are true for obstructed marked Schottky maps, it is
not too hard to see that obstructed marked Schottky maps never have an
FGCA\@. Indeed, one can show that obstructed marked Schottky maps have
infinitely many periodic homotopy classes of curves (with period $2$);
compare \Cref{cor:periodic-arcs}.

Let $T=(Q,E)$ be a marked topological Tischler graph and $(f,Q)$ 
be an associated marked Schottky map. Recall that the
\emph{complexity} $\gamma \cdot T$ of a curve $\gamma$ is defined as
the minimal number of intersections with $T$ in the homotopy class
of $\gamma$, and that $\gamma$ is called \emph{minimal} if it
realizes this minimal intersection number. Recall also that
we are using the topological notion of transversality; in particular, we say that $\gamma$ and $T$ have transverse intersections whenever $\gamma\cap T$ is finite and they cross each other at every intersection point $p\in \gamma\cap T$ (see \Cref{sec:preliminaries}).

We will say that a curve $\gamma$ and an edge $e\in E$ \emph{form a
  bigon} if there is a Jordan domain $U\subset S^2\setminus Q$ (the
bigon) whose boundary is given by the disjoint union $\gamma'\bigsqcup
e'\bigsqcup \{z_1,z_2\}$, where $\gamma'$ and $e'$ are open subarcs of
$\gamma$ and $e$, respectively, connecting $z_1$ and $z_2$.
Furthermore, we will say that $\gamma$ \emph{forms a bigon with the graph} $T$ if $\gamma$
forms a bigon with at least one edge $e \in E$.
It is easy
to see that if a curve $\gamma$ is minimal, then  $\gamma$ has transverse intersections with $T$ and does not
form a bigon with $T$. (Otherwise, we can homotope $\gamma$ to strictly decrease the number of intersections
of $\gamma$ with $T$.) In fact, the converse is also
true: if a curve $\gamma$ has transverse intersections with $T$ and does not form a bigon with $T$, then $\gamma$ is minimal; compare \cite[Secs.~1.2.4 and
1.2.7]{farbPrimerMappingClass2012}.

\Cref{lem:complexity} shows that complexity of curves can never increase under taking
pullbacks. In particular, the set $\cC_m$ of the homotopy classes of curves of 
complexity at most $m$ satisfies $f^{-1}(\cC_m)
\subseteq \cC_m$. (Note also that the set $\cC_m$ is finite, because $T$ is connected.) We will say that a minimal curve $\gamma$ is of \emph{stable
  complexity} (equal to $\gamma \cdot T$) if it has a pullback orbit 
$(\gamma_n)_{n\geq 0}$ with $\gamma_n \cdot T = \gamma \cdot T$ for all $n\geq 0$.  By
\Cref{lem:complexity}, this implies that each curve $\gamma_n$ is minimal
with $\gamma_n \cap T = \gamma \cap T$, and that all other pullbacks
of $\gamma$ under $f^n$ are null-homotopic. The
FGCA conjecture for (unmarked) critically fixed anti-rational maps is then equivalent to the following conjecture.

\begin{conjecture}[Decreasing complexity conjecture]
  \label{conj:decreasing-complexity}  
  Let $f$ be a critically fixed anti-rational map and $T_f$ be its reduced Tischler graph. Then there exists $N \in \N$ such that there are no minimal curves of
  stable complexity (with respect to $T_f$) greater than $N$.
\end{conjecture}

\begin{remark}
  The proof of the FGCA conjecture in the case of critically fixed
  rational maps in \cite{hlushchankaTischlerGraphsCritically2019}
  establishes a version of the above conjecture, with the complexity
  of a curve $\gamma$ defined as the intersection number
  $\gamma\cdot G_f$, where $G_f$ is a ``charge graph'' associated with
  a given rational map $f$. In that setting, complexity strictly
  decreases under pullbacks whenever the curve $\gamma$ has at least
  one double intersection with an edge of $G_f$. In the
  orientation-reversing case, however, double intersections with the
  edges of the reduced Tischler graph may persist under pullbacks, as
  illustrated by the example in \Cref{fig:pullback-2}.
\end{remark}

\begin{figure}[t]
  \centering
  \begin{tikzpicture} [line width=.8pt]
    \def\r{1cm}
    
    \def\s{.15*\r}
    
    \tikzstyle{vertex}=[circle,fill=black,minimum size=\s,inner sep=0pt]
    \begin{scope}
      \foreach \name/\angle in {A/36, B/108, C/180, D/-108, E/-36} {
        \node[vertex] (\name) at (\angle:\r) {};
      }
      \draw circle (\r);
      \foreach \i/\j in {A/B, B/C, C/D, D/E} {
        \draw (\i) to [bend left=20] (\j);
      }
      
      \draw[red,
      ]  plot [closed hobby]
      coordinates { (-72:\r) (-108:1.3*\r) (-144:\r) (-144:.7*\r)
        (-72:.2*\r) (12:\r) (36:1.2*\r) (72:\r) (72:.7*\r) (144:.7*\r)
        (144:\r) (108:1.3*\r) (36:1.6*\r) (-12:\r) (-72:.7*\r) };
      \node at (144:\r) [red, above left] {$\gamma_1$};
    \end{scope}
    
    \begin{scope} [xshift=4*\r]
      \foreach \name/\angle in {A/36, B/108, C/180, D/-108, E/-36} {
        \node[vertex] (\name) at (\angle:\r) {};
      }
      \draw circle (\r);
      \foreach \i/\j in {A/B, B/C, C/D, D/E} {
        \draw (\i) to [bend left=20] (\j);
      }
      
      \draw[red,
      ]  plot [closed hobby]
      coordinates { (72:\r) (108:1.3*\r) (144:\r) (144:.7*\r)
        (72:.2*\r) (-12:\r) (-36:1.2*\r) (-72:\r) (-72:.7*\r) (-144:.7*\r)
        (-144:\r) (-108:1.3*\r) (-36:1.6*\r) (12:\r) (72:.7*\r) };
      \node at (144:\r) [red, above left] {$\gamma_2$};
  \end{scope}
  \end{tikzpicture}
  \caption{Example of an unobstructed reduced topological Tischler graph $T$ (in black) and a
    pair of minimal curves $\gamma_1$ and $\gamma_2$ (in red) satisfying the following two conditions: (i) both $\gamma_1$ and $\gamma_2$ have a double intersection with an edge of $T$; and (ii) $\gamma_1$ and $\gamma_2$ are pullbacks of each other (up to
    homotopy) under the corresponding (realizable)
    Schottky map $f_T$. In particular, the curves $\gamma_1$ and  $\gamma_2$ are periodic with period $2$.}
  \label{fig:pullback-2}
\end{figure}

We have the following partial result toward
\Cref{conj:decreasing-complexity}. (Recall that a face of $T$ is
called \emph{digonal} or \emph{triangular} if it is simply connected
and its boundary circuit has length 2 or 3, respectively.)

\begin{proposition}
  \label{thm:gca-triangles}
  Let $T=(Q,E)$ be a marked topological Tischler graph and $(f,Q)$ be
  an associated marked Schottky map. If $\gamma$ is a minimal curve of
  stable complexity, then $\gamma$ passes through every digonal or
  triangular face $A$ of $T$ at most once, that is, $\gamma\cap A$ has
  at most one connected component.
\end{proposition}

Before proving \Cref{thm:gca-triangles}, let us first state and prove
a couple of auxiliary results.

\begin{lemma}
  \label{lem:no-repeats}
  Let $T=(Q,E)$ be a marked topological Tischler graph and $(f,Q)$ be
  an associated marked Schottky map. Suppose $\gamma$ is a minimal
  curve of stable complexity that passes through a face $A$ of $T$. If
  $z_1$ and $z_2$ are two consecutive intersections of $\gamma$ with
  $\partial A$ when following along $\gamma$, then $z_1$ and $z_2$ lie
  on two different edges of $T$.
\end{lemma}
In other words, if, while following the curve $\gamma$, we crossed an
edge $e\subset \partial A$ twice, then we must have crossed a
different edge $e'\subset \partial A$ in between.
\begin{proof}
  Suppose to the contrary that $z_1,z_2\in e$ for some edge $e\subset
  \partial A$, and let $\alpha$ be an open subarc of $\gamma$
  connecting $z_1$ and $z_2$ with $\alpha\cap \partial A =
  \emptyset$. Since $\gamma$ intersects the Jordan curve $\partial A$
  transversely, we have $z_1\neq z_2$, and either $\alpha \subset A$
  or $\alpha \subset S^2\setminus \overline{A}$.  In the former case,
  $\gamma$ and $e$ form a bigon, which contradicts the minimality of
  $\gamma$. In the latter case, there is a (unique) pullback $\gamma'$
  of $\gamma$ under $f$ with $z_1,z_2\in \gamma'$, as well as a
  (unique) preimage arc $\alpha'\subset \gamma'$ of $\alpha$ under $f$
  connecting $z_1$ and $z_2$ with $\alpha'\subset A$. It follows that
  $\gamma'$ and $e$ form a bigon, which implies that $\gamma'$ is not
  minimal, contradicting the assumption that the curve $\gamma$ is of stable
  complexity.
\end{proof}

\begin{figure}[t]
  \centering
  \begin{tikzpicture}[scale=1.8,
    every node/.style={inner sep=0pt}
    ]
    \def\r{.5pt}
    \def\s{3pt}
    \draw (-1.6,0) -- (0,0) node[label={[label
      distance=\s]above:$0$}](z){} -- (1.6,0);
    \fill (z) circle[radius=\r];
    \coordinate[label={[label distance=\s]260:$z_1$}] (z1) at (-1,0);
    \coordinate[label={[label distance=\s]-80:$z_2$}] (z2) at (1,0);
    \coordinate[label={[label distance=\s]260:$z_3$}] (z3) at (-.5,0);
    \coordinate[label={[label distance=\s]-80:$z_4$}] (z4) at (.5,0);
    \coordinate[label={[label distance=\s]260:$z_5$}] (z5) at (-.2,0);
    \draw[red] plot [hobby] coordinates { (-.95,-.2) (z1) (0,1) (z2) (0,-.7)
      (z3) (0,.5) (z4) (0,-.3) (z5) (-.1,.15) };
    \foreach \x in {z1,z2,z3,z4,z5} \fill[red] (\x) circle[radius=\r];
    \node (A) at (-1.3,.7) {$A$};
    \node (e1) at (-1.4,.1) {$e_1$};
    \node (e2) at (1.4,.1) {$e_2$};
    \node[red] (alpha1) at (.9,.8) {$\alpha_1$};
    \node[red] (alpha2) at (.9,-.7) {$\alpha_2$};
    \node[red] (alpha3) at (.3,.6) {$\alpha_3$};
  \end{tikzpicture}
  \hspace{1cm}
  \begin{tikzpicture}[scale=2.2,
    every node/.style={inner sep=0pt}
    ]
    \def\r{.5pt} \def\s{4pt} \draw (-1.5,0) -- (0,0)
    node[label={[label distance=\s]above:$0$}](a) 
    {} -- (1.2,0)
    node[label={[label distance=\s]above:$1$}](b){} -- (2.2,0);
    \foreach \x in {a,b} \fill (\x) circle[radius=\r];
    \coordinate[label={[label distance=\s]260:$z_1$}] (z1) at (-.5,0);
    \coordinate[label={[label distance=\s]-90:$z_2$}] (z2) at (.6,0);
    \coordinate[label={[label distance=\s]-80:$z_3$}] (z3) at (1.5,0);
    \coordinate[label={[label distance=\s]260:$z_4$}] (z4) at (-1,0);
    \coordinate[label={[label distance=\s]270:$z_5$}] (z5) at (.2,0);
    \draw[red] plot [hobby] coordinates { (-.5, -.2) (z1) (0,.4) (z2)
      (1.3,-.4) (z3) (0,.8) (z4) (-.5,-.5) (z5) (.7,.3) (1.1,.3)};
    \foreach \x in {z1,z2,z3,z4,z5} \fill[red] (\x) circle[radius=\r];
    \node (A) at (-1.3,.6) {$A$};
    \node (e1) at (-1.3,.1) {$e_1$};
    \node (e2) at (.8,.1) {$e_2$};
    \node (e3) at (2,.1) {$e_3$};
    \node[red] (alpha1) at (0,.5) {$\alpha_1$};
    \node[red] (alpha2) at (1.2,-.55) {$\alpha_2$};
    \node[red] (alpha3) at (1.4,.6) {$\alpha_3$};
    \node[red] (alpha4) at (-.6,-.6) {$\alpha_4$};
    \node[red] (alpha5) at (.7,.4) {$\alpha_5$};
  \end{tikzpicture}
  \caption{Illustrations of the arguments in the proofs of
    \Cref{lem:no-spiraling} (left) and \Cref{thm:gca-triangles}
    (right), both in the case where $\alpha_1 \subset A$.} 
  \label{fig:proof-illustrations}
\end{figure}

\begin{lemma}
  \label{lem:no-spiraling}
  Let $T=(Q,E)$ be a marked topological Tischler graph and $(f,Q)$ be
  an associated marked Schottky map. Suppose $\gamma$ is a minimal
  curve of stable complexity that passes through a face $A$ of $T$,
  and let $z_1,z_2,z_3$ be three consecutive intersections of $\gamma$
  with $\partial A$ when following along $\gamma$. If $z_1\in e_1$ and
  $z_2\in e_2$ for two adjacent edges $e_1,e_2$ of $T$, then either
  $z_3=z_1$ or $z_3\notin e_1\cup e_2$.
\end{lemma}

\begin{proof}
  By \Cref{lem:no-repeats}, we immediately have that $z_3 \notin
  e_2$. Suppose to the contrary that $z_3 \neq z_1$ and $z_3 \in
  e_1$. We may assume, for notational convenience, that $T$ is a plane
  graph in $\CC$, the face $A$ is the upper half-plane, $e_1 =
  (-2,0)$, and $e_2 = (0,2)$ (or $e_1 = (-\infty,0)$ and $e_2 =
  (0,\infty)$ when $A$ is a digonal face). By our assumptions, $z_3\in
  (z_1,0)$ or $z_1 \in (z_3,0)$. Up to changing the direction in which
  we are following the curve $\gamma$, we may assume without loss of
  generality that $z_3 \in (z_1,0)\subset e_1$.
  
  Let $(z_n)_{n\geq 1}$ be the sequence of consecutive intersections
  of $\gamma$ with $\partial A$ when following along $\gamma$ (in the
  chosen direction). For each $n\geq 1$, we denote by $\alpha_n$ the
  open subarc of $\gamma$ connecting $z_n$ and $z_{n+1}$ with $z_{n+2}
  \notin \alpha_n$. Since $\gamma$ has transverse intersections with
  $\partial A$, we have that $\alpha_3$ is contained in the Jordan
  domain in $\C$ bounded by $\alpha_1 \cup [z_1,z_2]$. Since $z_4 
  \notin e_1$ (by \Cref{lem:no-repeats}), we thus get $z_4 \in
  (0,z_2)\subset e_2$. Repeating this argument, we conclude that $z_5
  \in (z_3,0)$, $z_6 \in (0,z_4)$, etc.; see the left-hand side of
  \Cref{fig:proof-illustrations} for an illustration. This shows that
  $\gamma$ has infinitely many intersections with $\partial A
  \subset T$, contradicting the minimality of $\gamma$.
\end{proof}

Now the proof of \Cref{thm:gca-triangles} is a simple consequence.

\begin{proof}[Proof of \Cref{thm:gca-triangles}]
  The proof in the digonal case follows immediately from
  Lemmas~\ref{lem:no-repeats} and \ref{lem:no-spiraling}, so let us
  suppose that $A$ is a triangular face.  We may again assume that $T$
  is a plane graph in $\CC$, the face $A$ is the upper half-plane, and
  that the three edges in $\partial A$ are $e_1 = (-\infty,0)$, $e_2 =
  (0,1)$, and $e_3 = (1,\infty)$.

  Suppose to the contrary that $\gamma$ is a minimal curve of stable
  complexity that passes through $A$ at least twice, so that
  $\gamma\cap A$ has at least two connected components. As before, let
  $(z_n)_{n\geq 1}$ be the sequence of consecutive intersections of
  $\gamma$ with $\partial A$ when following along $\gamma$, and denote
  by $\alpha_n$ (for $n\geq 1$) the open subarc of $\gamma$ connecting
  $z_n$ and $z_{n+1}$ with $z_{n+2} \notin \alpha_n$. By
  \Cref{lem:no-repeats}, we may assume that $z_1 \in e_1$ and $z_2 \in
  e_2$. Using Lemmas~\ref{lem:no-repeats} and \ref{lem:no-spiraling}
  together with the assumptions on $\gamma$, we conclude that $z_3 \in
  e_3$, $z_4 \in e_1\setminus\{z_1\}$, $z_5 \in e_2\setminus\{z_2\}$,
  and $z_6 \in e_3\setminus\{z_3\}$. Assuming that the subarc
  $\alpha_1$ of $\gamma$ connecting $z_1$ and $z_2$ is contained in
  $A$, we get that $z_4 < z_1 < 0 < z_5 < z_2 < 1 < z_3$; see the
  right-hand side of \Cref{fig:proof-illustrations} for an
  illustration. However, the subarc $\alpha_5$ connecting $z_5$ and
  $z_6$ is then also contained in $A$ and must intersect $\alpha_1$,
  contradicting the assumption that $\gamma$ is a simple closed
  curve. The proof in the case where $\alpha_1 \subset S^2 \setminus
  \overline{A}$ is completely analogous.
\end{proof}

Finally, let us provide a simple result about certain curves that are
always periodic (with period $1$), and hence are of stable complexity
and contained in the FGCA of the respective map (if the FGCA exists).

\begin{lemma}
  \label{lem:fixed-curves}
  Let $T=(Q,E)$ be a marked topological Tischler graph and $(f,Q)$ be
  an associated marked Schottky map. Suppose $\gamma$ is a minimal
  curve with $\gamma\cdot T > 0$ that passes through every face of $T$ at most once. Then $\gamma$ has a
  pullback $\gamma'$ that is homotopic to $\gamma$, while all other pullbacks are null-homotopic.
\end{lemma}

Note that, since every face of $T$ is a Jordan domain, we have that a curve $\delta$ is null-homotopic if and only if $\delta\cdot T=0$. Furthermore, if $\delta \cdot T > 0$, then we must have $\delta \cdot T \geq 2$.

\begin{proof}
  Let $A_1, \dots, A_n$ be all the faces of $T$ intersecting
  $\gamma$. By assumption, $\alpha_k := \gamma \cap A_k$ and
  $\beta_k := \gamma \setminus \overline{A_k}$ are open subarcs of
  $\gamma$ for each $k=1,\ldots,n$. Moreover, we have $n\geq 2$, and
  (up to relabeling of the faces) $\alpha_k$ and $\beta_k$ connect
  $z_{k-1},z_k\in \partial A_k$ with $z_n = z_0$. By the definition of
  Schottky maps,
  $\alpha_k' := f^{-1}(\gamma) \cap A_k = f^{-1}(\beta_k) \cap A_k$ is
  also an open arc in $A_k$ connecting $z_{k-1}$ to $z_k$. Since $A_k$
  is a Jordan domain, $\alpha_k'$ is homotopic to $\alpha_k$ in
  $\overline{A_k}$ (with the endpoints $z_{k-1}$ and $z_k$ fixed), and
  thus the concatenation $\gamma'$ of all the arcs $\alpha_k'$ (with
  the points $z_k$ included) is a simple closed curve in
  $S^2\setminus Q$ that is homotopic to $\gamma$. Furthermore, the
  fact that $\gamma'$ is a simple closed curve implies that it is a
  connected component of $f^{-1}(\gamma)$, i.e., a pullback of
  $\gamma$, as claimed. The final assertion of the lemma now follows
  from \Cref{lem:complexity} (or directly from the definition of
  Schottky maps).
\end{proof}

The lemma above combined with the previous discussion immediately implies the following corollary.

\begin{corollary}
  \label{cor:gca-triangles}
  Let $T=(Q,E)$ be a marked topological Tischler graph and $(f,Q)$ be
  an associated marked Schottky map. If every edge of $T$ is on the
  boundary of a digonal or triangular face, then $(f,Q)$ has a finite
  global curve attractor. Moreover, if every face of $T$ is digonal or triangular, then the FGCA is given by the homotopy classes of curves that pass through
  every face of $T$ at most once.
\end{corollary}
\begin{remark}
  Note that the assumptions in the corollary imply that the marked
  topological Tischler graph $T$ is unobstructed, so that $(f,Q)$ is
  equivalent to a marked critically fixed anti-rational
  map. Consequently, we have shown the following special case of the
  FGCA conjecture: if $g$ is a critically fixed anti-rational map such
  that every edge of its reduced Tischler graph $T_g$ is on the
  boundary of a digonal or triangular face, then $(g,C_g)$ has an
  FGCA\@. This is \Cref{thm:d} from the introduction.  
\end{remark}

\subsection{Twisting problem}
\label{subsec:twisting}

In this subsection, we will show how our classification enables us to
solve certain instances of the \emph{twisting problem} for critically
fixed anti-Thurston maps. Our discussion here parallels the one for
the orientation-preserving case in
\cite{hlushchankaCriticallyFixedThurston2022}, where one can find some
more background and references.

The general setup is as follows. Let
$(f,Q)$ be a marked \antiThurston map and $\phi\in \Homeo^+(S^2,Q)$. Then the \emph{twist of $(f,Q)$ by $\phi$} (or just the \emph{twisted map} if the context is unambiguous) is the marked \antiThurston map
$(f_\phi,Q)$, where $f_\phi = \phi \circ f$. Note that $f_\phi$ has the same dynamics on $Q$ as $f$; in particular, $f_\phi$ is an \antiThurston map with  $P_{f_\phi}=P_f$. The \emph{twisting problem} is
the problem of finding the combinatorial  equivalence class of the
twisted map $(f_\phi,Q)$, given the (isotopy) data of $(f,Q)$ and $\phi$. (Originally, the twisting problem was investigated for unmarked maps, i.e.,
with $Q=P_f$, but the setup is exactly the same for marked maps.) 

In this subsection, we study the twisting problem in the setting of
marked Schottky maps. More specifically, given a marked topological
Tischler graph $T = (Q,E)$ and an associated marked Schottky map
$(f,Q)$, our goal is to determine conditions under which the twist of
$(f,Q)$ by $\phi\in \Homeo^+(S^2,Q)$ is again a marked Schottky map
(up to isotopy), and to identify the corresponding marked topological
Tischler graph $T_\phi$. In fact, we will show below that this is the
case when $\phi$ is a \emph{Dehn twist} (or, more generally, a power
of a Dehn twist) about an essential simple closed curve
$\gamma\subset S^2\setminus Q$ that passes through every face of $T$
at most once. We begin by recalling the definition of a Dehn twist and
then introducing a ``rotation operation'' on the given graph $T$,
which will be used to describe the sought-for graph $T_\phi$.

Suppose $\gamma$ is an arbitrary simple closed curve in $S^2\setminus Q$. We can choose coordinates on $S^2$ (that is, identify $S^2$ with $\CC$ by an orientation-preserving homeomorphism) so that $\gamma$ is the circle $\{z: |z|=e^\pi\}$ and so that the closed annulus $R:=\{z: 1\leq |z|\leq e^{2\pi}\}$ satisfies $R\cap Q = \emptyset$. (Note that $\gamma$ is a core curve of $R$.) A \emph{Dehn twist about $\gamma$} is the map $\phi_\gamma: S^2 \to S^2$ defined by
\begin{equation}\label{eq:Dehn-twist}
  \phi_\gamma(z) =
  \begin{cases}
    z\cdot e^{ix} & \text{ if } z=e^{x+iy} \in R, \\
    z & \text{ otherwise}.
  \end{cases}
\end{equation}
It is immediate that $\phi_\gamma\in \Homeo^+(S^2,Q)$, and one can check that $\phi_\gamma\in \Homeo^+_0(S^2,Q)$ if and only if $\gamma$ is non-essential. We also note that, up to isotopy rel.\ $Q$, the Dehn twist $\phi_\gamma$ depends only on the free homotopy class of $\gamma$ in $S^2\setminus Q$ and is independent of the choice of coordinates on $S^2$. Furthermore, it is known that the \emph{pure mapping class group} of $(S^2,Q)$, defined as the quotient group $\Homeo^+(S^2,Q)/\Homeo^+_0(S^2,Q)$, is generated by finitely many Dehn twists; see \cite[Chs.~3 and 4]{farbPrimerMappingClass2012} for general background.

From now on, we restrict ourselves to the case where the curve $\gamma$ is essential, minimal with respect to the given marked topological Tischler graph $T=(Q,E)$, and also passes through every face of $T$ at most
once. By modifying the chosen coordinates on $S^2$, we may further assume
that $T \cap R$ consists of $m\geq 2$ radial segments $e_k := \{ \exp(x + i\theta_k) : 0 \le x \le 2\pi \}$ with
  arguments $\theta_k := \frac{2\pi k}{m}$ for $k=0,1,\ldots,m-1$. Note that, by our  standing assumptions on $\gamma$, the segments $e_k$ subdivide the annulus $R$ into $m$ sectors, each of which belongs to a separate face of $T$.

For an arbitrary $n \in \Z$, we introduce the \emph{$n/m$-twisted radial segments}
\[
  e_{k,n} := \left\{ \exp\left[ x + i \left(\theta_k + \frac{n}{m}x
      \right) \right] :  0 \le x \le 2\pi \right\},
\]
where $k=0,1,\ldots,m-1$. Note that each $e_{k,n}$ is a closed
arc in $R$ connecting the point with argument $\theta_k$ on the inner boundary of $R$ to the point with argument $\theta_{k+n}$ on the outer boundary of $R$ (with the convention that indices are understood modulo $m$). Furthermore, for every fixed $n$, the $n/m$-twisted radial segments $e_{0,n}, e_{1,n}, \ldots, e_{m-1,n}$ are pairwise disjoint. Since the endpoints of these segments are the same
as those of the radial segments  $e_{0}, e_{1}, \ldots, e_{m-1}$, only connected in a different way, we
can replace the radial segments $e_k$ in $T$ by the $n/m$-twisted
radial segments $e_{k,n}$ to obtain a new plane graph $T' =
(Q,E')$. Following \cite[Def.~5.3]{hlushchankaCriticallyFixedThurston2022}, we
call $T'$ the \emph{$n$-rotation} of the graph $T$ about the curve
$\gamma$; see \Cref{fig:graph-rotation} for an illustration. 

As with Dehn twists, the $n$-rotation $T'$ is well-defined up to isotopy, i.e., it is independent of the choice of coordinates on $S^2$. Using the fact that $\gamma$ passes through every face of $T$ at most once, it is an easy exercise to show that $T'$ is again a
marked topological Tischler graph. We also note that when $n=m\cdot p$ for $p\in \Z$, we have $T'=\phi_\gamma^{p}(T)$, where $\phi_\gamma$ is the Dehn twist about $\gamma$ given by \eqref{eq:Dehn-twist}. In particular, in this case the $n$-rotation $T'$ is isomorphic to the original graph $T$.

\begin{figure}[t]
  \centering
  \begin{tikzpicture}[scale=.8]
    \def\r{1cm}

    \def\s{1.8}

    \def\g{1.3}    

    \def\vi{.5}
    \def\ve{2.2}

    \def\vm{2.6}

    \begin{scope}

    \path [fill=blue!20](0,0) circle (\r*\s);    
    \path [fill=white](0,0) circle (\r);  

      \foreach \name/\angle in {A/0, B/90, C/180, D/270} {
        \coordinate (\name1) at (\angle:\r);
        \coordinate (\name2) at (\angle:\r*\s);
      }
      
      \coordinate (V1) at (45:\vi*\r);
      \coordinate (V2) at (225:\vi*\r);
      \coordinate (V3) at (135:\ve*\r);
      \coordinate (V4) at (315:\ve*\r);
      
      \coordinate (W1) at (120:\vm*\r);
      \coordinate (W2) at (45: \vm*\r);
      \coordinate (W3) at (330:\vm*\r);
      
       \draw[red] circle (\r*\g);
      
      \foreach \name in {A,B,C,D} { \draw (\name1) -- (\name2); }
      
      \draw (V1) -- (V2);
      
      \draw plot [hobby] coordinates {(A1) (V1) (B1)};
      \draw plot [hobby] coordinates {(C1) (V2) (D1)};
      
      \draw plot [hobby] coordinates {(V3) (W1) (W2) (W3) (V4)};
      
      \draw plot [hobby] coordinates {(A2) (V4) (D2)};
      \draw plot [hobby] coordinates {(B2) (V3) (C2)};
      
      \foreach \name in {V1,V2,V3,V4} {\fill (\name) circle (2pt);}
    \end{scope}

    \begin{scope} [xshift=6*\r]

    \path [fill=blue!20](0,0) circle (\r*\s);    
    \path [fill=white](0,0) circle (\r);

      \foreach \name/\angle in {A/0, B/90, C/180, D/270} {
        \coordinate (\name1) at (\angle:\r);
        \coordinate (\name2) at (\angle:\r*\s);
      }
      
      \coordinate (V1) at (45:\vi*\r);
      \coordinate (V2) at (225:\vi*\r);
      \coordinate (V3) at (135:\ve*\r);
      \coordinate (V4) at (315:\ve*\r);
      
      \coordinate (W1) at (120:\vm*\r);
      \coordinate (W2) at (45: \vm*\r);
      \coordinate (W3) at (330:\vm*\r);

      \draw (V1) -- (V2);

      \foreach \angle in {0,90,180,270} {
        \draw[domain=\angle:90+\angle] plot (\x :
        {\r*exp((\x-\angle)/90*ln(\s))}); }
      
      \draw plot [hobby] coordinates {(A1) (V1) (B1)};
      \draw plot [hobby] coordinates {(C1) (V2) (D1)};
      
      \draw plot [hobby] coordinates {(V3) (W1) (W2) (W3) (V4)};
      
      \draw plot [hobby] coordinates {(A2) (V4) (D2)};
      \draw plot [hobby] coordinates {(B2) (V3) (C2)};
      
      \foreach \name in {V1,V2,V3,V4} {\fill (\name) circle (2pt);}
    \end{scope}

    \begin{scope}[xshift=12cm]

    \path [fill=blue!20](0,0) circle (\r*\s);    
    \path [fill=white](0,0) circle (\r);

      \foreach \name/\angle in {A/0, B/90, C/180, D/270} {
        \coordinate (\name1) at (\angle:\r);
        \coordinate (\name2) at (\angle:\r*\s);
      }
      
      \coordinate (V1) at (45:\vi*\r);
      \coordinate (V2) at (225:\vi*\r);
      \coordinate (V3) at (135:\ve*\r);
      \coordinate (V4) at (315:\ve*\r);
      
      \coordinate (W1) at (120:\vm*\r);
      \coordinate (W2) at (45: \vm*\r);
      \coordinate (W3) at (330:\vm*\r);

      \draw (V1) -- (V2);

     \foreach \angle in {0,90,180,270} {
        \draw[domain=\angle:180+\angle] plot (\x :
        {\r*exp((\x-\angle)/180*ln(\s))}); }
      
      \draw plot [hobby] coordinates {(A1) (V1) (B1)};
      \draw plot [hobby] coordinates {(C1) (V2) (D1)};
      
      \draw plot [hobby] coordinates {(V3) (W1) (W2) (W3) (V4)};
      
      \draw plot [hobby] coordinates {(A2) (V4) (D2)};
      \draw plot [hobby] coordinates {(B2) (V3) (C2)};
      
      \foreach \name in {V1,V2,V3,V4} {\fill (\name) circle (2pt);}
    \end{scope}

  \end{tikzpicture}
  \caption{Illustration of graph rotation. The left picture shows the original marked topological Tischler graph $T=(Q,E)$ (in black) and an 
  essential simple closed curve $\gamma\subset S^2\setminus Q$ (in red).  The middle and right pictures illustrate the $1$-rotation and $2$-rotation of $T$ about $\gamma$, respectively. The closed annulus $R\supset \gamma$ used in  the construction is depicted in 
  blue and the resulting graphs are in black. Note that in this example the $1$-rotation of the unobstructed graph $T$
  becomes obstructed. Moreover, the $2$-rotation of $T$ is isomorphic to the original graph.}
  \label{fig:graph-rotation}
\end{figure}

With this preparation, we can now provide an answer to a special
instance of the twisting problem for $(f,Q)$.

\begin{theorem}
  \label{thm:twisting}
  Let $T=(Q,E)$ be a marked topological Tischler graph and $(f,Q)$ be
  an associated marked Schottky map. Suppose $\gamma$ is an essential
  simple closed curve in $S^2 \setminus Q$ that is minimal with
  respect to $T$ and that passes through each face of $T$ at most
  once. Then for all $n\in \Z$, the twist of $(f, Q)$ by the $n$-th
  power of a Dehn twist $\phi_\gamma$ about $\gamma$ is isotopic to a
  marked Schottky map associated with the $n$-rotation of $T$ about
  $\gamma$.

  Furthermore, the twisted maps $(\phi^{n_1}_\gamma\circ f, Q)$ and
  $(\phi^{n_2}_\gamma\circ f, Q)$ are combinatorially equivalent
  whenever $n_1$ equals $n_2$ modulo $m:=\gamma\cdot T$.
\end{theorem}

\begin{proof}
  Below, we will follow the setup and notation used in the definition
  of the $n$-rotation $T'$ of $T$ about~$\gamma$; in particular, we
  assume that we have chosen appropriate coordinates on $S^2$.

  To simplify the argument, we would like the given Schottky map $f$
  to have a simple standard form in the annulus $R$. (We have the
  flexibility to do so since the isotopy class of the twisted map
  $(\phi^n_\gamma\circ f, Q)$ depends only on the isotopy class of
  $(f,Q)$.) To this end, we first define $f$ in $R$ by
  \begin{equation}
    \label{eq:annulus-schottky}
    f(\exp(x+iy)) = \exp(x -i(m-1)y) \qquad \text{for }
    0 \le x \le 2\pi \text{ and } y \in \R.
  \end{equation}
  In other words, the map $f|_{R}$ keeps the modulus of each $z\in R$
  fixed and multiplies the argument by $-(m-1)$. In particular,
  $f|_{R}$ fixes pointwise all the radial segments
  $e_0,e_1,\dots,e_{m-1}$. Furthermore, for each $k=0,1,\dots, m-1$,
  the sector
  \[
    R_k:=\left\{\exp(x+iy): 0 \le x \le 2\pi,\, \theta_k\le y \le
    \theta_{k}+\frac{2\pi}{m}\right\}
  \]
  in $R$ between $e_k$ and $e_{k+1}$ is mapped homeomorphically onto
  $\overline{R\setminus R_k}$. (Here and below, all indices are
  understood modulo $m$.) It is now straightforward to extend $f$ to
  the whole sphere, so that it defines a Schottky map associated with
  $T$. (This easily follows from
  \Cref{thm:alexander}\ref{item:alexander_1} and the fact that, for
  every face $A$ of $T$ meeting $R$, the set $A\setminus R$ has two
  connected components, each being a Jordan domain.)

  Suppose the Dehn twist $\phi_\gamma$ about $\gamma$ is given by \eqref{eq:Dehn-twist}. With our special form of $f$ in $R$, we then have the following claim.

  \begin{claim*}
    For all $n\in\Z$, the map $\phi_\gamma^n\circ f$ sends each edge
    of the $n$-rotation $T'$ of $T$ about~$\gamma$ homeomorphically
    onto itself.
  \end{claim*}

  Indeed, since $f|_T = \id_T$ and
  $\phi_\gamma|_{S^2\setminus R} = \id_{S^2\setminus R}$, it is enough
  to check that $\phi_\gamma^n\circ f$ maps each of the $n/m$-twisted
  radial segments $e_{0,n}, e_{1,n},\dots,e_{m-1,n}$ homeomorphically
  onto itself. Using \eqref{eq:Dehn-twist} and
  \eqref{eq:annulus-schottky}, we find that $\phi_\gamma$ and $f$ map
  twisted radial segments homeomorphically onto twisted radial
  segments as follows:
  \begin{equation*}
    \phi_\gamma(e_{k,n}) = e_{k,n+m}  \quad \text{and} \quad 
    f(e_{k,n}) = e_{k, -n(m-1)},
  \end{equation*}
  where $n\in \Z$ and $k=0,1,\dots,m-1$ are arbitrary. This
  immediately implies that $\phi_\gamma^n\circ f$ maps each $e_{k,n}$
  homeomorphically onto itself. (In fact, with our normalizations,
  $\phi_\gamma^n\circ f$ fixes each $e_{k,n}$ pointwise.)

  The first statement of the theorem now follows from
  \Cref{prop:fixed-graph-implies-schottky}; conditions
  \ref{sch-crit-1}--\ref{sch-crit-2} in the proposition are implied by
  $\phi_\gamma\in \Homeo^+(S^2,Q)$, and conditions
  \ref{sch-crit-3}--\ref{sch-crit-5} follow from the claim above. To
  conclude the second statement from the first one, we apply
  \Cref{prop:equivalent-tischler-graphs}\ref{item:schottky-4} and note
  that the $n_1$-rotation and $n_2$-rotation of $T$ about $\gamma$ are
  isomorphic whenever $n_1-n_2 = m\cdot p$ for $p\in\Z$ (as they
  differ by the $p$-th power of the Dehn twist about $\gamma$). This
  finishes the proof of the theorem.
\end{proof}

\subsection{Hurwitz existence problem for critically fixed anti-rational maps}
\label{subsec:hurwitz}

We start by introducing some terminology and conventions. To each
critically fixed anti-rational map of degree $d \ge 2$ we associate
its \emph{branch data} $(m_1, \ldots, m_r)$, where $m_1 \ge \ldots
\ge m_r$ are the multiplicities of its distinct critical points $c_1,
\ldots, c_r$. We then have $1 \le m_k \le d-1$ for each $k=1,\dots,r$,
and $\sum_{k=1}^r m_k = 2d-2$. We will refer to $d$ as the
\emph{mapping degree} to distinguish it from the vertex degrees of the
corresponding (topological) Tischler graphs or trees, as well as from
the local degrees at critical points. Note that a critically fixed
anti-rational map is Möbius conjugate to an anti-polynomial if and
only if $m_1 = d-1$.

The \emph{Hurwitz existence problem} in our context is the question of which tuples
of positive integers can be realized as the branch data of critically fixed anti-rational
maps. 

\begin{definition}
    We say that a tuple of positive integers $(m_1, \ldots, m_r)$ with $m_1 \ge \ldots \ge m_r$ is \emph{admissible for degree $d \ge
  2$} if $m_k \le d-1$ for all $k$, and $\sum_{k=1}^r m_k = 2d-2$. We say that such a tuple is \emph{polynomial} if $m_1 = d-1$; otherwise we call it \emph{non-polynomial}. 
\end{definition}

Admissible tuples as above are also sometimes called \emph{admissible
  abstract branch data}. In the polynomial case, we will also write an
admissible tuple $(m_1, m_2, \ldots, m_r)$ as
$\langle m_2, \ldots, m_r\rangle$, with $\sum_{k=2}^r m_k = d-1$,
omitting $m_1 = d-1$.

Admissibility is obviously necessary for the realizability of a tuple
$(m_1, m_2, \ldots, m_r)$ as branch data of a critically fixed
anti-rational map. The following theorem shows that it is also
sufficient, establishing \Cref{thm:c} from the introduction.

\begin{theorem}
  \label{thm:hurwitz}
  Let $(m_1, \ldots, m_r)$ be an admissible tuple for degree $d \ge
  2$. Then there exists a critically fixed anti-rational map of degree
  $d$ with branch data $(m_1, \ldots, m_r)$.
\end{theorem}

An analogous statement was established in the orientation-preserving
case in \cite[Thm.~1]{cordwellClassificationCriticallyFixed2015}; see
also a recent survey \cite{PetronioHurwitzProblem2020} for a
discussion and overview of results on the general Hurwitz existence
problem.

\begin{figure}[t]
  \centering
  \begin{tikzpicture}[
    scale=2.5,
    vertex/.style={circle, fill=black, inner sep=1pt},
    edge/.style={thick}
    ]

    \draw[edge] (1,0) -- (5,0);

    \foreach \x 
    in {1,2,3,4,5} 
    {
        \node[vertex, label=below:$\x$] (v\x) at (\x,0) {};
    }

    \def\segmentLength{0.2}

    \foreach \i in {1,2,3,4} {
        \draw[edge] (v2) -- ++({36*\i}:\segmentLength) node[vertex] {};
    }

    \foreach \i in {1,2,3} {
        \draw[edge] (v3) -- ++({45*\i}:\segmentLength) node[vertex] {};
    }

    \draw[edge] (v4) -- ++(90:\segmentLength) node[vertex] {};

  \end{tikzpicture}
  \caption{Example of a reduced topological Tischler tree for the
    polynomial admissible tuple $\langle 4,3,1 \rangle$ for degree
    $d=9$.}
  \label{fig:hurwitz-poly}
\end{figure}

\begin{proof}
  By \Cref{prop:equivalent-tischler-graphs} and
  \Cref{thm:schottky-obstructions}, it is sufficient to construct an
  unobstructed reduced topological Tischler graph $T$ that
  \emph{realizes} the admissible abstract branch data
  $(m_1, \ldots, m_r)$, that is, $T$ has exactly $r$ vertices of
  degrees $m_1+2,\ldots, m_r+2$.
  
  Let us first assume that the tuple $(m_1, \ldots, m_r)$ is
  polynomial, that is, $d=m_1+1=1+\sum_{k=2}^r m_k \ge 2$. We
  construct a reduced topological Tischler tree as follows: We start
  with the path graph of length $r$; say, the interval $[1,r+1]$ with
  vertices placed at the integers $1,\ldots,r+1$. At each interior
  vertex $k=2,\ldots,r$, we attach $m_k$ additional edges, each with
  its own distinct endpoint; for an example, see
  \Cref{fig:hurwitz-poly}. The resulting graph $\Delta$ is a reduced
  topological Tischler tree, and we denote by $T=T_\Delta=(P,E)$ the
  corresponding anti-polynomial reduced topological Tischler graph
  with $P=\{2,\dots, r\}\cup\{\infty\}$, obtained as in
  \Cref{lem:poly-graphs-and-trees}. By construction, we have
  $\deg_T(\infty) = 2+ \sum_{k=2}^r m_k = d + 1= m_1+2$ and
  $\deg_T(k)=m_k+2$ for each $k=2,\ldots,r$, that is, $T$ realizes the
  given tuple $(m_1,\dots, m_r)$.

  \begin{figure}[t]
    \centering
    \begin{tikzpicture}[scale=.5,thick, >=stealth,
      vertex/.style={circle, fill=black,inner sep=0pt, minimum
        size=5pt}, edge/.style={thick}]
      \def\rInner{1.5} %
      \def\rOuter{3.0} %
      \def\N{5}        %

      \begin{scope}
        \draw[red, ultra thick] (72:\rOuter)
        arc [start angle=72, end angle=360, radius=\rOuter]
        -- (0:\rInner)
        arc [start angle=0, end angle=288, radius=\rInner];
        
        \draw[dashed] (0,0) circle (\rInner);
        \draw[dashed] (0,0) circle (\rOuter);
        
        \foreach \i in {1,...,\N} {
          \pgfmathsetmacro{\angle}{(\i - 1) * 360 / \N}
          
          \draw [dashed] (\angle:\rInner) -- (\angle:\rOuter);
        }
        
        \foreach \i in {1,...,\N} {
          \pgfmathsetmacro{\angle}{(\i - 1) * 360 / \N}
          
          \filldraw (\angle:\rInner) circle (2pt);
          \filldraw (\angle:\rOuter) circle (2pt);
        }
      \end{scope}
      \begin{scope}[xshift=8cm]
        \draw[red, ultra thick] (72:\rOuter)
        arc [start angle=72, end angle=288, radius=\rOuter];

        \draw[red, ultra thick]
        (0:\rOuter) -- (0:\rInner)
        arc [start angle=0, end angle=72, radius=\rInner];
        
        \draw[red, ultra thick] (144:\rInner)
        arc [start angle=144, end angle=216, radius=\rInner];

        \draw[red, ultra thick]
        (288:\rInner) circle (2pt);
        
        \draw[dashed] (0,0) circle (\rInner);
        \draw[dashed] (0,0) circle (\rOuter);
        
        \foreach \i in {1,...,\N} {
          \pgfmathsetmacro{\angle}{(\i - 1) * 360 / \N}
          
          \draw [dashed] (\angle:\rInner) -- (\angle:\rOuter);
        }
      
        \foreach \i in {1,...,\N} {
          \pgfmathsetmacro{\angle}{(\i - 1) * 360 / \N}
          
          \filldraw (\angle:\rInner) circle (2pt);
          \filldraw (\angle:\rOuter) circle (2pt);
        }
      \end{scope}        
      \begin{scope}[xshift=16cm,scale=2]
        \node[vertex, red] (A) at (-1.5, 0) {};
        \node[vertex, red] (B) at (-0.5, 0) {};
        \node[vertex, red] (C) at (0.5, 0) {};
        \node[vertex, red] (D) at (1.5, 0) {};

        \draw[edge] (1.5,0) arc (0:180:1.5);  %
        \draw[edge] (1.5,0) arc (0:-180:1.5); %

        \draw[edge] (A) -- (B);
        \draw[edge] (B) -- (C);
        \draw[edge] (C) -- (D);

        \draw[edge] (A) .. controls (-1.0, -0.6) and (0.0, -0.6) .. (C);
        
        \draw[edge] (A) .. controls (-1.25, 0.25) and (-0.75, 0.25) .. (B);

        \draw[edge] (B) .. controls (0.0, 0.4) and (1.0, 0.4) .. (D);

        \draw[edge] (A) .. controls (-1.0, 1.0) and (1.0, 1.0) .. (D);
        
      \end{scope}
    \end{tikzpicture}
      
    \caption{Construction of an unobstructed reduced topological
      Tischler graph for the non-polynomial admissible tuple
      $(4,3,2,1)$ for degree $d=6$. The prism graph $\Pi_6$ (in black,
      dashed) with a Hamiltonian simple path $\alpha$ (in solid red)
      is depicted on the left. The middle picture shows $\Pi_6$ with
      the subpaths $\alpha_k$ (also in solid red) to be collapsed. The
      resulting unobstructed reduced topological Tischler graph $T$,
      obtained after collapsing, is illustrated on the right.  }
    \label{fig:prism-collapse}
  \end{figure}

  Let us now assume that we are given non-polynomial admissible
  abstract branch data $(m_1, \ldots, m_r)$. This forces $d \ge 3$,
  since the only admissible tuple for $d=2$ is the polynomial tuple
  $(1,1)$. For $d=3$ the unique admissible non-polynomial tuple is
  $(1,1,1,1)$, realized by the tetrahedral map
  $f_1(z) = 3\bar{z}^2 / (2\bar{z}^3+1)$; see
  \Cref{fig:Julia_sets,fig:tetr_map}.  So from now on let us assume
  $d \ge 4$.
  
  We start with the degree $d$ \emph{prism graph} $\Pi_d$, given by
  two concentric circles in $\R^2$ joined by $d-1$ radial segments;
  see \Cref{fig:prism-collapse} for an illustration. The graph $\Pi_d$
  has $2d-2$ vertices, each of vertex degree $3$. Moreover, all faces
  of $\Pi_d$ are Jordan domains and, since $d \ge 4$, every two
  adjacent faces share exactly one boundary edge, so $\Pi_d$ is an
  unobstructed reduced topological Tischler graph. The corresponding
  critically fixed anti-rational map has $2d-2$ simple critical
  points, and realizes the tuple $(1,1,\ldots,1)$. For a general
  non-polynomial admissible tuple $(m_1, \ldots, m_r)$ for degree $d$,
  we will collapse certain paths in $\Pi_d$ to points to obtain the
  graph $T$.

  Define a Hamiltonian simple path $\alpha$ in $\Pi_d$ as follows:
  First, traverse all vertices of $\Pi_d$ on the larger circle
  counterclockwise, then trace a radial segment to a vertex on the
  smaller circle, and finally traverse all vertices on the smaller
  circle counterclockwise. Now consider consecutive pairwise disjoint
  subpaths $\alpha_1, \ldots, \alpha_r$ of $\alpha$ such that each
  $\alpha_k$ contains exactly $m_k$ vertices. Since
  $\sum_{k=1}^r m_k=2d-2$, these subpaths cover all vertices of
  $\alpha$ (and $\Pi_d$). Let $T$ be a plane graph obtained by
  collapsing each subpath $\alpha_k$ to a point; see
  \Cref{fig:prism-collapse} for an example of the whole construction.

  Technically, collapsing the subpaths $\alpha_k$ in
  $\Pi_d\subset S^2$ is achieved by viewing $T= \Pi_d/\sim$ as a plane
  graph in the topological quotient space $S^2/\sim$, where the
  non-trivial equivalence classes of $\sim$ are exactly the paths
  $\alpha_k$. Note that, since $\sim$ is a closed equivalence relation
  on $S^2$ whose equivalence classes are connected and do not separate
  the sphere, the quotient $S^2/\sim$ is again a topological 2-sphere
  by Moore's theorem \cite{Moore1925}; see also
  \cite{TimorinMoore2010}. Moreover, the quotient $T=\Pi_d/\sim$ is
  easily seen to be a plane graph in $S^2/\sim$ with the vertex set
  $P:=\{v_k=\alpha_k/\sim : k=1,\dots,r\}$.
  
  We claim that $T$ is an unobstructed reduced topological Tischler
  graph representing the tuple $(m_1,\dots, m_r)$. Indeed, by
  construction, each vertex $v_k$ of $T$ has degree
  $3m_k-2(m_k-1)=m_k+2\geq 3$. Note also that the complement
  $(S^2/\sim) \setminus T$ is canonically homeomorphic to
  $S^2 \setminus \Pi_d$, because it does not intersect any non-trivial
  equivalence class of $\sim$. We claim that the faces of $T$ are
  Jordan domains, which is the one point in this proof where we have
  to use the fact that our abstract branch data is non-polynomial,
  i.e., that $m_k \le d-2$ for all $k$. The crucial property of
  $\alpha$ to observe here is that any subpath of length $\le d-3$
  intersects the boundary of each face $A$ of $\Pi_d$ in a connected
  proper subset. This shows that the quotient of the Jordan curve
  boundary $\partial A$ is again a Jordan curve, so that every face of
  $T$ is a Jordan domain, and therefore $T$ is a reduced topological
  Tischler graph. It remains to check that $T$ is
  unobstructed. Indeed, faces and edges of $T$ correspond to faces and
  uncollapsed edges of the original prism graph $\Pi_d$. Since
  adjacent faces of $\Pi_d$ share exactly one edge, the same must hold
  in $T$. This finishes the proof.
\end{proof}

\subsection{Census of critically fixed anti-rational maps of low
  degrees}
\label{subsec:census}

Throughout this subsection, all classifications and statements of
uniqueness are up to isomorphism for reduced topological Tischler
graphs (or trees) and are up to Möbius conjugacy for critically fixed
anti-rational maps (or up to affine conjugacy, in the case of
anti-polynomials). Moreover, reduced topological Tischler graphs and
trees are drawn to realize all possible symmetries. In particular, if
possible, they are drawn $\R$-symmetrically, which means that the
corresponding anti-rational or anti-polynomial map $f$ can be chosen
to have real coefficients. If, furthermore, all critical points of $f$
can be realized on the real line, the critical vertices of the
corresponding reduced topological Tischler graph or tree $T$ are
arranged on the real line as well. In that case, the
map $f$ is the complex conjugate of a critically fixed rational or
polynomial map. (Recall that a vertex of $T$ is called
\emph{critical} if it has degree $\geq 3$. Note that every vertex of a
reduced topological Tischler graph is critical, and so is every
non-leaf vertex of a reduced topological Tischler tree.)

\subsubsection{Critically fixed anti-polynomials}
\label{sec:crit-fixed-anti}

Instead of organizing our census of critically fixed anti-polynomials
by the mapping degree $d\geq 2$, it is easier to organize it by the
number $r$ of distinct finite critical points, which is also the
number of critical vertices in the corresponding reduced (topological)
Tischler trees, as well as the total number of vertices in the
corresponding Hubbard trees. Recall that we write polynomial branch
data as $\langle m_1, \ldots, m_r \rangle$, where $m_1 \ge m_2 \ge
\ldots \ge m_r$ are the multiplicities of the finite critical points
$c_1, \ldots, c_r$ with $\sum_{k=1}^r m_k= d-1$.

For $r \le 2$, we give a list of all possible critically fixed
anti-polynomials. For $r\ge 3$, we instead list all possible reduced
topological Tischler trees, with some remarks on the corresponding
anti-polynomials. A few of these anti-polynomials have been explicitly
calculated using ad hoc methods and computer algebra in
\cite{bshoutyCrofootSarasonsConjectureHarmonic2004}, but with our
combinatorial description it should be easier to systematically find
them numerically.

In our sketches of reduced topological Tischler trees, critical
vertices are represented by thick black dots, and leaves are evenly
distributed on a circle shown in dashed lines. This reflects the fact
that the leaves of these trees correspond to the landing points of the
fixed internal rays of the immediate basin of $\infty$. Additionally,
the edges of the Hubbard tree (as a subtree of the reduced Tischler
tree) are drawn slightly thicker for emphasis. Note that we use
\emph{angled} Hubbard trees introduced in
\cite{poirierHubbardTrees2010}, which are the usual Hubbard trees
together with the extra data about the \emph{turn angles} between any
two edges incident with the same critical vertex. (Such a turn angle
coincides with the angle between the corresponding internal rays.)

\smallskip

\textbf{Case} $\bm{r=1}$. In this case, the only possible branch
data is $\langle d-1 \rangle$, and the unique corresponding critically
fixed anti-polynomial is $f(z) = \bar{z}^d$. The reduced topological Tischler tree of such $f$ is a \emph{star}; see~\Cref{fig:tischler-trees-r-1}.

\begin{figure}[ht]
  \centering
  \begin{tikzpicture}
    \begin{scope}
      \coordinate (A) at (0,0);
      \coordinate (B) at (0:1);
      \coordinate (C) at (120:1);
      \coordinate (D) at (240:1);
      \draw (A) -- (B);
      \draw (A) -- (C);
      \draw (A) -- (D);
      \fill (A) circle [radius=1mm];
      \draw[dashed] (0,0) circle [radius=1];
      \node at (0,-1.4) {$\langle 1 \rangle$};
    \end{scope}
    \begin{scope}[xshift=3cm]
      \coordinate (A) at (0,0);
      \coordinate (B) at (0:1);
      \coordinate (C) at (90:1);
      \coordinate (D) at (180:1);
      \coordinate (E) at (270:1);
      \draw (A) -- (B);
      \draw (A) -- (C);
      \draw (A) -- (D);
      \draw (A) -- (E);
      \fill (A) circle [radius=1mm];
      \draw[dashed] (0,0) circle [radius=1];
      \node at (0,-1.4) {$\langle 2 \rangle$};
    \end{scope}
    \begin{scope}[xshift=6cm]
      \coordinate (A) at (0,0);
      \coordinate (B) at (0:1);
      \coordinate (C) at (72:1);
      \coordinate (D) at (144:1);
      \coordinate (E) at (-72:1);
      \coordinate (F) at (-144:1);
      \draw (A) -- (B);
      \draw (A) -- (C);
      \draw (A) -- (D);
      \draw (A) -- (E);
      \draw (A) -- (F);
      \fill (A) circle [radius=1mm];
      \draw[dashed] (0,0) circle [radius=1];
      \node at (0,-1.4) {$\langle 3 \rangle$};
    \end{scope}
    \begin{scope}[xshift=9cm]
      \coordinate (A) at (0,0);
      \coordinate (B) at (0:1);
      \coordinate (C) at (60:1);
      \coordinate (D) at (120:1);
      \coordinate (E) at (180:1);
      \coordinate (F) at (240:1);
      \coordinate (G) at (300:1);
      \draw (A) -- (B);
      \draw (A) -- (C);
      \draw (A) -- (D);
      \draw (A) -- (E);
      \draw (A) -- (F);
      \draw (A) -- (G);
      \fill (A) circle [radius=1mm];
      \draw[dashed] (0,0) circle [radius=1];
      \node at (0,-1.4) {$\langle 4 \rangle$};      
    \end{scope}
  \end{tikzpicture}
  \caption{Reduced topological Tischler trees with one critical vertex
    up to mapping degree~5, labeled by the corresponding branch data.}
\label{fig:tischler-trees-r-1}  
\end{figure}

\textbf{Case} $\bm{r=2}$.  We can always normalize such maps to have
critical points $c_1=0$ and $c_2=1$ with branch data $\langle m_1,
m_2 \rangle$, where $d= m_1 + m_2 +1$.  For every such pair of
multiplicities there is a unique associated reduced topological
Tischler tree, with vertex degrees $m_1 + 2$ at $c_1$ and $m_2 + 2$ at
$c_2$, and with a unique edge connecting the critical vertices $c_1$
and $c_2$. The corresponding anti-polynomials are the complex
conjugates of the critically fixed polynomials $g_{m_1, m_2}$ with the
same critical points and branch data, given by the regularized beta
function:
\begin{equation*}
  g_{m_1,m_2}(z) = \frac{B(z; \, m_1+1,m_2+1)}{B(m_1+1,m_2+1)} =
  \frac{(m_1+m_2+1)!}{m_1!\, m_2!} \int_0^z \zeta^{m_1} (1-\zeta)^{m_2}
  \, d\zeta.
\end{equation*}
Below, we explicitly list all such polynomials up to degree $5$:
\begin{align*}
  g_{1,1}(z) & = 3z^2-2z^3;   & g_{2,1}(z) & = 4z^3-3z^4; \\
  g_{3,1}(z) & = 5z^4 - 4z^5; & g_{2,2}(z) & = 10z^3 - 15z^4 + 6z^5;
\end{align*}
see also \Cref{fig:tischler-trees-r-2} for the corresponding reduced topological Tischler trees.

\begin{figure}[ht]
  \centering
  \begin{tikzpicture}
    \begin{scope}
      \coordinate (A1) at (-.4,0);
      \coordinate (A2) at (.4,0);
      \coordinate (B) at (45:1);
      \coordinate (C) at (-45:1);
      \coordinate (D) at (135:1);
      \coordinate (E) at (-135:1);
      \draw[very thick] (A1) -- (A2);
      \draw (A2) -- (B);
      \draw (A2) -- (C);
      \draw (A1) -- (D);
      \draw (A1) -- (E);
      \fill (A1) circle [radius=1mm];
      \fill (A2) circle [radius=1mm];
      \draw[dashed] (0,0) circle [radius=1];
      \node at (0,-1.4) {$\langle 1,1 \rangle$};
    \end{scope}
    \begin{scope}[xshift=3cm]
      \coordinate (A1) at (-0.2,0);
      \coordinate (A2) at (0.5,0);
      \coordinate (B) at (36:1);
      \coordinate (C) at (108:1);
      \coordinate (D) at (180:1);
      \coordinate (E) at (-108:1);
      \coordinate (F) at (-36:1);
      \draw[very thick] (A1) -- (A2);
      \draw (A2) -- (B);
      \draw (A2) -- (F);
      \draw (A1) -- (C);
      \draw (A1) -- (D);
      \draw (A1) -- (E);
      \fill (A1) circle [radius=1mm];
      \fill (A2) circle [radius=1mm];
      \draw[dashed] (0,0) circle [radius=1];
      \node at (0,-1.4) {$\langle 2,1 \rangle$};
    \end{scope}
    \begin{scope}[xshift=6cm]
      \coordinate (A1) at (0.5,0);
      \coordinate (A2) at (-0.2,0);
      \coordinate (B) at (30:1);
      \coordinate (C) at (-30:1);
      \coordinate (D) at (90:1);
      \coordinate (E) at (-90:1);
      \coordinate (F) at (150:1);
      \coordinate (G) at (-150:1);
      \draw[very thick] (A1) -- (A2);
      \draw (A1) -- (B);
      \draw (A1) -- (C);
      \draw (A2) -- (D);
      \draw (A2) -- (E);
      \draw (A2) -- (F);
      \draw (A2) -- (G);
      \fill (A1) circle [radius=1mm];
      \fill (A2) circle [radius=1mm];
      \draw[dashed] (0,0) circle [radius=1];
      \node at (0,-1.4) {$\langle 3,1 \rangle$};
    \end{scope}
    \begin{scope}[xshift=9cm]
      \coordinate (A1) at (0.4,0);
      \coordinate (A2) at (-0.4,0);
      \coordinate (B) at (0:1);
      \coordinate (C) at (60:1);
      \coordinate (D) at (120:1);
      \coordinate (E) at (180:1);
      \coordinate (F) at (240:1);
      \coordinate (G) at (300:1);
      \draw[very thick] (A1) -- (A2);
      \draw (A1) -- (B);
      \draw (A1) -- (C);
      \draw (A1) -- (G);
      \draw (A2) -- (D);
      \draw (A2) -- (E);
      \draw (A2) -- (F);
      \fill (A1) circle [radius=1mm];
      \fill (A2) circle [radius=1mm];
      \draw[dashed] (0,0) circle [radius=1];
      \node at (0,-1.4) {$\langle 2,2 \rangle$};
    \end{scope}
  \end{tikzpicture}
  \caption{Reduced topological Tischler trees with two critical
    vertices up to mapping degree 5, labeled by the corresponding
    branch data.}
\label{fig:tischler-trees-r-2}      
\end{figure}

\textbf{Case} $\bm{r=3}$.
In this case, we have branch data $\langle m_1,m_2,m_3 \rangle$
with $d=m_1+m_2+m_3+1$. There is only one possible Hubbard tree (up to isomorphism), with one vertex of degree 2, and two vertices of
degree 1. However, unless $m_1=m_2=m_3$, there are different ways to
assign these multiplicities to the vertices of the Hubbard tree, giving
inequivalent maps. Furthermore, in general there are different choices
for the ``turn angle'' at the vertex of degree 2. The different
inequivalent reduced topological Tischler trees encode all these
different choices.

For branch data $\langle 1,1,1 \rangle$ there is a unique associated
reduced topological Tischler tree (see \Cref{fig:tischler-trees-r-3}), and a unique corresponding critically fixed anti-polynomial of degree 4 with real coefficients. This anti-polynomial has one real and two complex
conjugate critical points.

For branch data $\langle 2,1,1 \rangle$ there are four associated
reduced topological Tischler trees (see
\Cref{fig:tischler-trees-r-3}), and correspondingly four different
critically fixed anti-polynomials. Two of them have the double
critical point $c_1$ at the vertex of degree 2 in the Hubbard
tree. For one of these, the two edges in the Hubbard tree are not
consecutive in the cyclic order of edges at $c_1$ in the corresponding
reduced topological Tischler tree. This corresponds to a turn angle of
$180^\circ$ between them in the angled Hubbard tree. In the other one,
they are consecutive, corresponding to a $90^\circ$ turn angle between
them. The two remaining Tischler trees have the double critical point
$c_1$ as a leaf in the Hubbard tree. This means that the turn angle at
the vertex of degree $2$ in the Hubbard tree is $120^\circ$. However,
because the Hubbard tree is not symmetric with respect to
interchanging its leaves (since one leaf is a simple and the other one
is a double critical point), the two different orientations of this
turn angle give inequivalent reduced topological Tischler trees. These
are the first examples in this list that cannot be realized by
anti-polynomials with real coefficients.

\begin{figure}[ht]
  \centering
  \begin{tikzpicture}
    \begin{scope}
      \coordinate (A1) at (0.3,0);
      \coordinate (A2) at (-0.2,0.4);
      \coordinate (A3) at (-0.2,-0.4);
      \coordinate (B) at (0:1);
      \coordinate (C) at (72:1);
      \coordinate (D) at (144:1);
      \coordinate (E) at (-72:1);
      \coordinate (F) at (-144:1);
      \draw[very thick] (A1) -- (A2);
      \draw[very thick] (A1) -- (A3);
      \draw (A1) -- (B);
      \draw (A2) -- (C);
      \draw (A2) -- (D);
      \draw (A3) -- (E);
      \draw (A3) -- (F);
      \fill (A1) circle [radius=1mm];
      \fill (A2) circle [radius=1mm];
      \fill (A3) circle [radius=1mm];
      \draw[dashed] (0,0) circle [radius=1];
      \node at (0,-1.4) {$\langle 1,1,1 \rangle$};
    \end{scope}
    \begin{scope}[xshift=3cm]
      \coordinate (A1) at (0.5,0);
      \coordinate (A2) at (0,0);
      \coordinate (A3) at (-0.5,0);
      \coordinate (B) at (30:1);
      \coordinate (C) at (-30:1);
      \coordinate (D) at (90:1);
      \coordinate (E) at (-90:1);
      \coordinate (F) at (150:1);
      \coordinate (G) at (-150:1);
      \draw[very thick] (A1) -- (A2) -- (A3);
      \draw (A1) -- (B);
      \draw (A1) -- (C);
      \draw (A2) -- (D);
      \draw (A2) -- (E);
      \draw (A3) -- (F);
      \draw (A3) -- (G);
      \fill (A1) circle [radius=1mm];
      \fill (A2) circle [radius=1mm];
      \fill (A3) circle [radius=1mm];
      \draw[dashed] (0,0) circle [radius=1];
      \node at (0,-1.4) {$\langle 2,1,1 \rangle$};
    \end{scope}
    \begin{scope}[xshift=6cm]
      \coordinate (A1) at (-0.2,0.4);
      \coordinate (A2) at (0.25,0);
      \coordinate (A3) at (-0.2,-0.4);
      \coordinate (B) at (30:1);
      \coordinate (C) at (90:1);
      \coordinate (D) at (150:1);
      \coordinate (E) at (-150:1);
      \coordinate (F) at (-90:1);
      \coordinate (G) at (-30:1);
      \draw[very thick] (A1) -- (A2) -- (A3);
      \draw (A1) -- (C);
      \draw (A1) -- (D);
      \draw (A2) -- (B);
      \draw (A2) -- (G);
      \draw (A3) -- (E);
      \draw (A3) -- (F);
      \fill (A1) circle [radius=1mm];
      \fill (A2) circle [radius=1mm];
      \fill (A3) circle [radius=1mm];
      \draw[dashed] (0,0) circle [radius=1];
      \node at (0,-1.4) {$\langle 2,1,1 \rangle$};
    \end{scope}
    \begin{scope}[xshift=9cm]
      \coordinate (A1) at (0.5,0);
      \coordinate (A2) at (-0.1,0);
      \coordinate (A3) at (-0.5,-0.3);
      \coordinate (B) at (0:1);
      \coordinate (C) at (60:1);
      \coordinate (D) at (120:1);
      \coordinate (E) at (180:1);
      \coordinate (F) at (240:1);
      \coordinate (G) at (300:1);
      \draw[very thick] (A1) -- (A2) -- (A3);
      \draw (A1) -- (B);
      \draw (A1) -- (C);
      \draw (A1) -- (G);
      \draw (A2) -- (D);
      \draw (A3) -- (E);
      \draw (A3) -- (F);
      \fill (A1) circle [radius=1mm];
      \fill (A2) circle [radius=1mm];
      \fill (A3) circle [radius=1mm];
      \draw[dashed] (0,0) circle [radius=1];
      \node at (0,-1.4) {$\langle 2,1,1 \rangle$};
    \end{scope}
    \begin{scope}[xshift=12cm]
      \coordinate (A1) at (0.5,0);
      \coordinate (A2) at (-0.1,0);
      \coordinate (A3) at (-0.5,0.3);
      \coordinate (B) at (0:1);
      \coordinate (C) at (60:1);
      \coordinate (D) at (120:1);
      \coordinate (E) at (180:1);
      \coordinate (F) at (240:1);
      \coordinate (G) at (300:1);
      \draw[very thick] (A1) -- (A2) -- (A3);
      \draw (A1) -- (B);
      \draw (A1) -- (C);
      \draw (A1) -- (G);
      \draw (A2) -- (F);
      \draw (A3) -- (D);
      \draw (A3) -- (E);
      \fill (A1) circle [radius=1mm];
      \fill (A2) circle [radius=1mm];
      \fill (A3) circle [radius=1mm];
      \draw[dashed] (0,0) circle [radius=1];
      \node at (0,-1.4) {$\langle 2,1,1 \rangle$};
    \end{scope}
  \end{tikzpicture}
  \caption{Reduced topological Tischler trees with three critical
    vertices up to mapping degree 5, labeled by the corresponding
    branch data. The first three are $\R$-symmetric, and the last two
    are mirror images under complex conjugation.}
    \label{fig:tischler-trees-r-3}    
\end{figure}

\textbf{Case} $\bm{r=4}$.  For branch data $\langle 1,1,1,1 \rangle$
we have four different reduced topological Tischler trees; see
\Cref{fig:tischler-trees-r-4} for an illustration. One of them induces
a Hubbard tree with 3 leaves and an extra vertex of degree 3, while
the other three induce Hubbard trees with 2 leaves and two extra
vertices of degree 2. The latter three (angled) Hubbard trees differ
only by the turn angles at the extra vertices. Even though these three
Hubbard trees are symmetric with respect to switching their leaves,
the two trees with opposite $120^\circ$ turn angles at the extra
vertices are not equivalent by an orientation-preserving
homeomorphism. This again leads to two critically fixed
anti-polynomials, which are conjugate by $\iota(z)=\bar{z}$ but not
by an affine map.

\begin{figure}[ht]
  \begin{tikzpicture}
    \begin{scope}
      \coordinate (A0) at (0,0);
      \coordinate (A1) at (0:0.58);
      \coordinate (A2) at (120:0.58);
      \coordinate (A3) at (240:0.58);
      \coordinate (B) at (30:1);
      \coordinate (C) at (-30:1);
      \coordinate (D) at (90:1);
      \coordinate (E) at (-90:1);
      \coordinate (F) at (150:1);
      \coordinate (G) at (-150:1);
      \draw[very thick] (A2) -- (A0) -- (A1);
      \draw[very thick] (A0) -- (A3);
      \draw (A1) -- (B);
      \draw (A1) -- (C);
      \draw (A2) -- (D);
      \draw (A3) -- (E);
      \draw (A2) -- (F);
      \draw (A3) -- (G);
      \fill (A0) circle [radius=1mm];
      \fill (A1) circle [radius=1mm];
      \fill (A2) circle [radius=1mm];
      \fill (A3) circle [radius=1mm];
      \draw[dashed] (0,0) circle [radius=1];
      \node at (0,-1.4) {$\langle 1,1,1,1 \rangle$};
    \end{scope}
    \begin{scope}[xshift=3cm]
      \coordinate (A0) at (0.3,0.5);
      \coordinate (A1) at (-0.2,0.2);
      \coordinate (A2) at (-0.2,-0.2);
      \coordinate (A3) at (0.3,-0.5);
      \coordinate (B) at (30:1);
      \coordinate (C) at (-30:1);
      \coordinate (D) at (90:1);
      \coordinate (E) at (-90:1);
      \coordinate (F) at (150:1);
      \coordinate (G) at (-150:1);
      \draw[very thick] (A0) -- (A1) -- (A2) -- (A3);
      \draw (A0) -- (B);
      \draw (A3) -- (C);
      \draw (A0) -- (D);
      \draw (A3) -- (E);
      \draw (A1) -- (F);
      \draw (A2) -- (G);
      \fill (A0) circle [radius=1mm];
      \fill (A1) circle [radius=1mm];
      \fill (A2) circle [radius=1mm];
      \fill (A3) circle [radius=1mm];
      \draw[dashed] (0,0) circle [radius=1];
      \node at (0,-1.4) {$\langle 1,1,1,1 \rangle$};
    \end{scope}
    \begin{scope}[xshift=6cm]
      \coordinate (A0) at (0.2,0.5);
      \coordinate (A1) at (-0.2,0.2);
      \coordinate (A2) at (0.2,-0.2);
      \coordinate (A3) at (-0.2,-0.5);
      \coordinate (B) at (30:1);
      \coordinate (C) at (-30:1);
      \coordinate (D) at (90:1);
      \coordinate (E) at (-90:1);
      \coordinate (F) at (150:1);
      \coordinate (G) at (-150:1);
      \draw[very thick] (A0) -- (A1) -- (A2) -- (A3);
      \draw (A0) -- (B);
      \draw (A2) -- (C);
      \draw (A0) -- (D);
      \draw (A3) -- (E);
      \draw (A1) -- (F);
      \draw (A3) -- (G);
      \fill (A0) circle [radius=1mm];
      \fill (A1) circle [radius=1mm];
      \fill (A2) circle [radius=1mm];
      \fill (A3) circle [radius=1mm];
      \draw[dashed] (0,0) circle [radius=1];
      \node at (0,-1.4) {$\langle 1,1,1,1 \rangle$};
    \end{scope}
    \begin{scope}[xshift=9cm, xscale=-1]
      \coordinate (A0) at (0.2,0.5);
      \coordinate (A1) at (-0.2,0.2);
      \coordinate (A2) at (0.2,-0.2);
      \coordinate (A3) at (-0.2,-0.5);
      \coordinate (B) at (30:1);
      \coordinate (C) at (-30:1);
      \coordinate (D) at (90:1);
      \coordinate (E) at (-90:1);
      \coordinate (F) at (150:1);
      \coordinate (G) at (-150:1);
      \draw[very thick] (A0) -- (A1) -- (A2) -- (A3);
      \draw (A0) -- (B);
      \draw (A2) -- (C);
      \draw (A0) -- (D);
      \draw (A3) -- (E);
      \draw (A1) -- (F);
      \draw (A3) -- (G);
      \fill (A0) circle [radius=1mm];
      \fill (A1) circle [radius=1mm];
      \fill (A2) circle [radius=1mm];
      \fill (A3) circle [radius=1mm];
      \draw[dashed] (0,0) circle [radius=1];
      \node at (0,-1.4) {$\langle 1,1,1,1 \rangle$};
    \end{scope}
  \end{tikzpicture}
  \caption{Reduced topological Tischler trees with four critical
    vertices up to mapping degree 5, labeled by the corresponding
    branch data. The first two are $\R$-symmetric, and the last two
    are mirror images under complex conjugation.}
    \label{fig:tischler-trees-r-4}     
\end{figure}

\subsubsection{Critically fixed non-polynomial anti-rational maps}
\label{sec:critically-fixed-non}

Let us fix a mapping degree $d\geq 3$. To compile a census of the
reduced topological Tischler graphs for critically fixed
anti-rational maps of degree $d$ that are not equivalent to
anti-polynomials, we have to consider branch data $(m_1, \ldots,
m_r)$ with $d-2 \ge m_1 \ge m_2 \ge \ldots \ge m_r \ge 1$ and
$\sum_{k=1}^r m_k = 2d-2$. Below, we list all such possible graphs for
mapping degrees 3 and 4. Similar to anti-polynomials, this
combinatorial census should make it possible to calculate the
corresponding maps numerically.

\textbf{Case} $\bm{d=3}$.  In this case, the only possible
non-polynomial branch data is $(1,1,1,1)$, and the only
unobstructed reduced topological Tischler graph is the complete
(simple) plane graph on 4 vertices. The corresponding critically fixed
anti-rational map is the \emph{tetrahedral map}
\begin{equation}\label{eq:tetr_map}
      f(z) = \frac{3\bar{z}^2}{2\bar{z}^3+1}.
\end{equation}

\begin{figure}[ht]
  \centering
  \begin{tikzpicture}
    \begin{scope}
      \coordinate (A) at (0,0);
      \coordinate (B) at (0:1);
      \coordinate (C) at (120:1);
      \coordinate (D) at (240:1);
      \draw (A) -- (B);
      \draw (C) -- (A) -- (D);
      \draw (A) circle [radius=1];
      \fill (A) circle [radius=1mm];
      \fill (B) circle [radius=1mm];
      \fill (C) circle [radius=1mm];
      \fill (D) circle [radius=1mm];
      \node at (0,-1.4) {$(1,1,1,1)$};
    \end{scope}
  \end{tikzpicture}
  \caption{Unique non-anti-polynomial unobstructed reduced topological
    Tischler graph for mapping degree $d=3$, labeled by the
    corresponding branch data. It is realized by the tetrahedral map
    \eqref{eq:tetr_map}.}
  \label{fig:tetr_map}
\end{figure}

\textbf{Case} $\bm{d=4}$.
Figure~\ref{fig:tischler-graphs-deg-4} shows all the unobstructed
reduced topological Tischler graphs for mapping degree 4 that do not
come from anti-polynomials. Note that the critically fixed anti-rational maps with
branch data $(2,2,2)$ and $(2,2,1,1)$ that correspond to the first two graphs in the figure are the complex
conjugates of certain critically fixed rational maps, since these graphs are
$\R$-symmetric and have all their critical vertices arranged on the
real line. Möbius conjugates of these two critically fixed rational
maps are explicitly calculated in
\cite[Sec.~11]{cordwellClassificationCriticallyFixed2015}. One can also calculate (explicitly or numerically) critically fixed anti-rational maps for the other graphs. For example, we calculated the critically fixed anti-rational map that corresponds to the third graph in \Cref{fig:tischler-graphs-deg-4} so that it has two critical points $\pm1$ on the real axis, and the other two critical points on the imaginary axis. Namely, this map is given by
\begin{equation}\label{eq:diam_map}
  f(z)=\frac{(4-2p)\,\bar{z}^3+(4+2p)\,\bar{z}}{(1-p)\,\bar{z}^{4}+6\,\bar{z}^{2}+(1+p)},  
\end{equation}
where $p\approx-1.5476618$ is the only real root of the cubic polynomial $2p^{3}-p^{2}+4p+16=0$. Numerically,
$$f(z)\approx \frac{7.095324\,\bar{z}^{3}+0.904676\,\bar{z}}{2.547662\,\bar{z}^{4}+6\bar{z}^{2}-0.547662}.$$

All but two of the graphs in Figure~\ref{fig:tischler-graphs-deg-4}
are $\R$-symmetric (namely, all except for the two Z-shaped graphs
with branch data $(2,2,1,1)$), so they correspond to critically
fixed anti-rational maps with real coefficients. The unique graph corresponding to the 
branch data $(1,1,1,1,1,1)$ is the prism graph $\Pi_4$.

\begin{figure}[ht]
  \centering
  \begin{tikzpicture}
    \begin{scope}
      \coordinate (A) at (0,0);
      \coordinate (B) at (1,0);
      \coordinate (C) at (-1,0);
      \draw (A) to [bend left=60] (B);
      \draw (A) to [bend right=60] (B);
      \draw (A) to [bend left=60] (C);
      \draw (A) to [bend right=60] (C);
      \draw (A) circle [radius=1];
      \fill (A) circle [radius=1mm];
      \fill (B) circle [radius=1mm];
      \fill (C) circle [radius=1mm];
      \node at (0,-1.4) {$(2,2,2)$};
    \end{scope}
    \begin{scope}[xshift=3cm]
      \coordinate (A) at (0,0);
      \coordinate (B) at (-1,0);
      \coordinate (C) at (-0.3,0);
      \coordinate (D) at (0.3,0);
      \coordinate (E) at (1,0);
      \draw (B) to [bend left=90] (C);
      \draw (B) to [bend right=90] (C);
      \draw (D) to [bend left=90] (E);
      \draw (D) to [bend right=90] (E);
      \draw (C) -- (D);
      \draw (A) circle [radius=1];
      \fill (B) circle [radius=1mm];
      \fill (C) circle [radius=1mm];
      \fill (D) circle [radius=1mm];
      \fill (E) circle [radius=1mm];
      \node at (0,-1.4) {$(2,2,1,1)$};
    \end{scope}
    \begin{scope}[xshift=6cm]
      \coordinate (A) at (0,0);
      \coordinate (B) at (-1,0);
      \coordinate (C) at (0,.4);
      \coordinate (D) at (0,-.4);
      \coordinate (E) at (1,0);
      \draw (B) -- (C) -- (E) -- (D) -- (B);
      \draw (C) -- (D);
      \draw (A) circle [radius=1];
      \fill (B) circle [radius=1mm];
      \fill (C) circle [radius=1mm];
      \fill (D) circle [radius=1mm];
      \fill (E) circle [radius=1mm];
      \node at (0,-1.4) {$(2,2,1,1)$};
    \end{scope}
    \begin{scope}[xshift=9cm]
      \coordinate (A) at (0,0);
      \coordinate (B) at (30:1);
      \coordinate (C) at (150:1);
      \coordinate (D) at (-150:1);
      \coordinate (E) at (-30:1);
      \draw (C) -- (B) -- (D) -- (E);
      \draw (A) circle [radius=1];
      \fill (B) circle [radius=1mm];
      \fill (C) circle [radius=1mm];
      \fill (D) circle [radius=1mm];
      \fill (E) circle [radius=1mm];
      \node at (0,-1.4) {$(2,2,1,1)$};
    \end{scope}
    \begin{scope}[xshift=12cm]
      \coordinate (A) at (0,0);
      \coordinate (B) at (30:1);
      \coordinate (C) at (150:1);
      \coordinate (D) at (-150:1);
      \coordinate (E) at (-30:1);
      \draw (B) -- (C) -- (E) -- (D);
      \draw (A) circle [radius=1];
      \fill (B) circle [radius=1mm];
      \fill (C) circle [radius=1mm];
      \fill (D) circle [radius=1mm];
      \fill (E) circle [radius=1mm];
      \node at (0,-1.4) {$(2,2,1,1)$};
    \end{scope}
    \begin{scope}[xshift=0cm, yshift=-3cm]
      \coordinate (A) at (-1,0);
      \coordinate (B) at (-.3,0);
      \coordinate (C) at (.3,.3);
      \coordinate (D) at (.3,-.3);
      \coordinate (E) at (1,0);
      \draw (A) -- (B) -- (C) -- (E) -- (D) -- (B);
      \draw (C) -- (D);
      \draw (0,0) circle [radius=1];
      \fill (A) circle [radius=1mm];
      \fill (B) circle [radius=1mm];
      \fill (C) circle [radius=1mm];
      \fill (D) circle [radius=1mm];
      \fill (E) circle [radius=1mm];
      \node at (0,-1.4) {$(2,1,1,1,1)$};
    \end{scope}
    \begin{scope}[xshift=3cm, yshift=-3cm]
      \coordinate (A) at (0,0);
      \coordinate (B) at (0:1);
      \coordinate (C) at (90:1);
      \coordinate (D) at (180:1);
      \coordinate (E) at (270:1);
      \draw (B) -- (A) -- (C);
      \draw (D) -- (A) -- (E);
      \draw (A) circle [radius=1];
      \fill (A) circle [radius=1mm];
      \fill (B) circle [radius=1mm];
      \fill (C) circle [radius=1mm];
      \fill (D) circle [radius=1mm];
      \fill (E) circle [radius=1mm];
      \node at (0,-1.4) {$(2,1,1,1,1)$};
    \end{scope}
    \begin{scope}[xshift=6cm, yshift=-3cm]
      \coordinate (A) at (0,0);
      \coordinate (B1) at (0:.5);
      \coordinate (C1) at (120:.5);
      \coordinate (D1) at (240:.5);
      \coordinate (B2) at (0:1);
      \coordinate (C2) at (120:1);
      \coordinate (D2) at (240:1);
      \draw (A) circle [radius=.5];
      \draw (A) circle [radius=1];
      \draw (B1) -- (B2);
      \draw (C1) -- (C2);
      \draw (D1) -- (D2);
      \fill (B1) circle [radius=1mm];
      \fill (C1) circle [radius=1mm];
      \fill (D1) circle [radius=1mm];
      \fill (B2) circle [radius=1mm];
      \fill (C2) circle [radius=1mm];
      \fill (D2) circle [radius=1mm];
      \node at (0,-1.4) {$(1,1,1,1,1,1)$};      
    \end{scope}
  \end{tikzpicture}
  \caption{Non-anti-polynomial unobstructed reduced topological Tischler graphs
    for mapping degree 4, labeled by the corresponding branch data.}
  \label{fig:tischler-graphs-deg-4}
\end{figure}

\subsection{Maps with symmetries}
\label{subsec:maps-with-symmetries}

If $\Theta$ is a finite group of spherical isometries, then every
element of $\Theta$ commutes with the \emph{antipodal map}
$\tau(z) = -1/\bar{z}$. Consequently, a rational map $f$ is
$\Theta$-symmetric (see \Cref{subsec:symmetries}) if and only if the
anti-rational map $g=\tau \circ f$ is $\Theta$-symmetric. This gives
an explicit one-to-one correspondence between $\Theta$-symmetric
rational and anti-rational maps, under which critically fixed
anti-rational maps correspond to ``critically antipodal'' rational
maps, i.e., rational maps that map each critical point $c$ to its
antipode $-1/\bar{c}$. Some of the references given in the following
concern only rational maps, but the respective results and formulas in
these references directly imply the corresponding results and formulas
for anti-rational maps.

We have already encountered the tetrahedral map \eqref{eq:tetr_map} as
the unique critically fixed anti-rational map with branch data
$(1,1,1,1)$. In a slightly different normalization, this map arises
without the need to use Thurston's theorem in the following way:
Consider a regular spherical tetrahedron $T$ in $\CC$, and in each
face $A$ of $T$, define $f$ to be the anti-conformal map from $A$ to
its complement $\CC \setminus \overline{A}$, fixing the vertices of
$T$ on $\partial A$. By symmetry, the extensions of these
anti-conformal maps to $\partial A$ fit together along the edges of
$T$, resulting in a critically fixed anti-rational map $f: \CC\to \CC$
of degree 3 with 4 simple critical points and tetrahedral
symmetry. The analogous construction can be carried out for all
Platonic solids. In \cite{doyleSolvingQuinticIteration1989}, the
corresponding ``critically antipodal'' rational maps were calculated
for the dodecahedron and icosahedron, as well as all rational maps of
degree $d \le 30$ with icosahedral symmetry. (The authors remark that
there is good evidence that these maps were already known to Felix
Klein.)  It turns out that all of these maps are postcritically
finite, and indeed are post-compositions of critically fixed
anti-rational maps with the antipodal map $\tau$. See
\cite[Fig.~1.2]{lodgeCirclePackingsKissing2022} for explicit formulas
of the critically fixed anti-rational maps corresponding to all five
Platonic solids, as well as pictures of their Julia sets.

In degree $d=31$, there is a non-trivial one-parameter family of
rational maps with icosahedral symmetry, most of which are not
postcritically finite. In \cite{crassDynamicsSoccerBall2014}, Crass
explicitly calculated two critically fixed anti-rational maps of
degree 31 with icosahedral symmetry, and conjectured that these are
the only such maps with simple critical points. The associated reduced
Tischler graphs are the spherical versions $T_1$ and $T_2$ of a truncated icosahedron (a ``soccer ball'') and a truncated dodecahedron, respectively. (Here and in the following, we identify a polyhedron
  with its 1-skeleton as a plane graph.) It is well
known that the only polyhedra with 32 faces and icosahedral symmetry
are those two and the icosidodecahedron, which has 20 triangular
faces and 12 pentagonal faces. We denote by $T_3$ the spherical version of the latter polyhedron. Note that, while $T_1$ and $T_2$ have 60
vertices, each of degree 3, the icosidodecahedron has only 30
vertices, each of degree 4. The following theorem gives a positive
answer to Crass's conjecture.

\begin{theorem}\label{thm:ico-31}
  Up to Möbius conjugacy, there are exactly three critically fixed
  anti-rational maps $f_k$, $k=1,2,3$, of degree $d=31$ with
  icosahedral symmetry, whose associated reduced Tischler graphs are
  given by the spherical polyhedra $T_k$. The maps
  $f_1$ and $f_2$ each have 60 simple critical points, and the map
  $f_3$ has 30 double critical points.
\end{theorem}

We note that Crass verified his conjecture numerically in
\cite{crassCriticallyFiniteDynamicsIcosahedron2020}, where he
continued his investigation to find numerically all rational maps of
degree 31 with icosahedral symmetry and with periodic critical points.

In spirit, \Cref{thm:ico-31} follows from the classification of
polyhedra applied to the corresponding reduced Tischler graphs, combined with
our classification of critically fixed anti-rational maps in
\Cref{thm:tischler-fixed-correspondence} and
\Cref{thm:symmetries}. However, since we cannot assume reduced
Tischler graphs to be polyhedra, we need to provide some more details.

\begin{proof}[Proof of \Cref{thm:ico-31}] Let $\Ico\subset\Mob(\CC)$
  be the group of orientation-preserving symmetries of a
  spherical icosahedron in the Riemann sphere $\hat{\mathbb
    C}$. We recall the following properties of $\Ico$: 
    \begin{enumerate}
        \item\label{item:ico1} $\Ico$ is isomorphic to the alternating group $\mathcal{A}_5$; in particular, $\#\Ico=60$. 
        \item\label{item:ico2} Every non-identity element $\theta\in\Ico$ is a rotation with exactly two fixed points.
        \item\label{item:ico3} Every non-cyclic subgroup of $\Ico$ contains a $180^\circ$-rotation interchanging the two fixed points of another non-trivial element of the subgroup. 
        \item\label{item:ico4} The points of $\CC$ with non-trivial stabilizer
  form exactly three orbits: the set $P_5$ of $12$ points with
  stabilizer of order $5$, the set $P_3$ of $20$ points with
  stabilizer of order $3$, and the set $P_2$ of $30$ points with
  stabilizer of order $2$.  All other orbits have $60$ points.
    \end{enumerate}

  Now assume that $f$ is an $\Ico$-symmetric critically fixed anti-rational map of degree $d=31$. Then the reduced Tischler graph $T$ of $f$ has
  32 faces and satisfies $\theta(T) = T$ for every $\theta \in \Ico$.
  In order to establish the claim in the theorem, we have to show that
  $T$ is isomorphic to one of the spherical polyhedra $T_1$, $T_2$, or $T_3$ (see \Cref{thm:symmetries}).

   Let $A$ be a face of $T$ and $\theta\in \Ico$ be a non-identity element with
  $\theta(A)=A$. It easily follows from \ref{item:ico2} and Brouwer's fixed point theorem that the rotation 
  $\theta$ has exactly one fixed point $c_A\in A$ (and the other one in $\CC\setminus \overline{A}$). 
  We then conclude from 
  \ref{item:ico3} that the stabilizer $\Ico_A$ of $A$ is cyclic and consists of rotations about a common axis through the point $c_A$. By 
  \ref{item:ico4}, we therefore have  $k_A:=\#\Ico_A\in \{1,2,3,5\}$ and $c_A\in P_{k_A}$ if $k_A>1$. 

  Since $\#F(T)=32$ and the orbit of each face $A$ consists of $60/k_A$ faces (containing the distinct points of the orbit of 
  $c_A$ when $k_A>1$), \ref{item:ico1} and \ref{item:ico4} imply that $P_3\cup P_5\subset \CC\setminus T$ while $P_2\subset T$. In particular, the orbit of each vertex is either free (of cardinality $60$) or equals $P_2$ (of cardinality $30$). At the same time, since $V(T)=C_f$, the Riemann--Hurwitz formula implies that $\#V(T)\leq 2d-2=60$. It follows that all vertices of $T$ belong to a single orbit, i.e., $\Ico$ acts transitively on $V(T)$. 
  
The desired statement now follows from the classification of \emph{isogonal} (i.e., vertex-transitive) tilings of the sphere with icosahedral symmetry in \cite{GrunbaumTilings1981}: among these, exactly
the spherical truncated icosahedron $T_1$, the spherical truncated dodecahedron $T_2$, and the
spherical icosidodecahedron $T_3$ have $32$ faces. Conversely, $T_1$, $T_2$, and $T_3$ are pairwise non-isomorphic unobstructed reduced topological Tischler graphs, so by \Cref{thm:schottky-obstructions,thm:tischler-fixed-correspondence} they correspond to three pairwise non-conjugate critically fixed anti-rational maps $f_1$, $f_2$, and $f_3$, unique up to Möbius conjugacy. By \Cref{thm:symmetries}, each of these maps $f_k$ inherits the icosahedral symmetry of $T_k$.
\end{proof}

\bibliographystyle{amsalpha}
\bibliography{classification-zotero,classification-additional}

\end{document}